\documentclass[12pt]{article}
\usepackage{amsmath}
\usepackage{amssymb}
\usepackage{amsthm}
\title{Kolmogorov's Theorem for Low-Dimensional Invariant Tori of Hamiltonian Systems}
\author{Plotnikov P.I.\footnote{plotnikov@hydro.nsc.ru}, Kuznetsov I.V.}

\date{}
\newtheorem{theorem}{Theorem}[section]
\newtheorem{lemma}[theorem]{Lemma}
\newtheorem{proposition}[theorem]{Proposition}
\newtheorem{definition}[theorem]{Definition}

\numberwithin{equation}{section}
\newtheorem{corollary}[theorem]{Corollary}

\begin{document}
\maketitle

 \noindent Lavrentyev Institute of Hydrodynamics
Siberian Branch of RAS
\\
Address: Lavrentyev pr.15 Novosibirsk  630090,  Russia,
\\
Novosibirsk State University
\\
Address:  Pirogova st. 2, Novosibirsk 630090,  Russia

\begin{center}
Abstract
\end{center}

 In this paper the problem of persistence of
invariant tori under small perturbations of integrable Hamiltonian
systems is considered. The existence of one-to-one correspondence
between hyperbolic invariant tori and critical points of the
function $\Psi$ of two variables defined on semi-cylinder is
established. It is proved that if unperturbed Hamiltonian has a
saddle point, then under arbitrary perturbations there persists at
least one hyperbolic torus.



 \tableofcontents

\section{Notation}

First, we introduce some notation which is used troughout of the paper.
Every  vector $\mathbf v$  in $m$-dimensional Euclidean space is
assumed
 to be a column vector. Row  vectors are denoted
by $\mathbf{v}^\top$. Hence $\mathbf v=(v_1,\dots,v_m)^\top$. For a
 matrix $\mathbf{A}=(A_{ij})$, $i,j=1,\dots,n$,
with entries $A_{ij}$, $i$ is the row index and $j$ the column
index, and $\mathbf{A}^\top=(A_{ji})$ stands for the transposed
matrix. The product of a matrix $\mathbf A$ by a column vector
$\mathbf v$ is denoted by $\mathbf{A} \mathbf{v}$. It is the  column
vector with the components $(\mathbf A\mathbf v)_i=A_{ij}v_j$. The
product of a row vector $\mathbf v^\top$ and a matrix $\mathbf A$ is
the row vector with the components $(\mathbf v^\top\mathbf A)_j=
v_iA_{ij}$. We thus get
$\mathbf{v}^\top\mathbf{A}=(\mathbf{A}^\top\mathbf{v})^\top$. The
product of two matrices $\mathbf{C}=\mathbf{A}\mathbf{B}$ is a
 matrix with the entries $C_{ij}=A_{ik}B_{kj}$ with the summation
 convention over the repeated indices.
The scalar product
 of two vectors is $\mathbf{v}^\top\cdot \mathbf{u}=v_i u_i$.
By abuse of notation, we also denote by $\mathbf u\cdot \mathbf v=u_iv_i$ the scalar product of two column vectors.

The tensor product
of two vectors $\mathbf u, \mathbf v\in \mathbb R^n$ is the matrix
$A:=\mathbf u\otimes\mathbf v$ with the entries $A_{ij}:=u_iv_j$,
$i,j=1,\dots, n$. For the product of this matrix with a vector we
have
\[
(\mathbf u\otimes\mathbf v)\mathbf w= (\mathbf v^\top\cdot\mathbf
w)\mathbf u\quad\text{and}\quad \mathbf w^\top (\mathbf
u\otimes\mathbf v)=(\mathbf w^\top\cdot\mathbf u)\mathbf v.
\]
 There is a difference between
the Jacobian $\mathbf v'(x)$ of a vector function
 $\mathbf v:\mathbb R^n\to\mathbb R^m$ and its gradient: the Jacobian
\index{Jacobian}%
is denoted by
$$\mathbf{v}':=(\partial_{x_j} v_i)=
\left[\displaystyle{\dfrac{\partial\mathbf{v}}{\partial x_1}},
 \displaystyle{\dfrac{\partial\mathbf{v}}{\partial x_2}},\dots,
  \displaystyle{\dfrac{\partial\mathbf{v}}{\partial x_n}}\right],$$
and the gradient
\index{gradient}%
 is its transpose
$$\nabla \mathbf{v}={\mathbf{v}'}^\top=(\partial_{x_i} v_j)=
[\nabla v_1, \nabla v_2, \nabla v_3]
$$
 If a function $\mathbf v: \mathbb R^{n-k}\times\mathbb R^k$
 depends on two variables $y\in \mathbb R^{n-k}$ and $z\in \mathbb R^k$, then
$$
\mathbf v_y'(y,z)=(\partial_{y_j}v_i)_{ij}=
\left[\displaystyle{\dfrac{\partial\mathbf{v}}{\partial y_1}},
\displaystyle{\dfrac{\partial\mathbf{v}}{\partial y_2}},\dots,
\displaystyle{\dfrac{\partial\mathbf{v}}{\partial y_{n-k}}}\right],
$$
and $\nabla_y\mathbf v=(\mathbf v_y')^\top$. For the
 derivatives of a scalar function $H$ we will use the notation
\begin{equation}\label{notation1}
\frac{\partial H}{\partial y}(y,z)=H_y'(y,z) \text{~and~}
\frac{\partial^2 H}{\partial z\partial y}(y,z) =\nabla_z(H_y')(y,z).
\end{equation}
In other words, $\partial^2 H/\partial z\partial y$ is the matrix
with the entries $\partial^2 H/\partial z_i\partial y_j$ and
$$
\frac{\partial^2 H}{\partial z\partial y}= \Big(\frac{\partial^2
H}{\partial y\partial z}\Big)^\top.
$$
Let $\varrho>0$. Denote by $\Sigma_\varrho$ the complex
neighborhood
\begin{equation}\label{notation2}\Re_\varrho=\{{\boldsymbol\xi} \in \mathbb
T^{n-1}+i\mathbb{R}^{n-1}:\, \mbox{Re}\, \xi\in \mathbb T^{n-1},\, |\mbox{Im}\,
{\xi} _j|\le \varrho,~~1\le j\le n-1\}
\end{equation}
of the (n-1)-dimensional torus $\mathbb T^{n-1}$. We will identify $
\Re_\varrho$ with the strip
$$
\{{\boldsymbol\xi} :\, \text{Re}\, \xi\in \mathbb R^{n-1},\,
|\text{Im}\, {\xi} _j|\le \varrho,~~1\le j\le n-1\}\subset \mathbb
C^{n-1}.
$$
We denote by  $B^n_\varrho$ the
complex neighborhood
\begin{equation}\label{notation3}
B^n_\varrho=\{\boldsymbol p \in \mathbb{C}^n: \, |\text{\rm Re}
\boldsymbol p\,|\leq \varrho\in D,\,\, |\mbox{Im}\, {p} _j|\le
\varrho,~~1\leq j\leq n\}\subset \mathbb C^n.
\end{equation}.

\section{Problem formulation. Results}
\renewcommand{\theequation}{\arabic{section}.\arabic{equation}}
\setcounter{equation}{0}
\paragraph{Invariant tori.}

In this paper we consider the problem of  persistence of
quasi-periodic motions spanning lower dimensional tori in a
nearly-integrable  Hamiltonian system
\begin{equation}\label{ee1.01bis}
\dot{\boldsymbol p}=-\partial_{\boldsymbol q}\, H(\boldsymbol
p,\boldsymbol q),\quad \dot{\boldsymbol q}=\partial_{\boldsymbol
p}\, H(\boldsymbol p,\boldsymbol q)
\end{equation}
with the Hamiltonian
\begin{equation}
H(\boldsymbol{p},\boldsymbol{q})=H_0({\boldsymbol{p}})+\varepsilon
H_1(\boldsymbol{p},\boldsymbol{q}). \label{ee1.01}
\end{equation}
Here the functions  $H_0$ and $H_1$ are analytic on the set $D\times
\mathbb{T}^n$, where $D\subset \mathbb R^n$ is a neighborhood of
zero and $\mathbb{T}^n=\mathbb{R}^n/2\pi\mathbb{Z}^n$ is
$n$-dimensional torus. In particular,  $H_1$ is an analytic
$2\pi$-periodic function of the angle variable  $\boldsymbol
q=(q_1,\dots, q_n)$. For $\varepsilon =0$ the system is integrable
and the  phase space is foliated by invariant  $n$-dimensional
invariant tori $\{\mathbf q=\overline{\boldsymbol \omega}\,
t+\text{const.}, t\in \mathbb R\}$ with the frequency vectors
$\overline{\boldsymbol\omega}=\nabla H_0(\boldsymbol{p})$,
$\boldsymbol{p}\in D$. If  all components of
$\overline{\boldsymbol\omega}$ are rationally independent, then KAM
theory, see \cite{Kolmogorov} and  \cite{Arnol}, shows that, under
suitable non-degeneracy assumptions, such invariant tori persist
under small analytic perturbations.

An invariant $n$-dimensional invariant torus of the unperturbed
system is said to be resonance if the number of rationally
independent components of $\overline{\boldsymbol\omega}$ is $m<n$. A
resonance invariant torus  of the unperturbed system is foliated by
invariant $m$-dimensional tori. A resonance torus breaks-up under
small perturbation and only  a few of its constituent invariant
$m$-dimensional tori survive a perturbation.

The mathematical study of low-dimensional tori in Hamiltonian systems dates back to the late 1960s. The intensive treatment of the problem starts with pioneering papers
\cite{Melnikov} and \cite{Moser}. We refer reader to the monographs
\cite{Arnold-Kozlov-Neisch}, \cite{Broer}, \cite{Kuksin},  \cite{Nirenberg}, \cite{Poschel}
and the papers \cite{Eliasson}, \cite{Graff}, \cite{Treschev}, and \cite{Zender} for the state of art in the domain.

In this paper we deal with
the problem of persistence of (n-1)- dimensional tori in setting
close to the original Kolmogorov theorem.
Our goal is to find
conditions on the unperturbed Hamiltonian $H_0$ under which system
\eqref{ee1.01bis} has an invariant torus of dimension $n-1$ for
\emph{every analytic perturbation and all sufficiently small
$\varepsilon$}. This problem is still poor investigated. We refer to
papers \cite{Cheng96} and \cite{Corsi} for  results and discussion.

We focus on the problem on persistence of hyperbolic invariant tori
for  arbitrary analytic perturbation of a Hamiltonian function.
Throughout of the paper we assume  that the frequency vector
$\overline{\boldsymbol\omega}=\nabla_p H(0,0)$ has $n-1$ rationally
independent components and the Hessian $H_0''(0,0)$ is
nondegenerate. Then there exists, see \cite{Treschev}, an affine
symplectic transformation with rational coefficients such that in
new variables
$$
\boldsymbol{p}=(\boldsymbol{y},z_2)\in\mathbb{R}^{n-1}\times\mathbb{R},
\quad\boldsymbol{q}=(\boldsymbol{x},z_1)\in\mathbb{T}^{n-1}\times
\mathbb{T}, \quad \boldsymbol{z}=(z_1,z_2),
$$
the Hamiltonian $H(\boldsymbol{x}, \boldsymbol{y},
\boldsymbol{z})=H_0(\boldsymbol y, z_1)+H_1(\boldsymbol{x},
\boldsymbol{y}, \boldsymbol{z})$ satisfies the following condition.
\begin{itemize}
\item[{\bf H.1}] Recall \eqref{notation2} and \eqref{notation3}.
There is $\varrho>0$ with the properties. The function
$H_0(\boldsymbol y, z_2)$ is analytic in the complex ball $(\mathbf
y, z_2)\in B^n_{3\varrho}$; the perturbation  $H_1(\boldsymbol{x}, \boldsymbol{y},
\boldsymbol{z})$ is analytic on the cartesian product of the complex
strip $(\mathbf x, z_1)\in \Re_{3\varrho}$
 and the ball $(\mathbf y, z_2)\in
B^n_{3\varrho}$. In particular, $H_1$ is $2\pi$-periodic in
$\boldsymbol x$ and $z_1$. Moreover, there is $c>0$ such that
\begin{equation}
\sup\limits_{(\boldsymbol{y},z_2)\in B^n_{3\rho}}|
H_0(\boldsymbol{y},z_2)|+\sup\limits_{(\boldsymbol{y},z_2)\in
B^n_{3\rho},~(\boldsymbol{x},z_1)\in \Re_{3\rho }}|
H_1(\boldsymbol{x},\boldsymbol{y},\boldsymbol{z})|\le
c,\label{ee1.07}
\end{equation}

\item[{\bf H.2}] The frequency vector
$\overline{\boldsymbol\omega}=(\boldsymbol\omega, 0)$ and
Hamiltonian $H_0$ satisfy
\begin{equation}
\nabla_{\boldsymbol{y}} H_0(0,0)=\boldsymbol\omega\in
\mathbb{R}^{n-1},\quad \partial_{z_2}H_{0}(0,0)=0. \label{ee1.03}
\end{equation}
The components of $\boldsymbol\omega$ are rationally independent.
 and satisfy the diophantine condition
\begin{equation}\label{diophantine}
    \big|\,(\boldsymbol \omega^\top \cdot \boldsymbol s)^{-1}\,\big|\leq
    c_0 |\boldsymbol s|^{-n}\text{~~for all ~~}\boldsymbol s\in
    \mathbb Z^{n-1}\setminus \{0\}.
\end{equation}
\end{itemize}
In the  variables $(\mathbf x,\mathbf y,\mathbf z)$ system \eqref{ee1.01bis} reads
\begin{equation}
\dot{\boldsymbol{x}}= \nabla_{\boldsymbol{y}}H,\quad
\dot{\boldsymbol{y}}=- \nabla_{\boldsymbol{x}}H, \quad
\dot{\boldsymbol{z}}=\mathbf{J}\nabla_{\boldsymbol{z}}H,\text{~~where~~}
\mathbf{J}=\left(
\begin{array}{cc}
0&1\\
-1&0
\end{array}
\right).\label{ee1.02}
\end{equation}

\medskip
\begin{definition}
\label{def1.1} Hamiltonian system \eqref{ee1.02} has an analytic
$(n-1)$-dimensional invariant torus with the frequency vector
$\boldsymbol\omega$ if there exists an analytic canonical
transform
$\boldsymbol\vartheta:(\boldsymbol\xi,\boldsymbol\eta,\boldsymbol\zeta)
\mapsto (\boldsymbol{x},\boldsymbol{y},\boldsymbol{z})$,
\begin{equation}\begin{split}\label{canonical0}
\boldsymbol{x}=\boldsymbol\xi+\boldsymbol{u}(\boldsymbol\xi),\quad
\boldsymbol{y}=\boldsymbol{v}(\boldsymbol\xi)+
O(|\boldsymbol\eta|,|\boldsymbol\zeta|),\quad
\boldsymbol{z}=\boldsymbol{w}(\boldsymbol\xi)+O(|\boldsymbol\zeta|),\\
 \boldsymbol\xi\in\mathbb{T}^{n-1},\quad
\boldsymbol\eta\in\mathbb{R}^{n-1},\quad
\boldsymbol\zeta\in\mathbb{R}^2,
\end{split}\end{equation}
such that $\boldsymbol\vartheta$   puts $H$ into the normal form, i.e.,
\begin{equation}\label{normalform0}
N(\boldsymbol\xi,\boldsymbol\eta,\boldsymbol\zeta)\equiv
H\circ\boldsymbol\vartheta=\boldsymbol\omega\cdot{\boldsymbol\eta}+
\frac12\boldsymbol\Omega\, \boldsymbol\zeta\cdot\boldsymbol\zeta+
o(|{\boldsymbol\eta} |,|{\boldsymbol\zeta} |^2).
\end{equation}
Here $\boldsymbol \Omega$ is a constant symmetric matrix. Without
loss of generality we may assume that
\begin{equation}\label{zabyli}
\boldsymbol\Omega={\rm diag}(-k,1),\quad k\in\mathbb R,
\end{equation}
and
$$
\overline{\mathbf {u}}:={(2\pi)^{1-n}}\int_{{\mathbb
T}^{n-1}}\boldsymbol{u}({\boldsymbol\xi})\,d{\boldsymbol\xi}=0,\quad \overline{\boldsymbol w}=(\alpha,0),\quad \alpha\in
\mathbb R.
$$
The invariant torus is weakly-hyperbolic if $k\geq 0$. In the normal
coordinates  system \eqref{ee1.02}   reads
$$ \dot {\boldsymbol\xi}
=\boldsymbol\omega  +O(|{\boldsymbol\eta} |,|{\boldsymbol\zeta}
|),\quad \dot {\boldsymbol\eta} = O(|{\boldsymbol\eta}
|,|{\boldsymbol\zeta} |),\quad\dot {\boldsymbol\zeta}
=\mathbf{J}\boldsymbol\Omega {\boldsymbol\zeta}
+O(|{\boldsymbol\eta} |,|{\boldsymbol\zeta} |^2).
$$
It  has  the  solution ${\boldsymbol\xi} =\boldsymbol\omega t+const$,
${\boldsymbol\eta}=0$, ${\boldsymbol\zeta}=0$, which defines an
invariant torus.
\end{definition}
\medskip
The invariant tori which meet all requirements of Definition \ref{def1.1}
are known as reducible tori. It was proved in \cite{Bol-Tresh} that every analytic hyperbolic invariant torus is reducible.

If  $\varepsilon=0$,  then the canonical transformation
 $\boldsymbol\vartheta$ is trivial
 $$
\boldsymbol{x}=\boldsymbol\xi,\quad \boldsymbol{y}=
\boldsymbol\eta,\quad \boldsymbol{z}=(\alpha,0)+\boldsymbol{\zeta},
$$
 and the normal form is given by
\begin{multline}
N(\boldsymbol\eta, \zeta_2)\equiv
H_0(\boldsymbol\eta,\zeta_2)\\=\boldsymbol\omega\cdot\boldsymbol\eta+
\frac12\zeta_2^2+\boldsymbol{t}_0\cdot\boldsymbol\eta\zeta_2+
\frac12\mathbf{S}_0\boldsymbol\eta\cdot\boldsymbol\eta+
o(|{\boldsymbol\eta} |^2,|{\boldsymbol\eta} ||\zeta_2 |,|\zeta_2
|^2). \label{ee1.04}
\end{multline}
Here  $\boldsymbol t_0\in \mathbb R^{n-1}$ and the Hessian $\mathbf
S_0$ are defined by
\begin{equation}\label{sandt}
\boldsymbol{t}_0=\nabla_y\partial_{z_2}H(0,0),\quad
\mathbf{S}_0=\partial_y^2H_{0}(0,0,0).
\end{equation}
If $\varepsilon =0$, then
$\mathbf \Omega=\text{~diag~}\{0,1\}$ is degenerate.  In other
words, the problem of persistence of lower dimensional tori is
degenerate in the original Kolmogorov- Melnikov formulation.
Introduce the constant matrix
$$\mathbf{K}_0\,=\,\mathbf{S}_0-\boldsymbol{t}_0\otimes\boldsymbol{t}_{0},$$
where the vector $\boldsymbol t_0$ and the matrix $\mathbf S_0$ are
defined
 by equalities \eqref{sandt}.
\begin{theorem}
\label{theo1.4} Let Conditions $(\mathbf H.\boldsymbol{1})$ - $(\mathbf H.\boldsymbol{1})$   be
satisfied.
 Furthermore,  assume that
\begin{equation}
\mathbf{K}_0\boldsymbol\eta\cdot\boldsymbol\eta<0\quad\text{~~for
all~~}\boldsymbol\eta\in \mathbb
R^{n-1}\setminus\{0\}.\label{ee1.06}
\end{equation}
Then there is $\varepsilon_0>0$ such that for all $\varepsilon\in
(-\varepsilon_0,\varepsilon_0)$, Hamiltonian system \eqref{ee1.02}
has an invariant $(n-1)$-dimensional weakly hyperbolic  torus which meets all
requirements of Definition \ref{def1.1}.
\end{theorem}

Now we can characterize the contents of the paper.
In this paper we apply to low-dimensional tori problem  a modified version
of the variational principle proposed in \cite{Percival1979} and developed in \cite{Plotnikov}.
In Section \ref{anna} we introduce   the special group of canonical transforms
which put the Hamiltonina $H$ in the normal form. We investigate in many details the structure of this group and its tangent space.

In the  Section \ref{luna} we employ the version
of the Lyapunov-Schmidt method developed in \cite{Moser} in order to reduce the problem to the finite system of functional equations named bifurcation equations. To this end
 we add the modification term $m z_1+M z_1^2/2$ to the
original  Hamiltonian. We obtain  the  modified Hamiltonian $H_m$,
\begin{equation*}
    H_m(\mathbf x, \mathbf y, \mathbf z)=H_0(\mathbf  y, z_2)+\varepsilon
    H_1(\mathbf x, \mathbf y, \mathbf z)+m z_1
+\frac12{M}z_1^2
\end{equation*}
Thus we come to the following modified problem.
For given $\alpha\in \mathbb
R^1$ and $k\in[0,1]$ to find  parameters $m,M$, and  a canonical
mapping
\begin{gather*}
\boldsymbol\theta:(\boldsymbol\xi,
\boldsymbol \eta, \boldsymbol
 \zeta)\to (\mathbf x, \mathbf y, \mathbf z),\\
\nonumber \boldsymbol x=\boldsymbol\xi+\mathbf
u(\boldsymbol\xi),\quad \mathbf y=\mathbf v+\mathbf
V(\boldsymbol\xi)\boldsymbol\eta+ \boldsymbol\Lambda(\boldsymbol\xi)
\boldsymbol\zeta+\frac{1}{2}\boldsymbol\zeta^\top \mathbf
R(\boldsymbol\xi)\boldsymbol\zeta,\\ \mathbf z=\mathbf
w(\boldsymbol\xi)+\mathbf W(\boldsymbol\xi)\boldsymbol\zeta.
\end{gather*}
 which  puts
 the modified hamiltonian in the normal form ,i.e.,
\begin{equation}\label{luna7}
 H_m\circ \boldsymbol\theta= \boldsymbol\omega^\top \cdot
 \boldsymbol\eta+\frac{1}{2}\boldsymbol\zeta^\top \boldsymbol\Omega
  \boldsymbol\zeta+ o(|\boldsymbol \eta|, |\boldsymbol \eta||\boldsymbol\zeta|,
 |\boldsymbol \zeta|^2),
 \end{equation}
In this framework $\alpha$ and $k$ are given, while
$m$, $M$, and $e$  are unknown and should be defined along with a
solution.  The solvability of
the modified problem  can be established  by using the Nash-Moser
implicit function theorem. The obtained  solution
to  the modified   problem is
 a function of $\alpha$ and $k$. In particular, we have $m=m(\alpha,k)$
 and $M=M(\alpha,k)$. Obviously, if $m=M=0$, then  the canonical mapping
  $\boldsymbol\theta$ puts the original  hamiltonian  $H(\mathbf x, \mathbf y, \mathbf z)$
  in the normal form.   Thus  we  reduce
  the original  problem to the system of two scalar equations
  \begin{equation*}
    m(\alpha,k)=M(\alpha,k)=0
  \end{equation*}
  named the bifurcation equations.

In Sections  \ref{lita}-\ref{ira} we prove that the modified problem has an analytic solution for all sufficiently small $\varepsilon$. In Sections \ref{kira}-\ref{zara} we investigate in details the dependence of solutions to the modified problem on the parameters $\alpha$ and $k$. We define the Jacobi vector fields associated with the derivatives of these solutions with respect to $\alpha$ and $k$.
We also investigate the properties of the quadratic form of the differential of the action functional.
Finally, in Section \ref{sveta} we prove the existence of the critical point of the action functional
and complete the proof of Theorem \ref{theo1.4}.

\section{Basic group of  of canonical transformations}\label{anna}
\medskip
\setcounter{equation}{0} \indent Recall that our task is to find  the canonical transformation
\eqref{canonical0}
 which puts $H$ into the
normal form \eqref{normalform0}.  An essential tool in our approach
is a special group of canonical transformations. In this section we
define such a group as a manifold in the space of analytic mappings
and investigate a structure of this manifold.

\paragraph{Canonical transformations} Let us consider the totality  of all analytic mappings
$\boldsymbol\vartheta: \mathbb T^{n-1}\times \mathbb
R^{n-1}\times\mathbb R^2\to \mathbb T^{n-1}\times \mathbb
R^{n-1}\times\mathbb R^2$ of the form
\begin{gather}\nonumber
\boldsymbol \vartheta : (\boldsymbol \xi,
\boldsymbol\eta,\boldsymbol\zeta) \to (\boldsymbol x, \boldsymbol
y,\boldsymbol z),\\\nonumber \boldsymbol x=\boldsymbol\xi+\mathbf
u(\boldsymbol\xi),\quad \mathbf y=\mathbf v+\mathbf
V(\boldsymbol\xi)\boldsymbol\eta+\boldsymbol\Lambda(\boldsymbol\xi)
\boldsymbol\zeta+\frac{1}{2}\boldsymbol\zeta^\top \mathbf
R(\boldsymbol\xi)\boldsymbol\zeta,
\\\label{anna1}
\mathbf z=\mathbf w(\boldsymbol\xi)+\mathbf
W(\boldsymbol\xi)\boldsymbol\zeta.
\end{gather}
 Here
$(n-1)\times (n-1)$- matrix  valued function $\mathbf V$,
$(n-1)\times 2$- matrix valued function $\boldsymbol \Lambda$,
$2\times 2$- matrix values function $\mathbf W$,
  and vector valued functions $\mathbf u, \mathbf v :\mathbb T^{n-1}\to \mathbb R^{n-1}$,
  $\mathbf v:\mathbb T^{n-1}\to \mathbb R^{n-1}$,
$\mathbf w:\mathbb T^{n-1}\to \mathbb R^{2}$,
 are analytic and $2\pi$ periodic in $\boldsymbol\xi$, $\mathbf R$ is a
  vector valued quadratic form given by
\begin{equation}\label{anna2}
 \boldsymbol\zeta^\top \mathbf R\boldsymbol\zeta=\big(\,
 \boldsymbol\zeta^\top \mathbf R_1\boldsymbol\zeta,\dots,\boldsymbol\zeta^\top
  \mathbf R_{n-1}\boldsymbol\zeta\big)^\top, \quad \mathbf R_i=\mathbf R_i^\top.
\end{equation}
where $2\times 2$ matrix valued functions $\mathbf R_i(\boldsymbol
\xi)$ are analytic and $2\pi$ periodic.

Recall that  $\boldsymbol\vartheta:\mathbb
T^{n-1} \times \mathbb R^{n-1}\times\mathbb R^2\to \mathbb
T^{n-1}\times \mathbb R^{n-1}\times\mathbb R^2$ is a
 canonical if $\boldsymbol \vartheta^{-1}$ takes solutions of  hamiltonian
 system with a hamiltonian $H$ to solutions of the hamiltonian system
with the hamiltonian $H\circ\boldsymbol\vartheta$. The
 mapping $\boldsymbol \vartheta$ is a canonical if and only if its Jacobi
 matrix is symplectic, i.e.,
\begin{equation}
(\boldsymbol\vartheta')^\top\,
\mathbf{J}_{2n}\,\boldsymbol\vartheta'\,=\,\mathbf{J}_{2n},
\label{anna3}
\end{equation}
where
$$\mathbf{J}_{2n}=\left(
\begin{array}{ccc}
0&\mathbf{I}&0\\
-\mathbf{I}_{n-1}&0&0
\\0&0&\mathbf{J}
\end{array}
\right), \quad \mathbf J=\left(
\begin{array}{cc}
0&1\\
-1&0
\end{array}
\right),$$
 $\mathbf I$ is the identity matrix.
 The following First Structure  Theorem gives the complete description
 of the set of canonical mappings $\boldsymbol\vartheta\in \mathcal A$.
\begin{theorem}\label{anna4} {\bf First Structure Theorem.} Let $\boldsymbol \vartheta \in \mathcal A$
is given by \eqref{anna1}. Furthermore assume that the there are
analytic inverses $(\text{Id}+\mathbf u)^{-1}$,
 $\mathbf V^{-1}$, and $\mathbf W^{-1}$. Then $\boldsymbol \vartheta $ is canonical
  if and only if
\begin{subequations}\label{anna5}
\begin{gather}\label{anna6}
\mathbf V=(\mathbf I_{n-1}+\mathbf u')^{-\top},\\\label{anna7}
\text{\rm det~} \mathbf W=1 \Leftrightarrow \mathbf W^\top \,
\mathbf J\, \mathbf W= \mathbf J,\\\label{anna8} \boldsymbol\Lambda=
-\mathbf V\, (\mathbf w')^\top\,\mathbf J\, \mathbf W\equiv -\mathbf
V\nabla\mathbf w\, \mathbf J\,\mathbf W,\\\label{anna9} \mathbf R_i=
-V_{ik}\, \frac{\partial}{\partial \xi_k}\Big(\mathbf W^\top\Big)\,
\mathbf J\,\mathbf W,\\\label{anna10} d(\xi_k+u_k)\wedge dv_k+d
w_1\wedge d w_2=0.
\end{gather}
 Moreover, there exist  $\beta\in \mathbb R^{n-1}$ and  $2\pi$ -periodic scalar
 function $\varphi_0(\boldsymbol\xi)$ such that
\begin{equation}\label{anna11}
    \mathbf v=\beta+\mathbf V\,\big( \,\nabla\varphi_0-w_2\nabla w_1\,\big).
\end{equation}
\end{subequations}

\end{theorem}
\begin{proof}The proof is given in Appendix \ref{lena}.
\end{proof}
 It is clear that the totality of the canonical mappings $\boldsymbol\vartheta$  contains the identity mapping and it is
closed
 with respect to the composition. In other words, it  can be regarded
 as a subgroup of the group of analytical diffeomorphisms.
Every  mapping $\boldsymbol\vartheta$ is completely characterized by the
 vector $\boldsymbol\Theta$ of its coefficients,
\begin{equation}\label{anna14}
    \boldsymbol\Theta=\big(\mathbf u, \mathbf v, \mathbf w, \mathbf V,
    \boldsymbol \Lambda, \mathbf W,
    \mathbf R_i\big).
\end{equation}
 The group structure on the set of the  canonical mappings $\boldsymbol \vartheta$ induces the group structure on the set of the corresponding vectors $\boldsymbol\Theta$. Hence the totality of coefficient vectors $\Theta$
 corresponding to the canonical transforms $\boldsymbol\vartheta$ can be regarded as a nonlinear manifold
 in the linear space of all vector-valued analytic $2\pi$-periodic vector-valued functions $\boldsymbol\Theta$. We denote this manifold by $\mathcal G$.
 Now our task is to  supply $\mathcal G$  with a local chart.

In view of Theorem \ref{anna4} the  vector $\boldsymbol \Theta$ are completely defined by  a constant
 vector $\beta$, a scalar function $\varphi_0$, vector functions $\mathbf u$, $\mathbf w$,
 and by three elements of a symplectic matrix $\mathbf W$. We can consider these
  quantities as local coordinates for the manifold $\mathcal G$.
More precisely, introduce the vector function
\begin{equation}\label{anna15}
    \boldsymbol\varphi=(\beta,\varphi_0, \mathbf u, \mathbf w, W_{11}, W_{12}, W_{21})
\end{equation}
such that the functions $W_{11}, W_{12}, W_{21}, \varphi_0: \mathbb
T^{n-1}\to\mathbb R$,
 the vector functions $\mathbf u: \mathbb T^{n-1}\to \mathbb R^{n-1}$,
  $\mathbf w: \mathbb T^{n-1}\to \mathbb R^2$ are analytic and
$$
\overline{\mathbf u}\equiv (2\pi)^{1-n}\int_{\mathbb T^{n-1}}\mathbf
u\, d\xi=0, \quad \overline{\psi_0}\equiv (2\pi)^{1-n}\int_{\mathbb
T^{n-1}}\psi_0\, d\xi=0,
$$
Next, introduce the mapping $ \boldsymbol \Theta(\boldsymbol\varphi)$
defined by the formulae
\begin{equation}\label{anna16}
    \boldsymbol \Theta:\boldsymbol\varphi\mapsto\boldsymbol\Theta(\boldsymbol\varphi)=
    \big(\mathbf u, \mathbf v, \mathbf w, \mathbf V,\boldsymbol \Lambda, \mathbf W,
    \mathbf R_i\big),
\end{equation}
where
\begin{equation}\label{anna17}\begin{split}
 \mathbf V=(\mathbf I_{n-1}+\mathbf u')^{-\top}, \quad \mathbf v=\beta +
 \mathbf V(\nabla\varphi_0-w_2\nabla w_1),\\
 \mathbf W=\left(
\begin{array}{cc}
W_{11}&W_{12}\\
W_{21}& \frac{1}{W_{11}}(1+W_{12} W_{21})
\end{array}
\right),\\
\boldsymbol\Lambda =-\mathbf V(\mathbf w')^\top\mathbf J\mathbf W,
\quad \mathbf R_i=-V_{ik} \Big (\frac{\partial}{\partial \xi_k}
\mathbf W^{\top}\Big)\, \mathbf J\,\mathbf W.
\end{split}\end{equation}
In view of Theorem \ref{anna4} the mapping
$
    \boldsymbol\varphi\, \mapsto\, \boldsymbol \Theta(\boldsymbol\varphi)$
takes the vectors $\boldsymbol \varphi$ to the elements of the
 manifold $\mathcal G$. Hence, this mapping defines a local chart on $\mathcal G$.
 This chart is local since the mapping \eqref{anna16} develops singularities at
 the points where $W_{11}=0$ and $\text{\rm det~}(\mathbf I_{n-1}+ \mathbf u')=0$.

\paragraph{Tangent space. }
Denote by $\mathcal L$ the linear space of all coordinate vectors $\boldsymbol \varphi$.
Introduce that Gateaux differential
\begin{equation}\label{galaadd1}
    D\boldsymbol \Theta(\boldsymbol\varphi)[\delta\boldsymbol\varphi]
    =\lim\limits_{t\to 0} t^{-1}\big(\,\boldsymbol \Theta(\boldsymbol\varphi+t\delta\boldsymbol\varphi)-\boldsymbol \Theta(\boldsymbol\varphi)\,\big).
\end{equation}
The totality of all vectors
\begin{equation}\label{galaadd2}
    \delta\boldsymbol \Theta= D\boldsymbol \Theta(\boldsymbol\varphi)[\delta\boldsymbol\varphi],
    \quad \delta\boldsymbol\varphi\in \mathcal L,
\end{equation}
can be regarded as the tangential space $\text{Tan}_\Theta\mathcal G$ to the manifold $\mathcal G$ at the point $\boldsymbol\Theta(\boldsymbol\varphi)$. Direct calculations show that for every
\begin{equation}\label{gala1}
\delta\boldsymbol\varphi=(\delta\mathbf \beta, \delta\mathbf
\varphi_0,
 \delta\mathbf u, \delta\mathbf
w, \delta W_{11},\delta W_{12},\delta W_{21}).
\end{equation}
the vector
\begin{equation}\label{gala3}
    \delta\boldsymbol\Theta=\big(\, \delta\mathbf u, \delta\mathbf v,
    \delta\mathbf w, \delta\mathbf V, \delta\mathbf W, \delta\boldsymbol\Lambda,
    \delta\mathbf R_i\,\big).
\end{equation}
is defined by the following formulae
\begin{gather}\nonumber
\delta\mathbf V=-\mathbf V\,\delta \nabla_\xi\delta\mathbf
u\,\mathbf V, \quad \delta W_{22}=\frac{1}{W_{11}}\big(W_{12}\delta
W_{21}+W_{21}\delta W_{12}-
W_{22}\delta W_{11}\big),\\
\label{gala2}\delta \mathbf W=\left(
\begin{array}{cc}
\delta W_{11}&\delta W_{12}\\
\delta W_{21}&\delta W_{22}
\end{array}
\right),\\\nonumber \delta\mathbf v=\delta\beta+\mathbf
V(\nabla\delta\varphi_0-w_2\nabla\delta w_1- \delta w_2 \nabla w_1)+
\delta \mathbf V\,(\nabla \varphi_0-w_2\nabla w_1),\\\nonumber
\delta\boldsymbol\Lambda=\delta \mathbf V \, \nabla w\, \mathbf
J\mathbf W- \mathbf V\, \nabla(\delta\mathbf w)\mathbf J\mathbf
W-\mathbf W\, \nabla\mathbf w\,
 \mathbf J\, \delta\mathbf W,\\\nonumber
\delta\mathbf R_i=-\delta V_{ik}
\frac{\partial}{\partial\xi_k}(\mathbf W^{\top}) \mathbf J\mathbf W-
 V_{ik} \frac{\partial}{\partial\xi_k}(\delta\mathbf W^{\top})\mathbf J\mathbf W-
 V_{ik} \frac{\partial}{\partial\xi_k}(\mathbf W^{\top})\mathbf J\, \delta\mathbf W.
\end{gather}
The right hand sides of \eqref{gala2} are linear differential
operators acting on $\delta\boldsymbol\varphi$. They can be regarded
as the G\^{a}teaux derivatives of the components of
$\boldsymbol\Theta(\boldsymbol\varphi)$ at the point $\boldsymbol\varphi$ in the direction
$\delta\boldsymbol\varphi$.
 These relations can be  simplified in the particular case when
 $$
\boldsymbol\varphi=\boldsymbol\varphi_0(\alpha)\equiv
(0,0,0,\alpha\mathbf e_1,1,0,0).
$$
Notice that for $\boldsymbol\varphi=\boldsymbol\varphi_0$ we have
$$
\mathbf V=\mathbf I_{n-1},\quad \mathbf W=\mathbf I,
\quad\boldsymbol\Lambda=0,
 \quad\mathbf R=0.
$$
We consider this very special case in many details. For
$\boldsymbol\varphi=\boldsymbol\varphi_0$  we will use the special
notation for the components of the vector
$\delta\boldsymbol\varphi$:
\begin{equation}\label{sima4}
    \boldsymbol\Upsilon\equiv \delta\boldsymbol\varphi=
    (\boldsymbol\nu,  \psi_0,  \boldsymbol\chi,\boldsymbol\lambda, \Gamma_{11},
    \Gamma_{12}, \Gamma_{21}).
\end{equation}
This means that
\begin{equation}\label{sima3}\begin{split}
 \boldsymbol\nu=\delta\boldsymbol \beta,  \quad \psi_0=
 \delta\varphi_0, \quad   \boldsymbol\chi =\delta \boldsymbol u,
 \quad \boldsymbol\lambda =\delta\mathbf w, \\ \Gamma_{11}=\delta W_{11}, \quad
    \Gamma_{12}=\delta W_{12}, \quad \Gamma_{21}=\delta W_{21},
\end{split}\end{equation}
It is easily seen that in this case
\begin{equation}\label{sima5}
    \mathfrak Z:=\delta\boldsymbol\Theta\equiv D_\varphi
    \boldsymbol\Theta(\boldsymbol\varphi_0)[\boldsymbol\Upsilon]
\end{equation}
is defined by the equalities
\begin{equation}\label{sima6}
\mathfrak Z\equiv \Big(\boldsymbol\chi,\, \boldsymbol\mu,\, \boldsymbol
\lambda,\,
 -\nabla\boldsymbol\chi,\, \boldsymbol\Gamma,\, \nabla (\mathbf J\boldsymbol\lambda),\, \partial_{\xi_i}
 (\mathbf J\, \boldsymbol\Gamma)\, \Big),\end{equation}
where
\begin{equation}\label{sima8}
\boldsymbol\mu=\boldsymbol\nu+\nabla\psi_0, \quad \text{Tr~}\boldsymbol\Gamma=0.
\end{equation}
This means that in this case
$$
\delta\mathbf V=-\nabla_\xi\boldsymbol\chi, \quad
\delta\boldsymbol\Lambda= -\nabla\boldsymbol\lambda\, \mathbf
J\equiv \nabla(\mathbf J\boldsymbol\lambda),\quad
 \mathbf R_i=-\frac{\partial}{\partial\xi_i}\boldsymbol\Gamma^\top\,
 \mathbf J\equiv \frac{\partial}{\partial\xi_i}(\mathbf J\boldsymbol\Gamma).
$$
$$
\delta\mathbf W=\boldsymbol\Gamma=\left(
\begin{array}{cc}
\Gamma_{11}&\Gamma_{12}\\
\Gamma_{21}&\Gamma_{22}
\end{array}
\right),\text{~~where~~} \Gamma_{22}=-\Gamma_{11}.
$$
In other words, we have the equality $\text{Tan~}_{\Theta(\boldsymbol \varphi_0)}\mathcal G=\{\mathfrak Z\}.$ In particular, all tangent spaces to the manifold $\mathcal G$ at points $\boldsymbol\Theta(\boldsymbol\varphi_0)$ coincide.
It is a remarkable fact of the theory that  for every $\boldsymbol\varphi$ there is a
canonical algebraic isomorphism
 between the tangent space to the manifold $\mathcal G$ at point $\boldsymbol\Theta(\boldsymbol\varphi)$ and the tangent space to $\mathcal G$ at $\boldsymbol\Theta(\boldsymbol\varphi_0)$.
This result is given by the following theorem

\begin{theorem}\label{sima12} {\bf Second Structure Theorem.}
$(\mathbf i)$ Let $\boldsymbol\varphi\in \mathcal L$.
Let   $\boldsymbol\Upsilon$ and    $\mathfrak Z$
      be given by
      \eqref{sima4} and  \eqref{sima6}. Then there exists a vector field
    $$\delta\boldsymbol\varphi=(\delta\mathbf \beta, \delta\mathbf \varphi_0,
     \delta\mathbf u, \delta\mathbf
w, \delta W_{11},\delta W_{12},\delta W_{21})\in X_{\sigma,d-1}$$
such that the corresponding vector of coefficients
 $$\delta \boldsymbol \Theta(\boldsymbol\varphi)=
\big(\, \delta\mathbf u, \delta\mathbf v, \delta\mathbf w,
\delta\mathbf V, \delta\mathbf W, \delta\boldsymbol\Lambda,
\delta\mathbf R_i\,\big)$$
  given by \eqref{gala3}-\eqref{gala2}
 are connected with the  vector fields  $\boldsymbol\Upsilon$  and $\mathfrak Z$
  by the relations
\begin{subequations}\label{sima13}
\begin{gather}
\label{sima13a} \delta\mathbf u\,=\, \chi_i\frac{\partial}{\partial
\xi_i} \,(\boldsymbol\xi +
\mathbf u),\\
\label{sima13b} \delta\mathbf v\,=\, \mathbf V
\boldsymbol\mu+\boldsymbol\Lambda \boldsymbol\lambda+ \chi_i\,
\frac{\partial}{\partial \xi_i} \mathbf v,\\\label{sima13c}
\delta\mathbf w\,=\, \mathbf W \boldsymbol \lambda+\chi_i\,
\frac{\partial}{\partial \xi_i} \mathbf w,\\\label{sima13d}
\delta\mathbf W\,=\, \mathbf W \boldsymbol \Gamma+\chi_i\,
\frac{\partial}{\partial \xi_i} \mathbf W,\\\label{sima13e}
\delta\mathbf V\,=\, -\mathbf V\nabla_\xi\boldsymbol \chi+\chi_i\,
\frac{\partial}{\partial \xi_i} \mathbf V,\\\label{sima13f}
\delta\boldsymbol\Lambda\,=\, \mathbf V\nabla_\xi(\mathbf
J\boldsymbol \lambda)+ \chi_i\, \frac{\partial}{\partial \xi_i}
\boldsymbol \Lambda+\boldsymbol\lambda^\top\mathbf R+
\boldsymbol\Lambda\boldsymbol\Gamma,\\\label{sima13g} \delta\mathbf
R_i=\frac{\partial}{\partial \xi_i}(\mathbf J\boldsymbol \Gamma)+
\mathbf R_i \boldsymbol\Gamma+(\mathbf R_i
\boldsymbol\Gamma)^\top+\chi_i\,
\frac{\partial}{\partial \xi_i} \mathbf R_i,\\
\label{sima13h} \delta\boldsymbol \beta=\boldsymbol \nu, \quad
\delta \varphi_0=\psi_0 +w_2\delta w_1+\chi_i\,
\frac{\partial}{\partial \xi_i} \varphi_0-w_2\chi_i\,
\frac{\partial}{\partial \xi_i} w_1-\boldsymbol \nu \cdot \mathbf u.
\end{gather}
\end{subequations}
Here $\mathbf u$, $\varphi_0$, $\mathbf W$, $\mathbf V$,
$\boldsymbol \Lambda$, and $\mathbf R_i$ are the components of the
vector fields $\boldsymbol\varphi$ and
 $\boldsymbol\Theta(\boldsymbol\varphi)$.

\noindent $(\mathbf i\mathbf i)$ Conversely, let
 $\delta\boldsymbol\varphi$
and  $\delta\boldsymbol\Theta(\boldsymbol \varphi)$ be given by the relations \eqref{gala1}-\eqref{gala2} . Then there exist vector fields
$$
\boldsymbol\Upsilon=(\boldsymbol\nu,  \psi_0,
\boldsymbol\chi,\boldsymbol\lambda, \Gamma_{11},
    \Gamma_{12}, \Gamma_{21})\in X_{\sigma,d-1}$$
and
\begin{equation*}
\mathfrak Z\equiv \Big(\boldsymbol\chi,\, \boldsymbol\mu,\, \boldsymbol
\lambda,\,
 -\nabla\boldsymbol\chi,\, \boldsymbol\Gamma,\, \nabla (\mathbf J\boldsymbol\lambda),\, \partial_{\xi_i}
 (\mathbf J\, \boldsymbol\Gamma)\, \Big),\end{equation*}
which are   connected with the components of the
vector fields $\delta\boldsymbol \varphi$ and  $\delta\boldsymbol
\Theta$    by the recurrent relations
\begin{subequations}\label{sima15}
\begin{gather}\label{sima15a}
\boldsymbol\chi\,= \,\mathbf V^\top \,\delta\mathbf
u\\\label{sima15b} \boldsymbol\lambda\,=\, \mathbf
W^{-1}\delta\mathbf w-\chi_i\,\mathbf W^{-1}
\frac{\partial}{\partial \xi_i} \mathbf w\\\label{sima15c}
\boldsymbol\Gamma\,=\, \mathbf W^{-1}\delta\mathbf W-\chi_i\,\mathbf
W^{-1} \frac{\partial}{\partial \xi_i} \mathbf W,\\\label{sima15d}
\boldsymbol\mu\,=\, \mathbf V^{-1}\Big(\delta\mathbf W+\chi_i\,
\frac{\partial}{\partial \xi_i} \mathbf v-
\boldsymbol\Lambda\boldsymbol\lambda\Big)\\\label{sima15e}
\nabla\psi_0=\boldsymbol\mu-\boldsymbol\nu,
 \quad \boldsymbol\nu=\frac{1}{(2\pi)^{n-1}}\int_{\mathbb T^{n-1}}
  \boldsymbol\mu\, d\boldsymbol\xi=\delta\boldsymbol\beta.
\end{gather}
\end{subequations}
\end{theorem}
\begin{proof} The proof is in Appendix \ref{proofsima12}.
\end{proof}

\section{A modified problem and bifurcation equations}\label{luna}
 Recall that the  main problem is to prove  that  the Hamiltonian
 $H(\mathbf x, \mathbf y, \mathbf z)= H_0 (\mathbf y, z_2)+
 \varepsilon H_1(\mathbf x, \mathbf y, \mathbf z)$ has  $(n-1)$-dimensional
  weakly-hyperbolic invariant torus for all sufficiently small $\varepsilon$.
  In view of Definition \ref{def1.1} it is necessary to find
 a canonical mapping
 $\boldsymbol\theta:(\boldsymbol\xi, \boldsymbol\eta, \boldsymbol\zeta)\to
 (\mathbf x, \mathbf y, \mathbf z)$ which puts $H$ in the normal form
 \begin{equation}\label{luna1}
 H\circ \boldsymbol\theta= e+\boldsymbol\omega^\top \cdot \boldsymbol\eta+
 \frac{1}{2}\boldsymbol\zeta^\top \boldsymbol\Omega \boldsymbol\zeta+ o(|\boldsymbol \eta|, |\boldsymbol \eta||\boldsymbol\zeta|,
 |\boldsymbol \zeta|^2),
 \end{equation}
 where $\boldsymbol\Omega =\text{diag~}(-k,1)$, $k\in [0,1]$.
and $e=\text{const~}$.  For $\varepsilon=0$ the problem  has a family of solutions  given by
$$
\mathbf x=\boldsymbol\xi, \quad \mathbf y=\boldsymbol\eta, \quad
\mathbf z=\boldsymbol\alpha+\boldsymbol\zeta, \quad
\boldsymbol\alpha= (\alpha,0)^\top,  \quad e=k=0, \quad \alpha\in \mathbb R^1.
$$
It is easily seen that  $(n-1)$- dimensional manifold $\mathbb
T_\alpha=\{\mathbf x=\boldsymbol \xi, \mathbf y=0, \mathbf z=\boldsymbol\alpha\}$ is an invariant
torus of the unperturbed system.
 The totality of these tori forms a foliation of resonance $n$-dimensional
invariant torus, and  $\alpha$ can be considered as a label of a
leave of this foliation. Notice that $\alpha$ is nothing else but
the mean value of $z_1$ over the invariant torus $\mathbb T_\alpha$.
Perturbations destroy  $n$-dimensional resonance torus, and only a
few $(n-1)$-dimensional invariant tori survive for $\varepsilon\neq
0$. The label $\alpha$ of surviving torus  is unknown and should be defined along with
a solution to the problem. Therefore, the range of unknowns $(\alpha,k)$ is the
whole strip $\mathbb R\times [0,1]$. This means that  the problem of
finding of
   $\alpha$ and $k$ is not local and can not be solved by  using an  iteration process.
In order to cope with this difficulty we apply the version of the
Lyapunov-Schmidt method proposed in \cite{Moser}.  Following
\cite{Moser} we add the modification term $m z_1+M z_1^2/2$ to the
original  Hamiltonian. Thus we get the  modified Hamiltonian $H$
\begin{equation}\label{luna5}
    H_m(\mathbf x, \mathbf y, \mathbf z)=H_0(\mathbf  y, z_2)+\varepsilon
    H_1(\mathbf x, \mathbf y, \mathbf z)+\mathbf {m}^\top\cdot
{\mathbf {z}}+\frac12\mathbf z^\top\mathbf{M}\mathbf{z}
\end{equation}
where
\begin{equation}\label{luna6}
    \mathbf m=(m,0)^\top, \quad \mathbf M=\text{diag~}(M,0)
\end{equation}
Consider the following

\noindent {\sl Modified problem.} {\it For given $\alpha\in \mathbb
R^1$ and $k\in[0,1]$ to find  parameters $m,M,e$, and  a canonical
mapping
\begin{gather}
\boldsymbol\theta:(\boldsymbol\xi,
\boldsymbol \eta, \boldsymbol
 \zeta)\to (\mathbf x, \mathbf y, \mathbf z),\\
\nonumber \boldsymbol x=\boldsymbol\xi+\mathbf
u(\boldsymbol\xi),\quad \mathbf y=\mathbf v+\mathbf
V(\boldsymbol\xi)\boldsymbol\eta+ \boldsymbol\Lambda(\boldsymbol\xi)
\boldsymbol\zeta+\frac{1}{2}\boldsymbol\zeta^\top \mathbf
R(\boldsymbol\xi)\boldsymbol\zeta,\\\label{luna3} \mathbf z=\mathbf
w(\boldsymbol\xi)+\mathbf W(\boldsymbol\xi)\boldsymbol\zeta.
\end{gather}
 with the following properties. The mapping $\boldsymbol\theta$
 satisfies the  condition
 $$
 \frac{1}{(2\pi)^{n-1}}\int_{\mathbb T^{n-1}} w_1(\xi)\, d\xi=\alpha,
 $$
 and
 puts
 the modified hamiltonian in the normal form ,i.e.,}
\begin{equation}\label{luna7}
 H_m\circ \boldsymbol\theta= e+\boldsymbol\omega^\top \cdot
 \boldsymbol\eta+\frac{1}{2}\boldsymbol\zeta^\top \boldsymbol\Omega
  \boldsymbol\zeta+ o(|\boldsymbol \eta|, |\boldsymbol \eta||\boldsymbol\zeta|,
 |\boldsymbol \zeta|^2),
 \end{equation}
We stress that in this framework $\alpha$ and $k$ are given, while
$m$, $M$, and $e$  are unknown and should be defined along with a
solution. The advantage of this approach is that the solvability of
the modified problem  can be established  by using the Nash-Moser
implicit function theorem. The obtained  solution $ \boldsymbol\theta, e, m,M$
to  the modified   problem is
 a function of $\alpha$ and $k$. In particular, we have $m=m(\alpha,k)$
 and $M=M(\alpha,k)$. Obviously, if $m=M=0$, then  the canonical mapping
  $\boldsymbol\theta$ puts the original  hamiltonian  $H(\mathbf x, \mathbf y, \mathbf z)$
  in the normal form.   Thus  we  reduce
  the original  problem to the system of two scalar equations
  \begin{equation}\label{lunabifurcation}
    m(\alpha,k)=M(\alpha,k)=0
  \end{equation}
  named the bifurcation equations.
Hence our first task is to prove the local existence and uniqueness
of solutions to  the modified problem for all $(\alpha,k)$ and for all sufficiently small $\varepsilon$.

 We reduce the modified problem  to the system of nonlinear partial differential equations.
Notice that in view of Theorem \ref{anna4} the coefficient vectors
$\boldsymbol\Theta=\big(\mathbf u, \mathbf v, \mathbf w, \mathbf
V,\boldsymbol \Lambda,
 \mathbf W,
    \mathbf R_i\big)$
is  a function of  the coordinate vector
\begin{equation}\label{luna4}
    \boldsymbol\varphi=(\beta,\varphi_0, \mathbf u, \mathbf w, W_{11}, W_{12}, W_{21}).
\end{equation}
Moreover, the mapping $\boldsymbol\varphi\mapsto
\boldsymbol\Theta(\boldsymbol\varphi)$ is given  by explicit
formulae \eqref{anna5}. Thus relation \eqref{luna7} can be regarded
as a system of nonlinear  equations for $\boldsymbol\varphi$.
Substituting representation \eqref{luna3} with the coefficients
$\boldsymbol\Theta= \boldsymbol\Theta(\boldsymbol\varphi)$ into
\eqref{luna7} we arrive at the following system of equations for the
 vector $\boldsymbol\varphi$ given by \eqref{luna4}.
\begin{subequations}\label{luna8}
\begin{gather}\label{luna8a}
\Phi_1(\boldsymbol\varphi, e, m,M)\equiv H_m(\text{id}+\mathbf u,
\mathbf v, \mathbf w)-e=0,
\\
\label{luna8b} \Phi_2(\boldsymbol\varphi, m,M)\equiv
\big\{\frac{\partial H_m}{\partial \mathbf y} (\text{id}+\mathbf u,
\mathbf v, \mathbf w)\boldsymbol \Lambda+ \frac{\partial
H_m}{\partial \mathbf z} (\text{id}+\mathbf u, \mathbf v, \mathbf
w)\mathbf W \big\}^\top=0,
\\
\label{luna8c} \Phi_3(\boldsymbol\varphi, m,M)\equiv
\big\{\frac{\partial H_m}{\partial \mathbf y} (\text{id}+\mathbf u,
\mathbf v, \mathbf w)\mathbf V\big\}^\top-\boldsymbol\omega=0
\end{gather}
\begin{equation}
\label{luna8d}\begin{split} \Phi_4(\boldsymbol\varphi, m,M)\equiv
\frac{\partial H_m}{\partial y_i} (\text{id}+\mathbf u, \mathbf v,
\mathbf w)\mathbf R_i+ \boldsymbol\Lambda^\top \frac{\partial^2
H_m}{\partial \mathbf y^2}
(\text{id}+\mathbf u, \mathbf v, \mathbf w)\boldsymbol\Lambda+\\
\mathbf W^\top \frac{\partial^2 H_m}{\partial\mathbf z\partial
\mathbf y} (\text{id}+\mathbf u, \mathbf v, \mathbf
w)\boldsymbol\Lambda +\Big(\mathbf W^\top \frac{\partial^2
H_m}{\partial\mathbf z\partial \mathbf y} (\text{id}+\mathbf u,
\mathbf v, \mathbf w)\boldsymbol\Lambda\Big)^\top+
\\\mathbf W^\top \frac{\partial^2 H_m}{\partial\mathbf z^2}
(\text{id}+\mathbf u, \mathbf v, \mathbf w)\mathbf
W-\boldsymbol\Omega=0
\end{split}\end{equation}
\begin{equation}\label{luna8e}
    \Phi_5(\boldsymbol\varphi)\equiv\frac{1}{(2\pi)^{n-1}}\int_{\mathbb T^{n-1}} w_1(\xi)\, d\xi-\alpha=0.
\end{equation}

Here matrices $\partial^2 H_m/(\partial \mathbf y)^2$, $\partial^2
H_m/(\partial\mathbf z\partial \mathbf y)$, and
  $\partial^2 H_m/(\partial \mathbf z)^2$ have the entries
 \begin{equation*}
    \Big(\frac{\partial^2 H_m}{\partial \mathbf y^2}\Big)_{ij}=
    \frac{\partial^2 H_m}{\partial  y_i\partial y_j}, \quad
     \Big(\frac{\partial H_m}{\partial\mathbf z\partial \mathbf y}\Big)_{pj}=
    \frac{\partial^2 H_m}{\partial  z_p\partial y_j},\quad
     \Big(\frac{\partial H_m}{\partial \mathbf z^2}\Big)_{pq}=
    \frac{\partial H_m}{\partial z_p\partial z_q},
 \end{equation*}
 $$
 1\leq i,j\leq n-1, \quad 1\leq p,q\leq 2,
 $$
 the row vectors $\partial H_m/\partial \mathbf y$,
$\partial H_m/\partial\mathbf z$ are given by \eqref{notation1},
i.e.,
 $$
\frac{\partial H_m}{\partial \mathbf y}=(\nabla_y H_m)^\top, \quad
\frac{\partial H_m}{\partial \mathbf z}=(\nabla_z H_m)^\top
$$
The matrices $\boldsymbol \Lambda$, $\mathbf V$, $\mathbf W$ and the
vector $\mathbf v$ are expressed
  in terms  of $\boldsymbol\varphi$ by equalities \eqref{anna17}:
 \begin{equation}\label{luna9}\begin{split}
 \mathbf V=(\mathbf I_{n-1}+\mathbf u')^{-\top}, \quad \mathbf v=\beta +
 \mathbf V(\nabla\varphi_0-w_2\nabla w_1),\\
 \mathbf W=\left(
\begin{array}{cc}
W_{11}&W_{12}\\
W_{21}& \frac{1}{W_{11}}(1+W_{12} W_{21})
\end{array}
\right),\\
\boldsymbol\Lambda =-\mathbf V(\mathbf w')^\top\mathbf J\mathbf W,
\quad \mathbf R_i=-V_{ik} \Big (\frac{\partial}{\partial \xi_k}
\mathbf W^{\top}\Big)\, \mathbf J\,\mathbf W.
\end{split}\end{equation}
 Notice that the matrices $\mathbf R_i$ are symmetric and the operator $\Phi_4$ is a
  symmetric matrix valued function.
  Hence $\Phi_1$ is a scalar, $\Phi_2$ is a column vector of dimension $2$, $\Phi_3$
  is a column vector of dimension $n-1$, $\Phi_4$ is $2\times 2$ symmetric matrix.
  Therefore, equations \eqref{luna8a}-\eqref{luna8d} form a system of $(n+5)$
  nonlinear differential equations for $n+5$ functional components of the
  vector- valued function $\boldsymbol\varphi$.

Notice that a solution to the  modified problem is not unique. It is
easily seen that if a canonical mapping $ (\boldsymbol\xi,
\boldsymbol\eta, \boldsymbol\zeta)\to (\mathbf x, \mathbf y, \mathbf
z)$ puts the hamiltonian $H_m$ to the normal form, then the mapping
$$
(\boldsymbol\xi', \boldsymbol\eta', \boldsymbol\zeta')\to
(\boldsymbol\xi, \boldsymbol\eta, \boldsymbol\zeta)\to (\mathbf x,
\mathbf y, \mathbf z),
$$
$$
\boldsymbol\xi =\boldsymbol\xi'+c, \quad
\boldsymbol\eta=\boldsymbol\eta', \quad \boldsymbol\zeta= \left(
\begin{array}{cc}
(1+kb^2)^{1/2}&b\\
kb&(1+kb^2)^{1/2}
\end{array}
\right)\boldsymbol\zeta',
$$
 $c\in \mathbb R^{n-1}$, $b\in \mathbb R$, also puts $H_m$ in the normal form.
 Hence the modified problem has $n$-parametric family of solutions. In order
 to eliminate this arbitrariness, we impose the orthogonality conditions
\begin{equation}\label{luna8f}
 \int_{\mathbb T^{n-1}}\mathbf u(\boldsymbol\xi)\, d\xi=0,\quad
 \int_{\mathbb T^{n-1}}W_{12}(\boldsymbol\xi)\, d\xi=0.
\end{equation}
Recall that by definition of the local coordinate on the manifold
$\mathcal G$, the function  $\varphi_0$ in \eqref{luna4} has zero
mean. Thus we have to add the orthogonality condition
\begin{equation}\label{luna8g}
   \int_{\mathbb T^{n-1}}\varphi_0(\boldsymbol\xi)\, d\xi=0.
\end{equation}
\end{subequations}

 \emph{Equations \eqref{luna8a}-\eqref{luna8e} and conditions \eqref{luna8f}-\eqref{luna8g}
 form a closed system of equations for $\boldsymbol\varphi$ and parameters $m,M$, $e$.}

  Throughout of the paper we will use the following consequence of this equations.
 Introduce the important differential operator
 \begin{equation}\label{luna10}
 \boldsymbol\partial:= \boldsymbol\omega^\top\cdot \nabla_\xi\equiv \omega_i
  \frac{\partial}{\partial\xi_i}.
 \end{equation}
 \begin{lemma}\label{luna11}
 If $\boldsymbol\varphi$ is an analytic solution to equations \eqref{luna8a}-
 \eqref{luna8d} and $\mathbf v$ is given by \eqref{luna9},
 then
 \begin{equation}\label{luna12}\begin{split}
    \boldsymbol\omega +\boldsymbol\partial \mathbf u=\nabla_y H_m(\boldsymbol\xi+
    \mathbf u, \mathbf v, \mathbf w), \\ \boldsymbol\partial \mathbf v=-
    \nabla_x H_m(\boldsymbol\xi+\mathbf u, \mathbf v, \mathbf w),\quad
    \boldsymbol \partial\mathbf w=\mathbf J\nabla_z H_m(\boldsymbol\xi+\mathbf u,
    \mathbf v, \mathbf w).
\end{split} \end{equation}
  \end{lemma}
The proof is in Appendix \ref{proofluna11}.

\paragraph{Operator equation.} Now we reduce  system of differential equations
\eqref{luna8} to  the operator equation in Banach
spaces of analytic functions. Notice that unknowns are the vector-valued function
$\boldsymbol \varphi$ and the constants $e$, $m$, and $M$. Equations
\eqref{luna8} also depend on the parameters $\alpha$ and $k$. It is
convenient to introduce the vectors
\begin{equation}\label{ira1bis}
   {\mathfrak f}=(\alpha, k), \quad{\mathfrak u}=
    (\boldsymbol\varphi, e, m,M)
\end{equation}
where $\boldsymbol\varphi=(\boldsymbol\beta,\varphi_0, \mathbf u,
\mathbf w, W_{11}, W_{12}, W_{22})$.

Now we introduce some  notation which is used throughout the rest of the paper.
\begin{definition}\label{anna20xx} Let $\varrho>0$ be given by
Condition $(\mathbf H.\mathbf 1)$. Denote by $\Sigma_\varrho$ the slab
\begin{equation}\label{anna20xxx}
    \Sigma_\varrho=\big\{(\alpha,k):\,\, \text{\rm Re~} \alpha\in \mathbb R, |\text{\rm Im~}\alpha|\leq \varrho,\,\,\text{\rm Re~} k\in [0,1], |\text{\rm Im~} k|\leq \varrho\big\}
    \end{equation}

\end{definition}

\begin{definition}\label{anna20} Let $\varrho>0$ bi given by
Condition $(\mathbf H.\mathbf 1)$. For every $\sigma\in [0,1]$ and any integer
$d\geq 0$
 denote by $\mathcal A_{\sigma,d}$ the Banach space of all functions
$$
\mathbf u: \Re_{\sigma\varrho}\to\mathbb C, \quad
\Re_{\sigma\varrho}=\big\{\,\boldsymbol \xi:\,\, \text{\rm
Re~}\boldsymbol \xi\in \mathbb T^{n-1}, \,\,\, |\text{\rm Im~}
 \boldsymbol\xi|\leq \sigma\varrho\,\big\},
$$
with the finite norm
\begin{equation}\label{anna20}
    \|u\|_{\sigma,d}=\max\limits_{0\leq |k|\leq d} \,\,\,
    \sup\limits_{\boldsymbol\xi\in \Re_{\sigma\varrho}}\,\,
    |\partial^k u(\boldsymbol\xi)|.
\end{equation}
\end{definition}

Notice that  system
\eqref{luna8} consists of differential equations
\eqref{luna8a}-\eqref{luna8d} and orthogonality conditions
\eqref{luna8e}-\eqref{luna8f}. It is convenient to incorporate the
orthogonality conditions in the definition of the Banach spaces.
Thus we arrive at  the following
\begin{definition}\label{luna15} For every $\sigma\in [0,1]$ and any integer $d\geq 0$
denote by $\mathcal E_{\sigma,d}$ the subspace of the Banach space
$\mathbb C^{n-1} \times \mathcal A_{\sigma,d}\times \mathcal
A_{\sigma,d}^{n-1} \times \mathcal A_{\sigma,d}^{2}\times \mathcal
A_{\sigma, d}^3\times \mathbb C^3$ which consists of all vectors
\begin{equation*}
\mathfrak u=(\boldsymbol\varphi, e, m,M)  \text{~~with~~}
\boldsymbol\varphi= (\boldsymbol\beta,\varphi_0, \mathbf u, \mathbf
w, W_{11}, W_{12}, W_{21}),
\end{equation*}
satisfying the orthogonality conditions
\eqref{luna8f}-\eqref{luna8g}.
\end{definition}

\begin{definition}\label{luna16} For every $\sigma\in [0,1]$ and any
integer $d\geq 0$ denote by $\mathcal F_{\sigma,d}$  the subspace of
$\mathcal A_{\sigma,d}\times \mathcal A_{\sigma,d}^{n-1} \times
\mathcal A_{\sigma,d}^2\times \mathcal A_{\sigma,d}^4\times \mathbb
C$ which consists of all vectors
\begin{equation*}
 \boldsymbol F=(F_1, F_2, F_3, F_4, f_5)
\end{equation*}
such that $F_4$ is a symmetric $2\times 2$-matrix valued function..
\end{definition}
Introduce the nonlinear operator
\begin{equation}\label{luna17}\begin{array}{cc}
    \mathbf \Phi(\mathfrak u, \mathfrak f)=&\\\big(\Phi_1(\boldsymbol\varphi, e, m,M),
    &\Phi_2(\boldsymbol\varphi, m,M),
    \Phi_3(\boldsymbol\varphi, m,M),\Phi_4(\boldsymbol\varphi, m,M;k),
    \Phi_5(\mathbf w; \alpha)\big),
\end{array}
\end{equation}
where $\Phi_i$ are given by \eqref{luna8}. It is worthy noting that
$\boldsymbol\Phi$ is a linear function of the scalars $e, m, M$, and
$k$, $\alpha$. Hence the modified problem \eqref{luna8} can be
written in the form of the operator equation
\begin{equation}\label{luna8x}
   \mathbf \Phi(\mathfrak f, \mathfrak u)  =0.
\end{equation}
{\bf Remark} For $\varepsilon=0$ and $\mathfrak f\in \mathbb C^2$,
equation \eqref{luna8x} has the unique solution
\begin{equation}\label{luna8xxx}
    \mathfrak v(\mathfrak f)=(\boldsymbol\varphi_0(\alpha), -k\alpha-k\alpha^2,
    -k\alpha, -k),
\end{equation}
where
$$\boldsymbol\varphi_0(\alpha)=(0,0, \alpha\mathbf e_1, 1, 0,0)
$$

Notice that $\Phi$ is a nonlinear analytic differential operator.
 The general theory of the nonlinear differential operators in
spaces of analytic functions $\mathcal A_{\sigma,d} $ is considered in the monograph \cite{Nirenberg}
and \cite{Poschel}. We use the following   proposition,
see \cite{Nirenberg} Ch.6,  which  constitute the continuity and differentiability
of the operator $\Phi$.

\begin{proposition}\label{luna18}
Let $\varrho>0$ be given by Condition ({\bf H.1}). For every
$\sigma\in [0,1]$ and $d\geq 1$, there are $r>0$ and $c>0$ with the
properties. If $\mathfrak f\in \Sigma_{\varrho}$ and
 $\mathfrak u\in \mathcal E_{\sigma,d}$
satisfy
$$
\|\boldsymbol\varphi-\boldsymbol\varphi_0(\alpha)\|_{\sigma, d}\leq r,
\quad \boldsymbol\varphi_0= (0,0,0,\alpha\boldsymbol e_1,1,0,0),
$$
then $\boldsymbol\Phi(\mathfrak f, \mathfrak u)\in \mathcal
F_{\sigma,0}$ and
$$
\|\boldsymbol\Phi(\mathfrak f, \mathfrak u)\|_{\sigma,0}\leq c+c
(|e|+|m|+|M|).
$$
If in addition
$$
\|\boldsymbol\varphi+\delta \boldsymbol\varphi-\boldsymbol\varphi_0(\alpha)\|_{\sigma, d}\leq r
$$
then
\begin{multline*}
\Phi_i(\boldsymbol\varphi+\delta\boldsymbol\varphi, e,m, M)-
\Phi_i(\boldsymbol\varphi, e,m, M)=\\D_\varphi
\Phi_i(\boldsymbol\varphi, e, m,M)
[\delta\boldsymbol\varphi]+Q_{\Phi_i}(\boldsymbol\varphi,
\delta\boldsymbol\varphi, e,m, M).
\end{multline*}
The linear operator $D_\varphi \Phi_i$ and the remainder
$Q_{\Phi_i}$ admit the estimates
\begin{equation}\label{luna23}
  \| D_\varphi \Phi_i(\boldsymbol\varphi, e, m,M)[\delta\boldsymbol\varphi]\|_{\sigma,d-1}
  \leq c(1+|e|+|m|+|M| )\,\,\|\delta\boldsymbol\varphi\|_{\sigma,d},
  \end{equation}
 \begin{equation}\label{luna24}
 \| Q_{\Phi_i}(\boldsymbol\varphi,\delta\boldsymbol\varphi, e,m, M)\|_{\sigma,d-1}
  \leq c(1+|e|+|m|+|M| )\,\,\|\delta\boldsymbol\varphi\|_{\sigma,d}^2.
 \end{equation}

\end{proposition}

\paragraph{Third structural theorem.}
In order to apply the Nash-Moser implicit function theorem to
operator equation \eqref{luna8x}, we have to prove that  the linear
operator $D_{\mathfrak u}\boldsymbol\Phi$ has an approximate inverse. In other
words, we have to show that for every
\begin{equation}\label{luna25x}
    \mathbf F=(F_1,F_2,F_3, F_4, f_5)\in \mathcal F_{\sigma,d}
\end{equation}
the linear equation
\begin{equation}\label{luna25}
D_\varphi\boldsymbol\Phi(\boldsymbol\varphi)[\delta\boldsymbol\varphi]+
D_e\boldsymbol\Phi(\boldsymbol\varphi)[\delta e]+
D_m\boldsymbol\Phi(\boldsymbol\varphi)[\delta m]+
D_M\boldsymbol\Phi(\boldsymbol\varphi)[\delta M]=\mathbf F
\end{equation}
has an analytic approximate solution $(\delta\boldsymbol \varphi, e,
m,M)$. This problem looks like very difficult because of the
complexity of the expression for
$\delta\boldsymbol\Phi
=D_\varphi\boldsymbol\Phi(\boldsymbol\varphi)[\delta\boldsymbol\varphi]$.
The remarkable fact of KAM theory is that   linear equation
\eqref{luna25}  can be reduced to the triangle system of first order
differential equations with
 constant coefficients in the principle part.  This very special change of
 variables is given by the second structural Theorem \ref{sima12}. In order to
 define, it choose an arbitrary vector
$\delta\boldsymbol\varphi= (\delta\boldsymbol\beta, \delta \varphi_0,
\delta\mathbf u, \delta\mathbf w, \delta W_{11},\delta W_{12},\delta
W_{21}) $. Next notice that relations \eqref{sima15} in Theorem
\ref{sima12} read
\begin{subequations}\label{luna27}
\begin{gather}\label{luna27a}
\boldsymbol\chi\,= \,\mathbf V^\top \,\delta\mathbf
u\\\label{luna27b} \boldsymbol\lambda\,=\, \mathbf
W^{-1}\delta\mathbf w-\chi_i\,\mathbf W^{-1}
\frac{\partial}{\partial \xi_i} \mathbf w\\\label{luna27c}
\boldsymbol\Gamma\,=\, \mathbf W^{-1}\delta\mathbf W-\chi_i\,\mathbf
W^{-1} \frac{\partial}{\partial \xi_i} \mathbf W,\\\label{luna27d}
\boldsymbol\mu\,=\, \mathbf V^{-1}\Big(\delta\mathbf v+\chi_i\,
\frac{\partial}{\partial \xi_i} \mathbf
v-\boldsymbol\Lambda\boldsymbol\lambda\Big)
\\\label{luna27e}
\nabla\psi_0=\boldsymbol\mu-\delta\boldsymbol\beta, \quad
\delta\boldsymbol\beta= \frac{1}{(2\pi)^{n-1}}\int_{\mathbb T^{n-1}}
\boldsymbol\mu\, d\boldsymbol\xi,
\end{gather}
\end{subequations}
where $\delta\mathbf W$ is given in terms of
$\delta\boldsymbol\varphi$ by \eqref{gala1}. These equalities
completely  define the vector
$$\boldsymbol\Upsilon=(\delta\boldsymbol\beta,  \psi_0,
\boldsymbol\chi,\boldsymbol\lambda, \Gamma_{11},
    \Gamma_{12}, \Gamma_{21}).$$
Recall that $\mathbf \Gamma$ is $2\times 2$ matrix such that
$\Gamma_{22}=-\Gamma_{11}$. It is worth noting that equalities
\eqref{luna27} establishes linear algebraic relations between
$\delta\boldsymbol\varphi$
 and $\boldsymbol\Upsilon$
The following  theorem shows that the change of
variables \eqref{luna27} brings the linear operator
$D_{\boldsymbol\varphi}\boldsymbol\Phi$ to the simple canonical form.
\begin{theorem}\label{structural3} {\bf Third Structure Theorem.}
Let $\boldsymbol\varphi$, $\alpha$ and $k$ meet all requirements of
Proposition \ref{luna18}.  Furthermore assume that
$$
\delta\Phi_i= D_{\boldsymbol\varphi} \Phi_i(\boldsymbol\varphi)[\delta\boldsymbol\varphi]
$$
 Then for every
$\delta\boldsymbol\varphi\in \mathbb C^{n-1}\times \mathcal
A_{\sigma,d}\times
 \mathcal A^{n-1}_{\sigma,d}\times \mathcal A_{\sigma,d}^2\times\mathcal A_{\sigma,d}^3$,
\begin{subequations}\label{luna26}
\begin{equation}\label{luna26a}
    \delta\Phi_1=\boldsymbol\partial \psi_0+\boldsymbol\omega^\top\cdot
    \delta\boldsymbol\beta+\Pi_1[\boldsymbol\chi, \boldsymbol \mu, \boldsymbol\lambda],
\end{equation}
\begin{equation}\label{luna26b}
    \delta\Phi_2=\mathbf J\boldsymbol\partial\boldsymbol\lambda +\boldsymbol\Omega\boldsymbol\lambda+\mathbf T\boldsymbol\mu+
    \Pi_2[\boldsymbol\chi, \boldsymbol\lambda,\boldsymbol\Gamma]
\end{equation}

\begin{equation}\label{luna26c}
 \delta\Phi_3=-\boldsymbol\partial \boldsymbol\chi+\mathbf S
 \boldsymbol\mu+\mathbf T^\top \boldsymbol\lambda +\Pi_3[\boldsymbol\chi],
\end{equation}

\begin{equation}\label{luna26d}\begin{split}
    \delta\Phi_4=\boldsymbol\partial(\mathbf J\boldsymbol\Gamma)
    +\boldsymbol\Omega\boldsymbol\Gamma+
(\boldsymbol\Omega\boldsymbol\Gamma)^\top +\mathbf U_{ij}
\frac{\partial \lambda_i}{\partial\xi_j}+ \mu_i\mathbf
E_i+\lambda_i\mathbf K_i+   \Pi_4[\boldsymbol\chi,\boldsymbol\Gamma]
\end{split}\end{equation}
\begin{equation}\label{luna26e}\begin{split}
    \delta\Phi_5=\frac{1}{(2\pi)^{n-1}}
    \int_{\mathbb T^{n-1}}\Big
    ((\mathbf W\boldsymbol\lambda)_1+
    \boldsymbol\chi_i\frac{\partial w_1}{\partial\xi_i}\Big) \, d
    \boldsymbol\xi.
\end{split}\end{equation}
\end{subequations}
Here $\delta\boldsymbol\beta,  \psi_0,
\boldsymbol\chi,\boldsymbol\lambda, \Gamma_{11},
    \Gamma_{12}, \Gamma_{21}$ and  $\delta\boldsymbol\varphi $
    are interrelated  by relations \eqref{luna27} . The linear
    differential operators
$\Pi_i$ are defined by
\begin{subequations}\label{luna28}
\begin{equation}\label{luna28a}
    \Pi_1[\boldsymbol \chi,\boldsymbol\mu, \boldsymbol\mu,
    \boldsymbol\lambda]
    =\frac{\partial\Phi_1}{\partial \xi_i}\chi_i+\Phi_2^\top
    \cdot \boldsymbol\lambda+\Phi_3^\top\cdot\boldsymbol\mu,
\end{equation}
\begin{equation}\label{luna28b}
    \Pi_2[\boldsymbol\chi, \boldsymbol\lambda, \boldsymbol\Gamma]=
    \chi_i\frac{\partial\Phi_2}{\partial \xi_i}+\boldsymbol
    \Gamma^\top \Phi_2+\Phi_4\boldsymbol\lambda+(\mathbf J\boldsymbol\lambda)'_\xi\Phi_3,
    \end{equation}
\begin{equation}\label{luna28c}
     \Pi_3[\boldsymbol\chi]=\chi_i\frac{\partial\Phi_3}{\partial \xi_i}-
     \boldsymbol\chi'_\xi \Phi_3,
\end{equation}
\begin{equation}\label{luna28d}
  \Pi_4[\boldsymbol\chi,\boldsymbol\Gamma]=\chi_i\frac{\partial\Phi_4}{\partial \xi_i}
+\Phi_{3,i}\frac{\partial}{\partial \xi_i}(\mathbf
J\boldsymbol\Gamma)+ \Phi_4\boldsymbol\Gamma+
(\Phi_4\boldsymbol\Gamma)^\top,
\end{equation}
where $\Phi_i=\Phi_i(\boldsymbol\varphi, e,m,M)$.
\end{subequations}
The matrix $\mathbf  S$ and $\mathbf T$ are given in terms of
$\boldsymbol\varphi$ by the equalities
\begin{equation}\label{luna29}\begin{split}
    \mathbf S=\mathbf V^\top\frac{\partial^2 H}{\partial\mathbf  y^2}\big
(\boldsymbol\xi+\mathbf u, \mathbf v, \mathbf w)\mathbf V, \\
\mathbf T=\mathbf W^\top\frac{\partial^2 H}{\partial\mathbf
z\partial\mathbf  y}\big (\boldsymbol\xi+\mathbf u, \mathbf v,
\mathbf w)\mathbf V+ \boldsymbol\Lambda^\top\frac{\partial^2
H}{\partial\mathbf  y^2}\big (\boldsymbol\xi+\mathbf u, \mathbf v,
\mathbf w)\mathbf V.
\end{split}\end{equation}
They admit the estimates
\begin{equation}\label{luna30}
    \|\mathbf S-\mathbf S_0\|_{\sigma, d-1}+
    \|\mathbf T-\mathbf T_0\|_{\sigma, d-1}
\leq c(\varepsilon+r),
\end{equation}
where $c$ depends only on $\sigma$, $d$ and $H$, the constant
matrices $\mathbf S_0$, $\mathbf T_0$ are given by
\begin{equation}\label{luna31}
\mathbf S_0=\frac{\partial^2 H_0}{\partial\mathbf  y^2}\big (0,0),
\quad \mathbf T_0=\frac{\partial^2 H_0}{\partial\mathbf
z\partial\mathbf  y}\big (0,0)
\end{equation}
The symmetric matrix valued functions $\mathbf U_{ij}$, $\mathbf
E_i$, $\mathbf K_i$ admit the estimates
\begin{equation}\label{luna32}
    \|\mathbf U_{ij}\|_{\sigma, d-1}+ \|\mathbf E_i\|_{\sigma, d-1}+
    \|\mathbf K_i\|_{\sigma, d-1}
\leq c.
\end{equation}

\end{theorem}
\begin{proof}
The proof  is in Appendix \ref{astructural3}.
\end{proof}

Theorem \ref{structural3} gives the representation of  the
derivatives
$D_{\boldsymbol\varphi}\boldsymbol\Phi=\delta\boldsymbol\Phi$. As
corollary of this result we obtain the representation for the full
differential $D_\mathfrak u\boldsymbol\Phi(\mathfrak f, \mathfrak
u)$ defined by
\begin{equation}\label{lita203}
D_\mathfrak u\boldsymbol\Phi(\mathfrak f, \mathfrak
u)[\delta\mathfrak u]= \lim\limits_{\tau\to 0} \frac{1}{\tau}\big(
\Phi(\mathfrak f, \mathfrak u+\tau\delta\mathfrak u)-\Phi(\mathfrak
f, \mathfrak u)\big).
\end{equation}
The result is given by the following

\begin{corollary}\label{structuralcor3}
Under the assumptions of Theorem \ref{structural3} the operators
$D_\mathfrak u\Phi_i$ admit the representation
\begin{subequations}\label{lita201}
\begin{gather}
\label{lita201a} D_\mathfrak u\Phi_1(\mathfrak f, \mathfrak
u)[\delta\mathfrak u]\equiv\\\nonumber \boldsymbol\partial
\psi_0+\boldsymbol\omega^\top\cdot
  \delta\boldsymbol\beta+\delta e+ \delta m\cdot  w_1+
    \frac{1}{2}\delta M w_1^2 +\Pi_1[\boldsymbol\chi, \boldsymbol \mu,
   \boldsymbol\lambda],\\
    \label{lita201b}
D_\mathfrak u\Phi_2(\mathfrak f, \mathfrak u)[\delta\mathfrak
u]\equiv\\\nonumber
    \mathbf J\boldsymbol\partial\boldsymbol\lambda+\boldsymbol\Omega \boldsymbol\lambda +\mathbf
T\boldsymbol\mu
   +\delta m\mathbf W^\top\mathbf e_1+
    \delta M w_1\mathbf W^\top \mathbf e_1  +\Pi_2[\boldsymbol\chi, \boldsymbol\lambda,\boldsymbol\Gamma] ,
  \\
\nonumber \boldsymbol\mu =\nabla \psi_0+\delta\boldsymbol\beta
      \\
\label{lita201c}
 D_\mathfrak u\Phi_3(\mathfrak f, \mathfrak
u)[\delta\mathfrak u]\equiv\\\nonumber
    -\boldsymbol\partial \boldsymbol\chi+
 \mathbf S\boldsymbol\mu+\mathbf T^\top \boldsymbol\lambda +
 \Pi_3[\boldsymbol\chi],
\\\label{lita201d}
D_\mathfrak u\Phi_4(\mathfrak f, \mathfrak u)[\delta\mathfrak
u]\equiv\\\nonumber
    \boldsymbol\partial(\mathbf J\boldsymbol\Gamma)
+\boldsymbol\Omega\boldsymbol\Gamma+
(\boldsymbol\Omega\boldsymbol\Gamma)^\top + \mathbf
U_{ij}\frac{\partial \lambda_i}{\partial\xi_j}+ \mu_i\mathbf
E_i+\lambda_i\mathbf K_i+
\\\nonumber
\mathbf W^\top\delta\mathbf M\mathbf W+\Pi_4[\boldsymbol\chi,\boldsymbol\Gamma],\\
\label{lita201e} D_\mathfrak u\Phi_5(\mathfrak f, \mathfrak
u)[\delta\mathfrak u]\equiv\\\nonumber \frac{1}{(2\pi)^{n-1}}
    \int_{\mathbb T^{n-1}}\Big
    ((\mathbf W\boldsymbol\lambda)_1+
    \boldsymbol\chi_i\frac{\partial w_1}{\partial\xi_i}\Big) \, d
    \boldsymbol\xi
\end{gather}
\end{subequations}
\end{corollary}
\begin{proof}
Relations \eqref{luna8}, \eqref{luna17} lead to the following
expressions for  $ D_e\Phi$, $D_m\Phi$, and $D_M\Phi$
\begin{gather*}
D_e\Phi_1[\delta e]=\delta e, \quad D_e\Phi_i[\delta e]=0\text{~~for~~} i=2,3,4,5\\
 D_m\Phi_1[\delta  m]=\delta  m w_1, \quad
D_m\Phi_2[\delta m]= \delta m\mathbf W^\top \mathbf e_1,
\\
D_m\Phi_3[\delta  m]=0, \quad D_m\Phi_4[\mathbf m]=0,
\quad D_m\Phi_5[\delta  m]=0,\\
 D_M\Phi_1[\delta  M]=\frac{1}{2} w_1^2 \delta
 M , \quad D_M\Phi_2[\delta  M]= \mathbf W^\top\delta
 M\mathbf w,
\\
D_M\Phi_3[\delta  M]=0, \quad D_M\Phi_4[\delta  M]= \mathbf
W^\top\delta M \mathbf W, \quad D_M\Phi_5[\delta M]=0.
\end{gather*}
Substituting these expressions  into the identity
\begin{equation*}
D_\mathfrak u\boldsymbol\Phi(\mathfrak f,\mathfrak u)[\delta\mathfrak u]=
\delta \boldsymbol \Phi+
    D_e\boldsymbol \Phi(\mathfrak f, \mathfrak u)\delta e+
    D_m\boldsymbol \Phi(\mathfrak f, \mathfrak u)\delta m+
     D_M\boldsymbol \Phi(\mathfrak f, \mathfrak u)\delta M
\end{equation*}
and recalling expressions \eqref{luna26} for $\delta
\boldsymbol\Phi_i$ in Theorem \ref{structural3} we arrive at the
desired identity \eqref{lita201}.
\end{proof}

\section{Linear Problem}\label{lita}
In this section we prove the existence an approximate inverse of the
Gateux derivative of the operator $\boldsymbol\Phi$. This result plays the
crucial role in the proof of  solvability of the modified problem
\eqref{luna8}. Choose an arbitrary $\sigma\in [1/2, 1]$ and an
arbitrary integer $d\geq 1$. Next, fix $\mathfrak f=(\alpha, k)\in
\Sigma_\varrho$. In view of Corollary  \ref{structuralcor3} the operator
$\boldsymbol\Phi(\mathfrak f, \mathfrak u)$  is differentiable with
respect to $\mathfrak u$ for all $\mathfrak u=(\boldsymbol\varphi,
e,  m, M)$ satisfying the condition
$$\|\boldsymbol\varphi-\boldsymbol\varphi_0(\alpha)\|_{\sigma,
d}\leq r.$$ Notice that $\boldsymbol\Phi$  is a linear function of
the parameters $e, m,M$.  Consider the linear operator equation
\begin{equation}\label{lita1}
    D_\mathfrak u\boldsymbol\Phi(\mathfrak f, \mathfrak u)
    [\delta\mathfrak u]=\mathbf F\in \mathcal F_{\sigma,d}.
    \end{equation}
We are looking for a solution
$$
\delta\mathfrak u=(\delta\boldsymbol \varphi, \delta e, \delta
m,\delta M)$$
 to this equation in the space $\mathcal
E_{\sigma,d}$ given by Definition \ref{luna15}. Hence
$\delta\boldsymbol\varphi$ have to  satisfy the orthogonality conditions
\begin{equation} \label{lita2f}
\int_{\mathbb T^{n-1}}\delta\varphi_0\, d\xi=0, \quad \int_{\mathbb
T^{n-1}}\delta\mathbf u\, d\xi=0, \quad \int_{\mathbb T^{n-1}}\delta
W_{12}\, d\xi=0,
\end{equation}
which are similar condition \eqref{luna8g}-\eqref{luna8f} for
$\boldsymbol \varphi$. Relations \eqref{lita1}-\eqref{lita2f} form
the closed system of equations for the vector $\delta\mathfrak u$.
However, the operator $D_\mathfrak u\boldsymbol\Phi$ has the
complicated structure and its type is indefinite. In order to
simplify the equation, we exploit Corollary \ref{structuralcor3} of
Third Structural
 Theorem \ref{structural3} and reduce  equation \eqref{lita1} to the
 triangular  canonical form with constant
  coefficients in the principal part.
To this end we make the linear change of unknown functions and
introduce the new unknown vector field
$\Upsilon=(\delta\boldsymbol\beta,\psi_0, \boldsymbol\lambda,
 \boldsymbol\chi, \boldsymbol \Gamma,)$
which is connected with
$$\delta\boldsymbol \varphi=(\delta\boldsymbol\beta,\delta \varphi_0, \delta\mathbf u,
 \delta\mathbf w,\delta\mathbf W)
$$
by relations \eqref{luna27}. Using identities \eqref{lita201} we
rewrite equations \eqref{lita1}-\eqref{lita2f} in the equivalent
form
\begin{subequations}\label{lita3}
\begin{gather}\label{lita3a}
 \boldsymbol\partial \psi_0+\boldsymbol\omega^\top\cdot
  \delta\boldsymbol\beta+\Pi_1[\boldsymbol\chi, \boldsymbol \mu,
   \boldsymbol\lambda]+\delta e+ \delta\mathbf m\cdot \mathbf w+
    \frac{1}{2}\mathbf w^\top \delta\mathbf M \mathbf w=F_1,\\
    \label{lita3b}
\mathbf J\boldsymbol\partial\boldsymbol\lambda +\mathbf
T\boldsymbol\mu
    +\Pi_2[\boldsymbol\chi, \boldsymbol\lambda,\boldsymbol\Gamma]+
    \mathbf W^\top\delta\mathbf m+
    \mathbf W^\top \delta\mathbf M \mathbf w =F_2,
  \\
\nonumber \boldsymbol\mu =\nabla \psi_0+\delta\boldsymbol\beta
      \\
\label{lita3c}
 -\boldsymbol\partial \boldsymbol\chi+
 \mathbf S\boldsymbol\mu+\mathbf T^\top \boldsymbol\lambda +\Pi_3[\boldsymbol\chi]=F_3,
\\\label{lita3d}
\boldsymbol\partial(\mathbf J\boldsymbol\Gamma)
+\boldsymbol\Omega\boldsymbol\Gamma+
(\boldsymbol\Omega\boldsymbol\Gamma)^\top + \mathbf
U_{ij}\frac{\partial \lambda_i}{\partial\xi_j}+ \mu_i\mathbf
E_i+\lambda_i\mathbf K_i
\\\nonumber
+\mathbf W^\top\delta\mathbf M\mathbf W+\Pi_4[\boldsymbol\chi,\boldsymbol\Gamma]=F_4.
\end{gather}
\begin{gather}\label{lita3f}
(2\pi)^{1-n} \int_{\mathbb T^{n-1}}\Big
    ((\mathbf W\boldsymbol\lambda)_1+
    \chi_i\frac{\partial w_1}{\partial\xi_i}\Big) \, d
    \boldsymbol\xi=f_5,\\\label{lita3g}
(2\pi)^{1-n}\int_{\mathbb T^{n-1}}\big[ \psi_0-\delta \boldsymbol
\beta\cdot\mathbf u+
w_2\big(\mathbf W\boldsymbol\lambda\big)_1\big]\, d \xi=0\\
\label{lita3h} (2\pi)^{1-n}\int_{\mathbb T^{n-1}}\mathbf V^{-\top}\boldsymbol\chi\,
d\xi=0,\quad \int_{\mathbb T^{n-1}}(\mathbf
W\boldsymbol\Gamma)_{12}\, d\xi=0.
\end{gather}
\end{subequations}
Recall that the matrices $\mathbf S$ and $\mathbf T$ are given by
\eqref{luna29}.
 This matrices along with the matrices $\mathbf U$, $\mathbf E_i$, and $\mathbf K_i$
 admit estimates \eqref{luna31} and \eqref{luna32}. The linear operators
 $\Pi$ are
 defined by formulae \eqref{luna28}. They vanishes when $\Phi=0$.

\subsection{Approximate equations}
Relations \eqref{lita3} form the closed system of equations
 and the orthogonality
conditions for the functions $\psi_0$, $\boldsymbol\chi$,
$\boldsymbol \lambda$,  $\boldsymbol \Gamma$, and the parameters
 $\boldsymbol\beta$, $\delta m$, $\delta M$, and $\delta e$.
 This  system is not in a triangle form and it is
 inconvenient for the investigation.
In order to cope with this difficulty we notice that, in accordance
 with the basic principles of the KAM theory, we are looking for an approximate
 solution to equations \eqref{lita3}. This approximate solutions should satisfy  equations
 \eqref{lita3} with the accuracy
 up to the discrepancy. Hence we can omit the operators $\Pi_i$ in \eqref{lita3}.
Thus we arrive at the approximate equations
\begin{subequations}\label{lita202}
\begin{gather}\label{lita202a}
 \boldsymbol\partial \psi_0+\boldsymbol\omega^\top\cdot
  \delta\boldsymbol\beta+\delta e+ \delta\mathbf m\cdot \mathbf w+
    \frac{1}{2}\mathbf w^\top \delta\mathbf M \mathbf w=F_1,\\
    \label{lita202b}
\mathbf J\boldsymbol\partial\boldsymbol\lambda +\boldsymbol\Omega\boldsymbol\lambda+\mathbf
T\boldsymbol\mu
    + \mathbf W^\top\delta\mathbf m+
    \mathbf W^\top \delta\mathbf M \mathbf w =F_2,
  \\
\nonumber \boldsymbol\mu =\nabla \psi_0+\delta\boldsymbol\beta
      \\
\label{lita202c}
 -\boldsymbol\partial \boldsymbol\chi+
 \mathbf S\boldsymbol\mu+\mathbf T^\top \boldsymbol\lambda
 =F_3,
\\\label{lita202d}
\boldsymbol\partial(\mathbf J\boldsymbol\Gamma)
+\boldsymbol\Omega\boldsymbol\Gamma+
(\boldsymbol\Omega\boldsymbol\Gamma)^\top + \mathbf
U_{ij}\frac{\partial \lambda_i}{\partial\xi_j}+ \mu_i\mathbf
E_i+\\\nonumber\lambda_i\mathbf K_i+ \mathbf W^\top\delta\mathbf
M\mathbf W=F_4.
\end{gather}
\begin{gather}\label{lita202f}
 {(2\pi)^{1-n}}\int_{\mathbb T^{n-1}}\Big
    ((\mathbf W\boldsymbol\lambda)_1+
    \boldsymbol\chi\cdot \nabla w_1\Big) \, d
    \boldsymbol\xi=f_5,\\\label{lita202g}
{(2\pi)^{1-n}}\int_{\mathbb T^{n-1}}\big[
\psi_0-\delta\boldsymbol\beta\cdot\beta\mathbf u+
w_2\big(\mathbf W\boldsymbol\lambda\big)_1\big]\, d \xi=0\\
\label{lita202h} {(2\pi)^{1-n}}\int_{\mathbb T^{n-1}}\mathbf
V^{-\top}\boldsymbol\chi\, d\xi=0,\quad {(2\pi)^{1-n}}\int_{\mathbb T^{n-1}}(\mathbf
W\boldsymbol\Gamma)_{12}\, d\xi=0,
\end{gather}
\end{subequations}
where $\delta\mathbf m=(\delta m,0)^\top$ and $\delta\mathbf M=
\text{diag}(\delta M, 0)$.

 The  difficulty is that the "angle
variable" $w_1$ is a growing function of $\alpha$. Hence the extra
term $m w_1+2^{-1} Mw_1^2$ in the modified Hamiltonian has a
polynomial growth in $\alpha$. In other words, this means that the
system of equations \eqref{lita202} contain secular terms. To cope
with this difficulty
 we collect all secular  terms together and introduce the new parameters
 \begin{equation}\label{lita5}\begin{split}
     q= \delta e+\delta\boldsymbol\beta^\top\cdot \boldsymbol\omega+
    \alpha\delta m +2^{-1} \alpha^2 \delta M,\\
     p=\delta m+\alpha \delta M, \quad \delta\mathbf p=(\delta p, 0)
 \end{split}\end{equation}
 Recall that $\delta\boldsymbol m=(\delta m,0)$,
  $\delta\boldsymbol M=\text{diag~}(\delta M,0).$
 Thus we arrive at the following equations
\begin{subequations}\label{lita4}
\begin{gather}\label{lita4a}
 \boldsymbol\partial \psi_0+ q+  p\cdot ( w_1-\alpha)+
    \frac{1}{2} \delta M (w_1-\alpha)^2=F_1,\\
    \label{lita4b}
\mathbf J\boldsymbol\partial\boldsymbol\lambda +\boldsymbol\Omega\boldsymbol\lambda+\mathbf
T\boldsymbol\mu
    +p\mathbf W^\top\mathbf e_1+
    \delta M(w_1-\alpha)\mathbf W^\top\mathbf e_1 =F_2,\\\nonumber
    \boldsymbol\mu =\nabla \psi_0+\delta\boldsymbol\beta\\
\label{lita4c}
 -\boldsymbol\partial \boldsymbol\chi+\mathbf S\boldsymbol\mu+
 \mathbf T^\top \boldsymbol\lambda =F_3,
\\\label{lita4d}
\boldsymbol\partial(\mathbf J\boldsymbol\Gamma)
+\boldsymbol\Omega\boldsymbol\Gamma+
(\boldsymbol\Omega\boldsymbol\Gamma)^\top +\mathbf U_{ij}
\frac{\partial \lambda_i}{\partial\xi_j}+ \mu_i\mathbf
E_i+\lambda_i\mathbf K_i+   \\\nonumber +\mathbf W^\top\delta\mathbf
M\mathbf W=F_4,\quad \Gamma_{11}=-\Gamma_{22},
\end{gather}
\begin{gather}\label{lita4f}
 {(2\pi)^{1-n}}\int_{\mathbb T^{n-1}}\Big((\mathbf W\boldsymbol\lambda)_1+
    \boldsymbol\chi\cdot\nabla w_1\Big) \, d
    \boldsymbol\xi=f_5,\\\label{lita4g}
{(2\pi)^{1-n}}\int_{\mathbb T^{n-1}}\big[ \psi_0-\delta \boldsymbol
\beta\cdot\mathbf u+
w_2\big(\mathbf W\boldsymbol\lambda\big)_1\big]\, d \xi=0\\
\label{lita4h} {(2\pi)^{1-n}}\int_{\mathbb T^{n-1}}\mathbf V^{-\top}\boldsymbol\chi\,
d\xi=0,\quad {(2\pi)^{1-n}}\int_{\mathbb T^{n-1}}(\mathbf
W\boldsymbol\Gamma)_{12}\, d\xi=0.
\end{gather}

\end{subequations}
The following theorem constitutes the existence and uniqueness of
solutions to problem \eqref{lita4}
\begin{theorem}\label{lita210} Let a fixed $\sigma\in [1/4,1]$, $d\geq 2$,
and the matrix $\mathbf K_0=\mathbf S_0-\mathbf t_0\otimes \mathbf
t_0$ given by \eqref{sandt}, satisfies the condition $\text{~det~}
\mathbf K_0\neq 0$. Then there are $\varepsilon_0>0$ and  $r_0$ with
the following properties. For every
$$
(\alpha, k)\in \Sigma_{\varrho}, \quad\|\boldsymbol\varphi-
\boldsymbol\varphi_0(\alpha)\|_{\sigma, d}\leq r_0, \quad
\boldsymbol\varphi_0(\alpha)= (0,0,0,\alpha\mathbf e_1,1,0,0), \quad
|\varepsilon|\leq \varepsilon_0, \quad ,
$$
$$
0\leq \sigma_0<\sigma_1\leq \sigma, \quad \sigma_1\geq 1/4,
$$
and all
$ \mathbf F=(F_1, F_2, F_3, F_4,
f_5)\in \mathcal F_{\sigma_1,0},
$
 problem \eqref{lita4} has a unique solution $$
(\psi_0, \boldsymbol\lambda, \boldsymbol\chi, \boldsymbol\Gamma,
\delta\boldsymbol \beta, q,  p,  \delta M)\in
\mathcal A_{\sigma_0, 0}\times \mathcal A_{\sigma_0,0}^{2}\times
\mathcal A_{\sigma,d}^{n-1}\times \mathcal A_{\sigma_0,0}^4\times
\mathbb C^{n-1}\times \mathbb C^{3}.
$$
This solution admits the estimate
\begin{equation}\label{lita213}
\|(\psi_0, \boldsymbol\lambda, \boldsymbol\chi,
\boldsymbol\Gamma)\|_{\sigma_0,0}+ |(\delta\boldsymbol \beta,
q, \ p, \delta M)| \leq
{c}{(\sigma_1-\sigma_0)^{-8n-12}} \|\mathbf F\|_{\sigma_1,0}
\end{equation}
where  the constant $c$ is independent of $\varepsilon_0$, $r_0$,
and $\sigma_i$.
\end{theorem}
\begin{proof}
The proof is in Appendix \ref{prooflita205}.
\end{proof}
  Notice that problems
\eqref{lita202} and \eqref{lita4} are equivalent. The vector
$\boldsymbol \Upsilon=(\delta\boldsymbol \beta,\psi_0,
\boldsymbol\lambda, \boldsymbol\chi, \boldsymbol\Gamma)$ and the
parameters $\delta e$, $\delta m$, and $\delta M$ satisfy equations
\eqref{lita202} if and only if $\boldsymbol \Upsilon$, $\delta M$
and the parameters $\delta q$, $\delta p$,  given by \eqref{lita5},
satisfy equations \eqref{lita4}. Thus we obtain the following
\begin{corollary}\label{lita214}
Under the assumptions of Theorem \ref{lita210}, problem
\eqref{lita202} has a unique solution
$$
(\psi_0, \boldsymbol\lambda, \boldsymbol\chi, \boldsymbol\Gamma,
\delta\boldsymbol \beta,\delta e, \delta m, \delta M)\in
\mathcal A_{\sigma_0, 0}\times \mathcal A_{\sigma_0,0}^{2}\times
\mathcal A_{\sigma,d}^{n-1}\times \mathcal A_{\sigma_0,0}^4\times
\mathbb C^{n-1}\times \mathbb C^{3}.
$$
This solution admits the estimate
\begin{equation}\label{lita215}
\|(\psi_0, \boldsymbol\lambda, \boldsymbol\chi,
\boldsymbol\Gamma)\|_{\sigma_0,0}+ |(\delta\boldsymbol \beta,
 \delta M)| \leq {c}{(\sigma_1-\sigma_0)^{-8n-12}}
\|\mathbf F\|_{\sigma_1,0}
\end{equation}
\begin{equation}\label{lita216}
|( \delta e, \delta m)| \leq
(1+\alpha^2){c}{(\sigma_1-\sigma_0)^{-8n-12}} \|\mathbf
F\|_{\sigma_1,0}
\end{equation}

\end{corollary}

\subsection{ Differential of $\boldsymbol \Phi$. Approximate
inverse}\label{asol} In this section we construct an approximate
inverse to the operator $D\boldsymbol\Phi$. Our considerations are
based on the following construction. For given $\boldsymbol\varphi$
denote by $\Xi(\boldsymbol\varphi)$ the linear operator  defined by
the equalities
\begin{subequations}\label{asol1}
\begin{equation}\label{asol0}
\Xi:\boldsymbol\Upsilon=(\delta\boldsymbol\beta, \boldsymbol\psi_0,
\boldsymbol\lambda, \boldsymbol\chi, \boldsymbol\Gamma) \to
\delta\boldsymbol\varphi=(\delta\boldsymbol\beta, \delta\varphi_0,
\delta\mathbf u, \delta\mathbf w, \delta W_{11}, \delta W_{12},
\delta W_{21})
\end{equation}
\begin{gather}
\label{asol1c} \delta\mathbf u\,=\, \chi_i\frac{\partial}{\partial
\xi_i} \,(\boldsymbol\xi +
\mathbf u),\\
\label{asol1d} \delta\mathbf w\,=\, \mathbf W \boldsymbol
\lambda+\chi_i\, \frac{\partial}{\partial \xi_i} \mathbf
w,\\
\label{asol1e} \delta\mathbf W\,=\, \mathbf W \boldsymbol
\Gamma+\chi_i\, \frac{\partial}{\partial \xi_i} \mathbf
W,\\
\label{asol1a}  \delta \varphi_0=\psi_0 +w_2\delta w_1+
\Big(\frac{\partial\varphi_0}{\partial \xi_i} -w_2 \frac{\partial
w_1}{\partial \xi_i}\Big)\chi_i- \mathbf u\cdot \delta\boldsymbol
\beta .
\end{gather}
\end{subequations}
In view of the Second structural Theorem \ref{sima12} the operator
$\boldsymbol\Xi$ has the inverse $\boldsymbol\Xi^{-1}$ defined by
\begin{subequations}\label{asol15}
\begin{equation}\label{asol15a0}
\Xi^{-1}\delta\boldsymbol\varphi=(\delta\boldsymbol\beta,
\delta\varphi_0, \delta\mathbf u, \delta\mathbf w, \delta W_{11},
\delta W_{12}, \delta
W_{21})\boldsymbol\Upsilon=(\delta\boldsymbol\beta,
\boldsymbol\psi_0, \boldsymbol\lambda, \boldsymbol\chi,
\boldsymbol\Gamma)
\end{equation}

\begin{gather}\label{asol15a}
\boldsymbol\chi\,= \,\mathbf V^\top \,\delta\mathbf
u\\\label{asol15b} \boldsymbol\lambda\,=\, \mathbf
W^{-1}\delta\mathbf w-\chi_i\,\mathbf W^{-1}
\frac{\partial}{\partial \xi_i} \mathbf w\\\label{asol15c}
\boldsymbol\Gamma\,=\, \mathbf W^{-1}\delta\mathbf W-\chi_i\,\mathbf
W^{-1} \frac{\partial}{\partial \xi_i} \mathbf W,\\\label{asol15d}
\boldsymbol\mu\,=\, \mathbf V^{-1}\Big(\delta\mathbf v+\chi_i\,
\frac{\partial}{\partial \xi_i} \mathbf v-
\boldsymbol\Lambda\boldsymbol\lambda\Big)\\\label{sima15e}
\psi_0=\delta\varphi_0-\chi_i \frac{\partial
\varphi_0}{\partial\xi_i}+ w_2\chi_i \frac{\partial
w_1}{\partial\xi_i} +\delta\boldsymbol\beta^\top \cdot \mathbf u.
\end{gather}
\end{subequations}
 Recall the notation
 $$
 \mathfrak f=(\alpha,k), \quad \mathfrak u=(\boldsymbol\varphi, e, m,
 M)
 $$
Introduce also the notation
$$
\delta\mathfrak u=(\delta\boldsymbol \varphi, \delta e, \delta m,
\delta M).
$$

\begin{definition}\label{asoldef} Let  $\mathcal R(\mathfrak f, \mathfrak u)$ be a linear operator defined by the equality
\begin{equation}\label{asol4}
    \mathcal R(\mathfrak f, \mathfrak u)[\mathbf F]=
    \big(\boldsymbol\Xi(\boldsymbol\varphi)[\boldsymbol\Upsilon],
    \delta e, \delta m, \delta M)
\end{equation}
where $\boldsymbol\Upsilon$, $\delta e$, $\delta m$, $\delta M$ is a solution
to the approximate problem  \eqref{lita202}. Notice that the
coefficients of  equations \eqref{lita202} are completely defined by
$\mathfrak f$ and $\mathfrak u$.
\end{definition}

The following proposition constitutes the basic properties  of the
operator $\mathcal R$. In particular, it follows that $\mathcal R$
is an approximate inverse to the operator $D\boldsymbol \Phi$.

\begin{proposition}\label{asol6}
Let $\sigma\in (0,1]$. Then there are $\varepsilon_0>0$ and $r_0>0$
with the following properties. For all $\mathfrak f$, $\mathfrak u$,
and $\varepsilon$, and $\sigma_i$ satisfying the conditions
$$
\mathfrak f\in \Sigma_\varrho, |\varepsilon|\leq \varepsilon_0,
\|\boldsymbol\varphi-\boldsymbol \varphi_0(\alpha)\|_{\sigma,d}\leq
r_0,
$$
$$
0\leq \sigma_0<\sigma_1\leq \sigma,
$$
the operator $\mathcal R(\mathfrak f,\mathfrak u):\mathcal
F_{\sigma_1,0}\to \mathcal E_{\sigma_0,d}$ is bounded and admits the
estimate
\begin{equation}\label{asol7}
    \|\mathcal R(\mathfrak f, \mathfrak u)
    [\delta\mathfrak u]\|_{\sigma_0,d}\leq  c(1+\alpha^2)
    (\sigma_1-\sigma_0)^{-8n-12-d}
\|\delta\mathfrak u\|_{\mathcal E_{\sigma_1,d}}.
\end{equation}
Moreover, it satisfies the inequalities
\begin{multline}\label{asol8x}
    \big\|\,\,\big(\mathcal R(\mathfrak f, \mathfrak u)D\Phi(\mathfrak f, \mathfrak u)-\text{\rm Id}\big)
    \,[\delta \mathfrak u]
    \,\,\big\|_{\sigma_0,d}\leq \\ c
    (\sigma_1-\sigma_0)^{-8n-13-2d}
    \|\boldsymbol\Phi(\mathfrak f, \mathfrak u)\|_{\sigma,0}\,\,
\|\delta\mathfrak u\|_{\mathcal E_{\sigma_1,d}}.
\end{multline}
\begin{multline}\label{asol9x}
    \big\|\,\,\big(D\Phi(\mathfrak f, \mathfrak u)
    \mathcal R( \mathfrak f, \mathfrak u)-\text{\rm Id}\big)
   \, [\mathbf F]
    \,\,\big\|_{\sigma_0,0}\leq \\ c
    (\sigma_1-\sigma_0)^{-8n-14}
    \|\boldsymbol\Phi(\mathfrak f, \mathfrak u))\|_{\sigma_1,0}\,\,
\|\mathbf F\|_{\sigma_1,0},
\end{multline}
i.e., $\mathcal R$ is an approximate inverse to $D\boldsymbol\Phi$.
\end{proposition}
\begin{proof}
The proof is based on the following auxiliary lemma .
\begin{lemma}\label{asol2}
Let a fixed $\sigma\in (0,1)$ and  $d\geq 2$ and $r_0<1/2$. Then for
every
$$
\mathfrak f\in \Sigma_{\varrho}, \quad\|\boldsymbol\varphi-
\boldsymbol\varphi_0(\alpha)\|_{\sigma, d}\leq r_0, \quad
\boldsymbol\varphi_0(\alpha)= (0,0,0,\boldsymbol\alpha,1,0,0),
$$
and every
$$
0\leq \sigma_0<\sigma_1\leq \sigma,
$$
the operator
$$\boldsymbol \Xi:\mathbb C\times \mathcal A_{\sigma_1,0}\times
\mathcal A_{\sigma_1,0}^2\times \mathcal A_{\sigma_1,0}^{n-1}\times
\mathcal A_{ \sigma_1,0}^4\to \mathbb C\times \mathcal
A_{\sigma_0,d}\times \mathcal A_{\sigma_0,d}^2\times \mathcal
A_{\sigma_1,d}^{n-1}\times \mathcal A_{ \sigma_0,d}^3
$$ is well defined and admits the estimate
\begin{equation}\label{asol3}
    \big\|\,\boldsymbol
    \Xi(\boldsymbol\varphi)\,[\boldsymbol\Upsilon]\,\big\|_{\sigma_0,d}
    \leq c  (\sigma_1-\sigma_0)^{-d}\,\,\|\boldsymbol\Upsilon\|_{\sigma_1,0}.
\end{equation}
The  inverse operator
$$\boldsymbol \Xi^{-1}:\mathbb C\times \mathcal A_{\sigma_1,0}\times
\mathcal A_{\sigma_1,0}^2\times \mathcal A_{\sigma_1,0}^{n-1}\times
\mathcal A_{ \sigma_1,0}^4\to \mathbb C\times \mathcal
A_{\sigma_0,d}\times \mathcal A_{\sigma_0,d}^2\times \mathcal
A_{\sigma_1,d}^{n-1}\times \mathcal A_{ \sigma_0,d}^3
$$ is well defined and admits the estimate
\begin{equation}\label{asol3bis}
    \big\|\,\boldsymbol
    \Xi^{-1}(\boldsymbol\varphi)\,[\delta \boldsymbol\varphi]\,\big\|_{\sigma_0,d}
    \leq c  (\sigma_1-\sigma_0)^{-d}\, \|\delta\boldsymbol\varphi\|_{\sigma_1,0}.
\end{equation}

\end{lemma}
\begin{proof}
In order to avoid repetitions we prove estimate \eqref{asol3}. The
proof of estimate \eqref{asol3bis} is similar. Notice that for
$d\geq 1$ and $r_0\leq 1/2$, the equality
$W_{22}=W_{11}^{-1}(1-W_{1,2}W_{21})$ implies
$$
\|w_2\|_{\sigma,0}+\|\nabla w_1\|_{\sigma,0}+\|\nabla\boldsymbol
W\|_{\sigma,0} +\|\boldsymbol W-\mathbf I\|_{\sigma,0}\leq
c\|\boldsymbol \varphi-\boldsymbol\varphi(\alpha)\|_{\sigma,d}
$$
Obviously we have
$$
\|w_2\|_{\sigma,0}+\|\nabla
w_1\|_{\sigma,0}+\|\nabla(\boldsymbol\xi+\boldsymbol u)\|
+\|\boldsymbol W\|_{\sigma,0}\leq c\|\boldsymbol
\varphi-\boldsymbol\varphi(\alpha)\|_{\sigma,d}\leq c
$$
Since $\mathcal A_{\sigma_1,0}$ is a Banach algebra, it follows from
this and \eqref{asol1} that
\begin{equation*}
    \|\boldsymbol
    \Xi(\boldsymbol\varphi)[\boldsymbol\Upsilon]\|_{\sigma_1,0}
    \leq c  \|\Upsilon]\|_{\sigma_1,0}.
\end{equation*}
It remains to note that
\begin{equation*}
    \|\boldsymbol
    \Xi(\boldsymbol\varphi)[\boldsymbol\Upsilon]\|_{\sigma_0,d}
    \leq c (\sigma_1-\sigma_0)^{-d} \|\boldsymbol
    \Xi(\boldsymbol\varphi)[\boldsymbol\Upsilon]\|_{\sigma_1,0}.
\end{equation*}
\end{proof}
Let us turn to the proof of the proposition. The proof falls into
three steps.

{\bf Step 1.} We begin with proving of estimate \eqref{asol7}. It
follows from Corollary \ref{lita214} that, under the assumptions of
Proposition \ref{asol6}, for a suitable choice of $\varepsilon_0$
and $r_0$, problem \eqref{lita202} has the unique  solution $\Upsilon
=(\delta \boldsymbol\beta, \psi_0,
\boldsymbol\lambda,\boldsymbol\chi, \boldsymbol \Gamma)$, and $\delta
e$, $\delta m$, $\delta M$ satisfying the inequality
\begin{multline}\label{asol11x}
    \|\boldsymbol \Upsilon\|_{(\sigma_0+\sigma_1)/2,0}+|\delta e|+
    |\delta m|+|\delta M|\\\leq c(1+\alpha^2)
    (\sigma_1-\sigma_0)^{-8n-12}\,\|\mathbf F\|_{\sigma_1,0}.
\end{multline}
Using this estimate and applying Lemma \ref{asol2} with $\sigma_1$
replaced by $(\sigma_0+\sigma_1)/2$ we obtain
\begin{equation}\label{asol10}\begin{split}
\big\|\,\boldsymbol
    \Xi(\boldsymbol\varphi)\,[\boldsymbol\Upsilon]\,\big\|_{\sigma_0,d}
    \leq c\,
    (\sigma_1-\sigma_0)^{-d}\,\|\boldsymbol\Upsilon]\|_{(\sigma_0+\sigma_1)/2,0}\\\leq
   c(\sigma_1-\sigma_0)^{-8n-12-d}\,\|\mathbf F\|_{\sigma_1,0} .
\end{split}\end{equation}
Combining \eqref{asol11x} and  \eqref{asol10} and recalling
Definition \ref{asoldef} we arrive at \eqref{asol7}.

{\bf Step 2.} Now our task is  to prove  estimate \eqref{asol8}.
Assume that $\varepsilon$, $r_0$ meet all requirements  of
Corollary \ref{lita214}. Choose an arbitrary
$$\delta\boldsymbol\varphi
=(\delta\boldsymbol\beta, \delta\varphi_0, \delta \mathbf u, \delta
W_{11}, \delta W_{12}, \delta W_{21})\text{~~ and~~}(\delta e,
\delta m, \delta M)$$
 such that
$$
\delta \mathfrak u\equiv ( \delta\boldsymbol\varphi, \delta e,
\delta m,\delta M)\in \mathcal E_{\sigma_1,d}.
$$
Set
\begin{equation}\label{asol16}
\boldsymbol \Upsilon=\boldsymbol\Xi^{-1}(\boldsymbol\varphi) [
\delta\boldsymbol\varphi]
\end{equation}
By construction, the vector $$\boldsymbol\Upsilon=
(\delta\boldsymbol\beta, \psi_0, \boldsymbol\lambda,
\boldsymbol\chi, \boldsymbol\Gamma)$$
 is connected with
$\delta\boldsymbol\varphi$ by relations \eqref{luna27}. Hence
$\boldsymbol\Upsilon$ and the parameters $(\delta e, \delta m,
\delta M)$ meet all requirements of Corollary \ref{structuralcor3}.
It follows that they satisfy identities  \eqref{lita201}. Notice that
relations \eqref{lita201} can be regarded as system of equations
\eqref{lita202a}-\eqref{lita202f} with the right hand sides
\begin{equation}\begin{split}\label{asol11}
F_1=-\Pi_1[\boldsymbol\chi,
 \boldsymbol\lambda, \boldsymbol\mu]+D_{\mathfrak u}
\Phi_1(\mathfrak f, \mathfrak u)[\delta\mathfrak u],\\
F_2=-\Pi_2[\boldsymbol\chi,
 \boldsymbol\lambda, \boldsymbol \Gamma]+D_{\mathfrak u}
\Phi_2(\mathfrak f, \mathfrak u)[\delta\mathfrak u],\\
F_3=-\Pi_3[\boldsymbol\chi]+D_{\mathfrak u} \Phi_3(\mathfrak f,
\mathfrak u)[\delta\mathfrak u],\\
F_4=-\Pi_2[\boldsymbol\chi,
 \boldsymbol \Gamma]+D_{\mathfrak u}
\Phi_4(\mathfrak f, \mathfrak u)[\delta\mathfrak u],\\
f_5=D_{\mathfrak u} \Phi_5(\mathfrak f, \mathfrak u)[\delta\mathfrak
u]
\end{split}\end{equation}
 Moreover, since $
\delta\mathfrak u\in \mathcal E_{\sigma_1, d}$, its component $
\delta \boldsymbol \varphi$ satisfies the integral relations
\eqref{lita2f} which, in view of \eqref{asol1c}-\eqref{asol1a}, are
equivalent to the integral conditions
\eqref{lita202g}-\eqref{lita202h}. Therefore,
 the vector $\boldsymbol\Upsilon$ and the parameters
 $(\delta e, \delta m,\delta M)$ satisfy equations
 \eqref{lita202} with the right hand sides given by \eqref{asol11}.

 On the other hand,  Definition \ref{asoldef} implies the identity
\begin{equation}\label{asol12}
    \mathcal R(\mathfrak f, \mathfrak u)\, \big[\,D_{\mathfrak u}
    \boldsymbol \Phi(\mathfrak f, \mathfrak u)[\delta\mathfrak u]\,\big]=
    \big(\boldsymbol\Xi(\boldsymbol\varphi)[\boldsymbol\Upsilon'],
    \delta e', \delta m', \delta M')
\end{equation}
where the vector $\boldsymbol\Upsilon'$ and the parameters
 $(\delta e', \delta m',\delta M')$ satisfy equations
 \eqref{lita202} with the right hand sides
\begin{equation}\label{asol13}
    F'_i= D\Phi_i(\mathfrak f, \mathfrak u)[\delta\mathfrak u], \,\,
    i=1,\dots,4,\,\, f'_5=
    D\Phi_5(\mathfrak f, \mathfrak u)[\delta\mathfrak u]
\end{equation}
It follows from  \eqref{asol16} that
\begin{equation}\label{asol17}
    \mathcal R(\mathfrak f, \mathfrak u)\big[\,D_{\mathfrak u}
    \boldsymbol \Phi(\mathfrak f, \mathfrak u)[\delta\mathfrak u]\,\big]\,-\,
  \delta\mathfrak u  =\\
    \big(\boldsymbol\Xi(\boldsymbol\varphi)[\tilde{\boldsymbol\Upsilon}],
    \tilde{\delta e}, \tilde{\delta m}, \tilde{\delta M})
\end{equation}
Where
$$
\tilde{\boldsymbol\Upsilon}=\boldsymbol\Upsilon'-\boldsymbol\Upsilon,
\quad \tilde{\delta e}=\delta e'-\delta e, \quad \tilde{\delta m}=\delta
m'-\delta m,\quad \tilde{\delta M}=\delta M'-\delta M.
$$
Notice that $(\boldsymbol\Upsilon,\delta e, \delta m,\delta M)$
satisfies equations \eqref{lita202} with the right hand sides
\eqref{asol11}. On the other hand,  $(\boldsymbol\Upsilon',\delta
e', \delta m',\delta M')$ satisfies equations \eqref{lita202} with
the right hand sides \eqref{asol13}. It follows from this that
$(\tilde{\boldsymbol\Upsilon},\tilde{\delta e}, \tilde{\delta m},\tilde{\delta M})$
satisfies equations \eqref{lita202} with the right hand sides
\begin{equation}\label{asol18}
    \tilde{F}_i=\Pi_i[\boldsymbol\Upsilon ],\, i=1, \dots, \, 4,  \,\, \tilde{f}_5=0.
\end{equation}
Here the operators $\Pi_i$ are given by formulae \eqref{luna28}:
\begin{subequations}\label{asol20}
\begin{equation}\label{asol20a}
    \Pi_1[\boldsymbol\Upsilon]
    =\frac{\partial\Phi_1}{\partial \xi_i}\chi_i+\Phi_2^\top
    \cdot \boldsymbol\lambda+\Phi_3^\top\cdot\boldsymbol\mu,
\end{equation}
\begin{equation}\label{asol20b}
    \Pi_2[\boldsymbol\chi, \boldsymbol\lambda, \boldsymbol\Gamma]=
    \chi_i\frac{\partial\Phi_2}{\partial \xi_i}+\boldsymbol
    \Gamma^\top \Phi_2+\Phi_4\boldsymbol\lambda+(\mathbf J\boldsymbol\lambda)'_\xi\Phi_3,
    \end{equation}
\begin{equation}\label{asol20c}
     \Pi_3[\boldsymbol\chi]=\chi_i\frac{\partial\Phi_3}{\partial \xi_i}-
     \boldsymbol\chi'_\xi \Phi_3,
\end{equation}
\begin{equation}\label{asol20d}
  \Pi_4[\boldsymbol\chi,\boldsymbol\Gamma]=\chi_i\frac{\partial\Phi_4}{\partial \xi_i}
+\Phi_{3,i}\frac{\partial}{\partial \xi_i}(\mathbf
J\boldsymbol\Gamma)+ \Phi_4\boldsymbol\Gamma+
(\Phi_4\boldsymbol\Gamma)^\top,
\end{equation}
where $\Phi_i=\Phi_i(\mathfrak f, \mathfrak u)$ and $\boldsymbol\mu
=\delta\boldsymbol\beta +\nabla\psi_0$.
\end{subequations}
We thus get
\begin{subequations}\label{asol51}
\begin{gather}\label{asol51a}
 \boldsymbol\partial \tilde{\psi}_0+\boldsymbol\omega^\top\cdot
  \delta\tilde{\boldsymbol\beta}+\tilde{\delta e}+
  \tilde{\delta\mathbf m}\cdot \mathbf w+
    \frac{1}{2}\mathbf w^\top \tilde{\delta\mathbf M} \mathbf w=\tilde{F}_1,\\
    \label{asol51b}
\mathbf J\boldsymbol\partial\tilde{\boldsymbol\lambda}+
\boldsymbol\Omega \tilde{\boldsymbol\lambda} +\mathbf
T\tilde{\boldsymbol\mu}
    + \tilde{\delta m}\mathbf W^\top\mathbf e_1+
    \mathbf W^\top \tilde{\delta\mathbf M} \mathbf w =\tilde{F}_2,
  \\
\nonumber \tilde{\boldsymbol\mu} =\nabla \tilde{\psi}_0+\delta\tilde{\boldsymbol\beta}
      \\
\label{asol51c}
 -\boldsymbol\partial \tilde{\boldsymbol\chi}+
 \mathbf S\tilde{\boldsymbol\mu}+\mathbf T^\top \tilde{\boldsymbol\lambda}
 =\tilde{F}_3,
\\\label{asol51d}
\boldsymbol\partial(\mathbf J\tilde{\boldsymbol\Gamma})
+\boldsymbol\Omega\tilde{\boldsymbol\Gamma}+
(\boldsymbol\Omega\tilde{\boldsymbol\Gamma})^\top + \mathbf
U_{ij}\frac{\partial \lambda_i}{\partial\xi^*_j}+ \tilde{\mu}_i\mathbf
E_i+\\\nonumber\tilde{\lambda}_i\mathbf K_i+ \mathbf W^\top\delta\tilde{\mathbf
M}\mathbf W=\tilde{F}_4.
\end{gather}
\begin{gather}\label{asol51f}
 \int_{\mathbb T^{n-1}}\Big
    ((\mathbf W\tilde{\boldsymbol\lambda})_1+
    \tilde{\boldsymbol\chi}\cdot \nabla w_1\Big) \, d
    \boldsymbol\xi=0,\\\label{asol51g}
\int_{\mathbb T^{n-1}}\big[
\tilde{\psi}_0-\delta\tilde{\boldsymbol\beta}\cdot\mathbf u+
w_2\big(\mathbf W\tilde{\boldsymbol\lambda}\big)_1\big]\, d \xi=0\\
\label{asol51h} \int_{\mathbb T^{n-1}}\mathbf
V^{-\top}\tilde{\boldsymbol\chi}\, d\xi=0,\quad \int_{\mathbb
T^{n-1}}(\mathbf W\tilde{\boldsymbol\Gamma})_{12}\, d\xi=0,
\end{gather}
\end{subequations}
where  $\delta\tilde{\mathbf M}=
\text{diag}(\delta \tilde{M}, 0)$.
Let us estimate $\tilde{\mathbf F}$. Equality $\boldsymbol\Upsilon
=\boldsymbol\Xi^{-1} \boldsymbol\delta \varphi$ and estimate
\eqref{asol3bis} imply
$$
\|\boldsymbol\Upsilon\|_{(\sigma_1+\sigma_0)/2, d} \leq
c(\sigma_1-\sigma_0)^{-d}\,\,\|\delta \boldsymbol\varphi\|_{\sigma_1,0}
$$
It follows from this that
\begin{equation}\label{asol22}
\|\boldsymbol\Upsilon\|_{(2\sigma_1+\sigma_0)/3, 0}+\|\partial_\xi
\boldsymbol\Upsilon\|_{(2\sigma_1+\sigma_0)/3, 0}\leq\\
c(\sigma_1-\sigma_0)^{-d} \|\delta \boldsymbol
\varphi\|_{\sigma_1,0}.
\end{equation}
Next we have
\begin{equation}\label{asol21}
    \|\boldsymbol\Phi\|_{(2\sigma_1+\sigma_0)/3,0}+
    \|\partial_\xi\boldsymbol \Phi\|_{(2\sigma_1+\sigma_0)/3,0}
   \leq c(\sigma_1-\sigma_0)^{-1} \|\boldsymbol\Phi\|_{\sigma_1,0}.
\end{equation}
Combining estimates \eqref{asol22} and \eqref{asol21} and recalling
formulae \eqref{asol20} we arrive at
\begin{equation}\label{asol23}
    \|\Pi[\Upsilon]\|_{(2\sigma_1+\sigma_0)/3, 0}\leq
    c(\sigma_1-\sigma_0)^{-d-1} \|\boldsymbol\Phi\|_{\sigma_1,0}\,\,
\|\delta \boldsymbol\varphi\|_{\sigma_1,0}
\end{equation}
Since $\delta \boldsymbol\varphi$ is the component of the vector
$\mathfrak u$, it follows from this and \eqref{asol18} that
\begin{equation}\label{asol24}
 \|\tilde{\mathbf F}\|_{(2\sigma_1+\sigma_0)/3, 0}\leq
    c(\sigma_1-\sigma_0)^{-d-1} \, \|\boldsymbol\Phi\|_{\sigma_1,0}\,\,
\|\delta \mathfrak u\|_{\sigma_1,0}
\end{equation}
 Applying Corollary \ref{lita214} to problem \eqref{asol51} we conclude
that it has the unique solution. Moreover, estimates
\eqref{lita215} and  \eqref{lita216} in this corollary imply the
estimate
\begin{equation*}\begin{split}
c(1+|\alpha^2|)(\sigma_1-\sigma_0)^{-8n-12} \|\tilde{\mathbf
F}\|_{(2\sigma_1+\sigma_0)/3, 0}.
\end{split}\end{equation*}

This result along with  inequality \eqref{asol24} leads to the
estimate
 \begin{equation*}\begin{split}
    \|\tilde{\boldsymbol\Upsilon}\|_{(\sigma_1+\sigma_0)/2,0}+
    |\delta \tilde{e}|+|\delta \tilde{m}|+|\delta \tilde{M}|\leq\\
c(1+|\alpha^2|)(\sigma_1-\sigma_0)^{-8n-13-d}\,
\|\boldsymbol\Phi\|_{\sigma_1,0} \,\|\delta \mathfrak
u\|_{\sigma_1,0}.
\end{split}\end{equation*}
From this, relation \eqref{asol17} and estimate \eqref{asol3} for
the norm of the operator $\boldsymbol\Xi$ we obtain
\begin{gather*}
 \big\|\,\mathcal R(\mathfrak f, \mathfrak u)\big[\,D_{\mathfrak u}
    \boldsymbol \Phi(\mathfrak f, \mathfrak u)\,[\delta\mathfrak u]\,\big]-
  \delta\mathfrak u\,\big\|_{\sigma_0, d} =\\
    \big\|\, \boldsymbol\Xi(\boldsymbol\varphi)
    [\tilde{\boldsymbol\Upsilon}]\,\big\|_{\sigma_0,d}+|
    \delta \tilde{e}|+|\delta \tilde{m}|+| \delta \tilde{M}|\leq\\
(\sigma_1-\sigma_0)^{-d}
\|\tilde{\boldsymbol\Upsilon}\|_{(\sigma_1+\sigma_0)/2,0}+
    |\delta \tilde{e}|+|\delta \tilde{m}|+|\delta \tilde{M}|\leq\\
     c(1+|\alpha^2|)(\sigma_1-\sigma_0)^{-8n-13-2d}
\|\boldsymbol\Phi\|_{\sigma_1,0} \|\delta \mathfrak
u\|_{\sigma_1,0}.
\end{gather*}
which obviously leads to \eqref{asol8x}.

{\bf Step 3.} It remains to prove estimate \eqref{asol9x}. Choose an
arbitrary $\mathbf F\in \mathcal F_{\sigma_1,0}$ and set
$\delta\mathfrak u:= \mathcal R(\mathfrak f, \mathfrak u)[\mathbf
F]$. It follows from the definition \ref{asoldef} that
$$
\delta\mathfrak u=(\boldsymbol\Xi[\boldsymbol\Upsilon], \delta e,
\delta m, \delta M),
$$
where $(\boldsymbol\Upsilon, \delta e, \delta m, \delta M)$ is a
solution to problem \eqref{lita202}. On the other hand, in view of
Corollary \ref{structuralcor3}, the vector
$$D_{\mathfrak
u}\boldsymbol \Phi(\mathfrak f, \mathfrak u)\mathcal R(\mathfrak f,
\mathfrak u)[\mathbf F]\equiv D_{\mathfrak u}\boldsymbol
\Phi(\mathfrak f, \mathfrak u)\delta\mathfrak u
$$
satisfies identities \eqref{lita201}. Combining \eqref{lita201} and
\eqref{lita202} we arrive at the identity
\begin{equation}\label{asol26}
D_{\mathfrak u}\boldsymbol \Phi(\mathfrak f, \mathfrak u)\mathcal
R(\mathfrak f, \mathfrak u)[\mathbf F]-\mathbf F=\\
(\Pi_1[\boldsymbol \Upsilon], \,\Pi_2[\boldsymbol
\Upsilon],\,\Pi_3[\boldsymbol \Upsilon],\,\Pi_4[\boldsymbol
\Upsilon],\,\Pi_5[\boldsymbol \Upsilon]).
\end{equation}
Applying  Corollary \ref{lita214} to equations \eqref{lita202} we
conclude that
\begin{equation*}
    \|\boldsymbol\Upsilon\|_{(\sigma_0+\sigma_1)/2,0}+
   |\delta e|+|\delta m|+|\delta M|\leq\\
     c(1+|\alpha^2|)(\sigma_1-\sigma_0)^{-8n-12}
     \|\mathbf F\|_{\sigma_1,0},
\end{equation*}
which gives
\begin{equation}\label{asol27}
\|\boldsymbol\Upsilon\|_{\sigma_0, 0}+\|\partial_\xi
\boldsymbol\Upsilon\|_{\sigma_0, 0}\leq\\
c(\sigma_1-\sigma_0)^{-8n-13} \|\mathbf F\|_{\sigma_1,0}.
\end{equation}
Next we have
\begin{equation}\label{asol28}
    \|\boldsymbol\Phi\|_{\sigma_0,0}+
    \|\partial_\xi\boldsymbol \Phi\|_{\sigma_0,0}
   \leq c(\sigma_1-\sigma_0)^{-1} \|\boldsymbol \Phi\|_{\sigma_1,0}.
\end{equation}
Combining estimates \eqref{asol27} and \eqref{asol28} and recalling
formulae \eqref{asol20} we arrive at
\begin{equation}\label{asol29}
    \|\Pi[\Upsilon]\|_{\sigma_0, 0}\leq
    c(\sigma_1-\sigma_0)^{-8n-14} \|\Phi_i\|_{\sigma_1,0}
\|\mathbf F\|_{\sigma_1,0}.
\end{equation}
It remains to note that desired estimate \eqref{asol9x} is a
straightforward consequence of estimate \eqref{asol29} and equality
\eqref{asol26}.
\end{proof}

\section{Implicit function theorem}\label{ira}

\subsection{Nash-Moser-Zehnder Implicit Function  Theorem}
We prove the local solvability of  operator equation \eqref{luna8x}
 by using the nash-Moser implicit function theorem. There are many different
  versions of this celebrated theorem. Our considerations
  are based on the
  version of the Nash-Moser implicit function theorem
  proposed by Zhender, see \cite{Nirenberg}, \cite{Zender}.

Fix $d\geq 2$. Recall Definitions \ref{luna15} and \ref{luna16} of
spaces $\mathcal E_{\sigma,d}$ and $\mathcal F_{\sigma,d}$. Let us
consider the abstract operator equation
\begin{equation}\label{ira9}
    \boldsymbol\Phi(\mathfrak f, \mathfrak u)=0, \quad
    \mathfrak f\in \mathbb C^2, \quad \mathfrak u\in \mathcal
    E_{\sigma, d}.
\end{equation}
Here $\boldsymbol\Phi$ is a smooth operator. Assume that the
operator $\boldsymbol\Phi$ satisfies the following conditions, cf.
\cite{Nirenberg}, ch.6.1.

For fixed $R>0$, $N>0$, and $(\mathfrak f_0, \mathfrak u_0)\in
\mathbb C^2\times \mathcal E_{1,d}$ denote by $\mathcal
B_\sigma(\mathfrak f_0, \mathfrak u_0)\subset \mathbb C^2\times
\mathcal E_{\sigma,d}$ the ball
\begin{equation}\label{ira3}
    \mathcal B_\sigma(\mathfrak f_0, \mathfrak u_0)=\{(\mathfrak f, \mathfrak u):\,\,
    |\mathfrak f-\mathfrak f_0|<N,\quad
     \|\mathfrak u-\mathfrak u_0\|_{\sigma,d}<R\}.
\end{equation}

\begin{itemize}
\item[{\bf H.3}] The mapping $\boldsymbol \Phi$ is
defined in $\mathcal B_0(\mathfrak f_0,\mathfrak u_0)$. Moreover,
the mapping $\boldsymbol\Phi: \mathcal B_\sigma(\mathfrak
f_0,\mathfrak u_0)\to \mathcal F_{\sigma,0}$ is continuous for all
$\sigma\in (0,1]$. For every $\sigma'<\sigma$, the mapping
$\boldsymbol\Phi(\mathfrak f, \cdot): \mathcal B_\sigma (\mathfrak
f_0,\mathfrak u_0)\cap \mathcal E_{\sigma,d}\to \mathcal
F_{\sigma',0}$, $\sigma'<\sigma$
 is differentiable. For every $(\mathfrak f, \mathfrak u)\in
  \mathcal B_\sigma(\mathfrak f_0, \mathfrak u_0)$,
 the quantity
 $$
Q(\mathfrak f; \mathfrak u, \mathfrak v)= \boldsymbol\Phi(\mathfrak
f, \mathfrak u)- \boldsymbol\Phi(\mathfrak f, \mathfrak
v)-D_{\mathfrak v}\boldsymbol \Phi(\mathfrak f, \mathfrak v)
(\mathfrak u-\mathfrak v)
$$
admits the estimate
\begin{equation}\label{ira4}
    \|Q(\mathfrak f; \mathfrak u, \mathfrak v)\|_{\sigma',0}\leq c_0(\sigma-\sigma')^{-2\tau}
    \|\mathfrak u-\mathfrak v\|_{\sigma,d}^2.
\end{equation}
\item[{\bf H.4}]  For every $\sigma\in (0,1]$ and
$(\mathfrak f, \mathfrak u), (\mathfrak g, \mathfrak u)\in \mathcal
B_\sigma(\mathfrak f_0, \mathfrak u_0)$, we have
\begin{equation}\label{ira5}
    \|\boldsymbol\Phi(\mathfrak f, \mathfrak u)-
\boldsymbol\Phi(\mathfrak g, \mathfrak u)\|_{\sigma,0}\leq c_0
    |\mathfrak f-\mathfrak g|.
\end{equation}
\item[{\bf H.5}]  For every $\sigma\in (0,1]$, every
$0\leq \sigma'<\sigma$, and every $(\mathfrak f, \mathfrak u),\in
\mathcal B_\sigma(\mathfrak f_0,\mathfrak u_0)$,
 there exists the linear continuous mapping
$\mathcal R(\mathfrak f, \mathfrak u): \mathcal F_{\sigma,0}\to
\mathcal E_{\sigma',d}$ such that for all $\mathfrak h\in \mathcal
F_{\sigma,0}$ and all $\mathfrak v\in \mathcal E_{\sigma,d}$,
\begin{equation}\label{ira6}
\|\mathcal R(\mathfrak f, \mathfrak u)\mathfrak h\|_{\sigma',d} \leq
c_0(\sigma-\sigma')^{-\gamma}\|\mathfrak h\|_{\sigma,0}
\end{equation}
\begin{multline}\label{ira7}
\|\big(D_{\mathfrak u} \boldsymbol\Phi(\mathfrak f, \mathfrak u)\,
\mathcal R(\mathfrak f, \mathfrak u)-\mathbf I\big)\mathfrak
h\|_{\sigma',0}
\\\leq c_0(\sigma-\sigma')^{-2\tau-\gamma}\|\,\,
\boldsymbol\Phi(\mathfrak f, \mathfrak
u)\|_{\sigma,0}\,\,\|\mathfrak h\|_{\sigma,0}.
\end{multline}
\begin{multline}\label{ira8}
\|\big(\mathcal R(\mathfrak f, \mathfrak u)\,\, D_{\mathfrak u}
\boldsymbol\Phi(\mathfrak f, \mathfrak u)-\mathbf I\big) \mathfrak
v\|_{\sigma',d}
\\\leq c_0(\sigma-\sigma')^{-2\tau-\gamma}\|\,\,
\boldsymbol\Phi(\mathfrak f, \mathfrak
u)\|_{\sigma,0}\,\,\|\mathfrak v\|_{\sigma,d}.
\end{multline}
\end{itemize}
The following theorem, see \cite{Nirenberg}, Theorem 6.1 and
Corollaries, constitute the local existence and uniqueness of
solutions to  operator equation \eqref{ira9}

\begin{theorem}\label{ira10}
Assume that  $\boldsymbol \Phi$ satisfies Conditions $({\mathbf H.\mathbf 3})-({\mathbf H.\mathbf 5})$.
Then there exists a constant $C$, depending only on $c_0$, $\tau$,
and $\gamma$ with the following properties. If for some $\sigma\in
(0,1]$, the couple $(\mathfrak f, \mathfrak v(\mathfrak f))\in
\mathcal B_\sigma(\mathfrak f_0, \mathfrak u_0)$ satisfies the
conditions
\begin{equation}\label{ira11}
    \|\mathfrak v(\mathfrak f)-\mathfrak u_0\|_{\sigma,d}\leq r\leq R, \quad
\|\boldsymbol\Phi(\mathfrak f, \mathfrak v(\mathfrak
f))\|_{\sigma,0}\leq C(R-r)\sigma^{2(\tau+\gamma)},
\end{equation}
then equation \eqref{ira9} have a  solution $\mathfrak u=\mathfrak
u(\mathfrak f) \in \mathcal E_{\sigma/2,d}$ such that
\begin{equation}\label{ira11}
    \|\mathfrak v(\mathfrak f)-\mathfrak u(\mathfrak f)\|_{\sigma/2,d}\leq C^{-1}
\|\boldsymbol\Phi(\mathfrak f, \mathfrak v(\mathfrak
f))\|_{\sigma,0}\sigma^{-\gamma}.
 \end{equation}
Moreover, if the mapping $\mathcal B_\sigma \cap \mathbb
C^2\ni\mathfrak f\to \mathfrak v(\mathfrak f)\in \mathcal
E_{\sigma,d}$ is continuous, then the mapping $\mathcal B_\sigma
(\mathfrak f_0, \mathfrak u_0)\cap \mathbb C^2\ni\mathfrak f\to
\mathfrak u(\mathfrak f)\in \mathcal E_{\sigma/2,d}$ is continuous.
If, in addition, $\mathfrak v(\mathfrak f)$ satisfies the inequality
\begin{equation}\label{ira12}
    2C^{-1}\|\boldsymbol\Phi(\mathfrak f, \mathfrak v(\mathfrak
    f)\|_{\sigma,0}\sigma^{-\gamma}<1,
\end{equation}
then the solution $\mathfrak u(\mathfrak f)$ is unique.
\end{theorem}
\begin{proof} Existence of solution $\mathfrak u(\mathfrak f)$ is
exactly the statement Theorem 6.1 in \cite{Nirenberg}. The
continuity of this solution is a consequence of the Corollary of
this theorem, see \cite{Nirenberg}. The uniqueness also results from
ch.6 in\cite{Nirenberg}. However, since our formulation is slightly
different, we recall these arguments. Let $\mathfrak u(f), \mathfrak
u'(\mathfrak f)\in \mathcal B_\sigma(\mathfrak f_0,\mathfrak u_0)$
satisfy equation \eqref{ira9} and inequalities \eqref{ira11}. It
follows  that
$$
D_{\mathfrak u}\boldsymbol\Phi(\mathfrak f, \mathfrak u(\mathfrak
f))(\mathfrak u'(\mathfrak f)-\mathfrak u(\mathfrak f))=-Q(\mathfrak
f, \mathfrak u, \mathfrak u')
$$
Applying to both sides of this identity the operator $\mathcal
R(\mathfrak f, \mathfrak u(\mathfrak f)$ and recalling that
$\mathfrak u(\mathfrak f)$ satisfies equation \eqref{ira9} we obtain
$$
\mathfrak u'(\mathfrak f)-\mathfrak u(\mathfrak f)=-\mathcal
R(\mathfrak f, \mathfrak u(\mathfrak f)Q(\mathfrak f, \mathfrak u,
\mathfrak u').
$$
Choose an arbitrary $\sigma''<\sigma'\leq \sigma$.Applying
inequalities \eqref{ira4} and \eqref{ira6} we arrive at
\begin{multline*}
\|\mathfrak u'(\mathfrak f)-\mathfrak u(\mathfrak f)\|_{\sigma'',d}
\leq c(\sigma'-\sigma'')^{-\gamma}\|Q(\mathfrak f, \mathfrak u,
\mathfrak u')\|_{(\sigma'+\sigma'')/2,d}\leq\\
c(\sigma'-\sigma'')^{-2\tau-\gamma} \|\mathfrak u'(\mathfrak
f)-\mathfrak u(\mathfrak f)\|_{\sigma'',d}^2.
\end{multline*}
Now set $\sigma_n=2^{-n-1}\sigma$, $n\geq 0$. We have
$$
\|\mathfrak u'(\mathfrak f)-\mathfrak u(\mathfrak
f)\|_{\sigma_{n+1},d}\leq c 2^{n(\gamma+2\tau)} \|\mathfrak
u'(\mathfrak f)-\mathfrak u(\mathfrak f)\|_{\sigma_{n},d},
$$
which leads to the inequality
$$
\|\mathfrak u'(\mathfrak f)-\mathfrak u(\mathfrak
f)\|_{\sigma_{n+1},d}\leq c^n 2^{n^2(\gamma+2\tau)} \|\mathfrak
u'(\mathfrak f)-\mathfrak u(\mathfrak f)\|_{\sigma/2,d}^{2^n},
$$
On the other hand, estimate \eqref{ira11} imply
\begin{equation*}
    \|\mathfrak u'(\mathfrak f)-\mathfrak u(\mathfrak f)
    \|_{\sigma/2,d}\leq 2C^{-1}
\|\boldsymbol\Phi(\mathfrak f, \mathfrak v(\mathfrak
f))\|_{\sigma,0}\sigma^{-\gamma}.
 \end{equation*}
Combining the obtained results and recalling \eqref{ira12} we
finally obtain
\begin{multline*}
\|\mathfrak u'(\mathfrak f)-\mathfrak u(\mathfrak f)\|_{0,d}\leq
\|\mathfrak u'(\mathfrak f)-\mathfrak u(\mathfrak
f)\|_{\sigma_{n+1},d}\leq \\c^n 2^{n^2(\gamma+2\tau)} ( 2C^{-1}
\|\boldsymbol\Phi(\mathfrak f, \mathfrak v(\mathfrak
f))\|_{\sigma,0}\sigma^{-\gamma})^{2^n}\to 0\text{~~as~~}
n\to\infty.
\end{multline*}

\end{proof}

\subsection{ Solvability of the main operator equation}

We are now in a position to prove the local solvability and
uniqueness result for the main operator equation \eqref{luna8x}.
Recall the notation
\begin{equation}\label{ira15}
    \mathfrak f=(\alpha, k), \,\mathfrak u=
    (\boldsymbol\varphi, e, m,M)\,\text{where}
\,    \boldsymbol\varphi=(\boldsymbol\beta,\varphi_0, \mathbf u,
\mathbf w, W_{11}, W_{12}, W_{22}).
\end{equation}
and
\begin{equation}\label{ira16}
    \mathfrak v(\mathfrak f)=(\boldsymbol\varphi_0(\alpha),\,
     -k\alpha-k\alpha^2/2,\,
    k\alpha,\, -k),
\end{equation}
where
$$\boldsymbol\varphi_0(\alpha)=(0,0, \alpha\mathbf e_1, 1, 0,0).
$$
Fix an arbitrary $\sigma\in (0,1]$ and $d\geq 1$. Next, fix
$\varrho>0$ satisfying Condition ({\bf H.1}) of Theorem
\ref{theo1.4} and set
\begin{equation}\label{ira17}
    G=\{(\alpha, k)\in \mathbb C^2:\,\,
    |\text{Re~}\alpha|\leq 4\pi, \text{Re~}k\in[0,1], \,\,\,
     |\text{Im~}\alpha|\leq \varrho,  |\text{Im~} k|\leq\varrho\}.
\end{equation}
The following theorem is the main result of this section
\begin{theorem}\label{ira18}
Let conditions $(\mathbf H.\mathbf 1)$-$(\mathbf H.\mathbf 2)$ be satisfied. Then
there is $\varepsilon_0>0$ such that for all $|\varepsilon|\leq
\varepsilon_0$ and for all  $\mathfrak f\in G$ operator equation
\eqref{luna8x} has a unique solution $\mathfrak u(\mathfrak f)\in
\mathcal E_{\sigma/2,d}$ such that
\begin{equation}\label{ira19}
    \|\mathfrak u(\mathfrak f)-\mathfrak v(\mathfrak
    f)\|_{\sigma/2,d}\leq c|\varepsilon|.
\end{equation}
The mapping
$
G\ni\mathfrak f\to \mathfrak u(\mathfrak f)\in \mathcal
E_{\sigma/2,d}$ has a continuous  extension  to the strip
$$\Sigma_\rho=\{(\alpha, k)\in \mathbb C^2:\,\,
    \text{\rm Re~}\alpha\in \mathbb R, \text{\rm ~Re~}k\in[0,1], \,\,\,
     |\text{\rm Im~}\alpha|\leq \varrho,\,\,  |\text{\rm Im~} k|\leq\varrho\}.$$
 Moreover, the extended  mappings
$$
\mathfrak f\to \boldsymbol\varphi-\boldsymbol\varphi_0(\alpha), \quad
\mathfrak f\to M,\quad  \mathfrak f\to m+\alpha M
$$
are $2\pi$- periodic in $\alpha$.
\end{theorem}

\begin{proof}In view of Proposition \ref{luna18} for every
$\mathfrak f\in G$, there is $r>0$ independent of $\mathfrak f$ with
the following property. If
\begin{equation}\label{ira800}
\|\mathfrak v(\mathfrak f)-\mathfrak u\|_{\sigma,d}\leq r,
\end{equation}
 then the operator $\Phi(\mathfrak f, \mathfrak u)\in
\mathcal F_{\sigma, 0}$ is  differentiable at the point $(\mathfrak
f, \mathfrak u)$ with respect to $\mathfrak u$ and satisfies
estimate \eqref{ira4} with the exponent $\tau=0$. On the other hand,
it follows from Proposition \ref{asol6}, that there exists
$\varepsilon_0>0$ independent of $\mathfrak f$ with the following
properties. If $\mathfrak f\in G$, $\mathfrak u\in \mathcal
F_{\sigma,0}$ satisfy inequality \eqref{ira800}, and
$|\varepsilon|\leq \varepsilon_0$, then there is the linear operator
$\mathcal R(\mathfrak f, \mathfrak u)$ satisfying inequalities
\eqref{ira6}-\eqref{ira8} with the exponents $\tau=0$ and
$\gamma=8n+14+2d$. In other words, if $(\mathfrak f,\mathfrak u)$
satisfies inequality \eqref{ira800} and $|\varepsilon|\leq
\varepsilon_0$, then the operator $\Phi$ satisfies inequalities
\eqref{ira4} and \eqref{ira6}-\eqref{ira8} in Conditions $(\mathbf
H.\mathbf 3)$-$(\mathbf H.\mathbf 5)$ of Theorem \ref{ira10}.

Now choose an arbitrary $\mathfrak f_0=(\alpha_0, k_0)\in G$ and set
$\mathfrak u_0= \mathfrak v(\mathfrak f_0)$. It follows from this
and \eqref{ira16} that
\begin{equation}\label{ira23}
    \|\mathfrak v(\mathfrak f)-\mathfrak u_0\|_{\sigma,d}\leq 20
    |\mathfrak f-\mathfrak f_0|\text{~~for all~~}
    \mathfrak f\in G.
\end{equation}
Recall definition \eqref{ira3} of the ball $\mathcal
B_\sigma(\mathfrak f_0,\mathfrak u_0)$:
\begin{equation*}
    \mathcal B_\sigma(\mathfrak f_0, \mathfrak u_0)=\{(\mathfrak f, \mathfrak u):\,\,
    |\mathfrak f-\mathfrak f_0|<N,\quad
     \|\mathfrak u-\mathfrak u_0\|_{\sigma,d}<R\}.
\end{equation*}
We have
\begin{equation*}\begin{split}
\|\mathfrak v(\mathfrak f)-\mathfrak u\|_{\sigma,d}\leq \|\mathfrak
u_0-\mathfrak u\|_{\sigma,d}+\|\mathfrak v(\mathfrak f)-\mathfrak
u_0\|_{\sigma,d}\\\leq R+20 |\mathfrak f-\mathfrak f_0|\leq R+20N
\end{split}\end{equation*}
for $(\mathfrak f, \mathfrak u)\in \mathcal B_\sigma$.
 Hence for
\begin{equation}\label{ira803}
R+20N\leq r,\quad N<\rho
\end{equation}
the vectors  $\mathfrak f$ and $\mathfrak u$ satisfy inequality
\eqref{ira800} in the ball $\mathcal B_\sigma(\mathfrak
f_0,\mathfrak u_0)$.  Moreover, we have $\mathcal B_\sigma(\mathfrak
f_0,\mathfrak u_0)\cap \mathbb C^2\subset \Sigma_\rho$. Now fix $N$
and $R$ satisfying \eqref{ira803}. It follows from this and what was
mentioned above that for $|\varepsilon|\leq \varepsilon_0$, the
operator $\Phi$ satisfies Conditions $(\mathbf H.\mathbf 3)$-$(\mathbf H.\mathbf 5)$
with the exponents $\tau=0$, $\gamma=9n+12+d$, and the constant
$c_0$ independent of
 $\mathfrak f_0$. Hence $\Phi$ meets all requirements of
 Theorem \ref{ira10}.
 Applying this theorem
  we conclude that
there is a constant $C$, depending only on $c_0$ and $\gamma$, with
the following properties. If  the couple $(\mathfrak f, \mathfrak
v(\mathfrak f))\in \mathcal B_\sigma(\mathfrak f_0, \mathfrak u_0)$
satisfies the conditions
\begin{equation}\label{ira21}
    \|\mathfrak v(\mathfrak f)-\mathfrak u_0\|_{\sigma,d}\leq R/2, \quad
\|\boldsymbol\Phi(\mathfrak f, \mathfrak v(\mathfrak
f))\|_{\sigma,0}\leq 2^{-1} CR\sigma^{2\gamma},
\end{equation}
then equation \eqref{ira9} have a unique solution $\mathfrak
u=\mathfrak u(\mathfrak f) \in \mathcal E_{\sigma/2,d}$ such that
\begin{equation}\label{ira22}
    \|\mathfrak v(\mathfrak f)-\mathfrak u(\mathfrak f)\|_{\sigma/2,d}
    \leq C^{-1}
\|\boldsymbol\Phi(\mathfrak f, \mathfrak v(\mathfrak
f))\|_{\sigma,0}\sigma^{-\gamma}.
 \end{equation}
Moreover,  the mapping $\mathcal B_\sigma (\mathfrak f_0, \mathfrak
u_0)\cap \mathbb C^2\ni\mathfrak f\to \mathfrak u(\mathfrak f)\in
\mathcal E_{\sigma/2,d}$ is continuous. In order to prove the local
existence and uniqueness of solution to equation \eqref{luna8x}, we
have to show that conditions \eqref{ira21}is fulfilled in the ball
$\mathcal B_\sigma(\mathfrak f_0, \mathfrak u_0)$. First we notice
that, in view of relations \eqref{ira16} and  \eqref{luna8} , we
have
\begin{gather*}
\Phi_1(\mathfrak f, \mathfrak v(\mathfrak f ))= \varepsilon
H_1(\text{id},0, \alpha\mathbf e_1),
\\
 \Phi_2(\mathfrak f, \mathfrak v(\mathfrak f ))=
\varepsilon\frac{\partial H_1}{\partial \mathbf z} (\text{id},0,
\alpha\mathbf e_1)\big\}^\top,
\\
 \Phi_3(\mathfrak f, \mathfrak v(\mathfrak f ))=
\varepsilon\big\{\frac{\partial H_1}{\partial \mathbf y} (\text{id},
0, \alpha\mathbf e_1)\big\}^\top,
\end{gather*}
\begin{equation*}
 \Phi_4(\mathfrak f, \mathfrak v(\mathfrak f ))=
\varepsilon\frac{\partial^2 H_1}{\partial\mathbf z^2} (\text{id}, 0,
\alpha\mathbf e_1), \end{equation*}
\begin{equation*}
    \Phi_5(\mathfrak f, \mathfrak v(\mathfrak f))=0.
\end{equation*}
It follows from this and analyticity conditions \eqref{ee1.07}
imposed on $H_1$ that
$$
\|\Phi(\mathfrak f, \mathfrak v(f))\|_{\sigma,0}\leq
c_1|\varepsilon|
$$
for all $\mathfrak f\in \Sigma_\rho$. Here the constant $c_1$ is
independent on $\mathfrak f$. From this and \eqref{ira23} we obtain
that  condition \eqref{ira21} is fulfilled for all $\mathfrak f$
satisfying the inequality
\begin{equation}\label{ira25}
    |\mathfrak f-\mathfrak f_0|\leq N_1=N/40,
\end{equation}
and all $\varepsilon$ such that
$$
\varepsilon\leq \varepsilon_0<c_1^{-1} CN\sigma^{2\gamma}.
$$
Applying Theorem \ref{ira10} we conclude that for all $\mathfrak
f\in G$ satisfying \eqref{ira25} and for all $|\varepsilon|\leq
\varepsilon_0$ operator equation \eqref{luna8x} has a solution
$\mathfrak u(\mathfrak f)$ satisfying the inequality
\begin{equation}\label{ira26}
    \|\mathfrak v(\mathfrak f)-\mathfrak u(\mathfrak f)\|_{\sigma/2,d}
    \leq c_1 C^{-1}
\sigma^{-\gamma}|\varepsilon|\equiv c_2|\varepsilon|.
 \end{equation}
This solution is unique. Moreover, the mapping $\mathbb
C^2\ni\mathfrak f\to \mathfrak u(\mathfrak f)\in \mathcal
E_{\sigma/2, d}$ is continuous in the disk $\{|\mathfrak f-\mathfrak
f_0|\leq N_1\}$.

Our next task is to prove that this solution is defined for all
$\mathfrak f\in G$. Since $G$ is a compact set, there is a finite
collection of the balls
$$
G_i=\{|\mathfrak f-\mathfrak f_i|\leq N_1/2\}, \quad 1\leq i \leq m,
\quad \mathfrak f_i\in G
$$
such that $G\subset \cup_i G_i$. It follows from the local
solvability of operator equation \eqref{luna8x} that for
$|\varepsilon|\leq\varepsilon_0$ there is a unique continuous
mapping $G_i\ni \mathfrak f \to \mathfrak u_i(\mathfrak f)\in
\mathcal E_{\sigma/2,d}$ such that
\begin{equation}\label{ira801}
\Phi(\mathfrak f, \mathfrak u_i(\mathfrak f))=0, \quad \|\mathfrak
u_i(\mathfrak f)-\mathfrak v(\mathfrak f)\|_{\sigma/2, d}\leq
c_2|\varepsilon|.
\end{equation}
If $\mathfrak f\in G_i\cap G_j$  then $\mathfrak u_i(\mathfrak
f)=\mathfrak u_j(\mathfrak f)$. Indeed,  $\mathfrak u_j$ satisfies
inequality \eqref{ira801} and the equation $\Phi(\mathfrak f,
\mathfrak u_j)=0$. On the other hand,  $\mathfrak u_i(\mathfrak f)$
is the unique solution to the equation $\Phi(\mathfrak f, \mathfrak
u)=0$ with $\mathfrak f\in G_i$ satisfying this inequality. Hence
$\mathfrak u_i=\mathfrak u_j$. Therefore,  the relation $\mathfrak
u(\mathfrak f)=\mathfrak u_i(\mathfrak f)$ for $\mathfrak f\in G_i$
defines the continuous mapping $G\ni \mathfrak f\to\mathfrak
u(\mathfrak f)\in \mathcal E_{\sigma/2,d}$.

It remains to prove that this mapping has the analytic extension to
the strip $\Sigma_\rho$. Fix an arbitrary $\mathfrak f=(\alpha,
k)\in G$ such that
\begin{equation}\label{ira27}
    \text{~Re~}\alpha\in [0, \pi/2]
\end{equation}
and set
$$\tilde{\alpha}=\alpha+2\pi, \quad \tilde{ \mathfrak
f}=(\tilde \alpha, k).$$ Let $\tilde{\mathfrak u}=\mathfrak
u(\tilde{\mathfrak f})$ be a solution to the operator equation
$\Phi(\tilde{\mathfrak f},\tilde{\mathfrak u})=0$. We have
$$
\tilde{\mathfrak u}=(\tilde{\boldsymbol\varphi}, \tilde e, \tilde
m,\tilde M), \text{~~where~~}\tilde{\boldsymbol
\varphi}=(\tilde{\boldsymbol\beta},\tilde{\varphi}_0, \tilde{\mathbf
u}, \tilde{\mathbf w}, \tilde W_{11}, \tilde W_{12}, \tilde W_{22}).
$$
Now set
\begin{equation}\label{ira30}\begin{split}
{\mathfrak u}^*=(\boldsymbol \varphi^*,  e^*,  m^*,
M^*),\intertext{where}
\boldsymbol\varphi^*=\tilde{\boldsymbol\varphi} -(0,0,0, 2\pi
\mathbf e_1,
0,0,0),\\
e^*=\tilde e-(2\pi\tilde m+2\pi^2\tilde M), \quad m^*=\tilde m +2\pi
\tilde M, \quad M^*=\tilde M.
\end{split}
\end{equation}
Since the basic Hamiltonian $H(\mathbf x, \mathbf y, \mathbf z)$ is
$2\pi$-periodic in $z_1$ it follows from definition \eqref{luna8} of
the operator $\Phi$ that
\begin{equation}\label{ira28}
    \Phi(\mathfrak f, \mathfrak u^*)\equiv \Phi(
    \tilde{\mathfrak f},\tilde{\mathfrak u}).
\end{equation}
Next, we have
\begin{equation*}
    \mathfrak v(\tilde{\mathfrak f})=
\mathfrak v(\mathfrak f)-(2\pi k+2\pi^2 k, -2\pi k, k).
\end{equation*}
Relations \eqref{ira30} imply
\begin{equation}\label{ira31}
    \mathfrak u^*-\mathfrak v(\mathfrak f)=
    \tilde{\mathfrak u}-\mathfrak v(\tilde {\mathfrak f})
\end{equation}
Recall that $\tilde{\mathfrak u}=\mathfrak u(\tilde{\mathfrak f})$
satisfies the inequality
\begin{equation*}
    \|\tilde{\mathfrak u}-\mathfrak v(\tilde {\mathfrak
    f})\|_{\sigma/2,d}\leq c_2|\varepsilon|.
\end{equation*}
and $\Phi(
    \tilde{\mathfrak f},\tilde{\mathfrak u})=0$. From this, \eqref{ira28} and \eqref{ira31} we obtain
$$
 \Phi(\mathfrak f, \mathfrak u^*)=0, \quad
 \|{\mathfrak u}^*-\mathfrak v ({\mathfrak
    f})\|_{\sigma/2,d}\leq c_2|\varepsilon|, \quad \mathfrak f\in G.
$$
Hence $\mathfrak u^*=\mathfrak u(\mathfrak f)$, which along with
\eqref{ira30} and the equality $\tilde{\mathfrak u}= \mathfrak
u(\tilde{\mathfrak f})$ yields
\begin{equation*}\begin{split}
    \boldsymbol \varphi(\mathfrak f)+2\pi(0,0,0,\mathbf e_1, 0,0,0)=
\boldsymbol\varphi(\tilde{\mathfrak f}), \\
m(\mathfrak f)= m(\tilde{\mathfrak f})+2\pi M(\tilde{\mathfrak
f}),\quad M(\mathfrak f)= M(\tilde{\mathfrak f}). \end{split}
\end{equation*}
We can rewrite these relations in the equivalent form
\begin{equation}\label{ira32}\begin{split}
    \boldsymbol \varphi(\mathfrak f)-\alpha(0,0,0,\mathbf e_1, 0,0,0)=
\boldsymbol\varphi(\tilde{\mathfrak f})-(\alpha+2\pi)(0,0,0,\mathbf e_1, 0,0,0), \\
m(\mathfrak f)+\alpha M(\mathfrak f)= m(\tilde{\mathfrak f})+
(\alpha+2\pi) M(\tilde{\mathfrak f}),\quad M(\mathfrak f)=
M(\tilde{\mathfrak f}).
\end{split}
\end{equation}
Recall that these relations holds true for all $\mathfrak f$ and
$\tilde{\mathfrak f}$ satisfying the conditions
$$
\mathfrak f=(\alpha,k)\in G, \quad \tilde{\mathfrak
f}=(\alpha+2\pi,k),\quad \text{Re~}\alpha\in [0,\pi/2].
$$
Since the vector-function $\mathfrak u(\mathfrak f)$ and all its
components are holomorphic in $\alpha$, relations \eqref{ira32}
imply that the mappings $$\mathfrak f\to
\boldsymbol\varphi-\boldsymbol\varphi_0(\alpha), \quad \mathfrak f\to
M,\quad  \mathfrak f\to m+\alpha M$$ are holomorphic in the strip
$\Sigma_\rho$ and $2\pi$ periodic in $\alpha$. This completes the
proof of Theorem \ref{ira18}.
\end{proof}

\section{Dependence on parameters. Jacobi vector fields}\label{kira}

\subsection{Differentiability with respect to parameters}Theorem \ref{ira18} guarantees the local existence of
solutions to the operator equation
\begin{equation}\label{kira1}
\Phi(\mathfrak f, \mathfrak u)=0.
\end{equation}
In this section we investigate in details the dependence of this
solution on the parameter $\mathfrak f$. The result is given by the
following

\begin{theorem}\label{kira2}
Let conditions $(\mathbf H.\mathbf 1)$-$(\mathbf H.\mathbf 2)$ be satisfied and
$\varepsilon_0>0$ be given by Theorem \ref{ira18}. Let
$|\varepsilon|\leq \varepsilon_0$ and $\mathfrak u=\mathfrak
u(\mathfrak f)$ be a solution to equation \eqref{kira1},
\begin{equation}\label{kira3}\begin{split}
\mathfrak u(\mathfrak f)=(\boldsymbol\varphi(\alpha,k), e(\alpha,k),
m(\alpha,k), M(\alpha,k),\\
\boldsymbol\varphi(\alpha,k)= (\boldsymbol\beta, \varphi_0, \mathbf
u, \mathbf w, W_{11},W_{12},W_{21}).
\end{split}\end{equation}
Then the mappings
\begin{equation}\label{kira4}\begin{split}
\mathbb R\times [0,1]\ni(\alpha, k)\to \mathbf u(\alpha,k)\in
\mathcal A_{\sigma/4,d},\\
\mathbb R\times [0,1]\ni(\alpha, k)\to \varphi_0(\alpha,k)\in
X_{\sigma/4,d},\\ \mathbb R\times [0,1]\ni(\alpha, k)\to
\mathbf w(\alpha,k)-\alpha\mathbf e_1\in \mathcal A_{\sigma/2,d},\\
 \mathbb R\times [0,1]\ni(\alpha, k)\to
\mathbf W(\alpha,k)\in \mathcal A_{\sigma/2,d},
\\
\mathbb R\times [0,1]\ni(\alpha, k)\to \boldsymbol\beta(\alpha,k)\in
\mathbb R^{n-1},\\
 \mathbb R\times [0,1]\ni(\alpha, k)\to M(\alpha,k)\in
\mathbb R^1,
\\
\mathbb R\times [0,1]\ni(\alpha, k)\to m(\alpha,k)+\alpha
M(\alpha,k)\in \mathbb R^1,
\end{split}\end{equation}
are continuously differentiable and $2\pi$-periodic in $\alpha$.
Moreover, they are analytic in $\mathbb R\times (0,1)$ and satisfy
the inequalities
\begin{equation}\label{kira5}\begin{split}
    \|\partial_\alpha^r\mathbf u\|_{\sigma/2,d}+
    \|\partial_\alpha^r\varphi_0\|_{\sigma/2,d}\leq
    c(r)|\varepsilon|,\\
\|\partial_\alpha^r(\mathbf w-\alpha\mathbf e_1)\|_{\sigma/2,d}+ \|\partial_\alpha^r(\mathbf W-\mathbf I)
\|_{\sigma/2,d}\leq c(r)|\varepsilon|,\\
|\partial_\alpha^r\boldsymbol\beta|+|\partial_\alpha^r(m+\alpha M)|+
|\partial_\alpha^r (M+k)|\leq c(r)|\varepsilon|,
\end{split}\end{equation}
and
\begin{equation}\label{kira6}\begin{split}
    \|\partial_k\mathbf u\|_{\sigma/2,d}+
    \|\partial_k\varphi_0\|_{\sigma/2,d}\leq
    c|\varepsilon|,\\
\|\partial_k(\mathbf w-\alpha\mathbf e_1)\|_{\sigma/2,d}+\|\partial_k(\mathbf W-\mathbf I)
\|_{\sigma/2,d}\leq c|\varepsilon|,\\
|\partial_k\boldsymbol\beta|+|\partial_k(m+\alpha M)|+
|\partial_k(M+k)|\leq c|\varepsilon|,
\end{split}\end{equation}
where $r\geq 0$ is an arbitrary integer, the constant $c$ is
independent of $\alpha, k$, and $\varepsilon$.
\end{theorem}
\begin{proof}
By virtue of Theorem \eqref{ira18}, the mapping $\Sigma_\rho\ni
\mathfrak f\to \mathfrak u(\mathfrak f)\in \mathcal E_{\sigma/2,0}$ is continuous in the
complex strip $\Sigma_\rho$ which contains the real axis. Hence this
mapping is holomorphic function of $\alpha$. It  is also a
holomorphic function of $k$ on the interval $(0,1)$. Moreover,
Theorem \ref{ira18} shows that $\boldsymbol \beta$, $\varphi_0$,
$\mathbf u$, $\mathbf w-\alpha \mathbf e_1$, $W_{ij}$, $m+\alpha M$
and $M$ are periodic in $\alpha$. Since these functions are
holomorphic in $\alpha$ on the real axis, estimates \eqref{kira5}
obviously follows from estimates \eqref{ira19} in Theorem
\ref{ira18}.

However, this theorem does not guarantee the differentiability of
$\mathfrak u(\alpha,k)$ with respect to $k$ on the closed segment
$k\in [0,1]$.  In order to prove estimates \eqref{kira6} fix an
arbitrary $\alpha\in \mathbb R^1$ and consider the function
$\mathfrak u(\alpha,k)$. In view of Theorem \ref{ira18} this
function is continuous in the rectangular $\text{Re~}k\in [0,1]$,
$|\text{Im~}k|\leq \rho$. Hence it is holomorphic on the interval
$(0,1)$ and continuous on $[0,1]$. Next the function $\mathfrak
u(\alpha, k)$ satisfies the operator equation
\begin{equation}\label{kira7}
    \Phi(\mathfrak f, \mathfrak u)\equiv
    (\Phi_1(\mathfrak u),\Phi_2(\mathfrak u),\Phi_3(\mathfrak u),
    \Phi_4(k,\mathfrak u),\Phi_5(\mathfrak u))=0,
\end{equation}
where the differential operators $\Phi_i$ are defined by
\eqref{luna8}. It follows from this relation that the only $\Phi_4$
depends on $k$ via the matrix $\boldsymbol\Omega=\text{diag~}(-k,1)$
in the left hand side of \eqref{luna8d}. Since $\mathfrak u$ is
analytic function of $k$ on the interval $(0,1)$, we can
differentiate \eqref{kira7} with respect to $k$ to obtain
\begin{multline}\label{kira8}
    D_{\boldsymbol \varphi} \Phi(\mathfrak f,\mathfrak
    u)\partial_k\boldsymbol\varphi+
 D_{e} \Phi(\mathfrak f,\mathfrak
    u)\partial_k e+\\D_{m} \Phi(\mathfrak f,\mathfrak
    u)\partial_k m+D_{M} \Phi(\mathfrak f,\mathfrak
    u)(\partial_k M+1)=\mathcal Z\text{~~for~~} k\in (0,1).
\end{multline}
Here
$$
\mathcal Z=(0,0,0, \mathcal Z_4,0), \quad \mathcal
Z_4=\text{diag~}(1,0)-\mathbf W^\top \text{diag~}(1,0)\mathbf W.
$$
It follows from estimate \eqref{kira5} that
\begin{equation}\label{kira9}
    \|\mathcal Z\|_{\sigma/2, 0}\leq c|\varepsilon|.
\end{equation}
Relation \eqref{kira8} can be considered as equation for
$$\partial_k\mathfrak u=(\partial_k\boldsymbol\varphi, \partial_k e, \partial_k
m,\partial_k M).$$ Applying to \eqref{kira8} Proposition \ref{asol6}
with $\sigma$, $\sigma_1$, and $\sigma'$ replaced by $\sigma/2$,
$\sigma/2$, and $\sigma/4$ we conclude that for a suitable choice of
$\varepsilon_0$, equation \eqref{kira8} has a unique solution
$$
(\partial_k\boldsymbol\varphi,\partial_k e,\partial_k m, \partial_k
M+1)=\mathcal R(\mathfrak f, \mathfrak u)\mathcal Z
$$
Estimate \eqref{asol7}for the resolvent $\mathcal R$ in Proposition
\eqref{asol6} implies the inequality
$$
\|\partial_k\mathfrak \varphi\|_{\sigma/4,d}+|\partial_k m|+
|\partial_k M+1|\leq c \sigma^{-9n-12-d}\|\mathcal
Z\|_{\sigma/2,0}\leq c|\varepsilon|,
$$
which obviously yields \eqref{kira6}. This completes the proof of
Theorem \ref{kira2}.
\end{proof}

\subsection{Representation of derivatives. Jacobi vector fields}  In this section we obtain the representation for the derivatives of solutions to the modified problem with respect to the parameters $\alpha$ and $k$. Let $\boldsymbol\varphi(\alpha, k)= (\boldsymbol\beta, \varphi_0, \mathbf u, \mathbf w, W_{11},W_{12},
W_{21})$ be a solution to the operator equation \eqref{kira1} given by Theorem \ref{kira2}.
Let $\mathbf V$, $\mathbf W$, and $\mathbf v$ are given by \eqref{anna5}
\begin{equation}\label{kira10}\begin{split}
    \mathbf V=(\mathbf I_{n-1}+\mathbf u')^{-\top}, \quad  \mathbf v=\beta+\mathbf V\,\big( \,\nabla\varphi_0-w_2\nabla w_1\,\big),\\ W_{22}=W_{11}^{-1}(1+W_{12}W_{21}).
\end{split}\end{equation}
For $\tau=\alpha,k$ denote by $\boldsymbol\chi^{(\tau)}$, $\boldsymbol\lambda^{(\tau)}$, and $\boldsymbol\mu^{(\tau)}$ the functions
\begin{subequations}\label{kira11}
\begin{gather}\label{kira11a}
\boldsymbol\chi^{(\tau)}\,= \,\mathbf V^\top \,\partial_\tau\mathbf
u\\\label{kira11b} \boldsymbol\lambda^{(\tau)}\,=\, \mathbf
W^{-1}\partial_\tau\mathbf w-\chi_i^{(\tau)}\,\mathbf W^{-1}
\frac{\partial}{\partial \xi_i} \mathbf w\\\label{kira11c}
\boldsymbol\mu^{(\tau)}\,=\, \mathbf V^{-1}\Big(\partial_\tau\mathbf v+\chi_i^{(\tau)}\,
\frac{\partial}{\partial \xi_i} \mathbf v-
\boldsymbol\Lambda\boldsymbol\lambda^{(\tau)}\Big).
\end{gather}
\end{subequations}
The vector fields  $\boldsymbol\chi^{(\tau)}$, $\boldsymbol\lambda^{(\tau)}$, and $\boldsymbol\mu^{(\tau)}$ can be regarded as the Jacobi vector fields for the invariant tori problem.
We also set
\begin{equation}\label{kira12}\begin{split}
    p^{(\tau)}=\partial_\tau m\, +\, \alpha\partial_\tau M.
\end{split}\end{equation}
Here the matrix $\boldsymbol\Lambda$ are given by \eqref{anna17}:
$$
\boldsymbol\Lambda=- \mathbf V\, (\mathbf w')\,\mathbf J\,\mathbf W.
$$
Throughout of this an the next sections we will use the following notation. For every
integrable periodic function $f(\xi)$ we set
\begin{equation}\label{kiraa}
    \overline {f}=(2\pi)^{1-n}\int_{\mathbb T^{n-1}} f(\xi)\, d\xi, \quad
    f^*=f-\overline{f}.
\end{equation}
We have the following
\begin{lemma}\label{kira13} Under the assumptions of Theorem \ref{kira2}, the functions
$\boldsymbol\chi^{(\tau)}$, $\boldsymbol\lambda^{(\tau)}$,  $\boldsymbol\mu^{(\tau)}$, and the parameter
$q^{(\tau)}$ satisfy the equations
\begin{subequations}\label{kira14}
\begin{gather}\label{kira14a}\partial\, \boldsymbol\mu^{(\tau)}\,=\,
-p^{(\tau)}\nabla w_1 -\frac{1}{2}\partial_\tau M\, \nabla (w_1^*)^2,\\
\label{kira14b}
\partial\, \boldsymbol\lambda^{(\tau)}\,+\, \boldsymbol\Omega\, \boldsymbol\lambda^{(\tau)}\,+\,
\mathbf T\, \boldsymbol\mu^{(\tau)}+ p^{(\tau)}\mathbf W^\top\boldsymbol e_1\,=
\, -\, \partial_\tau M w_1^*\mathbf W^\top\boldsymbol e_1,\\
\label{kira14c}
-\partial\, \boldsymbol\lambda^{(\tau)}\,+\,\mathbf S\boldsymbol\mu^{(\tau)}\, +\, \mathbf T^\top
\boldsymbol\lambda^{(\tau)}=0
\end{gather}
and the orthogonality conditions
\begin{gather}\label{kira14d}
\overline{\{W\boldsymbol\lambda^{(\tau)}\cdot \mathbf e_1\}}\,\,+\,\,\overline{\{\boldsymbol\chi^{(\tau)}\cdot \nabla w_1\}}
=\delta_{\alpha\tau}, \\\label{kira14e}
\overline{\{\mathbf V^{-\top}\boldsymbol\chi^{(\tau)}\}}\,=\, 0
\end{gather}
\end{subequations}
Here the matrices $\mathbf S$ and $\mathbf T$ are defined by \eqref{luna29}
\begin{equation}\label{kira15}\begin{split}
    \mathbf S=\mathbf V^\top\frac{\partial^2 H}{\partial\mathbf  y^2}\big
(\boldsymbol\xi+\mathbf u, \mathbf v, \mathbf w)\mathbf V, \\
\mathbf T=\mathbf W^\top\frac{\partial^2 H}{\partial\mathbf
z\partial\mathbf  y}\big (\boldsymbol\xi+\mathbf u, \mathbf v,
\mathbf w)\mathbf V+ \boldsymbol\Lambda^\top\frac{\partial^2
H}{\partial\mathbf  y^2}\big (\boldsymbol\xi+\mathbf u, \mathbf v,
\mathbf w)\mathbf V.
\end{split}\end{equation}
In view of \eqref{luna30} and \eqref{kira4}, they admit the estimates
\begin{equation}\label{kira16}
    \|\mathbf S-\mathbf S_0\|_{\sigma/2, d-1}+
    \|\mathbf T-\mathbf T_0\|_{\sigma/2, d-1}
\leq c|\varepsilon|,
\end{equation}
where  the constant
matrices $\mathbf S_0$, $\mathbf T_0$ are given by
\begin{equation}\label{kira17}
\mathbf S_0=\frac{\partial^2 H_0}{\partial\mathbf  y^2}\big (0,0),
\quad \mathbf T_0=\frac{\partial^2 H_0}{\partial\mathbf
z\partial\mathbf  y}\big (0,0)
\end{equation}

\end{lemma}

\begin{proof}
Since $\mathfrak u=(\boldsymbol\varphi, e, m, M)$ is continuously differentiable with respect
to  $\tau=\alpha, k$  we can
differentiate \eqref{kira7} with respect to $\tau$ to obtain
\begin{multline}\label{kira18}
    D_{\boldsymbol \varphi} \Phi_i(\mathfrak f,\mathfrak
    u)\partial_\tau\boldsymbol\varphi+
 D_{e} \Phi_i(\mathfrak f,\mathfrak
    u)\partial_\tau e+\\D_{m} \Phi_i(\mathfrak f,\mathfrak  u)\partial_\tau m+D_{M} \Phi(\mathfrak f,\mathfrak
    u)(\partial_\tau M)+\partial_\tau \Phi_i=0.
\end{multline}
$i=1,\dots,5$. Introduce the functions
\begin{subequations}\label{kira19}
\begin{gather}\label{kira19a}
\boldsymbol\chi^{(\tau)}\,= \,\mathbf V^\top \,\partial_\tau\mathbf
u\\\label{kira19b} \boldsymbol\lambda^{(\tau)}\,=\, \mathbf
W^{-1}\partial_\tau\mathbf w-\chi_i^{(\tau)}\,\mathbf W^{-1}
\frac{\partial}{\partial \xi_i} \mathbf w\\\label{kira19c}
\boldsymbol\Gamma^{(\tau)}\,=\, \mathbf W^{-1}\partial_\tau\mathbf W-\chi_i^{(\tau)}\,\mathbf
W^{-1} \frac{\partial}{\partial \xi_i} \mathbf W,\\\label{kira19d}
\boldsymbol\mu^{(\tau)}\,=\, \mathbf V^{-1}\Big(\partial_\tau\mathbf v+\chi_i^{(\tau)}\,
\frac{\partial}{\partial \xi_i} \mathbf
v-\boldsymbol\Lambda\boldsymbol\lambda^{(\tau)}\Big)
\\\label{kira19e}
\nabla\psi_0^{(\tau)}=\boldsymbol\mu^{(\tau)}-\boldsymbol\beta^{(\tau)}, \quad
\boldsymbol\beta^{(\tau)}= \frac{1}{(2\pi)^{n-1}}\int_{\mathbb T^{n-1}}
\boldsymbol\mu^{(\tau)}\, d\boldsymbol\xi,
\end{gather}
\end{subequations}
Since $\mathbf u$, $\mathbf v$, $\mathbf w$, and $\mathbf W$ form a canonical mapping in the group $\mathcal G$, the vector functions $\partial_\tau \mathbf u$,  $\partial_\tau \mathbf v$,
 $\partial_\tau \mathbf w$, and  $\partial_\tau \mathbf W$ determine the element of the tangent space of $\mathcal G$. Hence we can apply  Corollary \ref{structuralcor3} of the structural theorem \ref{structural3}
 to obtain
\begin{subequations}\label{kira20}
\begin{gather}\label{kira20a}
    D_\varphi\Phi_1[\partial_\tau \boldsymbol\varphi]+
   D_e\Phi_1[\partial_\tau e] +D_m\Phi_1[\partial_\tau m]+D_M\Phi_1[\partial_\tau M]=\\
   \nonumber\boldsymbol\partial \psi_0^{(\tau)}+\boldsymbol\omega^\top\cdot
    \boldsymbol\beta^\tau+\partial_\tau e+\partial_\tau m\, w_1+\partial_\tau M\, \frac{1}{2}w_1^2,\\
   \label{kira20b}
D_\varphi\Phi_2[\partial_\tau \boldsymbol\varphi]+
   D_e\Phi_2[\partial_\tau e] +D_m\Phi_2[\partial_\tau m]+D_M\Phi_b[\partial_\tau M]\equiv\\\nonumber
    \mathbf J\boldsymbol\partial\boldsymbol\lambda^{(\tau)} +\boldsymbol\Omega\boldsymbol\lambda^{(\tau)}+\mathbf
T\boldsymbol\mu^{(\tau)}
    +\partial_\tau m \mathbf W^\top \mathbf e_1+
    \partial_\tau M w_1\mathbf W^\top \mathbf e_1,
  \\
\nonumber \boldsymbol\mu^{(\tau)} =\nabla \psi_0^{(\tau)}+\boldsymbol\beta^{(\tau)}
      \\
\label{kira20c}
 D_\varphi\Phi_3[\partial_\tau \boldsymbol\varphi]+
   D_e\Phi_3[\partial_\tau e] +D_m\Phi_3[\partial_\tau m]+D_M\Phi_M[\partial_\tau M]\equiv\\\nonumber
    -\boldsymbol\partial \boldsymbol\chi^{(\tau)}+
 \mathbf S\boldsymbol\mu^{(\tau)}+\mathbf T^\top \boldsymbol\lambda^{(\tau)},
\\
\label{kira20d} D_\varphi\Phi_5[\partial_\tau \boldsymbol\varphi]+
   D_e\Phi_5[\partial_\tau e] +D_m\Phi_3[\partial_\tau m]+D_M\Phi_M[\partial_\tau M]\equiv\\\nonumber \overline{
    \boldsymbol\lambda^{(\tau)}\cdot \mathbf W^\top \mathbf e_1}\,\,+\,\,
    \overline{\boldsymbol\chi^{(\tau)}\cdot\nabla w_1}
\end{gather}
\end{subequations}
Since $D_\tau \Phi_1=0$ we have from equality \eqref{kira18}
\begin{equation}\label{kira22}
  \boldsymbol\partial \psi_0^{(\tau)}+\boldsymbol\omega^\top\cdot
    \boldsymbol\beta^{(\tau)}+\partial_\tau e+\partial_\tau m\, w_1+\partial_\tau M\, \frac{1}{2}w_1^2=0
\end{equation}
Next we have
\begin{equation}\label{kira23}
\nabla\,(\boldsymbol\partial \psi_0^{(\tau)}+\boldsymbol\omega^\top\cdot
    \boldsymbol\beta^{(\tau)})=\boldsymbol\partial\boldsymbol\mu^{(\tau)}.
\end{equation}
On the other hand, relations $p^{(\tau)}=\partial_\tau m+\alpha\partial_\tau M$ and
$w_1=\alpha +w_1^*$ yield
\begin{equation}\label{kira24}
\nabla (\partial_\tau m\, w_1+\partial_\tau M\, \frac{1}{2}w_1^2)=p^{(\tau)}\nabla w_1^*+\partial_\tau
M\frac{1}{2}\nabla(w_1^*)^2
\end{equation}
Taking the gradient from both the sides of \eqref{kira22} and using equalities
\eqref{kira23}-\eqref{kira24} we arrive at \eqref{kira14a}.
Since $\partial_\tau \Phi_2=0$ it follows from \eqref{kira18} and \eqref{kira20b} that
\begin{equation}\label{kira25}
    \mathbf J\boldsymbol\partial\boldsymbol\lambda^{(\tau)} +\boldsymbol\Omega\boldsymbol\lambda^{(\tau)}+\mathbf
T\boldsymbol\mu^{(\tau)}
    +\partial_\tau m \mathbf W^\top \mathbf e_1+
    \partial_\tau M w_1\mathbf W^\top \mathbf e_1=0.
\end{equation}
It is easily seen that
$$
\partial_\tau m \mathbf W^\top \mathbf e_1+
    \partial_\tau M w_1\mathbf W^\top \mathbf e_1=p^{(\tau)}\mathbf W^\top\mathbf e_1+w_1^*
\mathbf W^\top\mathbf e_1.
$$
Substituting this relation into \eqref{kira25} we arrive at \eqref{kira14b}. Next \eqref{kira18} and \eqref{kira20c} yield \eqref{kira14c} since $\partial_\tau \Phi_3=0$.
Next, we have $\partial_\tau \Phi_5=-\delta_{\tau\alpha}$. Substituting this equality and relation
\eqref{kira20d} into \eqref{kira18} we arrive at \eqref{kira14d}. It remains to note that equation \eqref{kira14e} obviously follows from the equality  $\overline{\partial_\tau \mathbf u}=0$
and relation \eqref{kira19a}.

\end{proof}

\begin{corollary}\label{kira26}
Under the assumptions of Theorem \ref{kira2}, we have
\begin{equation}\label{kira27}\begin{split}
  ( \boldsymbol\mu^{(\alpha)},\boldsymbol\lambda^{(\alpha)}, \boldsymbol\chi^{(\alpha)},p^{(\alpha)})
  =(\boldsymbol\mu^{(1)},\boldsymbol\lambda^{(1)},\boldsymbol\chi^{(1)},p^{(1)})
  +\partial_\alpha M(\boldsymbol\mu^{(2)},\boldsymbol\lambda^{(2)},\boldsymbol\chi^{(2)},p^{(2)})\\
 ( \boldsymbol\mu^{(k)},\boldsymbol\lambda^{(k)}, \boldsymbol\chi^{(k)},p^{(k)})
  =\partial_k M(\boldsymbol\mu^{(2)},\boldsymbol\lambda^{(2)},\boldsymbol\chi^{(2)},p^{(2)}),
 \end{split}\end{equation}
Where $(\boldsymbol\mu^{(i)},\boldsymbol\lambda^{(i)},\boldsymbol\chi^{(i)},q^{(i)})$,
$i=1,2$ are solutions to the equations
\begin{subequations}\label{kira28}
\begin{gather}\label{kira28a}\boldsymbol\partial\, \boldsymbol\mu^{(1)}\,=\,
-p^{(1)}\nabla w_1,\\
\label{kira28b}
\mathbf J\boldsymbol\partial\, \boldsymbol\lambda^{(1)}\,+\, \boldsymbol\Omega\, \boldsymbol\lambda^{(1)}\,+\,
\mathbf T\, \boldsymbol\mu^{(1)}+ p^{(1)}\mathbf W^\top\boldsymbol e_1=0,\\
\label{kira28c}
-\boldsymbol\partial\, \boldsymbol\lambda^{(1)}\,+\,\mathbf S\boldsymbol\mu^{(1)}\, +\, \mathbf T^\top
\boldsymbol\lambda^{(1)}=0
\end{gather}
\begin{gather}\label{kira28d}
\overline{\{W\boldsymbol\lambda^{(1)}\cdot\mathbf e_1\}}\,\,+\,\,\overline{\{\boldsymbol\chi^{(1)}\cdot \nabla w_1\}}
=1, \\\label{kira28e}
\overline{\{\mathbf V^{-\top}\boldsymbol\chi^{(1)}\}}\,=\, 0,
\end{gather}
\end{subequations}
and
\begin{subequations}\label{kira29}
\begin{gather}\label{kira29a}\boldsymbol\partial\, \boldsymbol\mu^{(2)}\,=\,
-p^{(2)}\nabla w_1 -\frac{1}{2} \, \nabla (w_1^*)^2,\\
\label{kira29b}
\mathbf J\boldsymbol\partial\, \boldsymbol\lambda^{(2)}\,+\, \boldsymbol\Omega\, \boldsymbol\lambda^{(2)}\,+\,
\mathbf T\, \boldsymbol\mu^{(2)}+p^{(2)}\mathbf W^\top\boldsymbol e_1\,=
\, -\,w_1^*\mathbf W^\top\boldsymbol e_1,\\
\label{kira29c}
-\boldsymbol\partial\, \boldsymbol\lambda^{(2)}\,+\,\mathbf S\boldsymbol\mu^{(2)}\, +\, \mathbf T^\top
\boldsymbol\lambda^{(2)}=0
\end{gather}
\begin{gather}\label{kira29d}
\overline{\{W\boldsymbol\lambda^{(2)}\cdot \mathbf e_1\}}\,\,+\,\,\overline{\{\boldsymbol\chi^{(2)}\cdot \nabla w_1\}}
=0, \\\label{kira28e}
\overline{\{\mathbf V^{-\top}\boldsymbol\chi^{(2)}\}}\,=\, 0
\end{gather}
\end{subequations}

\end{corollary}

In the next section we investigate problems \eqref{kira28} and \eqref{kira29} in many details.

\section{Linear problem}\label{figa}
In this section we prove the solvability of problems \eqref{kira28} and \eqref{kira29} and
establish estimates for their solutions in the spaces of analytic functions and the Sobolev spaces.
First we recall the basic facts from the theory of Sobolev spaces on torus.
\subsection{Preliminaries. Sobolev spaces}
For every $s\in \mathbb R$, we denote by $H_s$ the Hilbert space which consists of all distributions $u$ on the tori $\mathbb T^{n-1}$ such that
\begin{equation}\label{figa1}
    |u|_s=\Big(\sum_{m\in \mathbb Z^{n-1}}(1+|m|^2)^s|\hat{u}(m)|^2\Big)^{1/2},
\end{equation}
where $\hat{u}$ is the Fourier transform of $u$. If $u$ is an integrable function, then
$$
\hat{u}(m)=(2\pi)^{-(n-1)/2}\int_{\mathbb T^{n-1}}e^{-i m\cdot \boldsymbol \xi} \,u(\boldsymbol\xi)\, d
\boldsymbol{\xi}.
$$
It is clear that $H_0=L^2(\mathbb T^{n-1})$ and
$$
 |u|_s=\|(-\Delta+1)^{s/2} u\|_{L^2(\mathbb T^{n-1})}.
$$
For every $s'<s''$ the embedding $H_{s''}\hookrightarrow H_{s'}$ is compact. For every integer $l\geq 0$
the embedding  $H_{n-1+l}\hookrightarrow C^l(\mathbb T^{n-1})$ is compact. Next, we have
\begin{equation}\label{figa2}
    |uv|_s\leq c(s)|u|_s\|v\|_s \text{~~for~~}s\geq n-1.
\end{equation}
For all $-\infty<r\leq s<\infty$ we have the interpolation inequality
\begin{equation}\label{figa3}
    \|u\|_{(s+r)/2}\leq c(s,r)\|u\|_{s}^{1/2}\|u\|_r^{1/2}.
\end{equation}
Recall the estimate $|\hat{ u}(m)|\leq \|u\|_{\sigma, 0} \exp(-\sigma|m|)$. It follows that
\begin{equation}\label{figa4}
    |u|_s\leq c(s, \sigma)\|u\|_{\sigma, 0}\text{~~for all ~~} s\in \mathbb R\text{~~and~~}
\sigma>0.
\end{equation}
\subsection{Basic linear problem.} In this section we prove the solvability of  linear problems
\eqref{kira28} and \eqref{kira29}. We investigate the qualitative properties of their solutions.
First we consider the general boundary value problem which includes problems
\eqref{kira28} and \eqref{kira29} as  particular cases.
\begin{subequations}\label{figa4}
\begin{gather}\label{figa4a}\boldsymbol\partial\, \boldsymbol\mu\,=\,
-p\nabla w_1 -  \nabla g_1,\\
\label{figa4b}
\mathbf J\boldsymbol\partial\, \boldsymbol\lambda\,+\, \boldsymbol\Omega\, \boldsymbol\lambda\,+\,
\mathbf T\, \boldsymbol\mu+p\mathbf W^\top \mathbf e_1=\mathbf g_2\\
\label{figa4c}
-\boldsymbol\partial\, \boldsymbol\lambda\,+\,\mathbf S\boldsymbol\mu\, +\, \mathbf T^\top
\boldsymbol\lambda=0
\end{gather}
\begin{gather}\label{figa4d}
\overline{\{W\boldsymbol\lambda\}\cdot \mathbf e_1}\,\,+\,\,\overline{\{\boldsymbol\chi\cdot \nabla w_1\}}
=\gamma, \\\label{figa4e}
\overline{\{\mathbf V^{-\top}\boldsymbol\chi\}}\,=\, 0.
\end{gather}
\end{subequations}
Assume that $w_1$ and $\mathbf W$ meet all requirements of Theorem \ref{kira2}, i.e.
\begin{equation}\label{kira5extra}\begin{split}
 \|\mathbf V^{-1}-\mathbf I\|_{\sigma/2,0}+\|(w_1-\alpha)\|_{\sigma/2,d}+ \|\mathbf W-\mathbf I
\|_{\sigma/2,d}\leq c|\varepsilon|.
\end{split}\end{equation}
We also assume that the matrix -valued functions $\mathbf S$ and $\mathbf T$ meet
all requirements of Lemma \ref{kira13}, i.e.
\begin{equation}\label{kira16extra}
    \|\mathbf S-\mathbf S_0\|_{\sigma/2, d-1}+
    \|\mathbf T-\mathbf T_0\|_{\sigma/2, d-1}
\leq c|\varepsilon|,
\end{equation}
where  the constant
matrices $\mathbf S_0$, $\mathbf T_0$ are given by
\begin{equation}\label{kira17extra}
\mathbf S_0=\frac{\partial^2 H_0}{\partial\mathbf  y^2}\big (0,0),
\quad \mathbf T_0=\frac{\partial^2 H_0}{\partial\mathbf
z\partial\mathbf  y}\big (0,0)
\end{equation}
Recall that  $d\geq 2$ is a fixed number.
The following theorem is the main result of this section.
\begin{theorem}\label{figa5} Let $s\in \mathbb R^1$ be an arbitrary number and $\sigma\in (1/2,1]$.
Then
  there is $\varepsilon_0> 0$ independent on $\alpha$ and $k$  with the following properties. For every
   $|\varepsilon|\leq \varepsilon_0$, $g_1,\mathbf g_2\in \mathcal A_{\sigma/2,0}$
   and $\gamma\in \mathbb R^1$, problem \eqref{figa4} has a unique solution
$(\boldsymbol\mu, \boldsymbol\lambda, \boldsymbol \chi)\in \mathcal A_{\sigma/4,0}$, $p\in \mathbb R^1$.
This solution admits the estimates
\begin{equation}\label{figa6}
    \|\boldsymbol\mu\|_{\sigma/4,0}+\|\boldsymbol\lambda\|_{\sigma/4,0}+\|\boldsymbol\chi\|_{\sigma/4,0}
+|p|\leq c\big(\|g_1\|_{\sigma/2,0}+\|\boldsymbol g_2\|_{\sigma/2,0}
+|\gamma|\big),
\end{equation}
\begin{equation}\label{figa7}
    |\boldsymbol\mu|_{s}+|\boldsymbol\lambda|_{s}+|\boldsymbol\chi|_{s}
+|p|\leq c\big(|g_1|_{s+4n+6}+|\boldsymbol g_2|_{s+4n+6}
+|\gamma|\big).
\end{equation}
\begin{equation}\label{figa7overline}
    |\overline{\boldsymbol\mu}|\leq c |\overline{\mathbf g_2}|+\varepsilon_0c\big(|g_1|_{s+4n+6}+|\boldsymbol g_2|_{s+4n+6}
+|\gamma|\big).
\end{equation}

\end{theorem}
\begin{proof} The proof is in Appendix \ref{proothfiga5}.\end{proof}

We employ Theorem \ref{figa5} in order to investigate the structure of solutions to problems
\eqref{kira28} and \eqref{kira29} and to establish robust estimates of these solutions. We start
with detailed analysis of problem \eqref{kira29}. The corresponding results is given by the following theorem
 which is the second main statement of this section.

\begin{theorem}\label{figa8} Under the assumptions of of Theorem \ref{figa5},
there exists $\varepsilon_0>0$ with the following property. For every $|\varepsilon|\leq\varepsilon_0$,
problem \eqref{kira29} has a unique analytic solution $(\boldsymbol{\mu}^{(2)},
\boldsymbol{\lambda}^{(2)}, \boldsymbol{\chi}^{(2)})\in \mathcal A_{\sigma/4,0}$, $p^{(2)}\in \mathbb C$.
This solution admits the estimate
\begin{equation}\label{figa9}
  \|\boldsymbol{\mu}^{(2)}\|_{\sigma/4,0}+  \|\boldsymbol{\lambda}^{(2)}\|_{\sigma/4,0}+
  \|\boldsymbol{\chi}^{(2)}\|_{\sigma/4,0}+|p^{(2)} |\leq c \|w_1^*\|_{\sigma/2,0}\leq c|\varepsilon|,
\end{equation}
\begin{equation}\label{figa10}
 |\boldsymbol{\mu}^{(2)}|_0\,\leq\,  c\varepsilon_0\,\,|w_1^*|_{-1}, \quad
     |{\lambda}^{(2)}_2|_0+ |k|\,\,|{\lambda}^{(2)}_1|_0\geq c^{-1}\,|w_1^*|_{-1}.
\end{equation}
Here the strictly positive constant $c$ is independent of $\alpha$, $k$, and $\varepsilon$.
\end{theorem}
\begin{proof}
Notice that problem \eqref{kira29} is a particular case of problem \eqref{figa4} with the right hand sides
\begin{equation}\label{figa11}
    g_1*=(w_1^*)^{2}, \quad  \mathbf g_2=- w_1^* \mathbf W\mathbf e_1, \quad \gamma=0.
\end{equation}
It follows from inequalities \eqref{kira5extra} and the identity $w_1^*=w_1-\alpha$ that
\begin{equation*}
    \|g_1^*\|_{\sigma/2,0}\leq c \|w_1^*\|_{\sigma/2,0}^2, \quad
    \|\mathbf g_2\|_{\sigma/2,0}\leq c \|w_1^*\|_{\sigma/2,0} .
\end{equation*}
Applying Theorem \ref{figa5} we conclude that problem \eqref{kira29} has a unique solution,
satisfying inequality \eqref{figa9}. It remains to prove estimates \eqref{figa10}
Choose an arbitrary $t>n$. Since the space $H_{t+1}$ is a Banach algebra, we have
$$
|\nabla(w_1^*)^2|_{t}\leq |(w_1^*)^2|_{t+1}\leq c |w_1^*|_{t+1}^2.
$$
On the other hand, the interpolation inequality implies
\begin{equation*}
    |w_1^*|_{t+1}\leq c|w_1^*|_{-1}^{1/2}\, \, |w_1^*|_{2t+3}^{1/2}\leq c
    |w_1^*|_{-1}^{1/2}\, \, \|w_1^*\|_{\sigma/2,0}^{1/2}\leq c \varepsilon_0^{1/2}\, |w_1^*|_{-1}^{1/2}.
\end{equation*}
We thus get
\begin{equation}\label{figa12}
    |\nabla(w_1^*)^2|_{t}\leq c |\varepsilon|\, |w_1^*|_{-1}.
\end{equation}
Obviously, this inequality holds true for all $t\in \mathbb R$. Next we have
\begin{equation}\label{figa13}
    |w_1^*\mathbf W\mathbf e_1|_{t}\leq c \|\mathbf W\|_{\sigma/2,0}\, |w_1^*|_{t} \leq
    c|w_1^*|_{t}.
\end{equation}
Now set $s=-1-4n-6$ and $t=-1$. Apply Theorem \ref{figa5} to problem \eqref{kira29}. It follows
from the estimate \eqref{figa7} in this theorem and relations \eqref{figa10}-\eqref{figa13}
that
\begin{equation}\label{figa14}
  |\boldsymbol{\mu}^{(2)}|_{-4n-7}+  |\boldsymbol{\lambda}^{(2)}|_{-4n-7}+
  |\boldsymbol{\chi}^{(2)}|_{-4n-7}+|p^{(2)}| \leq c|\mathbf w_1^*|_{-1}.
\end{equation}
Since $s+4n+6=-1$, estimate \eqref{figa7overline} implies
$$
|\overline{\boldsymbol\mu^{(2)}}|\leq |\overline{\mathbf g_2}|+c \varepsilon_0 |\mathbf w_1^*|_{-1}.
$$
We have
\begin{equation*}
 |\overline{\mathbf g_2}|=|\overline{w_1^* \mathbf W\mathbf e_1}|=
|\overline{w_1^* (\mathbf W-\mathbf I)\mathbf e_1}| \leq |w_1^*|_{-1}\, |\mathbf W-\mathbf I|_1
\leq c\varepsilon_0  |w_1^*|_{-1}.
\end{equation*}
Combining the obtained results we arrive at the estimate
\begin{equation}\label{figa15}
 |\overline{\boldsymbol\mu^{(2)}}|\leq c \varepsilon_0 |\mathbf w_1^*|_{-1}
\end{equation}
Let us estimate ${\boldsymbol \mu^{(2)}}^*$.  The first equation \eqref{kira29a} in
system \eqref{kira29} reads
$$
\boldsymbol\partial\, {\boldsymbol\mu^{(2)}}^*\,=\,
-p^{(2)}\nabla w_1 -\frac{1}{2} \, \nabla (w_1^*)^2
$$
Applying Lemma \ref{nisa6} to this equation we obtain
\begin{equation*}
    | {\boldsymbol\mu^{(2)}}^*|_0\leq c |p^{(2)}|\, |w_1^*|_{4n+7}+ c|\nabla (w_1^*)^2|_{4n+6}
\end{equation*}
It follows from inequality \eqref{figa12}  with $t=4n+6$ that
$$
 |w_1^*|_{4n+7}\leq c\|w_1^*\|_{\sigma/2,0}\leq c \varepsilon_0, \quad |\nabla (w_1^*)^2|_{4n+6}
\leq c\varepsilon_0 |w_1^*|_{-1}.
$$
It follows that
$$
 | {\boldsymbol\mu^{(2)}}^*|_0\leq c\varepsilon_0 |w_1^*|_{-1}.
$$
Combining this result with \eqref{figa15} we obtain the first estimate in \eqref{figa10}.
Let us prove the second. Notice that $|\boldsymbol\partial\lambda_2^{(2)}|_{-1}\leq c
|\boldsymbol\lambda_2^{(2)}|_{0}$. It follows that
\begin{gather}\nonumber
   |{\lambda}^{(2)}_2|_0+ |k|\,\,|{\lambda}^{(2)}_1|_0\geq
   c^{-1}(|\boldsymbol\partial\lambda_2^{(2)}|_{-1} + |k|\,\,|{\lambda}^{(2)}_1|_{-1})\geq\\\label{figa16}
   c^{-1}(|\boldsymbol\partial\lambda_2^{(2)}-k\,{\lambda}^{(2)}_1|_{-1})\geq
    c^{-1}(|
    \boldsymbol\partial{\lambda_2^{(2)}}^*-k\,{{\lambda}^{(2)}_1}^*
|_{-1}).
\end{gather}
On the other hand, the second  equation \eqref{kira29b} in
system \eqref{kira29} implies
\begin{equation*}
    \boldsymbol\partial{\lambda_2^{(2)}}^*-k\,{{\lambda}^{(2)}_1}^*=
    -w_1^* -(\mathbf T\boldsymbol\mu^{(2)})^*
    -p^{(2)}(\mathbf W^\top-\mathbf I)^*\mathbf e_1-(w_1^*(\mathbf W^\top-\mathbf I)\mathbf e_1)^*.
\end{equation*}
We have
\begin{gather*}
|(\mathbf W^\top-\mathbf I)^*|_{-1}\leq \|(\mathbf W^\top-\mathbf I)^*\|_{\sigma/2,0}
\leq c\varepsilon_0, \\ |w_1^*(\mathbf W^\top-\mathbf I)\mathbf e_1|_{-1}\leq
c |w_1^*|_{-1}\|(\mathbf W^\top-\mathbf I)^*\|_{\sigma/2,0}\leq  c\varepsilon_0 |w_1^*|_{-1}.
\end{gather*}
Next, the first estimate in \eqref{figa10} implies
$$
|\mathbf T\boldsymbol{\mu}^{(2)}|_0\,\leq\,  c\varepsilon_0\,\,|w_1^*|_{-1}.
$$
Combining obtained inequalities we arrive at the estimate
$$
|w_1^* +(\mathbf T\boldsymbol\mu^{(2)})^*
    +p^{(2)}(\mathbf W^\top-\mathbf I)^*\mathbf e_1+(w_1^*(\mathbf W^\top-\mathbf I)\mathbf e_1)^*|_{-1}
    \geq |w_1^*|_{-1}-c\varepsilon_0 |w_1^*|_{-1},
 $$
 which yields
$$
 | \boldsymbol\partial{\lambda_2^{(2)}}^*-k\,{{\lambda}^{(2)}_1}^*|_{-1}
 \geq |w_1^*|_{-1}-c\varepsilon_0 |w_1^*|_{-1}.
$$
This estimate along with \eqref{figa16} implies the estimate
$$
 |{\lambda}^{(2)}_2|_0+ |k|\,\,|{\lambda}^{(2)}_1|_0\geq c^{-1}(|w_1^*|_{-1}-c\varepsilon_0 |w_1^*|_{-1}).
$$
Choosing $\varepsilon_0$ sufficiently small we obtain the second estimate in \eqref{figa10}
and the theorem follows.
\end{proof}

Now turn to problem \eqref{kira28}.  Our goal is to prove the existence and uniqueness of solutions
to this problem and to obtain the asymptotic expansion of its solutions near $\varepsilon=0$.
In order to formulate the results we introduce the auxiliary denotations. Recall  denotations
\eqref{kira15}-\eqref{kira17} for the matrix $\mathbf T$. Denote by $\mathbf t_i^\top$, $i=1,2,$
the rows of  $\mathbf T$. They can be regarded al analytic periodic  functions
$\mathbf t_i:\mathbb T^{n-1}\to \mathbb C^{n-1}$.  Notice that for $\varepsilon=0$,  we have
$\mathbf t_1=0$ and $\mathbf t_2=\mathbf t_0$, where $\mathbf t_0$ is the only nonzero row of the matrix
$\mathbf T_0$. It follows from this and estimate \eqref{kira16} that
\begin{equation}\label{figa17}
    \|\mathbf t_1\|_{\sigma/2, d-1}+ \|\mathbf t_2-\mathbf t_0\|_{\sigma/2, d-1}\leq c|\varepsilon|.
\end{equation}
Recall the denotation
$$
\overline{\mathbf t_1}=\frac{1}{(2\pi)^{n-1}}\int_{\mathbb T^{n-1}} \mathbf t_1\, d\boldsymbol\xi.
$$
Introduce the constant vector
\begin{equation}\label{figa25}
    \boldsymbol\mu_0\,=\, \mathbf K_0^{-1}\, \overline{\mathbf t_1}.
\end{equation}
Our considerations are based on the following algebraic  lemma.

\begin{lemma}\label{figa26}
Let $k=0$. There is $\varepsilon_0>0$ with the following properties. For every $|\varepsilon|\leq \varepsilon_0$,
the system of linear equations
\begin{equation}\label{figa27}\begin{split}
 \boldsymbol\Omega  {\boldsymbol\lambda}_1+
 \overline{\mathbf T}\,\,\boldsymbol \mu_1+p_1\,\,\mathbf e_1=0,\\
\overline{\mathbf T}^\top{\boldsymbol\lambda}_1+
\overline{\mathbf S}\,\,\boldsymbol\mu_1
=0, \\
 \boldsymbol{\lambda}_1\mathbf e_1=1.
\end{split}\end{equation}
 has a unique solution $\boldsymbol\lambda_1\in \mathbb C^2$, $\boldsymbol \mu_1\in \mathbb C^{n-1}$, $p_1
 \in \mathbb C$ such that
\begin{equation}\label{figa28x}
 \boldsymbol\mu_1=\boldsymbol\mu_0+\boldsymbol\mu_2, \quad \boldsymbol\lambda_1=
 \mathbf e_1-(\mathbf t_0\cdot \boldsymbol\mu_0)\mathbf e_2+\boldsymbol\lambda_2
\end{equation}
\begin{equation}\label{figa28}\begin{split}
    |\boldsymbol\mu_2|\leq c |\varepsilon|\,\, |\boldsymbol \mu_0|,\quad
    |p_1|\leq  c |\varepsilon|\,\, |\boldsymbol \mu_0|,\quad
|\boldsymbol\lambda_2|\leq  c |\varepsilon|\,\, |\boldsymbol \mu_0|.
\end{split}\end{equation}

\end{lemma}
\begin{proof} Rewrite the first two equations  in the form
\begin{equation*}\begin{split}
   \boldsymbol\Omega  {\boldsymbol\lambda}_1+{\mathbf T_0}\,\,
   \boldsymbol\mu_1+ p_1\,\,\mathbf e_1=
(\mathbf T_0-\overline{\mathbf T})\boldsymbol\mu_1,\\
{\mathbf T_0}^\top\,\,{\boldsymbol\lambda}_1+{\mathbf S_0}\,\,
\boldsymbol\mu_1 =(\mathbf T_0^\top-\overline{\mathbf
T}^\top)\,{\boldsymbol\lambda}_1+
(\mathbf S_0-\overline{\mathbf S})\,\,\boldsymbol \mu_1.
\end{split}\end{equation*}
Notice that
$
{\boldsymbol\lambda}_1={\lambda}_{1,2}\mathbf e_2+
\mathbf e_1.
$
 Thus we  get the linear system of the
equations for $\lambda_{1,2}$ and $\boldsymbol\mu_1$
\begin{equation}\label{figa29}\begin{split}
\lambda_{1,2}+\mathbf t_0\cdot
\boldsymbol\mu_1=\{(\mathbf T_0-
\overline{\mathbf T})\boldsymbol\mu_1\}\cdot \mathbf e_2\\
\mathbf S_0\boldsymbol\mu_1+\lambda_{1,2}\, \mathbf t_0=
(\mathbf T_0^\top-\overline{\mathbf T}^\top)\mathbf e_1
+\\\lambda_{1,2}(\mathbf T_0^\top-\overline{\mathbf T}^\top)
\mathbf e_2+(\mathbf S_0-\overline{\mathbf
S})\boldsymbol\mu_1.
\end{split}\end{equation}
  Express
$\lambda_{1,2}$ in terms of $\boldsymbol\mu_1$ using the
first equation in \eqref{figa29}. Substituting the result into the second equation in
\eqref{figa29} we obtain the following equation for
$\boldsymbol\mu_1$
\begin{equation*}\begin{split}
   \mathbf K_0\boldsymbol\mu_1=(\mathbf S_0-\overline{\mathbf S})
    \boldsymbol\mu_1+
    \{(\mathbf T_0-\overline{\mathbf T})\boldsymbol\mu_1\}_2
    \big((\mathbf T_0^\top-\overline{\mathbf T}^\top){\boldsymbol e}_2- \mathbf t_0\big)\\-
   ( \boldsymbol\mu_1\cdot \mathbf t_0)
   (\mathbf T_0^\top-\overline{\mathbf T}^\top){\mathbf e}_2
    +(\mathbf T_0^\top-\overline{\mathbf T}^\top){\boldsymbol e}_1,
\end{split}\end{equation*}
where the matrix $\mathbf K_0= \mathbf S_0-\mathbf t_0\otimes
\mathbf t_0$ has a
 bounded inverse. Using the identities
 $(\mathbf T_0^\top-\overline{\mathbf T}^\top){\boldsymbol e}_2- \mathbf t_0=-
 \overline{\mathbf t}_2$ and
 $(\mathbf T_0^\top-\overline{\mathbf T}^\top){\boldsymbol e}_1=\overline{\mathbf t}_1$, we can
 rewrite this equation in the equivalent form
\begin{equation*}\begin{split}
   \mathbf K_0\boldsymbol\mu_1=(\mathbf S_0-\overline{\mathbf S})
    \boldsymbol\mu_1-
    \{(\mathbf T_0-\overline{\mathbf T})\boldsymbol\mu_1\}_2\, \overline{\mathbf t}_2
    \\-
   ( \boldsymbol\mu_1\cdot \mathbf t_0)
   (\mathbf T_0^\top-\overline{\mathbf T}^\top){\mathbf e}_2
    -\overline{\mathbf t}_1
\end{split}\end{equation*}
Thus we get the following equation for $\boldsymbol\mu_0$
\begin{equation}\label{figa30}
\boldsymbol\mu_1-\mathbf A\boldsymbol\mu_1=-\boldsymbol\mu_0,\
\end{equation}
where the linear mapping $\mathbf A: \mathbb C^2\to \mathbb C^2$ is
given by
$$
\mathbf A:\boldsymbol\mu_1\mapsto \mathbf K_0^{-1}
\Big\{(\mathbf S_0-\overline{\mathbf S})\boldsymbol\mu_1-
    \{(\mathbf T_0-\overline{\mathbf T})\boldsymbol\mu\}_2\, \overline{\mathbf t}_2
    -(\boldsymbol\mu_1\cdot \mathbf t_0)
   \big((\mathbf T_0^\top-\overline{\mathbf T}^\top){\boldsymbol e}_2 \Big\}.
$$
In view  of estimates \eqref{kira15}, the
mapping $\mathbf A$
 admits the estimate
\begin{equation}\label{figa30}
|\mathbf A\boldsymbol\mu_1|\leq
c|\varepsilon|\, |\boldsymbol\mu_1|.
\end{equation}
 Choosing $\varepsilon_0$ sufficiently small  we obtain that equation \eqref{figa30} has a
 unique  solution which admits decomposition estimate \eqref{figa28x} with
 the reminder
 $$
 \boldsymbol \mu_2=(\mathbf I-(\mathbf I-\mathbf A)^{-1}) \boldsymbol\mu_0.
 $$
 When equation  \eqref{figa30} is solved the vector ${\boldsymbol\lambda}_1$ and the scalar $p_1$ are restored
 by the relations
\begin{equation*}
{\boldsymbol\lambda}_1= \mathbf e_1 -(\mathbf t_0\cdot \boldsymbol\mu_1)\mathbf e_2
\end{equation*}
\begin{equation}\label{nisa33}
p_1=(\mathbf T_0-\overline{\mathbf T})\boldsymbol\mu_1\cdot \mathbf e_1
\end{equation}
which gives decomposition \eqref{figa28x} with the reminder
$$
\boldsymbol \lambda_2=-(\mathbf t_0\cdot \boldsymbol\mu_1)\mathbf e_2
$$
Estimate \eqref{figa28} for $\boldsymbol\lambda_2$ and $p_1$ obviously follows
from estimate \eqref{figa28} for $\boldsymbol\mu_2$ and the
decomposition \eqref{figa28x} for $\boldsymbol\mu_1$ .
\end{proof}

The following theorem, which is the third main result of this section, constitutes the
 properties of solutions to problem \eqref{kira29}
\begin{theorem}\label{figa18} Under the assumptions of of Theorem \ref{figa5},
there exists $\varepsilon_0>0$ with the following property. For every $|\varepsilon|\leq\varepsilon_0$,
problem \eqref{kira28} has a unique analytic solution $(\boldsymbol{\mu}^{(1)},
\boldsymbol{\lambda}^{(1)}, \boldsymbol{\chi}^{(1)})\in \mathcal A_{\sigma/4,0}$, $p^{(1)}\in \mathbb C$.
This solution admits the estimate
\begin{equation}\label{figa19}
  \|\boldsymbol{\mu}^{(1)}\|_{\sigma/4,0}+  \|\boldsymbol{\lambda}^{(1)}\|_{\sigma/4,0}+
  \|\boldsymbol{\chi}^{(1)}\|_{\sigma/4,0}+|p^{(1)} |\leq c.
\end{equation}
If $k=0$, then this solution has the decomposition
\begin{equation}\label{figa20}
 \boldsymbol{\mu}^{(1)}=-\boldsymbol\mu_0+\boldsymbol\mu_\varepsilon, \quad
 {\boldsymbol\lambda}^{(1)}= \mathcal C\mathbf e_1-(\boldsymbol \mu_0\cdot{\mathbf t_0})\mathbf e_2
 +\boldsymbol\lambda_\varepsilon.
\end{equation}
Here the  the constant $\mathcal C$ and functions $\boldsymbol\mu_\varepsilon$,
$\boldsymbol\lambda_\varepsilon$  admit the estimates
\begin{equation}\label{figa21}
|\mathcal C-1|\leq c |\varepsilon|,\quad\|\boldsymbol{\mu}_\varepsilon\|_{\sigma/4,0}+
    \|\boldsymbol{\lambda}_\varepsilon\|_{\sigma/4,0}\leq c |\varepsilon||\boldsymbol\mu_0|.
\end{equation}
The constant $c$ is independent of $\alpha$, $k$, and $\varepsilon$.
\end{theorem}
\begin{proof}
Notice that problem \eqref{kira29} is a particular case of problem \eqref{figa4} with the right hand sides
$g_1*=0$, $  \mathbf g_2=0$ , and $ \gamma=1$. Hence, the existence of a solution and estimate
\eqref{figa19} is a straightforward consequence of Theorem \ref{figa5}. It remains to prove  decompositiion
\eqref{figa21}.

Now our task is to  justify decomposition \eqref{figa20}. Let $\boldsymbol\mu_1$, $\boldsymbol \lambda_1$,
and $p_1$ be given by Lemma \ref{figa26}. Set
$$
\boldsymbol\chi_1=\overline{\boldsymbol{\chi}_1}+\boldsymbol{\chi}_1^*,
$$
Here   $\boldsymbol \chi_1^*$ is a solution solution to the equations
\begin{equation}\label{figa31}
    \boldsymbol \partial \boldsymbol\chi_1^*=\mathbf S\boldsymbol\mu_1+\mathbf T^\top \boldsymbol
    \lambda_1, \quad \overline{\boldsymbol\chi_1^*}=0.
\end{equation}
By virtue of \eqref{figa27}, the mean value of the right hand side of
equation \eqref{figa31} over the torus $\mathbb T^{n-1}$ is equal to zero.
Hence this problem have an analytic periodic solution. The constant vector $\overline{\boldsymbol\chi_1}$
 is defined by
 \begin{equation}\label{figa32}
    \overline{\boldsymbol \chi_1}= -\overline{(\mathbf V^{-\top}-\mathbf I)\boldsymbol\chi^*}
 \end{equation}
 Represent the solution $(\boldsymbol{\mu}^{(1)},
\boldsymbol{\lambda}^{(1)}, \boldsymbol{\chi}^{(1)},p^{(1)})$ to problem \eqref{kira28} in the form
\begin{equation}\label{figa33}
 \boldsymbol{\mu}^{(1)}=\mathcal C\boldsymbol{\mu}_1+\boldsymbol{\mu} ,\quad
 \boldsymbol{\lambda}^{(1)}=\mathcal C\boldsymbol{\lambda}_1+\boldsymbol{\lambda} ,\quad
  \boldsymbol{\chi}^{(1)}=\mathcal C\boldsymbol{\chi}_1+\boldsymbol{\chi} , \quad
  {p}^{(1)}=\mathcal C\boldsymbol{p}_1+\boldsymbol{p} ,
\end{equation}
where the constant $\mathcal C$ will be specified below.
The calculations show that
\begin{subequations}\label{figa4rt}\begin{gather}\label{figa4art}\boldsymbol\partial\, \boldsymbol\mu_1\,=\,0
,\\
\label{figa4brt}
\mathbf J\boldsymbol\partial\, \boldsymbol\lambda_1\,+\, \boldsymbol\Omega\, \boldsymbol\lambda_1\,+\,
\mathbf T\, \boldsymbol\mu_1+p_1\mathbf W^\top \mathbf e_1=p_1(\mathbf W^\top-\mathbf I) \mathbf e_1\\
\label{figa4crt}
-\boldsymbol\partial\, \boldsymbol\chi_1\,+\,\mathbf S\boldsymbol\mu_1\, +\, \mathbf T^\top
\boldsymbol\lambda_1=0
\end{gather}
Next, notice that $\mathbf V^{-\top}-\mathbf I=\mathbf u'$, which yields
$$
\overline{\mathbf V^{-\top}-\mathbf I}=\overline{\mathbf u'}=0, \quad
\overline{(\mathbf V^{-\top}-\mathbf I)\boldsymbol\chi_1}=
\overline{(\mathbf V^{-\top}-\mathbf I)\boldsymbol\chi_1^*}
$$
It follows from this and \eqref{figa32} that
\begin{gather}\label{figa4drt}
\overline{\{\mathbf V^{-\top}\boldsymbol\chi_1\}}\,=\, 0.
\end{gather}
Finally set
\begin{gather}\label{figa4ert}
\mathcal C=\Big\{\overline{\{\mathbf W\boldsymbol\lambda_1\}\cdot \mathbf e_1}\,\,+\,\,
\overline{\{\boldsymbol\chi_1\cdot \nabla w_1\}}
\Big\}^{-1},
\end{gather}
\end{subequations}
Substituting decomposition \eqref{figa33} in equations \eqref{kira28} and using relations \eqref{figa4rt}
we conclude that the quantities $(\boldsymbol\mu, \boldsymbol\lambda, \boldsymbol \chi, p)$ in \eqref{figa33}
satisfy equations \eqref{figa4} with the right hand sides
\begin{equation}\label{figa34}
    g_1=\mathcal C p_1 w_1^*, \quad\boldsymbol g_2=-\mathcal C p_1(\mathbf W^\top-\mathbf I) \mathbf e_1,\quad
    \gamma=0.
\end{equation}
In order to apply Theorem \eqref{figa5} to the obtained problem \eqref{figa4}, we have to estimate
the the constant $\mathcal C$ and functions $g_1$, $\mathbf g_2$.  In order to estimate $\mathcal C$, we
substitute decomposition \eqref{figa28x} into the right hand side of relation \eqref{figa4ert}. We get
\begin{equation*}
   \{\mathbf W\boldsymbol\lambda_1\}\cdot \mathbf e_1=1+\boldsymbol\lambda_2+
    ( \{(\mathbf W-\mathbf I)\boldsymbol\lambda_1\}\cdot \mathbf e_1).
\end{equation*}
Notice that
$$
|\boldsymbol\lambda_2|\leq c|\boldsymbol\mu_0|\leq c |\varepsilon|, \quad
|\boldsymbol\lambda_1|\leq c, \quad |\mathbf W-\mathbf I|\leq c |\varepsilon|,
$$
which yields the estimate
\begin{equation}\label{figa35}
    |\overline{\{\mathbf W\boldsymbol\lambda_1\}\cdot \mathbf e_1}-1|\leq c|\varepsilon|
\end{equation}
Next, estimate \eqref{kira16} and estimates \eqref{kira28} imply the inequality
$\| \mathbf S\boldsymbol\mu_1+\mathbf T^\top \boldsymbol
    \lambda_1\|_{\sigma/2,0}\leq c$. Hence we can apply Lemma \ref{nisa6} to equation \eqref{figa31}
 to obtain $\|\boldsymbol\chi_1^*\|_{\sigma/4,0}\leq c$.
    On the other hand, we have
    $|\nabla w_1*|\leq c|\varepsilon|$. It follows that
\begin{equation}\label{figa36}
|\overline{\boldsymbol\chi_1\cdot \nabla w_1^*}|=
|\overline{\boldsymbol\chi_1^*\cdot \nabla w_1^*}|\leq c|\varepsilon|.
\end{equation}
Combining \eqref{figa35}, \eqref{figa36} and recalling expression \eqref{figa4ert} we arrive at the estimate
$$
|\mathcal C^{-1}-1|\,\leq\, c|\varepsilon|
$$
Choosing $|\varepsilon|\leq \varepsilon_0$ sufficiently small we obtain the desired estimate \eqref{figa21}
for $\mathcal C$. We are now in a position to estimate the functions $g_1$ and $\mathbf g_2$ in
\eqref{figa34}. It follows from estimate \eqref{figa28} for $p_1$ and estimate \eqref{figa21} for $\mathcal C$ that
\begin{equation*}
    \|g_1\|_{\sigma/2,0}\leq c |p_1|\|w_1^*\|_{\sigma/2,0}\leq c |\boldsymbol\mu_0| |\varepsilon|, \quad
 \|g_2\|_{\sigma/2,0}\leq c |p_1|\|\mathbf W-\mathbf I\|_{\sigma/2,0}\leq c |\boldsymbol\mu_0| |\varepsilon|
\end{equation*}
Applying Theorem \ref{figa5} to the problem \eqref{figa4} with the right hand sides
\eqref{figa34} we arrive at the estimate
 \begin{equation}\label{figa37}
    \|\boldsymbol{\mu}\|_{\sigma/4,0}+
    \|\boldsymbol{\lambda}\|_{\sigma/4,0}\leq c |\varepsilon||\boldsymbol\mu_0|.
 \end{equation}
 Finally set
 $$
\boldsymbol{\mu}_\varepsilon= \boldsymbol\mu +(\mathcal C-1)\boldsymbol\mu_1+\boldsymbol\mu_2,\quad
\boldsymbol{\lambda}_\varepsilon=\boldsymbol\lambda -
(\mathcal C-1)(\mathbf t_0\cdot \boldsymbol\mu_0)\mathbf e_2+\mathcal C\boldsymbol\lambda_2.
$$
It remains to note that decomposition \eqref{figa20} follows from decompositions \eqref{figa28x} and \eqref{figa33},
and estimates \eqref{figa21} for $\boldsymbol{\mu}_\varepsilon$ and
$\boldsymbol{\lambda}_\varepsilon$ follow from estimate \eqref{figa21}\ for $\mathcal C$,
estimate \eqref{figa37} for $\boldsymbol{\mu}$ and $\boldsymbol\lambda$ and estimates \eqref{figa28}
for $\boldsymbol{\mu}_i$ and $\boldsymbol\lambda_i$.

\end{proof}

\section{Quadratic form}\label{zara}
Let  $\boldsymbol\mu^{(i)}$,  $\boldsymbol\lambda^{(i)}$ ,
given by Theorems \ref{figa8} and \ref{figa18}, be
solutions to problems \eqref{kira28} and \eqref{kira29}.
We define the quadratic form $L=(L_{ij})$, $1\leq i,j\leq 2$, as follows
\begin{equation}\label{zara1}
    L_{ij}=\int_{\mathbb T^{n-1}}\big\{\,\,\mathbf S\boldsymbol\mu^{(i)}\cdot\boldsymbol\mu^{(j)}-
    (\mathbf J\boldsymbol\partial
    \boldsymbol\lambda^{(i)}+\boldsymbol\Omega\boldsymbol\lambda^{(i)})\cdot
    \boldsymbol\lambda^{(j)}\,\,\big\}\, d \boldsymbol\xi.
\end{equation}
This form  plays the key role in our
analysis of the bifurcation equation \eqref{lunabifurcation}. As it will be shown in the next section,
the derivatives of the action functional are represented in the terms of this form. The following theorem, which
is the main result of this section, describe the properties of the quadratic form \eqref{zara1}.
\begin{theorem}\label{zara2}
Under the assumptions of of Theorem \ref{figa5},
there exists $\varepsilon_0>0$ with the following property. For every $|\varepsilon|\leq\varepsilon_0$,
the elements of the quadratic form $L$ satisfy the inequalities
\begin{equation}\label{zara3}\begin{split}
|L_{12}|\leq c |\varepsilon|, \quad
 |L_{11}|\leq c, \\ c^{-1} |w_1^*|_{-1}^2 \,\leq \,|L_{22}|\,\leq\, c\|w_1^*\|_{\sigma/2,0}^2\leq c |\varepsilon|^2.
\end{split}\end{equation}
If $k=0$, then
\begin{equation}\label{zara4}
  |L_{11}- (2\pi)^{n-1} \mathbf K_0\boldsymbol\mu_0\cdot \boldsymbol \mu_0|\leq c |\varepsilon|
  |\boldsymbol \mu_0|^2, \quad  |L_{12}|\leq c |\varepsilon|\,|\boldsymbol \mu_0| .
\end{equation}
Here the strictly positive constant $c$ is independent of $\alpha$, $k$, and $\varepsilon$,
the vector $\boldsymbol\mu_0$ is given by \eqref{figa25}.
\end{theorem}
\begin{proof}
We begin with the observation that
\begin{multline*}
    |L_{ij}|\leq c
    ( |\boldsymbol\mu^{(i)}|_0|\boldsymbol\mu^{(j)}|_0+ |\boldsymbol\lambda^{(i)}|_1|\boldsymbol\lambda^{(j)}|_0)\\
    \leq c( \|\boldsymbol\mu^{(i)}\|_{\sigma/4,0}\|\boldsymbol\mu^{(j)}\|_{\sigma/4,0}+
    \|\boldsymbol\lambda^{(i)}\|_{\sigma/4,0}|\boldsymbol\lambda^{(j)}|_{\sigma/4,0}).
\end{multline*}
Therefore, estimates of $|L_{ij}|$ from above in \eqref{zara4} obviously follows from estimates
\eqref{figa9} and   \eqref{figa19}   in Theorems \ref{figa8} and \ref{figa18}. Let us
estimate $L_{22}$ from below.
It follows from the expressions \eqref{ee1.02} and \eqref{zabyli} for $\mathbf J$ and $\boldsymbol\Omega$
that
\begin{equation*}
    L_{22}= \int_{\mathbb T^{n-1}} \big\{\,\,\mathbf S\boldsymbol\mu^{(2)}\cdot\boldsymbol\mu^{(2)}
  +k{\lambda_1^{(2)}}^2- {\lambda_2^{(2)}}^2
  -\lambda_1^{(2)}\boldsymbol\partial
\lambda_2^{(2)}+\lambda_2^{(2)}\boldsymbol\partial \lambda_1^{(2)}\,\big\}\, d \boldsymbol\xi.
\end{equation*}
Integrating by parts gives
\begin{equation}\label{zara5}
    L_{22}= \int_{\mathbb T^{n-1}} \big\{\,\,\mathbf S\boldsymbol\mu^{(2)}\cdot\boldsymbol\mu^{(2)}
  +k{\lambda_1^{(2)}}^2- {\lambda_2^{(2)}}^2
  +2\lambda_2^{(2)}\boldsymbol\partial \lambda_1^{(2)}\,\big\}\, d \boldsymbol\xi.
\end{equation}
Next,  multiplying both sides of equation \eqref{kira29b} by $\mathbf e_2$ we obtain
\begin{equation*}
-\boldsymbol\partial \lambda_1^{(2)}+ \lambda_2^{(2)}+\mathbf t_2\cdot \boldsymbol\mu^{(2)}=
-(p^{(2)}+w_1^*) \mathbf W^\top \mathbf e_1\cdot \mathbf e_2.
\end{equation*}
It follows that
\begin{equation*}
    2\boldsymbol\partial \lambda_1^{(2)}-\lambda_2^{(2)}=\lambda_2^{(2)}+
    2\mathbf t_2\cdot \boldsymbol\mu^{(2)}+2(p^{(2)}+w_1^*) \mathbf W^\top \mathbf e_1\cdot \mathbf e_2.
\end{equation*}
Substituting this expression into \eqref{zara5} and noting that $\mathbf W^\top\mathbf e_1\cdot \mathbf e_2
=W_{12}$  we arrive at the identity
\begin{multline}\label{zara6}
    L_{22}= \int_{\mathbb T^{n-1}} \big\{\,\,\mathbf S\boldsymbol\mu^{(2)}\cdot\boldsymbol\mu^{(2)}
  +k{\lambda_1^{(2)}}^2+ {\lambda_2^{(2)}}^2\\
  +2\lambda_2^{(2)}(\mathbf t_2\cdot \boldsymbol\mu^{(2)}+2(p^{(2)}+w_1^*))
   W_{12}\big\}\, d \boldsymbol\xi.
\end{multline}
Introduce the denotation $\varkappa=W_{12} W_{11}^{-1}$. Multiplying both side of \eqref{kira29b} by
$\varkappa\mathbf e_1$ we arrive at
\begin{equation*}
    -(p^{(2)}+w_1^*)W_{12}=-\varkappa(p^{(2)}+w_1^*)W_{11}=\varkappa(\boldsymbol\partial \lambda_2^{(2)}
    -k\lambda_1^{(2)} +\mathbf t_1\cdot\boldsymbol\mu^{(2)})
\end{equation*}
Thus we get
\begin{multline*}
   \int_{\mathbb T^{n-1}} 2\lambda_2^{(2)}(\mathbf t_2\cdot \boldsymbol\mu^{(2)}+2(p^{(2)}+w_1^*))
   W_{12}\, d \boldsymbol\xi=\\
  \int_{\mathbb T^{n-1}} \{\,2\lambda_2^{(2)}(\mathbf t_2\cdot \boldsymbol\mu^{(2)}) -
  \varkappa\lambda_2^{(2)}\boldsymbol\partial \lambda_2^{(2)} +\varkappa k\lambda_2^{(2)}\lambda_1^{(2)}
-\varkappa \lambda_2^{(2)}(\mathbf t_1\cdot\boldsymbol\mu^{(2)})\}\, d \boldsymbol\xi=\\
\int_{\mathbb T^{n-1}} \{\,2\lambda_2^{(2)}(\mathbf t_2\cdot \boldsymbol\mu^{(2)}) +
 \frac{1}{2}\boldsymbol\partial \varkappa(\lambda_2^{(2)})^2 +\varkappa k\lambda_2^{(2)}\lambda_1^{(2)}
-\varkappa \lambda_2^{(2)}(\mathbf t_1\cdot\boldsymbol\mu^{(2)})\}\, d \boldsymbol\xi:=\mathbf I_0.
\end{multline*}
Let us estimate the quantity $\mathbf I_0$. It follows from estimates \eqref{kira5} for the matrix $\mathbf W$
in Theorem \ref{kira2} and estimate \eqref{kira16}  for the matrix
$\mathbf T^\top=[\mathbf t_1, \mathbf t_2]$ in Lemma \ref{kira13} that
\begin{gather}\label{zara7}
    |\varkappa|+|\boldsymbol\partial\varkappa|\leq c|\varepsilon|, \quad |\mathbf t_2|\leq c, \quad
    |\mathbf t_1|\leq c|\varepsilon|.
\end{gather}
Thus we get
\begin{multline}\label{zara8}
    |\mathbf I_0|\leq c \int_{\mathbb T^{n-1}} |\lambda_2^{(2)}|| \boldsymbol\mu^{(2)}|
    d \boldsymbol\xi+c|\varepsilon| \int_{\mathbb T^{n-1}}( |\lambda_2^{(2)}|^2+
    k|\lambda_1^{(2)}|^2)d \boldsymbol\xi\leq\\
   c|\varepsilon| \int_{\mathbb T^{n-1}}( |\lambda_2^{(2)}|^2+
    k|\lambda_1^{(2)}|^2)d \boldsymbol\xi+
    |\varepsilon|^{-1} \int_{\mathbb T^{n-1}}| \boldsymbol\mu^{(2)}|^2d \boldsymbol\xi
\end{multline}
On the other hand, we have
\begin{equation*}
    L_{22}= \int_{\mathbb T^{n-1}} \big\{\,\,\mathbf S\boldsymbol\mu^{(2)}\cdot\boldsymbol\mu^{(2)}
  +k{\lambda_1^{(2)}}^2+ {\lambda_2^{(2)}}^2\,\,\big\}+ \mathbf I_0.
\end{equation*}
It follows from this and \eqref{zara8} that
\begin{equation*}
    L_{22}\geq (1-c|\varepsilon|)|\lambda_2^{(2)}|_0^2+k|\lambda_1^{(2)}|^2-c|\varepsilon|^{-1}
    |\boldsymbol\mu^{(2)}|_0^2
\end{equation*}
Applying estimates \eqref{figa9} and \eqref{figa10} in Theorem \ref{figa8} we arrive at the inequality
\begin{equation*}
     L_{22}\geq c^{-1}(1-c|\varepsilon|) |w_1^*|_{-1}^2- c |\varepsilon|  |w_1^*|_{-1}^2.
\end{equation*}
Choosing $\varepsilon$ sufficiently small we obtain the desired estimate \eqref{zara3}.
It remains to prove inequalities \eqref{zara4}.
Notice that for  $k=0$, the equality $\boldsymbol\Omega \mathbf a=a_2\mathbf e_2$
holds for every vector $\mathbf a$. Substituting decomposition \eqref{figa20} into the expression
\eqref{zara1} we arrive at the identity
\begin{equation}\label{zara9}
    L_{11}=\int_{\mathbb T^{n-1}} \big\{\,\,\mathbf S_0\boldsymbol\mu_0\cdot\boldsymbol\mu_0
-(\boldsymbol\mu_0\cdot \mathbf t_0)^2\,\big\}\, d\boldsymbol\xi+\mathbf I_1,
\end{equation}
where
\begin{multline}\label{zara10}
    \mathbf I_1=\int_{\mathbb T^{n-1}}(\mathbf S-\mathbf S_0)\boldsymbol\mu_0\cdot\boldsymbol\mu_0
\, d\boldsymbol\xi+\int_{\mathbb T^{n-1}}\mathbf S\boldsymbol\mu_\varepsilon\cdot
(2\boldsymbol\mu_0+\boldsymbol \mu_\varepsilon)\, d\boldsymbol\xi\\
-\int_{\mathbb T^{n-1}}\big\{(\mathbf J\boldsymbol\partial\boldsymbol\lambda _\varepsilon +
\boldsymbol\Omega\boldsymbol\lambda _\varepsilon)\cdot \boldsymbol\lambda _\varepsilon+2(\boldsymbol\mu_0\cdot \mathbf t_0)
\lambda_{\varepsilon, 2}\big\}, d\boldsymbol\xi
\end{multline}
Estimates \eqref{kira16}  for the matrices $\mathbf S$ and
$\mathbf T^\top=[\mathbf t_1, \mathbf t_2]$ in Lemma \ref{kira13} imply
\begin{equation*}
    |\mathbf S-\mathbf S_0|\leq c|\varepsilon|, \quad |\mathbf t_2|\leq c.
\end{equation*}
In its turn, estimates \eqref{figa21} in Theorem \ref{figa18} imply
\begin{equation*}
|\boldsymbol{\mu}_\varepsilon|+|\boldsymbol\partial\boldsymbol{\mu}_\varepsilon+|
    |\boldsymbol{\lambda}_\varepsilon|\leq c |\varepsilon||\boldsymbol\mu_0|.
\end{equation*}
Combining these estimates we finally obtain
\begin{equation}\label{zara11}
|\mathbf I_1|\leq c |\varepsilon| |\boldsymbol\mu_0|^2.
\end{equation}
Next notice that
$$
\int_{\mathbb T^{n-1}} \big\{\,\,\mathbf S_0\boldsymbol\mu_0\cdot\boldsymbol\mu_0
-(\boldsymbol\mu_0\cdot \mathbf t_0)^2\,\big\}\, d\boldsymbol\xi=(2\pi)^{n-1}
\mathbf K_0\boldsymbol\mu_0\cdot \boldsymbol\mu_0.
$$
Substituting this equality into \eqref{zara9} and using inequality  \eqref{zara11} we obtain estimate \eqref{zara4}
for $L_{11}$. It remains to estimate $L_{12}$.
Since $k=0$ we have
\begin{equation}\label{zara12}
    L_{12}=\int_{\mathbb T^{n-1}}\big\{\,\,\mathbf S\boldsymbol\mu^{(1)}\cdot\boldsymbol\mu^{(2)}-
    (\mathbf J\boldsymbol\partial
    \boldsymbol\lambda^{(2)}+\boldsymbol\Omega\boldsymbol\lambda^{(2)})\cdot
    (\boldsymbol\lambda^{(1)}-\mathcal C\mathbf e_1)\,\,\big\}\, d \boldsymbol\xi.
\end{equation}
In view  of estimates \eqref{figa21} in Theorem \ref{figa18}, we have
\begin{equation*}
|\boldsymbol\mu^{(1)}|\leq c |\boldsymbol\mu_0, \quad
|\boldsymbol\lambda^{(1)}-\mathcal C\mathbf e_1||\leq c |\boldsymbol\mu_0|, \quad
|\boldsymbol\partial \boldsymbol\lambda^{(1)}|\leq c |\boldsymbol\mu_0|
\end{equation*}
On the other hand, estimates \eqref{figa9} in Theorem \ref{figa18} yield
\begin{equation*}
 \|\boldsymbol{\lambda}^{(2)}\|_{\sigma/4,0}+
  \|\boldsymbol{mu}^{(2)}\|_{\sigma/4,0}+|p^{(2)} |\leq c|\varepsilon|.
\end{equation*}
This result and the expression \eqref{zara12} for $L_{12}$ lead to the estimate $|L_{12}|\leq
 c |\varepsilon||\boldsymbol\mu_0|$. This completes the proof of Theorem
\end{proof}

\section{Action functional. Proof of Theorem \ref{theo1.4}}\label{sveta}
Theorem \ref{kira2} constitutes the existence and uniqueness of the solution
 \begin{equation}\label{sveta1}\begin{split}
\mathfrak u=\big(\boldsymbol\varphi(\alpha,k), e(\alpha,k)\big),
m(\alpha,k), M(\alpha,k),\\
\boldsymbol\varphi(\alpha,k)= (\boldsymbol\beta, \varphi_0, \mathbf
u, \mathbf w, W_{11},W_{12},W_{21}).
\end{split}\end{equation}
of the main operator equation \eqref{luna8x} for every $(\alpha, k)\in \mathbb T\times [0,1]$
and for all sufficiently small $\varepsilon$. This means  that
the vector
\begin{equation}\label{sveta2}
 \boldsymbol\Theta(\boldsymbol\varphi)=   \big(\mathbf u, \mathbf v, \mathbf w, \mathbf V,\boldsymbol \Lambda, \mathbf W,
    \mathbf R_i\big),
\end{equation}
with the components $\mathbf v$, $\mathbf V$, $\mathbf W$, and $\mathbf R_i$ defined
by relations \eqref{anna17},
determines the canonical mapping $\boldsymbol\vartheta$ given by \eqref{anna1}.
By the definition of the operator $\Phi$, this mapping puts the modified Hamiltonian $ H_m$
into the normal form
\eqref{normalform0}. By virtue of Definition \ref{def1.1} of the normal form, the modified Hamiltonian
has a weakly hyperbolic invariant torus. In the phase  space, these torus has the
the parametric representation
\begin{equation}\label{sveta3}
    \mathbf x(\boldsymbol\xi)=\boldsymbol\xi+\mathbf u(\boldsymbol\xi), \quad
    \mathbf y(\boldsymbol\xi)=\mathbf v(\boldsymbol\xi),\quad \mathbf z=\mathbf w(\boldsymbol\xi),\quad
  \boldsymbol\xi\in \mathbb T^{n-1} ,
\end{equation}
in which $\mathbf u$ , $\mathbf w$ are the components of the vector $\boldsymbol\varphi(\alpha,k)$, and
the component $\mathbf v$ of the vector $\boldsymbol\Theta$ is defined  by the relation \eqref{anna16}, i.e.,
\begin{equation}\label{sveta4}\begin{split}
\mathbf v=\boldsymbol \beta +
 \mathbf V(\nabla\varphi_0-w_2\nabla w_1), \quad \mathbf V=(\mathbf I+\mathbf u')^{-\top}.
 \end{split}\end{equation}
By virtue of Theorem \ref{ira18}, the main operator equation \eqref{luna8x}  has an analytic
periodic solution for all  $(\alpha, k)\in \mathbb T\times [0,1]$
and for all sufficiently small $\varepsilon$. In other words, the modified Hamiltonian has
the two-parametric family of weakly hyperbolic invariant tori labeled by $(\alpha,k)$.
Notice that $H_m(\mathbf x,\mathbf y,\mathbf z)= H(\mathbf x,\mathbf y,\mathbf z)+
mz_1+2^{-1} M z_1^2$, where $m=m(\alpha,k)$,  $M=M(\alpha,k)$ are the components of the vector
$\mathfrak u$ satisfying \eqref{luna8x}. Therefore, the hamiltonian $H$ has a weakly
hyperbolic invariant torus if and only if
\begin{equation}\label{sveta5}
m(\alpha,k)=0,  M(\alpha,k)=0
\end{equation}
Relations \eqref{sveta5} give the system of two scalar bifurcation equations
for $(\alpha,k)\in \mathbb T^1\times [0,1]$. The difficulty is that the scalars
$(m,M)$ are  the integral part of a solution to the complicated operator equation,
and we know nothing about their properties.
In order to cope with this problem, we notice  that every  Hamiltonian system has
a variational formulation, and its solutions are critical points of the action functional.
For quasi-periodic solutions,  the action functional can be written as the Perceval functional,
see \cite{Percival1979},
\begin{equation}\label{sveta6}
    \Psi=\int_{\mathbf T^{n-1}}\big( (\boldsymbol\omega+\boldsymbol\partial \mathbf u)\mathbf v+
    w_2\boldsymbol\partial w_1-H(\boldsymbol \xi+\mathbf u, \mathbf v, \mathbf w)\big)\,
d\boldsymbol\xi.
\end{equation}
In view of Theorem \ref{kira2}, the mappings
\begin{equation}\label{sveta7}\begin{split}
\mathbb R\times [0,1]\ni(\alpha, k)\to \mathbf u(\alpha,k)\in
\mathcal A_{\sigma/4,d},\\
\mathbb R\times [0,1]\ni(\alpha, k)\to \mathbf v(\alpha,k)\in
\mathcal A_{\sigma/4,d},\\ \mathbb R\times [0,1]\ni(\alpha, k)\to
\mathbf w(\alpha,k)-\alpha\mathbf e_1\in \mathcal A_{\sigma/2,d},\\ \mathbb R\times [0,1]\ni(\alpha, k)\to
\mathbf W(\alpha,k)\in \mathcal A_{\sigma/2,d},
\\
 \mathbb R\times [0,1]\ni(\alpha, k)\to M(\alpha,k)\in
\mathbb R^1,
\\
\mathbb R\times [0,1]\ni(\alpha, k)\to m(\alpha,k)+\alpha
M(\alpha,k)\in \mathbb R^1,
\end{split}\end{equation}
are continuously differentiable and $2\pi$-periodic in $\alpha$.
Moreover, they are analytic in $\mathbb R\times (0,1)$ and satisfy
the inequalities
\begin{equation}\label{sveta8}\begin{split}
    \|\partial_\alpha^r\mathbf u\|_{\sigma/2,d}+
    \|\partial_\alpha^r\mathbf v\|_{\sigma/2,d}\leq
    c(r)|\varepsilon|,\\
\|\partial_\alpha^r(\mathbf w-\alpha\mathbf e_1)\|_{\sigma/2,d}++\|\partial_\alpha^r(\mathbf W-\mathbf I)
\|_{\sigma/2,d}\leq c(r)|\varepsilon|,\\
|\partial_\alpha^r\boldsymbol\beta|+|\partial_\alpha^r(m+\alpha M)|+
|\partial_\alpha^r (M+k)|\leq c(r)|\varepsilon|,
\end{split}\end{equation}
and
\begin{equation}\label{sveta9}\begin{split}
    \|\partial_\alpha^r\partial_k\mathbf u\|_{\sigma/2,d}+
    \|\partial_\alpha^r\partial_k\mathbf v\|_{\sigma/2,d}\leq
    c|\varepsilon|,\\
\|\partial_\alpha^r\partial_k(\mathbf w-\alpha\mathbf e_1)\|_{\sigma/2,d}+
\|\partial_\alpha^r\partial_k(\mathbf W-\mathbf I)
\|_{\sigma/2,d}\leq c|\varepsilon|,\\|\partial_\alpha^r\partial_k(m+\alpha M)|+
|\partial_\alpha^r\partial_k(M+k)|\leq c(r)|\varepsilon|,
\end{split}\end{equation}
where $r\geq 0$ is an arbitrary integer, the constant $c$ is
independent of $\alpha, k$, and $\varepsilon$. It follows from this that
the mapping $\Psi$ is a function of the variables $(\alpha, k)$, which belong to the Banach space
$C^1([0,1]; \mathcal A_{\sigma/2,0})$. In particular, $\Psi$ and $\partial_k \Psi$ are continuous in
$k$ and analytic in $\alpha$ "uniformly" with respect to $k\in[0,1]$.

The following Theorem constitutes the relations between $\Psi$, $m$, $M$ and the quadratic form $(L_{ij})$.
\begin{theorem}\label{sveta10} Under the above assumptions, there is $\varepsilon_0>0$  with the properties.
For every $|\varepsilon|\leq \varepsilon_0$ the derivatives of the function $\Psi$
admit the representation
\begin{gather}\label{sveta11}
    \partial_\alpha\Psi(\alpha,k)=(2\pi)^{n-1} (m+\alpha M)+\varsigma_1 M,
    \\\label{sveta12}
     \partial_k\Psi(\alpha,k)= M\, \partial_k M\, L_{22},
    \\\label{sveta13}
\partial_\alpha^2\Psi(\alpha,k)=(2\pi)^{n-1}  M+\varsigma_2 M +L_{11}+2L_{12} \partial_\alpha M+
L_{22}( \partial_\alpha M)^2.
\end{gather}
Here $L_{ij}$ are given by \eqref{zara1}, the quantities $\varsigma_i(\alpha,k)$ satisfy the inequalities
\begin{equation}\label{sveta14}
    |\varsigma_i|+|\partial_\alpha\varsigma_i|\leq c |\varepsilon|^2.
\end{equation}
\end{theorem}

\begin{proof}
Differentiation \eqref{sveta6} with respect to the variable $\tau$,
$\tau=\alpha, k$, and integrating by parts gives
\begin{gather}\nonumber
    \partial_\tau \Psi(\alpha,k)=\int_{\mathbf
    T^{n-1}}\big((\boldsymbol\omega +\boldsymbol\partial\mathbf
    u-\nabla_y H(\boldsymbol \xi+\mathbf u, \mathbf v, \mathbf
    w))\cdot \partial_\tau \mathbf v\, d\boldsymbol \xi\\\nonumber-
\int_{\mathbf
    T^{n-1}}\big(\boldsymbol\partial\mathbf
    v+\nabla_x H(\boldsymbol \xi+\mathbf u, \mathbf v, \mathbf
    w))\cdot \partial_\tau \mathbf u\, d\boldsymbol \xi+\\\nonumber \int_{\mathbf
    T^{n-1}}\big(\boldsymbol\partial w_1-
    \partial_{z_2} H(\boldsymbol \xi+\mathbf u, \mathbf v, \mathbf
    w))\cdot \partial_\tau w_2\, d\boldsymbol \xi\\\label{sveta15}
-\int_{\mathbf
    T^{n-1}}\big(\boldsymbol\partial w_2+
    \partial_{z_1} H(\boldsymbol \xi+\mathbf u, \mathbf v, \mathbf
    w))\cdot \partial_\tau w_1\, d\boldsymbol \xi
\end{gather}
Next notice that
\begin{gather}\nonumber
\nabla_y H=\nabla_y H_m, \quad \nabla_x H=\nabla_x H_m, \quad
\partial_{z_2} H=\partial_{z_2} H_m, \\\label{sveta16}
\partial_{z_1} H(\boldsymbol \xi+\mathbf u, \mathbf v, \mathbf
    w))+m +Mw_1= \partial_{z_1} H_m(\boldsymbol \xi+\mathbf u, \mathbf v, \mathbf
    w)).
\end{gather}
Recall that the vector $\boldsymbol\varphi$ given by \eqref{sveta1}
serves as a solution to operator equation \eqref{luna8x}. From this
and Lemma \ref{luna11} we conclude that $\mathbf u$, $\mathbf v$,
and $\mathbf w$ satisfy equations \eqref{luna12}, i.e.,
\begin{equation*}\begin{split}
    \boldsymbol\omega +\boldsymbol\partial \mathbf u=\nabla_y H_m(\boldsymbol\xi+
    \mathbf u, \mathbf v, \mathbf w), \\ \boldsymbol\partial \mathbf v=-
    \nabla_x H_m(\boldsymbol\xi+\mathbf u, \mathbf v, \mathbf w),\quad
    \boldsymbol \partial\mathbf w=\mathbf J\nabla_z H_m(\boldsymbol\xi+\mathbf u,
    \mathbf v, \mathbf w).
\end{split} \end{equation*}
Substituting these equality along with \eqref{sveta16} into
\eqref{sveta15} we arrive at the identity
\begin{equation}\label{sveta17}
     \partial_\tau \Psi(\alpha,k)=\int_{\mathbf
    T^{n-1}}(m +w_1 M)\partial_\tau w_1\, d\boldsymbol \xi, \quad
    \tau=\alpha,k.
\end{equation}
Recall the denotations
$$
\overline{w_1}=(2\pi)^{-n+1}\int_{\mathbb T^{n-1}} w_1\,
d\boldsymbol \xi, \quad w_1^*=w_1-\overline{w_1}. $$ Notice that
$$
w_1=\alpha+w_1^*, \quad \partial_\alpha w_1=1+\partial_\alpha w_1^*,
\quad
\partial_k w_1=\partial_k w_1^*.
$$
From this and \eqref{sveta17} we obtain
\begin{equation}\label{sveta18}\begin{split}
\partial_\alpha\Psi(\alpha,k)=(2\pi)^{n-1} (m+\alpha M)+\varsigma_1 M,
\\ \partial_k\Psi(\alpha,k)=M\,\int_{\mathbb T^{n-1}}
w_1^*\partial_k w_1^*\, d\boldsymbol \xi,
\end{split}\end{equation}
where
$$
\varsigma_1=\int_{\mathbb T^{n-1}}w_1^*\partial_\alpha w_1^*\,
d\boldsymbol \xi.
$$
It obviously follows from estimates \eqref{sveta8} that
$\varsigma_1$ satisfies inequality \eqref{sveta14}. This leads to
representation \eqref{sveta11}.

Let us prove representation
\eqref{sveta12}. Recall  formulae \eqref{kira11} which constitutes
the linear algebraic relation between the vector field
$(\partial_\tau \mathbf u,\partial_\tau \mathbf v, \partial_\tau
\mathbf w)$ and the vector field  $(\boldsymbol \mu^{(\tau)},
\boldsymbol \lambda^{(\tau)}, \boldsymbol \chi^{(\tau)})$. In view
of the Second Structure Theorem \eqref{sima12} this relation has the
inverse given by \eqref{sima13}. In particular,  we have
\begin{equation}\label{sveta19}
    \partial_\tau \mathbf w\,=\, \mathbf W \boldsymbol \lambda^{(\tau)}+\chi^{(\tau)}_i\,
\frac{\partial}{\partial \xi_i} \mathbf w.
\end{equation}
On the other hand,  relation \eqref{kira27} in Corollary
\ref{kira26} yields
\begin{equation*}
(\boldsymbol \mu^{(k)}, \boldsymbol \lambda^{(k)}, \boldsymbol
\chi^{(k)})=\partial_k M(\boldsymbol \mu^{(2)}, \boldsymbol
\lambda^{(2)}, \boldsymbol \chi^{(2)})
\end{equation*}
Substituting this relation into identity \eqref{sveta19} with $\tau$
replaced by $k$, we obtain
\begin{equation}\label{sveta20}
    \partial_k w_1^*\,=\, \partial_k \mathbf w\cdot \mathbf e_1 \,=\,\partial_k M\big\{( \mathbf W^\top \mathbf
    e_1)\cdot\boldsymbol\lambda^{(2)}+\nabla w_1^*\cdot
    \boldsymbol\chi^{(2)}\big\}.
\end{equation}
It follows that
\begin{gather*}
    \int_{\mathbb T^{n-1}}w_1^*\partial_k w_1^*\, d
    \boldsymbol\xi= \int_{\mathbb T^{n-1}}(p^{(2)}+w_1^*)\partial_k w_1^*\, d
    \boldsymbol\xi=\\
\partial_k M\int_{\mathbb
    T^{n-1}}\big\{(p^{(2)}+w_1^*) \mathbf W^\top \mathbf
    e_1\big\}\cdot\boldsymbol\lambda^{(2)}\, d
    \boldsymbol\xi+\partial_k M\int_{\mathbb
    T^{n-1}}\big\{(p^{(2)}+w_1^*)\nabla w_1^*\big\}\cdot
    \boldsymbol\chi^{(2)}\, d
    \boldsymbol\xi
\end{gather*}
Next, equations \eqref{kira29} imply the equalities
\begin{gather*}
 (p^{(2)}+ w_1^*) \nabla w_1^*=-   \boldsymbol\partial\, \boldsymbol\mu^{(2)},\\
( p^{(2)}+w_1^*)\mathbf W^\top\boldsymbol e_1 =-
 \mathbf J\boldsymbol\partial\,
\boldsymbol\lambda^{(2)}\,-\, \boldsymbol\Omega\,
\boldsymbol\lambda^{(2)}\,-\, \mathbf T\boldsymbol\mu^{(2)}.
\end{gather*}
Combining the obtained results we arrive at the identity
\begin{multline}\label{sveta23}
\int_{\mathbb T^{n-1}}w_1^*\partial_k w_1^*\, d
    \boldsymbol\xi=-\partial_k M\int_{\mathbb
    T^{n-1}}\Big( (\mathbf J\boldsymbol\partial\,
\boldsymbol\lambda^{(2)}\,\\+\, \boldsymbol\Omega\,
\boldsymbol\lambda^{(2)}\,+\, \mathbf
T\boldsymbol\mu^{(2)})\cdot\boldsymbol\lambda^{(2)}+\boldsymbol\partial\,
\boldsymbol\mu^{(2)}\cdot \boldsymbol\chi^{(2)}\Big)\, d
    \boldsymbol\xi.
\end{multline}
Equation \eqref{kira29c} yields the identity
\begin{gather*}
 -\partial_k M\int_{\mathbb
    T^{n-1}}\boldsymbol\partial\,
\boldsymbol\mu^{(2)}\cdot \boldsymbol\chi^{(2)}\, d
    \boldsymbol\xi=\partial_k M\int_{\mathbb
    T^{n-1}}\boldsymbol\partial\,
\boldsymbol\mu^{(2)}\cdot \boldsymbol\chi^{(2)}\, d
    \boldsymbol\xi\\=
\partial_k M\int_{\mathbb
    T^{n-1}}\big(\mathbf S\boldsymbol\mu^{(2)}\, +\, \mathbf T^\top
\boldsymbol\lambda^{(2)}\big)\cdot\boldsymbol\mu^{(2)}\, d
    \boldsymbol\xi
\end{gather*}
Substituting this equality into \eqref{sveta23} we finally obtain
\begin{multline}\label{sveta24}
    \int_{\mathbb T^{n-1}}w_1^*\partial_k w_1^*\, d
    \boldsymbol\xi=\partial_k M\int_{\mathbb
    T^{n-1}}\Big(\mathbf
    S\boldsymbol\mu^{(2)}\cdot\boldsymbol\mu^{(2)}\\-(\mathbf J\boldsymbol\partial\,
\boldsymbol\lambda^{(2)}\,+\, \boldsymbol\Omega\,
\boldsymbol\lambda^{(2)})\cdot\boldsymbol\lambda^{(2)}\Big)\, d
    \boldsymbol\xi=\partial_k M\,\, L_{22}.
\end{multline}
It remains to notice that desired identity \eqref{sveta12} obviously
follows from \eqref{sveta24} and identity \eqref{sveta18}.

Our next task is to prove identity \eqref{sveta13}. Differentiating
identity \eqref{sveta11} with respect to $\alpha$ and noting that
$\overline{\partial_\alpha^2 w_1}=0$ we obtain
\begin{equation}\label{sveta25}
    \partial_\alpha^2\Psi(\alpha,k)= M\int_{\mathbb
    T^{n-1}}((\partial_\alpha w_1)^2+w_1^* \partial_\alpha^2 w_1^*)\, d
    \boldsymbol\xi +R,
\end{equation}
where
\begin{equation}\label{sveta26}
    R= \int_{\mathbb
    T^{n-1}}(\partial_\alpha m +w_1\partial_\alpha M)\, d
    \boldsymbol\xi.
\end{equation}
Since $(\partial_\alpha w_1)^2=1+2\partial_\alpha
w_1^*+(\partial_\alpha w_1^*)^2$, we have
\begin{equation}\label{sveta27}
\int_{\mathbb
    T^{n-1}}((\partial_\alpha w_1)^2+w_1^* \partial_\alpha^2 w_1^*)\, d
    \boldsymbol\xi=(2\pi)^{n-1} +\varsigma_2,
\end{equation}
where
\begin{equation}\label{sveta28}
\varsigma_2= \int_{\mathbb
   T^{n-1}}((\partial_\alpha w_1^*)^2+w_1^* \partial_\alpha^2 w_1^*)\, d
    \boldsymbol\xi.
\end{equation}
Estimate \eqref{sveta8} implies that $\varsigma_2$ satisfies
inequality \eqref{sveta14}. It remains to calculate $R$. Equality
\eqref{kira12} yields
\begin{equation*}
\partial_\alpha m +w_1\partial_\alpha M= \partial_\alpha m +\alpha \partial_\alpha M
+\partial_\alpha M w_1^*= p^{(\alpha)}+\partial_\alpha M w_1^*.
\end{equation*}
Thus we get
\begin{equation}\label{sveta29}
    R=\int_{\mathbb
   T^{n-1}}(p^{(\alpha)}+\partial_\alpha M w_1^*)\partial_\alpha w_1\, d
    \boldsymbol\xi.
\end{equation}
Next, relation \eqref{sveta19} with $\tau=\alpha$ gives the identity
\begin{equation}\label{sveta30}
 \partial_\alpha  w_1  =  \partial_\alpha \mathbf w\cdot
 \mathbf e_1
    \,=\, \mathbf W^\top \mathbf e_1\cdot  \boldsymbol \lambda^{(\alpha)}+\boldsymbol
    \chi^{(\alpha)}\cdot\nabla \mathbf w_1^*.
\end{equation}
Substituting this result into \eqref{sveta29} we arrive at the
identity
\begin{multline}\label{sveta32}
    R=\int_{\mathbb
   T^{n-1}}(p^{(\alpha)}+\partial_\alpha M w_1^*)\mathbf
   W^{\top}\mathbf e_1\cdot \boldsymbol\lambda^{(\alpha)}\,d
    \boldsymbol\xi\\+\int_{\mathbb
   T^{n-1}}(p^{(\alpha)}+\partial_\alpha M w_1^*)\nabla w_1^* \cdot \boldsymbol\chi^{(\alpha)}
 \,d\boldsymbol\xi  .
\end{multline}
Next, equations \eqref{kira14a} and \eqref{kira14b} in Lemma
\ref{kira13} imply the equalities
\begin{gather*}
(p^{(\alpha)}+ w_1^*) \nabla w_1^*=-   \boldsymbol\partial\, \boldsymbol\mu^{(\alpha)},\\
 (p^{(\alpha)}+ w_1^*) \mathbf W^\top \mathbf e_1 =-
 \mathbf J\boldsymbol\partial\,
\boldsymbol\lambda^{(\alpha)}\,-\, \boldsymbol\Omega\,
\boldsymbol\lambda^{(\alpha)}\,-\, \mathbf
T\boldsymbol\mu^{(\alpha)},\\
\end{gather*}
Substituting these equalities into \eqref{sveta29} we arrive at the
identity
\begin{equation}\label{sveta31}
R=-\int_{\mathbb
    T^{n-1}}\Big( (\mathbf J\boldsymbol\partial\,
\boldsymbol\lambda^{(\alpha)}\,+\, \boldsymbol\Omega\,
\boldsymbol\lambda^{(\alpha)}\,+\, \mathbf
T\boldsymbol\mu^{(2)})\cdot\boldsymbol\lambda^{(\alpha)}+\boldsymbol\partial\,
\boldsymbol\mu^{(\alpha)}\cdot \boldsymbol\chi^{(\alpha)}\Big)\, d
    \boldsymbol\xi.
\end{equation}
On the other hand, equation \eqref{kira14c} with $\tau=\alpha$ implies
\begin{gather*}
 -\int_{\mathbb
    T^{n-1}}\boldsymbol\partial\,
\boldsymbol\mu^{(\alpha)}\cdot \boldsymbol\chi^{(\alpha)}\, d
    \boldsymbol\xi=\int_{\mathbb
    T^{n-1}}\boldsymbol\partial\,
\boldsymbol\mu^{(\alpha)}\cdot \boldsymbol\chi^{(\alpha)}\, d
    \boldsymbol\xi\\=
\int_{\mathbb
    T^{n-1}}\big(\mathbf S\boldsymbol\mu^{(\alpha)}\, +\, \mathbf T^\top
\boldsymbol\lambda^{(\alpha)}\big)\cdot\boldsymbol\mu^{(2)}\, d
    \boldsymbol\xi
\end{gather*}
Substituting this equality into \eqref{sveta31} we  obtain
\begin{equation}\label{sveta33}
 R=\int_{\mathbb
    T^{n-1}}\Big(\mathbf
    S\boldsymbol\mu^{(\alpha)}\cdot\boldsymbol\mu^{(\alpha)}-(\mathbf J\boldsymbol\partial\,
\boldsymbol\lambda^{(\alpha)}\,+\, \boldsymbol\Omega\,
\boldsymbol\lambda^{(\alpha)})\cdot\boldsymbol\lambda^{(\alpha)}\Big)\,
d \boldsymbol\xi= L_{22}.
\end{equation}
Combining this relation with the equality
$$
(\boldsymbol\mu^{(\alpha)},\boldsymbol\lambda^{(\alpha)})=
(\boldsymbol\mu^{(1)},\boldsymbol\lambda^{(1)}) +\partial_\alpha
M\,(\boldsymbol\mu^{(2)},\boldsymbol\lambda^{(2)})
$$
and recalling formula \eqref{zara1} for $L_{ij}$ we arrive at the
expression for $R$,
\begin{equation*}
    R= L_{11}+\partial_\alpha M 2L_{12}+ (\partial_\alpha M)^2
    L_{22},
\end{equation*}
Substituting this expression and equality \eqref{sveta27} into
\eqref{sveta25} we obtain desired identity \eqref{sveta14}. This
completes the proof of Theorem \ref{sveta10}.
\end{proof}

The following proposition is a direct consequence of this theorem.
\begin{proposition}\label{sveta100}
Let all assumptions of Theorem \ref{sveta10} be satisfied and $k=0$. Then the third
derivative of $\Psi$ admits the estimate
\begin{equation}\label{sveta101}
    |\partial_\alpha^3\Psi(\alpha,0)-(2\pi)^{n-1}\partial_\alpha M|\leq c
     |\varepsilon|(|M|+|\boldsymbol\mu_0|+|\partial_\alpha M|).
\end{equation}
\end{proposition}
\begin{proof} It follows from
representation \eqref{sveta13} in Theorem \ref{sveta10} that
\begin{equation}\label{sveta102}
    \partial^3_\alpha\Psi=(2\pi)^{n-1}  \partial_\alpha M+\varsigma_2 \partial_\alpha M
    +\partial_\alpha \varsigma_2 M
 +\mathbf I_1+\mathbf I_2,
\end{equation}
where
\begin{gather}\label{sveta103}
    \mathbf I_1= \partial_\alpha L_{11}+2\partial_\alpha M\partial_\alpha L_{12}+
(\partial_\alpha M)^2\partial_\alpha L_{22}, \\\label{sveta104}
 \mathbf I_2= 2\partial_\alpha^2 M L_{12}+
2(\partial_\alpha M)\partial_\alpha^2 M L_{22}
\end{gather}
The rest of the proof is based on the following
\begin{lemma}\label{sveta105}
Under the assumptions of Theorem \ref{sveta100},
\begin{equation}\label{sveta106}
    |\partial_\alpha L_{11}|\leq c|\varepsilon||\boldsymbol\mu_0|, \quad
   |\partial_\alpha L_{12}|+|\partial_\alpha L_{22}|\leq c|\varepsilon| .
\end{equation}

\end{lemma}
\begin{proof} First we estimates the vector fields $\boldsymbol\mu^{(i)}$ and
$\boldsymbol\lambda^{(i)}$, $i=1,2$. We begin with the observation that relations
in view of the identities \eqref{kira27},
\begin{equation}\label{sveta107}\begin{split}
(\boldsymbol\mu^{(2)},\boldsymbol\lambda^{(2)})=\partial_k M^{-1}\,\,
( \boldsymbol\mu^{(k)},\boldsymbol\lambda^{(k)}), \\
(\boldsymbol\mu^{(1)},\boldsymbol\lambda^{(1)})=
  ( \boldsymbol\mu^{(\alpha)},\boldsymbol\lambda^{(\alpha)})-
  \partial_\alpha M\,\,\,(\boldsymbol\mu^{(2)},\boldsymbol\lambda^{(2)}).
\end{split}\end{equation}
In its turn, equalities \eqref{kira11} establish the following relations between the vector fields
$\boldsymbol\mu^{(\tau)}$,
$\boldsymbol\lambda^{(\tau)}$ and the derivatives of the vector
fields $\mathbf u$, $\mathbf v$, and $\mathbf w$, i.e.,
\begin{gather*}
\boldsymbol\chi^{(\tau)}\,= \,\mathbf V^\top \,\partial_\tau\mathbf
u\\\boldsymbol\lambda^{(\tau)}\,=\, \mathbf
(W^{-1}\partial_\tau\mathbf w-\chi_i^{(\tau)}\,\mathbf W^{-1}
\frac{\partial}{\partial \xi_i} \mathbf w\\
\boldsymbol\mu^{(\tau)}\,=\, \mathbf V^{-1}\Big(\partial_\tau\mathbf v+\chi_i^{(\tau)}\,
\frac{\partial}{\partial \xi_i} \mathbf v-
\boldsymbol\Lambda\boldsymbol\lambda^{(\tau)}\Big).
\end{gather*}
These identities along with estimates
\eqref{sveta8} and \eqref{sveta9} imply
\begin{equation}\label{sveta109}
    \|\partial_\alpha  \boldsymbol\mu^{(\tau)}\|_{\sigma/4,0}+
     \|\partial_\alpha  \boldsymbol\lambda^{(\tau)}\|_{\sigma/4,0}\leq c|\varepsilon|, \quad\tau=\alpha, k.
\end{equation}
On the other hand, inequalities \eqref{sveta9} yield
\begin{equation}\label{sveta110}
    |\partial_k M^{-1}|\leq (1-c|\varepsilon|)^{-1}\leq c,\quad  |\partial _k\partial_\alpha M|\leq
    c|\varepsilon|.
\end{equation}
Inequalities \eqref{sveta109}-\eqref{sveta110} and
 the identity
$$
\partial_\alpha(\boldsymbol\mu^{(2)},\boldsymbol\lambda^{(2)})=-(\partial_k M)^{-2}
\partial_\alpha\partial_k M\,\,
( \boldsymbol\mu^{(k)},\boldsymbol\lambda^{(k)})+
\partial_k M^{-1}
\partial_\alpha
( \boldsymbol\mu^{(k)},\boldsymbol\lambda^{(k)})
$$
lead to  the estimate
\begin{equation}\label{sveta111}
    |\partial_\alpha(\boldsymbol\mu^{(2)},\boldsymbol\lambda^{(2)})|\leq c|\varepsilon|.
\end{equation}
Next, it follows from \eqref{sveta107} that
\begin{multline*}
 \partial_\alpha (\boldsymbol\mu^{(1)},\boldsymbol\lambda^{(1)})=
  \partial_\alpha( \boldsymbol\mu^{(\alpha)},\boldsymbol\lambda^{(\alpha)})
  -\partial_\alpha^2 M\,
(\boldsymbol\mu^{(2)},\boldsymbol\lambda^{(2)})\\-
  \partial_\alpha M\partial_\alpha(\boldsymbol\mu^{(2)},\boldsymbol\lambda^{(2)}).
\end{multline*}
This relation along with inequalities \eqref{sveta9} and  \eqref{sveta111}
yields the estimate
\begin{equation}\label{sveta112}
      |\partial_\alpha(\boldsymbol\mu^{(1)},\boldsymbol\lambda^{(1)})|\leq c|\varepsilon|.
\end{equation}
The differentiation both sides of equality \eqref{zara1} with respect to $\alpha$
leads to the representation
\begin{multline}\label{sveta113}
   \partial_\alpha L_{ij}=\sigma_{ij}+
   \int_{\mathbb T^{n-1}}\big\{\,\,\mathbf S\mu^{(i)}\cdot\partial_\alpha\boldsymbol\mu^{(j)}-
    (\mathbf J\boldsymbol\partial
    \boldsymbol\lambda^{(i)}+\boldsymbol\Omega\boldsymbol\lambda^{(i)})\cdot
    \partial_\alpha\boldsymbol\lambda^{(j)}\,\,\big\}\, d \boldsymbol\xi\\+.
\int_{\mathbb T^{n-1}}\big\{\,\,\mathbf S\partial_\alpha\boldsymbol\mu^{(i)}\cdot\boldsymbol\mu^{(j)}-
    (\mathbf J\boldsymbol\partial
    \boldsymbol\lambda^{(j)}+\boldsymbol\Omega\boldsymbol\lambda^{(j)})\cdot
    \partial_\alpha\boldsymbol\lambda^{(i)}\,\,\big\}\, d \boldsymbol\xi,
\end{multline}
where
$$
\sigma_{ij}=\int_{\mathbb T^{n-1}}\partial_\alpha \mathbf S\boldsymbol\mu^{(i)}\cdot \boldsymbol\mu^{(i)}
\, d \boldsymbol\xi.
$$
Estimate \eqref{figa9} in Theorem \ref{figa8} and estimate \eqref{figa21}
in Theorem\ref{figa18} yield the inequality.
that
\begin{equation}\label{sveta114}
    |\sigma_{12}|+|\sigma_{22}|\leq c |\boldsymbol\mu^{(1)}|_0|\boldsymbol\mu^{(2)}|_0
   + |\boldsymbol\mu^{(2)}|_0^2\leq c|\varepsilon^2, \quad |\sigma_{11}|\leq c|\boldsymbol\mu^{(1)}|_0^2
   \leq c|\varepsilon| |\boldsymbol\mu_0|.
\end{equation}
On the other hand, we have
\begin{multline*}
\Big|\int_{\mathbb T^{n-1}}\big\{\,\,\mathbf S\mu^{(i)}\cdot\partial_\alpha\boldsymbol\mu^{(j)}-
    (\mathbf J\boldsymbol\partial
    \boldsymbol\lambda^{(i)}+\boldsymbol\Omega\boldsymbol\lambda^{(i)})\cdot
    \partial_\alpha\boldsymbol\lambda^{(j)}\,\,\big\}\, d \boldsymbol\xi\Big|\\\leq
c( |\boldsymbol\mu^{(i)}|_0|\partial_\alpha\boldsymbol\mu^{(j)}|_0
+|\boldsymbol\lambda^{(i)}|_1|\partial_\alpha\boldsymbol\lambda^{(j)}|_0).
\end{multline*}
Combining this inequality with estimates  \eqref{sveta111}-\eqref{sveta112} we arrive at
the estimate
\begin{multline*}
\Big|\int_{\mathbb T^{n-1}}\big\{\,\,\mathbf S\mu^{(i)}\cdot\partial_\alpha\boldsymbol\mu^{(j)}-
    (\mathbf J\boldsymbol\partial
    \boldsymbol\lambda^{(i)}+\boldsymbol\Omega\boldsymbol\lambda^{(i)})\cdot
    \partial_\alpha\boldsymbol\lambda^{(j)}\,\,\big\}\, d \boldsymbol\xi\Big|\\\leq
c|\varepsilon|( |\boldsymbol\mu^{(i)}|_0
+|\boldsymbol\lambda^{(i)}|_1)\leq c|\varepsilon|,
\end{multline*}
which along with \eqref{sveta114} yields the estimate \eqref{sveta106} for $\partial_\alpha L_{12}$
and $\partial_\alpha L_{22}$.

It remains  to estimate $\partial_\alpha L_{11}$.
Since  $k=0$, we have
$$
\mathbf J\boldsymbol\partial
    \boldsymbol\lambda^{(1)}+\boldsymbol\Omega\boldsymbol\lambda^{(1)}=
  \mathbf J\boldsymbol\partial
    \boldsymbol\lambda^{(1)}+ \boldsymbol\lambda^{(1)}\cdot \mathbf e_2.
 $$
Substituting  decomposition \eqref{figa20} into this relation and using estimate \eqref{figa21}
we obtain
\begin{equation*}
    |\mathbf J\boldsymbol\partial
    \boldsymbol\lambda^{(1)}+\boldsymbol\Omega\boldsymbol\lambda^{(1)}|=
    |\mathbf J\boldsymbol\partial
    \boldsymbol\lambda_\varepsilon+(\boldsymbol \mu_0\cdot \mathbf t_0)+
    \boldsymbol\lambda_\varepsilon \cdot \mathbf e_2|\leq
c|\boldsymbol\mu_0|.
\end{equation*}
Combining this result with the estimate $|\boldsymbol\mu^{(1)}|\leq c |\boldsymbol\mu_0|$ and estimates
\eqref{sveta111}-\eqref{sveta112} we arrive at the inequality
\begin{equation*}
\Big|\int_{\mathbb T^{n-1}}\big\{\,\,\mathbf S\mu^{(1)}\cdot\partial_\alpha\boldsymbol\mu^{(1)}-
    (\mathbf J\boldsymbol\partial
    \boldsymbol\lambda^{(1)}+\boldsymbol\Omega\boldsymbol\lambda^{(1)})\cdot
    \partial_\alpha\boldsymbol\lambda^{(1)}\,\,\big\}\, d \boldsymbol\xi\Big|\leq
c|\varepsilon| |\boldsymbol\mu_0|.
\end{equation*}
Combing this inequality with estimate \eqref{sveta114} for $\sigma_{11}$ and identity \eqref{sveta113}
we finally obtain that $|\partial_\alpha L_{11}|\leq c |\varepsilon\ |\boldsymbol\mu_0|$.
This completes the proof of the lemma.
\end{proof}
Let us turn to the proof of Proposition \ref{sveta105}.  Applying Lemma \ref{sveta106} we
obtain that the quantity $\mathbf I_1$ given by \eqref{sveta103} satisfies the inequality
\begin{equation}\label{sveta117}
    |\mathbf I_1|\leq c |\partial_\alpha L_{11}| +c |\partial_\alpha M|(|\partial_\alpha L_{12}|+
    |\partial_\alpha L_{22}|)\leq c |\varepsilon||\boldsymbol \mu_0|+c |\varepsilon| |\partial_\alpha M|.
\end{equation}
On the other hand, estimates \eqref{zara3} and \eqref{zara4} in Theorem \ref{zara5} imply the
following estimate for the quantity $\boldsymbol I_2$ given by  \eqref{sveta104}
\begin{equation}\label{sveta118}
     |\mathbf I_2|\leq c |L_{12}|+c |\partial_\alpha M||L_{22}|\leq  c
     |\varepsilon||\boldsymbol \mu_0|+c |\varepsilon| |\partial_\alpha M|.
\end{equation}
Substituting  \eqref{sveta117} and \eqref{sveta118} into identity \eqref{sveta102} we arrive at
the estimate
\begin{multline*}
    |\partial_\alpha^3\Psi(\alpha,0)-(2\pi)^{n-1}\partial_\alpha M|\\\leq
    c |\varepsilon ||\varsigma_2||\partial_\alpha M|+c |\varepsilon ||\partial_\alpha\varsigma_2|
    | M|
 +c   |\varepsilon|(|\partial_\alpha M|+|\boldsymbol\mu_0|)
\end{multline*}
Recalling estimate \eqref{sveta14} in Theorem \ref{sveta10} for  the quantity
$\varsigma_2$ we finally obtain the desired inequality
\begin{equation*}
    |\partial_\alpha^3\Psi(\alpha,0)-(2\pi)^{n-1}\partial_\alpha M|\leq
   c   |\varepsilon|(|M|+|\partial_\alpha M|+|\boldsymbol\mu_0|),
\end{equation*}
and the proposition follows.

\end{proof}

\subsection{Proof of Theorem \ref{theo1.4}}\label{mara}
We are now in a position to prove  Theorem \ref{theo1.4} which is the main result of this work.
We begin with the observation that by virtue of Theorem \ref{ira18} the main operator equation
\eqref{luna8x} has an analytic periodic solution $(\boldsymbol\varphi, e, m, M)$ for all
$(\alpha, k)\in \mathbb T^1\times [0,1]$ and all $|\varepsilon|\leq \varepsilon _0$. This solution
define the vector $\boldsymbol \Theta(\varphi)=(\mathbf u, \mathbf v, \mathbf w, \mathbf V,
\mathbf W, \boldsymbol\Lambda, \mathbf R)$ such that the corresponding canonical
mapping $\boldsymbol\vartheta$
put the modified Hamiltonian $H_m=H+m z_1+2^{-1} Mz_1^2$ into the normal form \eqref{normalform0}.
In particular, the modified Hamiltonian has the weakly hyperbolic invariant torus for every
$(\alpha, k)\in \mathbb T^1\times [0,1]$. Recall that $m$ and $M$ are  functions of the variables
$\alpha$ and $k$. Hence Theorem \ref{theo1.4} will be proved if we prove that the bifurcation equations
\begin{equation}\label{mara1}
    m(\alpha, k)= M(\alpha,k)=0
\end{equation}
have a solution $(\alpha_0, k_0)\in \mathbb T^1\times [0,1]$. We claim that such a solution can be
defined as a minimizer of the action functional $\Psi(\alpha,k)$ given by \eqref{sveta6}, i.e.,
\begin{equation}\label{mara2}
    \Psi(\alpha_0, k_0)=\min\limits_{(\alpha, k)\in \mathbb T^1\times [0,1]} \Psi(\alpha,k).
\end{equation}
The proof of this fact falls into the sequence of  lemmas.
The first lemma shows that every minimizer $(\alpha_0, k_0)$, $k_0>0$,  of the action functional
serves as a solution to the bifurcation equations.
\begin{lemma}\label{mara3} Let all assumptions of Theorem \eqref{sveta10} be satisfied. Furthermore, assume that
\begin{equation}\label{mara4}
  \Psi(\alpha_0, k_0)=\min\limits_{(\alpha, k)\in \mathbb T^1\times [0,1]} \Psi(\alpha,k)\text{~~and~~}
  \Psi(\alpha_0, 0)> \Psi(\alpha_0, k_0).
\end{equation}
Then $m(\alpha_0, k_0)=M(\alpha_0, k_0)=0.$
\end{lemma}
\begin{proof} Let us prove that $L_{22}(\alpha_0,k)>0$ everywhere on $(0,1)$ but countable discrete set.
 Notice that the mapping
$[0,1]\ni k \to w_1^*(\alpha_0, k)\in \mathcal A_{\sigma/2,0}$ is continuously differentiable on $[0,1]$
an is analytic
on $(0,1)$. If $w_1^*$ is not identically equal to zero on $(0,1)$, then there is a countable discrete
set $\{k_i\}\subset (0,1)$ such that $|w_1^*|_{-1}>0$ outside of this set. It follows
from this and estimate \eqref{zara3} in Theorem \ref{zara2} that $L_{22}(\alpha_0,k)>0$ for $k\neq k_i$.

Let us show that the case when $w_1^*(\alpha_0, \cdot)$ identically equals zero on (0,1)
is impossible. In view of \eqref{zara3}, the identity  $w_1^*(\alpha_0,\cdot)\equiv 0$ on $(0,1)$ implies
the identity  $L_{22}(\alpha_0, \cdot)\equiv 0$ on $(0,1)$, and hence implies the identity $\partial_k\Psi(\alpha_0,\cdot)
=M \partial_k M L_{22}(\alpha_0,\cdot)\equiv 0$ on $(0,1)$. The latter is impossible since
$\Psi(\alpha_0,0)> \Psi(\alpha_0, k_0)$.
Therefore, $L_{22}(\alpha_0,k)>0$ for all $k\neq k_i$.

Let us prove that $k_0<1$. It follows from inequalities \eqref{sveta9} and \eqref{sveta10} that
$$
|\partial_k M+1|+|M+k| \leq c|\varepsilon|.
$$
Hence for all sufficiently small $\varepsilon$ we have
$$
\partial_k M\leq -1/3, \quad M<-1/3 \text{~~for~~} k\geq 1/2.
$$
Since $L_{22}(\alpha_0,k)$ is strictly positive a.e. in $(0,1)$, it follows that $\Psi(\alpha_0,k)$
strictly increases on $(1/2,1]$. Hence $k_0<1$.  From this and the conditions of the lemma we conclude
that $0<k_0<1$.

Let us prove that $M(\alpha_0, k_0)=0$. Notice that for all  sufficiently small $|\varepsilon|$, we have
$\partial_k M\geq -1+c|\varepsilon| <0$. Hence $\partial_k M L_{22}(\alpha_0,k)<0$ almost everywhere on $(0,1)$.
  Next, identity \eqref{sveta12} in Theorem \ref{sveta10} implies
  $\partial_k\Psi(\alpha_0,k)= M\partial_k M L_{22}(\alpha_0,k)$. If
  $M(\alpha_0, k_0)\neq 0$, then the function
 $\Psi(\alpha_0,k)$ is strictly monotone in a neighborhood of $k_0$. This contradicts to the assumption
 that $(\alpha_0,k_0)$ is a minimizer of $\Psi$. Hence $M(\alpha_0,k_0)$ equals zero.

 It remains to prove
 that $m(\alpha_0, k_0)=0$. To this end, notice that identity \eqref{sveta11} in Theorem \ref{sveta10} along with the equality
 $M(\alpha_0, k_0)=0$ yields
$$
 0=   \partial_\alpha\Psi(\alpha_0,k_0)=(2\pi)^{n-1} (m+\alpha M)+\varsigma_1 M=
 (2\pi)^{n-1}m(\alpha_0, k_0),
$$
which completes the proof.
\end{proof}

\begin{lemma}\label{mara5} Let all assumptions of Theorem \eqref{sveta10} be satisfied. Furthermore, assume that
\begin{equation}\label{mara6}
  \Psi(\alpha_0, 0)=\min\limits_{(\alpha, k)\in \mathbb T^1\times [0,1]} \Psi(\alpha,k)\text{~~and~~}
  \mathbf K_0 \boldsymbol\mu_0\cdot \boldsymbol\mu_0<0.
\end{equation}
Then $m(\alpha_0, 0)=M(\alpha_0, 0)=0.$
\end{lemma}
\begin{proof} We split the proof into two steps.

\noindent
{\bf Step 1.} First we prove the lemma under the additions assumptions that
\begin{equation}\label{mara7}
   \| w_1^*(\alpha_0, k)\|_{\sigma/2, 0}>0 \text{~~for a.e. ~~} k\in (0,1).
\end{equation}
Let us prove that $M(\alpha_0,0)=0$. Suppose, contrary to our claim, that $M(\alpha_0,0)\neq 0$.
Notice that $\partial_k M\leq -1+c|\varepsilon| <0$. On the other hand, estimate \eqref{zara3}
in Theorem \ref{zara2} yields
$L_{22}\geq c^{-1} |w_1^*|_{-1}$.
Hence $\partial_k M L_{22}(\alpha_0,k)<0$ almost everywhere on $(0,1)$. If $M(\alpha_0,0)>0$, then $\Psi(\alpha_0,k)$
strictly decreases in a neighborhood of $k=0$. This contradicts to the assumptions that
$(\alpha_0,0)$ is a minimizer of $\Psi$. Hence $M(\alpha_0,0)\leq 0$. Next, representation \eqref{sveta13}
in Theorem \ref{sveta10} yields the inequality
\begin{multline}\label{mara8x}
0\leq
\partial_\alpha^2\Psi(\alpha_0,0)=(2\pi)^{n-1}  M(\alpha_0,0)+
\varsigma_2 M(\alpha_0,0)\\ +L_{11}(\alpha_0,0)+2L_{12} \partial_\alpha M(\alpha_0,0)+
L_{22} \partial_\alpha M(\alpha_0,0)^2
\end{multline}
In view of Theorem \ref{sveta10}, we have $|\varsigma_2|\leq c|\varepsilon|$. Hence the inequality
\begin{equation}\label{mara8}
    |M(\alpha_0,0)|=-M(\alpha_0,0)\leq L_{11}(\alpha_0,0)+2L_{12} \partial_\alpha M(\alpha_0,0)+
L_{22} \partial_\alpha M(\alpha_0,0)^2
\end{equation}
holds true for all sufficiently small $\varepsilon$. Next, estimate \eqref{zara4} in Theorem
\eqref{zara2} implies the inequality
\begin{multline*}\label{mara9}
    L_{11}+ 2|\partial_\alpha M ||L_{12}|\leq (2\pi)^{n-1}
    \mathbf K_0\boldsymbol\mu_0\cdot\boldsymbol\mu_0 +c|\varepsilon| |\boldsymbol\mu_0|^2+
c|\varepsilon| |\boldsymbol\mu_0||\partial_\alpha M|\\\leq
(2\pi)^{n-1}
    \mathbf K_0\boldsymbol\mu_0\cdot\boldsymbol\mu_0 +c|\varepsilon| |\boldsymbol\mu_0|^2+
c|\varepsilon||\partial_\alpha M|^2.
\end{multline*}
Recall that
$$
-c^{-1}|\boldsymbol\mu_0|^2 \leq |K_0\boldsymbol\mu_0\cdot\boldsymbol\mu_0|\leq -c|\boldsymbol\mu_0|^2
$$
for some positive constant $c$ independent of $\varepsilon$. From this  we conclude that
\begin{equation}\label{mara10}
   L_{11}+ 2|\partial_\alpha M ||L_{12}|\leq
    \mathbf K_0\boldsymbol\mu_0\cdot\boldsymbol\mu_0 +
c|\varepsilon||\partial_\alpha M|^2
\end{equation}
for all sufficiently small $\varepsilon$. Substituting this inequality into \eqref{mara8}
and noting that $L_{22}\leq c\varepsilon^2$ we conclude that
\begin{equation*}
    |M(\alpha_0,0)|-\mathbf K_0\boldsymbol\mu_0\cdot\boldsymbol\mu_0\leq c |\varepsilon|
    |\partial_\alpha M|^2,
   \end{equation*}
which gives
\begin{equation}\label{mara11}
    |M(\alpha_0,0)|+|\boldsymbol\mu_0|^2\leq c |\varepsilon|
    |\partial_\alpha M|^2,
   \end{equation}
It follows from representation \eqref{mara8x} and estimates  \eqref{zara3}, \eqref{zara4} in
Theorem \ref{zara2}
that
\begin{multline*}
0\leq \partial_\alpha^2\Psi(\alpha_0,0)\leq c| M(\alpha_0,0)|+
|L_{11}(\alpha_0,0)|+\\2|L_{12}| |\partial_\alpha M(\alpha_0,0)|+
|L_{22} |\partial_\alpha M(\alpha_0,0)^2\leq c| M(\alpha_0,0)|+c |\boldsymbol\mu_0|^2\\+
c |\varepsilon||\boldsymbol\mu_0||\partial_\alpha M(\alpha_0,0)|+c
|\varepsilon||\partial_\alpha M(\alpha_0,0)|^2
\end{multline*}
Substituting estimate \eqref{mara11} into the right hand side of this inequality we obtain
\begin{equation}\label{mara12}
0\leq \partial_\alpha^2\Psi(\alpha_0,0)\,\leq\,
c|\varepsilon|\,|\partial_\alpha M(\alpha_0,0)|^2.
\end{equation}
Let us estimate the third derivatives of $\Psi$. Combining inequality \eqref{sveta101} in Proposition
\ref{sveta100} and inequality \eqref{mara11} we obtain
\begin{equation*}
    |\partial_\alpha^3\Psi(\alpha_0,0)-(2\pi)^{n-1}\partial_\alpha M(\alpha_0,0)|\leq c
     |\varepsilon||\partial_\alpha M(\alpha_0,0)|.
\end{equation*}
It follows that the inequality
\begin{equation}\label{mara13}
|\partial_\alpha^3\Psi(\alpha_0,0)|\geq |\partial_\alpha M(\alpha_0,0)|
\end{equation}
holds for all small $\varepsilon$. Since $\Psi_0(\alpha_0, 0)$ is analytic and takes the
 minimum at point $\alpha_0$
we have the Taylor expansion
\begin{equation}\label{mara14}
    0\leq \Psi(\alpha_0+t, 0)-\Psi(\alpha_0,0)= A t^2+Bt^3+ Ct^4,
\end{equation}
where
$$ A=2^{-1}\partial_\alpha^2\Psi(\alpha_0,0), \quad B =6^{-1}\partial_\alpha^2\Psi(\alpha_0,0), \quad,
|C|\leq c_4,
$$
where $c_4$ is independent of $\varepsilon$. Notice that this relation holds true for all
$t\in \mathbb R^1$.  Estimates \eqref{mara12} and \eqref{mara13} imply the inequalities
$$
|A|\leq c_2|\varepsilon| |\partial_\alpha M|^2, \quad |B|\geq 6^{-1}|\partial_\alpha M|.
$$
Now set $t=- \delta B$, where $\delta$ are an arbitrary positive number.
We have
\begin{multline}\label{mara15}
 0\leq \Psi(\alpha_0+t, 0)-\Psi(\alpha_0,0)= A t^2+Bt^3+ Ct^4= \\B^4\delta^3( A B^{-2}\delta^{-1}
 -1+\delta C)
\end{multline}
Notice that $|AB^{-2}|\leq c|\varepsilon|$.  Obviously,
 the
right hand side of \eqref{mara15} is negative for $\delta=|\varepsilon |^{1/2}$ and small $\varepsilon$ .
This contradiction proves the equality
$M(\alpha_0,0)=0$. It remains to note that the equality $m(\alpha_0,k)=0$
obviously follows from the relations
$$
 0=   \partial_\alpha\Psi(\alpha_0,0)=(2\pi)^{n-1} (m+\alpha M)+\varsigma_1 M=
 (2\pi)^{n-1}m(\alpha_0,0).
$$

\noindent
{\bf Step 2.} It remains to consider the degenerate case when $w_1^*(\alpha_0,k)=0$ for all
$k\in [0,1]$. In view of estimate \eqref{figa9} we have in this case
$$
\boldsymbol\lambda^{(2)}(\alpha_0,k)=0, \quad \boldsymbol\mu^{(2)}(\alpha_0,k)=0\text{~~for all~~}
k\in [0,1].
$$
Recalling representation \eqref{zara1} for $L_{ij}$ we conclude that
\begin{equation}\label{mara20}
    L_{12}(\alpha_0,k)=0, \quad  L_{22}(\alpha_0,k)=0 \text{~~for all~~}
k\in [0,1].
\end{equation}
It follows from this and formula \eqref{sveta12}  that
$\partial_k \Psi(\alpha_0,k)=0$  and hence
\begin{equation}\label{mara21}
  \Psi(\alpha_0, k)=\min\limits_{(\alpha, l)\in \mathbb T^1\times [0,1]} \Psi(\alpha,l)\text{~~for all~~}
k\in [0,1].
\end{equation}
In other words, $\Psi$,  the whole  segment $\{\alpha_0\}\times [0,1]$
consists of the minimizers of $\Psi$.

Let us prove that there is $k_0\in [0,1]$ such that
$M(\alpha_0, k_0)=0$.  We begin with the observation that $|M(\alpha_0,1)+1|\leq c|\varepsilon|$.
Hence, $M(\alpha_0,1)<0$  for all sufficiently small $\varepsilon$. Therefore,
it suffices to prove that $M(\alpha_0,0)\geq 0$. Suppose, contrary to our claim, that $M(\alpha_0,0)<0$
Since $(\alpha_0,0)$ is a minimizer of $\Psi$ and $L_{12}=L_{22}=0$ for $\alpha=\alpha_0$,
representation \eqref{sveta13} in Theorem \ref{sveta10} implies the inequality
\begin{equation*}
0\leq
\partial_\alpha^2\Psi(\alpha_0,0)=(2\pi)^{n-1}  M(\alpha_0,0)+
\varsigma_2 M(\alpha_0,0) +L_{11}(\alpha_0,0).
\end{equation*}
It follows from this, estimate \eqref{zara4} in Theorem \ref{zara2}, and
the inequality $|\varsigma_2|\leq c|\varepsilon|$
that
\begin{multline}\label{mara22}
0\leq
\partial_\alpha^2\Psi(\alpha_0,0)=(2\pi)^{n-1}  M(\alpha_0,0)+\\
\mathbf K_0\boldsymbol\mu_0\cdot \boldsymbol\mu_0+c|\varepsilon||M(\alpha_0,0)|+
c |\varepsilon| |\boldsymbol\mu_0|^2.
\end{multline}
Recall that $ M(\alpha_0,0)<0$ and $\mathbf K_0\boldsymbol\mu_0\cdot \boldsymbol\mu_0\leq -c
|\boldsymbol\mu_0|^2$. It follows from this that the right hand side of \eqref{mara22} is negative
for all sufficiently small $\varepsilon$. The contradiction prove the inequality $M(\alpha_0,0)\geq 0$.
Since $M(\alpha_0,1)$ is negative, there is $k_0$ such that $M(\alpha_0,k_0)=0$. On the other hand,
$\Psi$ takes the minimum at the point $(\alpha_0,k_0)$. It follows from this that $\partial_\alpha
\Psi(\alpha_0,k_0)= (2\pi)^{n-1} m(\alpha_0,k_0)=0$.  This completes the proof of the lemma.
\end{proof}

It remains to note that the statement of Theorem \ref{theo1.4}
is a straightforward consequence of Lemmas \ref{mara3} and \ref{mara5}.


\begin{appendix}

\setcounter{equation}{0}
\renewcommand{\theequation}{\arabic{equation} A}

\section{Proof of Theorem \ref{anna4}}\label{lena}
Let $\boldsymbol  \vartheta$ is defined by \eqref{anna1}.
  It follows that
\begin{equation}\label{lena2}
\boldsymbol \vartheta'=\left(
\begin{array}{ccc}
A_{11}&0&0\\
A_{21}&A_{22}&A_{23}
\\A_{31}&0& A_{33}
\end{array}
\right), \quad {\boldsymbol \vartheta'}^\top=\left(
\begin{array}{ccc}
A_{11}^\top&A_{21}^\top&A_{31}^\top\\
0&A_{22}^\top&0
\\0&A_{23}^\top& A_{33}^\top
\end{array}
\right),
\end{equation}
where
\begin{equation}\label{lena3}\begin{split}
  A_{11}=\mathbf I+\mathbf u' ,\quad A_{22}=\mathbf V, \quad A_{33}=\mathbf W,\\
  A_{21}=\mathbf v'_{\xi} +(\mathbf V\boldsymbol\eta)_\xi'+
  (\boldsymbol \Lambda \boldsymbol\zeta)_\xi'
  +\frac{1}{2}(\boldsymbol\zeta^\top \mathbf R\boldsymbol\zeta)_\xi',\\
  A_{31}=\mathbf w'+(\mathbf W\boldsymbol\zeta)_\xi',\quad A_{23}=
  \boldsymbol\Lambda+\mathbf R\boldsymbol\zeta.
  \end{split}\end{equation}
Here notation $\mathbf R\boldsymbol\zeta$ stands for $(n-1)\times 2$
matrix with the entries
\begin{equation}\label{lena4}
    \{\mathbf R\boldsymbol\zeta\}_{ip}=\{\mathbf R_i\}_{ pq}\zeta_q,
    \quad 1\leq i\leq n-1, \quad
    p=1,2.
\end{equation}
Substituting \eqref{lena2} into the equation
${\boldsymbol\vartheta'}^\top \mathbf
J_{2n}\boldsymbol\vartheta'=\mathbf J_{2n}$ we obtain nine matrix
equations. Four of those are nontrivial:
\begin{gather}\label{lena5}
A_{11}^\top A_{22}=\mathbf I, \quad A_{33}^\top\mathbf JA_{33}=\mathbf J,\\
\label{lena6}
A_{11}^\top A_{23}+A_{31}^\top \mathbf J A_{33}=0, \\
\label{lena7} A_{11}^\top A_{21}-A_{21}^\top
A_{11}+A_{31}^\top\mathbf J A_{31}=0.
\end{gather}

The remaining five equations are either trivial or can be obtained
from \eqref{lena5}-\eqref{lena7} by transposition. Equations
\eqref{lena5} along with \eqref{lena3} give
\begin{equation}\label{lena8}\begin{split}
    \mathbf V=(\mathbf I_{n-1}+\mathbf u' )^{-\top},\\
\mathbf W^\top\mathbf J\mathbf W=\mathbf J\text{~or
equivalently~~}\text{det~}\mathbf W=1.
\end{split}\end{equation}
In view of \eqref{lena3} we can rewrite equation \eqref{lena6} in
the form
$$
\boldsymbol\Lambda +\mathbf R\, \boldsymbol\zeta= -\boldsymbol
V\big(\, \mathbf w_{\xi}'+ (\mathbf
W\boldsymbol\zeta)_\xi'\,\big)^\top\, \mathbf J\mathbf W\text{~~for
all~~} \boldsymbol\zeta\in \mathbb R^2.
$$
It follows that
\begin{equation}\label{lena9}
 \boldsymbol\Lambda =-\boldsymbol V\big(\, \mathbf w_{\xi}'\,\big)^\top\, \mathbf J\mathbf W,
\end{equation}
and
$$
\mathbf R\, \boldsymbol\zeta=-\boldsymbol V\,{ (\mathbf
W\boldsymbol\zeta)_\xi'}^\top\, \mathbf J \mathbf W\text{~~for
all~~} \boldsymbol\zeta\in \mathbb R^2.
$$
Recalling the identities
$$
\{\mathbf R\, \boldsymbol\zeta\}_{ij}=R_{i,jq}\,\, \zeta_q, \quad
\{{(\mathbf
W\boldsymbol\zeta)_\xi'}^\top\}_{kp}=\frac{\partial}{\partial\xi_k}
(W_{pq}\,\zeta_q)
$$
we obtain
$$
R_{i,qj} \zeta_q=-V_{ik}\,\,\frac{\partial}{\partial\xi_k}
(W_{pq}\zeta_q) \{\mathbf J\mathbf W\}_{pj} \text{~~for all~~}
\boldsymbol\zeta\in \mathbb R^2.
$$
Since the matrix $\mathbf R_i$ is symmetric, we obtain
\begin{equation}\label{lena10}
    \mathbf R_i=-V_{ik}\,\frac{\partial}{\partial\xi_k} (\mathbf W^\top)\,\mathbf J\,\mathbf W.
\end{equation}
Let us turn to equation \eqref{lena7}. It is equivalent to the
system of four matrix equations
\begin{gather}\label{lena11}
\frac{1}{2}\mathbf V^{-1}(\boldsymbol\zeta ^\top\mathbf
R\boldsymbol\zeta)_\xi'- \frac{1}{2}{(\boldsymbol\zeta ^\top\mathbf
R\boldsymbol\zeta)_\xi'}^\top\mathbf V^{-\top}+ {\big(\mathbf
W\boldsymbol\zeta\big)_\xi'}^\top \mathbf J\big(\mathbf
W\boldsymbol\zeta\big)_\xi' =0,\\\label{lena12} \mathbf
V^{-1}(\boldsymbol\Lambda\boldsymbol\zeta)_\xi'-
{(\boldsymbol\Lambda\boldsymbol\zeta)_\xi'}^\top\mathbf V^{-\top} +
{\big(\mathbf w\big)_\xi'}^\top \mathbf J\big(\mathbf
W\boldsymbol\zeta\big)_\xi' +
{\big(\mathbf W\boldsymbol\zeta\big)_\xi'}^\top \mathbf J\big(\mathbf w\big)_\xi'=0\\
\label{lena13} \mathbf V^{-1}(\mathbf V\boldsymbol\eta)_\xi'-
({\mathbf V\boldsymbol\eta)_\xi'}^\top\mathbf V^{-\top}=0,\\
\label{lena14} \mathbf V^{-1}(\mathbf v)_\xi'-{(\mathbf v)'_\xi
}^\top\mathbf V^{-\top}+ {(\mathbf w)_\xi'}^\top\mathbf J(\mathbf
w)_\xi'=0,
\end{gather}
which hold true for all $\boldsymbol\zeta\in \mathbb R^2$ and for all
$\boldsymbol\eta\in \mathbb R^{n-1}$. Let us proof that equations
\eqref{lena11}- \eqref{lena13} are consequence of
\eqref{lena8}-\eqref{lena10}.
 We start with equation \eqref{lena11}.  It follows from \eqref{lena8} that
$$
\mathbf x=\boldsymbol\xi+\mathbf u(\boldsymbol\xi), \quad \mathbf V=
(\mathbf I+\mathbf u')^{-\top}\equiv
(\mathbf x'_\xi)^{-\top}=(\boldsymbol \xi_x')^\top,
$$
and hence
\begin{equation}\label{lena15}
    \frac{\partial}{\partial x_i}=V_{ik}\frac{\partial}{\partial \xi_k}, \quad
    \nabla_x=\mathbf V\nabla_\xi.
\end{equation}
Multiplying both sides of equation\eqref{lena11} by $\mathbf V^\top$
from the right and by $\mathbf V$ from the left we can rewrite this
equation in the equivalent form
\begin{equation}\label{lena16}
  \frac{1}{2}(\boldsymbol\zeta ^\top\mathbf R\boldsymbol\zeta)_x'-
\frac{1}{2}{(\boldsymbol\zeta ^\top\mathbf
R\boldsymbol\zeta)_x'}^\top+ {\big(\mathbf
W\boldsymbol\zeta\big)_x'}^\top \mathbf J \big(\mathbf
W\boldsymbol\zeta\big)_x' =0
\end{equation}
In view of \eqref{lena10}, we have
$$
\boldsymbol \zeta^\top\mathbf
R_i\boldsymbol\zeta=-\boldsymbol\zeta^\top
V_{ik}\frac{\partial}{\partial\xi_k}(\mathbf W^\top) \mathbf J
\mathbf W\boldsymbol\zeta=-\boldsymbol\zeta^\top
\frac{\partial}{\partial x_i}(\mathbf W^\top) \mathbf J \mathbf
W\boldsymbol\zeta,
$$
which implies
$$
\{(\boldsymbol \zeta^\top\mathbf R_i\boldsymbol\zeta)_x'\}_{ij}=
-\boldsymbol\zeta^{\top}\frac{\partial}{\partial x_j}
\Big(\frac{\partial\mathbf W^\top}{\partial x_i}\mathbf J\mathbf
W\Big)\boldsymbol\zeta.
$$
Thus we get
\begin{equation}\begin{split}\label{lena17}
\frac{1}{2}\Big\{(\boldsymbol\zeta ^\top\mathbf
R\boldsymbol\zeta)_x'- {(\boldsymbol\zeta ^\top\mathbf
R\boldsymbol\zeta)_x'}^\top\Big\}_{ij}=
-\frac{1}{2}\boldsymbol\zeta^{\top}\frac{\partial}{\partial x_j}
\Big(\frac{\partial\mathbf W^\top}{\partial x_i}\mathbf J\mathbf W\Big)\boldsymbol\zeta+\\
\frac{1}{2}\boldsymbol\zeta^{\top}\frac{\partial}{\partial x_i}
\Big(\frac{\partial\mathbf W^\top}{\partial x_j}\mathbf J\mathbf
W\Big)\boldsymbol\zeta
=\frac{1}{2}\boldsymbol\zeta^{\top}\Big(\frac{\partial\mathbf
W^\top}{\partial x_j}\mathbf J \frac{\partial\mathbf W}{\partial
x_i}-\frac{\partial\mathbf W^\top}{\partial x_i}\mathbf J
\frac{\partial\mathbf W}{\partial x_j}\Big)\boldsymbol\zeta.
\end{split}\end{equation}
On the other hand, we have
$$
\{{\big(\mathbf W\boldsymbol\zeta\big)_x'}^\top \mathbf J
\big(\mathbf W\boldsymbol\zeta\big)_x'\}_{ij}
=\boldsymbol\zeta^\top\Big(\frac{\partial\mathbf W^\top}{\partial
x_i}\mathbf J \frac{\partial\mathbf W}{\partial
x_j}\Big)\boldsymbol\zeta.
$$
Noting that  $\boldsymbol\zeta^{\top}\mathbf A\boldsymbol\zeta=
\boldsymbol\zeta^{\top}\mathbf A^\top\boldsymbol\zeta$ for every
matrix $\mathbf A$ and setting
$$
\mathbf A=\frac{\partial\mathbf W^\top}{\partial x_i}\mathbf J
\frac{\partial\mathbf W}{\partial x_j}, \quad \mathbf A^\top=
-\frac{\partial\mathbf W^\top}{\partial x_j}\mathbf J
\frac{\partial\mathbf W}{\partial x_i},
$$
we arrive at the identity
\begin{equation}\label{lena18}
\{{\big(\mathbf W\boldsymbol\zeta\big)_x'}^\top \mathbf J
\big(\mathbf W\boldsymbol\zeta\big)_x'\}_{ij}=
\frac{1}{2}\boldsymbol\zeta^\top\Big(\frac{\partial\mathbf
W^\top}{\partial x_i}\mathbf J \frac{\partial\mathbf W}{\partial
x_j}-\frac{\partial\mathbf W^\top}{\partial x_j}\mathbf J
\frac{\partial\mathbf W}{\partial x_i}\Big)\boldsymbol\zeta.
\end{equation}
Combining \eqref{lena17} and \eqref{lena18} gives \eqref{lena16}.
 Hence \eqref{lena8} and \eqref{lena10} imply \eqref{lena11}.
  Let us turn to equation \eqref{lena12}. Arguing as before we can
  rewrite it in the equivalent form
\begin{equation}\label{lena19}
(\boldsymbol\Lambda\boldsymbol\zeta)_x'-
{(\boldsymbol\Lambda\boldsymbol\zeta)_x'}^\top + {\big(\mathbf
w\big)_x'}^\top \mathbf J\big(\mathbf W\boldsymbol\zeta\big)_x' +
{\big(\mathbf W\boldsymbol\zeta\big)_x'}^\top \mathbf J\big(\mathbf
w\big)_x'=0
\end{equation}
Next, equality \eqref{lena9} yields
$$
\boldsymbol\Lambda=-\mathbf V\nabla_\xi \mathbf w\mathbf J\mathbf
W=- \nabla_x \mathbf w\mathbf J\mathbf W
$$
which leads to
$$
\{(\boldsymbol\Lambda\boldsymbol\zeta)_x'\}_{ij}=
-\frac{\partial}{\partial x_j}\Big(\frac{\partial \mathbf
w^\top}{\partial x_i}\, (\mathbf J\, \mathbf
W)\,\boldsymbol\zeta\,\Big),
$$
and hence
\begin{equation*}
   \{(\boldsymbol\Lambda\boldsymbol\zeta)_x'-
{(\boldsymbol\Lambda\boldsymbol\zeta)_x'}^\top    \}_{ij}=
\frac{\partial \mathbf w^{\top}}{\partial x_j}\, \mathbf
J\,\frac{\partial}{\partial x_i}\Big( \mathbf
W\,\boldsymbol\zeta)\,\Big)- \frac{\partial \mathbf w^\top}{\partial
x_i}\,\mathbf J\, \frac{\partial}{\partial x_j} \Big( \mathbf
W\,\boldsymbol\zeta)\,\Big).
\end{equation*}
This relations can be rewritten in the matrix form
\begin{equation}\label{lena20}
(\boldsymbol\Lambda\boldsymbol\zeta)_x'-
{(\boldsymbol\Lambda\boldsymbol\zeta)_x'}^\top=-(\mathbf w_x')^\top
\mathbf J(\mathbf W\boldsymbol\zeta)_x' +\big((\mathbf w_x')^\top
\mathbf J(\mathbf W\boldsymbol\zeta)_x'\big)^\top.
\end{equation}
Notice that
$$
\big((\mathbf w_x')^\top J(\mathbf W\boldsymbol\zeta)_x'\big)^\top
={(\mathbf W\boldsymbol\zeta)_x'}^\top \mathbf J^\top w_x'=-
{(\mathbf W\boldsymbol\zeta)_x'}^\top \mathbf J^\top w_x'
$$
Combining this result with \eqref{lena20} we arrive at equation
\eqref{lena19}. Since the latter is equivalent to \eqref{lena12}, we
obtain that \eqref{lena12} follows from \eqref{lena8}-\eqref{lena9}.
Now consider the equation \eqref{lena13}.  In view of identities
\eqref{lena15} it can be rewritten in the equivalent form
\begin{equation}\label{lena21}
    (\mathbf V\boldsymbol\eta)_\xi' \mathbf V^\top-
    \mathbf V{(\mathbf V\boldsymbol\eta)_\xi'}^\top
\equiv (\mathbf V\boldsymbol\eta)_x' -{(\mathbf
V\boldsymbol\eta)_x'}^\top=0.
\end{equation}
On the other hand, formula \eqref{lena8} yields
$$
\{\mathbf V \boldsymbol \eta\}_i= \frac{\partial\boldsymbol
\xi^\top}{\partial x_i}\cdot\boldsymbol \eta.
$$
Hence
$$
\{ (\mathbf V\boldsymbol\eta)_x' \}_{ij}=
\frac{\partial^2\boldsymbol \xi^\top}{\partial x_i\partial
x_j}\cdot\boldsymbol \eta.
$$
Hence the matrix $(\mathbf V\boldsymbol\eta)_x'$ is symmetric, and
\eqref{lena21} is a straightforward consequence of \eqref{lena8}.
Thus we show that equations \eqref{lena11}-\eqref{lena13} directly
follows from \eqref{lena8}-\eqref{lena10}.  Let us consider equation
\eqref{lena14}.
 We prove that it is not trivial and leads to the desired representation  $\mathbf v$. It  follows from \eqref{lena8} that equation
 \eqref{lena14} can be written in the form
\begin{equation*}
\frac{\partial}{\partial \xi_i}(\xi_k+u_k)\frac{\partial
v_k}{\partial \xi_j}- \frac{\partial}{\partial
\xi_j}(\xi_k+u_k)\frac{\partial v_k}{\partial \xi_i}+ \frac{\partial
w_1}{\partial \xi_i}\frac{\partial w_2}{\partial \xi_j}-
\frac{\partial w_1}{\partial \xi_j}\frac{\partial w_2}{\partial
\xi_i}=0, \quad 1\leq i,j\leq n-1,
\end{equation*}
which  is equivalent to
\begin{equation*}
    d(\xi_k+u_k)\wedge d v_k+d w_1\wedge w_2=0.
\end{equation*}
Next, multiplying \eqref{lena14} by $\mathbf V$ from the left and
by
 $\mathbf V^\top$ from the right  we obtain
$$
\mathbf v_\xi'\mathbf V^\top-\mathbf V{\mathbf v_\xi'}^\top+ \mathbf
V(\mathbf w_\xi')^\top\mathbf J \mathbf w_\xi'\, \mathbf
V^\top\equiv \mathbf v_x-{\mathbf v_x'}^\top+(\mathbf
w_x')^\top\mathbf J \mathbf w_x'=0.
$$
This means that
\begin{equation*}\begin{split}\frac{
    \partial v_i}{\partial x_j}- \frac{\partial v_j}{\partial x_i}
    +\frac{\partial w_1}{\partial x_i}\frac{\partial w_2}{\partial x_j}-
\frac{\partial w_1}{\partial x_j}\frac{\partial w_2}{\partial x_i}=\\
\frac{ \partial }{\partial x_j}\Big(v_i+w_2\frac{\partial
w_1}{\partial x_i}\Big) -\frac{ \partial }{\partial
x_i}\Big(v_j+w_2\frac{\partial w_1}{\partial x_j}\Big)=0
\end{split}\end{equation*}
Since $\mathbf v$ and $\mathbf w$ are analytic an $2\pi$-periodic,
 it follows that there exist analytic $2\pi$-periodic function $\varphi_0$ with
  zero mean value and a constant $\boldsymbol \beta\in \mathbb R^{n-1}$ such that
$$
\beta_i+\frac{
    \partial \varphi_0}{\partial x_i}=v_i+w_2\frac{
    \partial w_1}{\partial x_i},
$$
this leads to the desired representation for $\mathbf v$
\begin{equation}\label{lena22}
    \mathbf v=\beta+\nabla_x\varphi_0-w_2\nabla_x w_2=\beta +\mathbf V\, \big(\nabla_\xi\varphi_0-w_2\nabla_\xi w_2\big).
\end{equation}
Formulae \eqref{lena8}-\eqref{lena10} and \eqref{lena22} give
general solution to equation ${\boldsymbol\vartheta'}^\top \mathbf
J_{2n}\boldsymbol\vartheta'=\mathbf J_{2n}$ and completely determine
the totality of canonical mappings $\boldsymbol\vartheta$. $\square$

\setcounter{equation}{0}
\renewcommand{\theequation}{\arabic{equation}B}

\section{Proof of Theorem \ref{sima12}}\label{proofsima12}

\subsection{Proof of $(\mathbf i)$.} Choose an arbitrary analytic
\begin{equation}\label{rita1}
    \boldsymbol\varphi=(\beta,\varphi_0, \mathbf u, \mathbf w, W_{11}, W_{12},
     W_{21})
\end{equation}
and consider the vector field
\begin{equation}\label{rita2}
    \boldsymbol \Theta(\boldsymbol\varphi)=\big(\mathbf u, \mathbf v, \mathbf w,
    \mathbf V,\boldsymbol \Lambda, \mathbf W,
    \mathbf R_i\big),
\end{equation}
defined by \eqref{anna17}.   Next,
choose an arbitrary
\begin{equation}\label{rita3}
    \boldsymbol\Upsilon=(\boldsymbol\nu,  \psi_0,  \boldsymbol\chi,\boldsymbol\lambda,
    \Gamma_{11},
    \Gamma_{12}, \Gamma_{21})\in X_{\sigma, d-1}(r),
\end{equation}
and set  $\boldsymbol\mu=\boldsymbol\nu+\nabla\psi_0$, $
\Gamma_{22}=-\Gamma_{11}$.
Our task is to find the vector field
\begin{equation}\label{rita5}
\delta\boldsymbol\varphi=(\delta\mathbf \beta, \delta\mathbf
\varphi_0,
 \delta\mathbf u, \delta\mathbf
w, \delta W_{11},\delta W_{12},\delta W_{21})
\end{equation}
such that the corresponding mapping
\begin{equation}\label{rita5x}
    \delta\boldsymbol\Theta=\big(\, \delta\mathbf u, \delta\mathbf v,
    \delta\mathbf w, \delta\mathbf V, \delta\mathbf W, \delta\boldsymbol\Lambda,
    \delta\mathbf R_i\,\big).
\end{equation}
given by \eqref{gala3} and \eqref{gala2}, satisfies the equations
\begin{subequations}\label{rita8}
\begin{gather}\label{rita8a}
\delta \mathbf
u(\boldsymbol\xi)=\chi_i(\boldsymbol\xi)\frac{\partial}{\partial\xi_i}\,
\big(\boldsymbol \xi+\mathbf u(\boldsymbol\xi)\big)\Leftrightarrow
\boldsymbol\chi=
\mathbf V^\top\delta\mathbf u,\\
\label{rita8b} \delta \mathbf
v=\chi_i\,\frac{\partial}{\partial\xi_i}\mathbf v+\mathbf
V\boldsymbol\mu+
\boldsymbol\Lambda\boldsymbol\lambda,\\
\label{rita8c}
 \delta \mathbf w=\chi_i\,\frac{\partial}{\partial\xi_i}\mathbf w+
 \mathbf W\boldsymbol\lambda,\\
\label{rita8d}\delta \mathbf
W=\chi_i\,\frac{\partial}{\partial\xi_i}\mathbf W+ \mathbf
W\boldsymbol\Gamma,
\end{gather}
\end{subequations}
and
\begin{subequations}\label{rita9}
\begin{gather}\label{rita9a}
\delta \mathbf V=\chi_i\,\frac{\partial}{\partial\xi_i}\mathbf V-
\mathbf V\nabla_\xi\boldsymbol\chi,\\
\label{rita9b} \delta
\boldsymbol\Lambda=\chi_i\,\frac{\partial}{\partial\xi_i}\boldsymbol\Lambda
+\mathbf V\nabla_\xi(\mathbf
J\boldsymbol\lambda)+\boldsymbol\lambda^\top\mathbf R-
\boldsymbol\Lambda\boldsymbol\Gamma,
\\
\label{rita9c} \delta\mathbf
R_i=\chi_p\frac{\partial}{\partial\xi_p}\mathbf R_i+\mathbf R_i
\boldsymbol\Gamma+(\mathbf R_i \boldsymbol\Gamma)^\top+
V_{ik}\frac{\partial}{\partial\xi_k}(\mathbf J\boldsymbol\Gamma),
\end{gather}
where $\boldsymbol\lambda^\top\mathbf R$  is
$(n-1)\times 2$ matrix with the entries
$(\boldsymbol\lambda^\top\mathbf R)_{ij}=R_{i, jp}\,\lambda_p$.
Recall that by definition of $\boldsymbol \Upsilon,$
\begin{equation}\label{rita9d}
    \boldsymbol\mu =\boldsymbol\nu+\nabla\psi_0
\end{equation}
\end{subequations}
Notice that relation \eqref{rita8}-\eqref{rita9} coincide with
desired relations \eqref{sima13a}- \eqref{sima13g} in Theorem
\ref{sima12}.  In order to obtain the full system of relations
\eqref{sima13} we have supplement  \eqref{rita8}-\eqref{rita9} with
the expression \eqref{sima13h} for $\delta\boldsymbol\beta$ and
$\delta\varphi_0$:
\begin{equation}
\label{rita9x} \delta\boldsymbol \beta=\boldsymbol \nu, \quad \delta
\varphi_0=\psi_0 +w_2\delta w_1+\chi_i\, \frac{\partial}{\partial
\xi_i} \varphi_0-w_2\chi_i\, \frac{\partial}{\partial \xi_i}
w_1-\boldsymbol \nu \cdot \mathbf u,
\end{equation}
where $\boldsymbol \nu$ and $\psi_0$ are the components of the given
vector $\boldsymbol\Upsilon$. The system of equalities
\eqref{rita8}-\eqref{rita9x} is equivalent to \eqref{sima13}. It is
important to note that the left hand sides of these equalities are
the components of vector fields $\delta\boldsymbol\varphi$ and
$\delta\boldsymbol\Theta$. They are not independent and should
satisfy relations \eqref{gala2} which constitutes the connection
between
$\delta\boldsymbol\Theta$
 $\delta\boldsymbol\varphi$.
Hence,  relations  \eqref{gala2}  give  the extra five equations
\begin{subequations}\label{rita10}
\begin{gather}
\label{rita10a} \delta\mathbf V=-\mathbf V\, \nabla_\xi\delta\mathbf
u\,\mathbf V,
\\
\label{rita10b} \delta W_{22}=\frac{1}{W_{11}}\big(W_{12}\delta
W_{21}+W_{21}\delta W_{12}-
W_{22}\delta W_{11}\big),\\
\label{rita10c} \delta\mathbf v=\delta\beta+\mathbf
V(\nabla\delta\varphi_0-w_2\nabla\delta w_1- \delta w_2 \nabla w_1)+
\delta \mathbf V\,(\nabla \varphi_0-w_2\nabla w_1),\\\label{rita10d}
\delta\boldsymbol\Lambda=-\delta \mathbf V \, \nabla \mathbf w\,
\mathbf J\mathbf W- \mathbf V\, \nabla(\delta\mathbf w)\mathbf
J\mathbf W-\mathbf W\, \nabla\mathbf w\,
 \mathbf J\, \delta\mathbf W,\\\label{rita10e}
\delta\mathbf R_i=-\delta V_{ik}
\frac{\partial}{\partial\xi_k}(\mathbf W^{\top}) \mathbf J\mathbf W-
 V_{ik} \frac{\partial}{\partial\xi_k}(\delta\mathbf W^{\top})\mathbf J\mathbf W-
 V_{ik} \frac{\partial}{\partial\xi_k}(\mathbf W^{\top})\mathbf J\, \delta\mathbf W.
\end{gather}
\end{subequations}
The obtained extended system \eqref{rita8}-\eqref{rita10} is
overdetermined and we have to prove that it  satisfies compatibility
conditions. The assertion $(\mathbf i)$ of Theorem \ref{sima12} will
be proved if we prove that
 system  \eqref{rita10} is a consequence of equations \eqref{rita8}-
 \eqref{rita9x}. The proof falls into five steps.
\paragraph{Step 1.} Let us show that equalities \eqref{rita8}
and  \eqref{rita9} yield \eqref{rita10a}. We begin with the
observation that equality \eqref{rita8a} along with relation
\eqref{anna6} in Theorem \ref{anna4} implies
$\boldsymbol\chi=\mathbf V^\top \delta\mathbf u$. From this we
obtain
\begin{equation}\label{rita11}
    \chi_k\,\frac{\partial}{\partial\xi_k}\mathbf V=V_{nk}\,
    \delta u_n \,\frac{\partial}{\partial\xi_k}\mathbf V
\end{equation}
Next we have
$$
\{\mathbf V\nabla\boldsymbol \chi\}_{ij}=V_{ik}
\frac{\partial}{\partial \xi_k}\big( V_{nj}\,\delta u_n\big)=
V_{ik}\,\frac{\partial V_{nj}}{\partial \xi_k}\,\delta
u_n+V_{ik}V_{nj}\, \frac{\partial\,\delta u_n}{\partial \xi_k}
$$
Combining this result with \eqref{rita11} we arrive at
\begin{equation}\label{rita12}
\Big\{\chi_k\frac{\partial \mathbf V}{\partial \xi_k}- \mathbf
V\nabla\boldsymbol\chi\Big\}_{ij}= \Big(\frac{V_{nk}\partial
V_{ij}}{\partial \xi_k}-V_{ik} \frac{\partial V_{nj}}{\partial
\xi_k}\Big) -V_{ik}V_{nj}\frac{\partial \delta u_n}{\partial \xi_k}.
\end{equation}
Next notice that $V_{nk}\partial/\partial\xi_k=\partial /\partial
x_n$ and $V_{ij}=\partial \xi_j/\partial x_i$. It follows that
$$
V_{nk}\frac{\partial V_{ij}}{\partial \xi_k}-V_{ik}\frac{\partial
V_{nj}}{\partial \xi_k} =\frac{\partial}{\partial
x_n}\Big(\frac{\partial \xi_j}{\partial x_i}\Big)-
\frac{\partial}{\partial x_i}\Big(\frac{\partial \xi_j}{\partial
x_n}\Big)=0
$$
Substituting this equality into \eqref{rita12} we arrive at the
equality
$$
\chi_k\frac{\partial \mathbf V}{\partial \xi_k}-\mathbf
V\nabla\boldsymbol\xi=- \mathbf V\, \nabla\delta\mathbf u\, \mathbf
V,
$$
which along with \eqref{rita9a} implies \eqref{rita10a}.
\paragraph{Step 2.} Let us prove that \eqref{rita8} and \eqref{rita9} imply
\eqref{rita10b}. Notice that \eqref{rita10b} is equivalent to the
equality
\begin{equation}\label{rita13}
    \delta\mathbf W^\top\mathbf J\mathbf W+\mathbf W^\top\mathbf J\delta\mathbf W=0.
\end{equation}
It follows from \eqref{rita8d} that
\begin{equation}\label{rita14}\begin{split}
     \delta\mathbf W^\top\mathbf J\mathbf W+\mathbf W^\top\mathbf J\delta\mathbf W=
     \chi_i\Big(\frac{\partial}{\partial \xi_i}\mathbf W^\top\mathbf J\mathbf W+
     \mathbf W^\top\mathbf J\frac{\partial}{\partial \xi_i}\mathbf W\Big)+\\
     (\mathbf W\boldsymbol \Gamma)^\top\mathbf J\mathbf W+\mathbf W^{\top}\mathbf J
     \mathbf W\boldsymbol\Gamma.
\end{split}\end{equation}
Next, relation \eqref{anna7} in Theorem \ref{anna4} yields $\mathbf
W^\top\mathbf J\mathbf W=\mathbf J$. It follows that
$$
\frac{\partial}{\partial \xi_i}\mathbf W^\top\mathbf J\mathbf W+
     \mathbf W^\top\mathbf J\frac{\partial}{\partial \xi_i}\mathbf W=0
$$
and
$$
(\mathbf W\boldsymbol \Gamma)^\top\mathbf J\mathbf W+\mathbf
W^{\top}\mathbf J
     \mathbf W\boldsymbol\Gamma=\boldsymbol\Gamma^\top\mathbf J+\mathbf J\boldsymbol\Gamma
     =\mathbf J\boldsymbol\Gamma-(\mathbf J\boldsymbol\Gamma)^\top=0
 $$
 since the matrix $\mathbf J\boldsymbol\Gamma$ is symmetric and
 $\mathbf J^\top=-\mathbf J.$
Substituting these relations into \eqref{rita14} we obtain
\eqref{rita13} and the assertion follows.
\paragraph{Step 3.} Let us prove that equalities \eqref{rita8}-\eqref{rita9x} imply
\eqref{rita10c}. It is necessary to prove that the right hand side
of \eqref{rita8b} coincide with the right hand side of
\eqref{rita10c}. Let us calculate step by step all terms in the
right hand side of \eqref{rita8b}.  Relation \eqref{rita8a} and the
identity
 $V_{ik}\partial/ \partial \xi_k=\partial/\partial x_i$  imply
\begin{equation}\label{rita15e}
 \chi_i\frac{\partial}{\partial \xi_i}\mathbf v=
 \delta u_i\frac{\partial }{\partial x_i}\mathbf v,
 \end{equation}
On the other hand, relation  \eqref{anna11} in Theorem \ref{anna4}
yields
$$
\mathbf v=\boldsymbol\beta+\nabla_x\varphi_0-w_2\nabla_x w_1.
$$
Combining these results we arrive at
\begin{equation}\label{rita15}
   \chi_i\frac{\partial}{\partial \xi_i}\mathbf v=\delta u_i
    \frac{\partial }{\partial x_i}(\nabla_x\varphi_0-w_2\nabla_x w_1).
\end{equation}
Next, relation \eqref{anna8} in Theorem \ref{anna4} implies
$$
\boldsymbol\Lambda=-\mathbf V\,\nabla_\xi\mathbf w\, \mathbf J\,
 \mathbf W=-\nabla_x \mathbf w\, \mathbf J\mathbf W.
$$
In its turn, it follows from  \eqref{rita8c} that
\begin{equation}\label{rita16e}
\boldsymbol\lambda=\mathbf W^{-1}\delta \mathbf w- \chi_i\mathbf
W^{-1}\frac{\partial}{\partial \xi_i}\mathbf w= \mathbf W^{-1}\delta
\mathbf w-\delta u_i\mathbf W^{-1}\frac{\partial}{\partial
x_i}\mathbf w.
\end{equation}
Thus we get
\begin{equation*}\begin{split}
 \boldsymbol\Lambda\boldsymbol\lambda =\nabla_x\mathbf w\, \mathbf J
  \Big(\delta u_i\frac{\partial}{\partial x_i}\mathbf w\Big)-
  \nabla_x\mathbf w\, \mathbf J\, \delta\mathbf w=
\nabla_x(w_2\delta w_1)-\\(\delta w_2\nabla_x w_1+w_2\nabla_x \delta
w_1)+ \Big(\delta u_i\frac{\partial}{\partial
x_i}w_2\Big)\nabla_xw_1- \Big(\delta u_i\frac{\partial}{\partial
x_i}w_1\Big)\nabla_xw_2
\end{split}\end{equation*}
Combining this result with \eqref{rita15} we arrive at
\begin{equation}\label{rita16}\begin{split}
\chi_i\frac{\partial}{\partial \xi_i}\mathbf
v+\boldsymbol\Lambda\boldsymbol\lambda
=\nabla_x(w_2\delta w_1)-(\delta w_2\nabla_x w_1+w_2\nabla_x \delta w_1)+\\
\delta u_i\,\nabla_x\big(\frac{\partial}{\partial x_i}\varphi_0- w_2
\frac{\partial}{\partial x_i}w_1\big).
\end{split}\end{equation}
Next notice that
\begin{gather*}
\delta u_i\,\nabla_x\big(\frac{\partial}{\partial x_i}\varphi_0- w_2
\frac{\partial}{\partial x_i}w_1\big) =\\\nabla_x\Big(\delta
u_i\big(\frac{\partial}{\partial x_i}\varphi_0- w_2
\frac{\partial}{\partial x_i}w_1\big)\Big)-\nabla_x(\delta \mathbf
u)(\nabla_x\varphi_0-w_2\nabla_x w_1).
\end{gather*}
Hence we can rewrite \eqref{rita16} in the form
\begin{equation}\label{rita17}\begin{split}
\chi_i\frac{\partial}{\partial \xi_i}\mathbf
v+\boldsymbol\Lambda\boldsymbol\lambda =\nabla_x\Big(w_2\delta
w_1+\delta u_i\frac{\partial}{\partial x_i}\varphi_0-
\delta u_i w_2\frac{\partial}{\partial x_i}w_1\Big)-\\
(\delta w_2\nabla_x w_1+w_2\nabla_x \delta w_1)-\nabla_x(\delta
\mathbf u) (\nabla_x\varphi_0-w_2\nabla_x w_1).
\end{split}\end{equation}
On the other hand, the identity $\nabla_x=\mathbf V\nabla_\xi$ and
relation \eqref{rita10a}, which  we have been proved in Step 1, imply
$$
\nabla_x(\delta \mathbf u)(\nabla_x\varphi_0-w_2\nabla_x
w_1)=-\delta\mathbf V (\nabla_\xi\varphi_0-w_2\nabla_\xi w_1)
$$
Moreover, as it was mentioned above, we have $\delta
u_k\partial/\partial x_k= \chi_k\partial/\partial\xi_k$. It follows
from this and \eqref{rita17} that
\begin{equation}\label{rita18}\begin{split}
\chi_i\frac{\partial}{\partial \xi_i}\mathbf
v+\boldsymbol\Lambda\boldsymbol\lambda =\mathbf
V\nabla_\xi\Big(w_2\delta w_1+\chi_i\frac{\partial}{\partial
\xi_i}\varphi_0-
\chi_i w_2\frac{\partial}{\partial \xi_i}w_1\Big)-\\
\mathbf V(\delta w_2\nabla_\xi w_1+w_2\nabla_\xi \delta w_1)+ \delta
\mathbf V(\nabla_\xi\varphi_0-w_2\nabla_\xi w_1).
\end{split}\end{equation}
Let us consider the term $\mathbf V\boldsymbol \mu$ in expression
\eqref{rita8b} for $\mathbf v$. It follows from \eqref{rita9d} that
$$
\mathbf V\boldsymbol \mu=\mathbf V\boldsymbol \nu+\mathbf
V\nabla_\xi\psi_0.
$$
Next notice that
$$
\mathbf V\boldsymbol\nu=\mathbf V\nabla_\xi(\boldsymbol \nu\cdot
\boldsymbol\xi)= \nabla_x(\boldsymbol\nu\cdot
\boldsymbol\xi)=\nabla_x(\boldsymbol\nu\mathbf x- \boldsymbol \nu
\cdot \mathbf u)=\boldsymbol \nu-\nabla_x(\boldsymbol \nu \cdot
\mathbf u)= \boldsymbol\nu-\mathbf V\nabla_\xi(\boldsymbol\nu \cdot
\mathbf u).
$$
Thus we get
$$
\mathbf V\boldsymbol \mu= \boldsymbol \nu+\mathbf V\nabla_\xi(\psi_0-\boldsymbol
\nu \cdot \mathbf u).
$$
Combining this result with \eqref{rita18} and recalling formulae
\eqref{rita8b} for $\delta\mathbf v$ and \eqref{rita9x} for
$\varphi_0$ we finally obtain
\begin{equation*}\begin{split}
    \delta \mathbf v=\chi_i\frac{\partial}{\partial \xi_i}\mathbf v+
    \boldsymbol\Lambda\boldsymbol\lambda
    +\mathbf V\boldsymbol\mu=\boldsymbol\nu+\mathbf V\nabla\delta\varphi_0-\\
   \mathbf V(\delta w_2\nabla_\xi w_1+w_2\nabla_\xi \delta w_1)+
   \delta \mathbf V(\nabla_\xi\varphi_0-w_2\nabla_\xi w_1),
\end{split}\end{equation*}
which obviously yields \eqref{rita10c}.
\paragraph{Step 4.} Let us prove that equalities \eqref{rita8}-\eqref{rita9} imply
\eqref{rita10d}. It follows from \eqref{rita9b} that
\begin{equation}\label{rita19}
\delta
\boldsymbol\Lambda=\chi_i\,\frac{\partial\boldsymbol\Lambda}{\partial\xi_i}
+\mathbf V\nabla_\xi(\mathbf
J\boldsymbol\lambda)+\boldsymbol\lambda^\top\mathbf R-
\boldsymbol\Lambda\boldsymbol\Gamma.
\end{equation}
Arguing as in the proof of \eqref{rita15e} and using  the identity
 $V_{ik}\partial/ \partial \xi_k=\partial/\partial x_i$  we obtain
\begin{equation}\label{rita20}
 \chi_i\frac{\partial \boldsymbol\Lambda}{\partial \xi_i}=
 \delta u_i\frac{\partial \boldsymbol\Lambda}{\partial x_i},\quad
 \mathbf V\nabla_\xi(\mathbf J\boldsymbol\lambda)=
 \nabla_x(\mathbf J\boldsymbol\lambda).
 \end{equation}
Thus we get
\begin{equation}\label{rita19-0}
\delta \boldsymbol\Lambda=\delta
u_i\,\frac{\partial\boldsymbol\Lambda}{\partial x_i}+
\nabla_x(\mathbf J\boldsymbol\lambda)+\boldsymbol\lambda^\top\mathbf
R- \boldsymbol\Lambda\boldsymbol\Gamma.
\end{equation}
Next, it follows from the relation \eqref{anna9} in Theorem
\ref{anna4} that
\begin{equation}\label{rita21}
    \mathbf R_i=-V_{ik}\frac{\partial\mathbf W^\top}{\partial \xi_k}\,
    \mathbf J\, \mathbf W=
    -\frac{\partial\mathbf W^\top}{\partial x_i}\, \mathbf J\, \mathbf W
\end{equation}
Here $\boldsymbol \lambda^\top \mathbf R$
is $(n-1)\times 2$ -matrix with the rows
$\boldsymbol\lambda^\top\mathbf R_i$,
 $1\leq i\leq n-1$. It follows from \eqref{rita21} that
$$
\boldsymbol\lambda^\top \mathbf R_i=-\boldsymbol\lambda^\top
\frac{\partial\mathbf W^\top}{\partial x_i}\, \mathbf J\, \mathbf W=
-\frac{\partial}{\partial x_i}(\boldsymbol\lambda^\top \mathbf
W^\top)\mathbf J\mathbf
W+\frac{\partial\boldsymbol\lambda^\top}{\partial x_i} \mathbf
W^\top\mathbf J\mathbf W.
$$
Noting that in view of relation \eqref{anna7} in Theorem
\ref{anna4}, we have $ \mathbf W^\top\mathbf J\mathbf W=\mathbf
J=-\mathbf J^\top$, we obtain the expression for the rows
$\boldsymbol\lambda^\top\mathbf R_i$ of the matrix
$\boldsymbol\lambda^\top\mathbf R$
$$
\boldsymbol\lambda^\top\mathbf R_i=-\frac{\partial}{\partial
x_i}(\mathbf W\boldsymbol\lambda )^\top\mathbf J\mathbf
W-\frac{\partial}{\partial x_i}(\mathbf J\boldsymbol\lambda)^\top.
$$
which leads to
\begin{equation*}
   \boldsymbol\lambda^\top \mathbf R
   =-\nabla_x(\mathbf W\boldsymbol\lambda)\mathbf J\mathbf W-
   \nabla_x(\mathbf J\boldsymbol\lambda).
\end{equation*}
On the other hand, equality \eqref{rita8c} yields
$$
\boldsymbol\lambda=\mathbf W^{-1}\Big(\delta\mathbf w-
\chi_i\frac{\partial \mathbf w}{\partial \xi_i}\Big)= \mathbf
W^{-1}\Big(\delta\mathbf w-\delta u_i\frac{\partial \mathbf
w}{\partial x_i}\Big)
$$
Substituting this relation into the latter identity we arrive at
\begin{equation}\label{rita12-00}
    \boldsymbol\lambda^\top \mathbf R=-\nabla_x(\delta\mathbf w)\mathbf J\mathbf W
    +\nabla_x\big(\delta u_i\frac{\partial \mathbf w}{\partial x_i}\big)\mathbf J\mathbf W
    -\nabla_x(\mathbf J\boldsymbol\lambda).
\end{equation}
Let us calculate $\boldsymbol \Lambda\boldsymbol\Gamma$. Equalities
\eqref{rita8d}, the identity $\mathbf V\nabla_\xi=\nabla_x$, and
equality \eqref{rita10a}, which we proved in Step 1, implies
\begin{equation}\label{rita23e}
\boldsymbol\Gamma=\mathbf W^{-1}\delta\mathbf W-\chi_i \mathbf
W^{-1} \frac{\partial\mathbf W}{\partial\xi_i} =\mathbf
W^{-1}\delta\mathbf W-\delta u_i\mathbf W^{-1} \frac{\partial\mathbf
W}{\partial x_i} .
\end{equation}
On the other hand, equality \eqref{anna8} yields
$\boldsymbol\Lambda= -\nabla_x(\mathbf w)\mathbf J\mathbf W$. Thus
we get
$$
\boldsymbol\Lambda\boldsymbol\Gamma=-\nabla_x(\mathbf w)\mathbf
J\delta\mathbf W+ \delta u_i(\nabla_x \mathbf w)\mathbf J
\frac{\partial\mathbf W}{\partial x_i} .
$$
Substituting this result along with \eqref{rita12-00} into
\eqref{rita19-0} we arrive at
\begin{equation}\label{rita23}\begin{split}
 \delta\boldsymbol\Lambda=\delta u_i\Big( \frac{\partial \boldsymbol\Lambda}{\partial x_i}+
\nabla_x \mathbf w\mathbf J \frac{\partial\mathbf W}{\partial x_i}
\Big)+\nabla_x\Big(\delta u_i\frac{\partial\mathbf w}{\partial
x_i}\Big) \mathbf J\mathbf W-\\\nabla_x (\delta\mathbf w)\mathbf
J\mathbf W- \nabla_x \mathbf w \,\mathbf J\,(\delta\mathbf W).
\end{split}\end{equation}
Let us calculate separately the term containing $\delta u_i$. We begin
with the observation that
$$
\nabla_x\Big(\delta u_i\frac{\partial\mathbf w}{\partial
x_i}\Big)=\delta u_i \frac{\partial}{\partial x_i}\nabla_x\mathbf
w+\nabla_x(\delta\mathbf u)\, \nabla_x\mathbf w.
$$
Thus we get
\begin{equation}\label{rita24}
\nabla_x\Big(\delta u_i\frac{\partial\mathbf w}{\partial x_i}\Big)
\mathbf J\mathbf W=\delta u_i \frac{\partial}{\partial
x_i}(\nabla_x\mathbf w)\mathbf J\mathbf W+ \nabla_x(\delta\mathbf
u)\, \nabla_x\mathbf w\mathbf J\mathbf W.
\end{equation}
Next, the identity $\boldsymbol\Lambda=-\nabla_x \mathbf w \,
\mathbf J\, \mathbf W$ implies
$$
\delta u_i \frac{\partial \boldsymbol\Lambda}{\partial x_i}=-\delta
u_i \frac{\partial}{\partial x_i}(\nabla_x\mathbf w)\mathbf J\mathbf
W- \delta u_i\nabla_x\mathbf w\, \mathbf J\frac{\partial\mathbf
W}{\partial x_i}
$$
Combining this result with \eqref{rita24} we obtain
\begin{equation*}
   \delta u_i\Big( \frac{\partial }{\partial x_i}\boldsymbol\Lambda+
\nabla_x \mathbf w\mathbf J \frac{\partial\mathbf W}{\partial x_i}
\Big)+\nabla_x\Big(\delta u_i\frac{\partial\mathbf w}{\partial
x_i}\Big) \mathbf J\mathbf W=\nabla_x(\delta\mathbf
u)\nabla_x\mathbf w\, \mathbf J\mathbf W.
\end{equation*}
Substituting this result in \eqref{rita23} we obtain
\begin{equation}\label{rita26}
 \delta\boldsymbol\Lambda=-\nabla_x(\delta\mathbf w)\mathbf J\mathbf W-
 \nabla_x\mathbf w\, \mathbf J\,  \delta \mathbf W+\nabla_x(\delta\mathbf u)
 \nabla_x\mathbf w\, \mathbf J\mathbf W.
\end{equation}
Next, relation \eqref{rita10a}, which was proved in Step 1, and the
identity $\mathbf V\nabla_\xi=\nabla_x$ imply
$$
\nabla_x(\delta\mathbf u)\nabla_x\mathbf w\, \mathbf J\mathbf W=
(\mathbf V\nabla_\xi \delta\mathbf u) \nabla_x \mathbf w\, \mathbf
J\mathbf W=-\delta\mathbf V\, \nabla_x\mathbf w\, \mathbf J\,\mathbf
W,
$$
$$
\nabla_x(\delta \mathbf w)=\mathbf V\nabla_\xi(\delta\mathbf w),
\quad \nabla_x\mathbf w=\mathbf V\nabla_\xi \mathbf w.
$$
Substituting these equalities into \eqref{rita26} gives desired
relation \eqref{rita10d}
 and the assertion follows.
\paragraph{Step 5.} It remains to prove that equalities \eqref{rita8}-\eqref{rita9}
imply \eqref{rita10e}. We proved yet that $\delta\mathbf V$
satisfies relation \eqref{rita10a}, i.e.,
$$
\delta V_{ik}=-V_{in}\Big(\frac{\partial}{\partial \xi_n}\delta
u_p\Big)V_{pk}.
$$
It follows from this \eqref{rita10a} and the identity $\mathbf
V\nabla_\xi=\nabla_x$ that
$$
\delta V_{ik}\partial/ \partial\xi_k=-V_{in}\Big(
\frac{\partial}{\partial \xi_n}\delta
u_p\Big)V_{pk}\partial/\partial \xi_k=-
\Big(\frac{\partial}{\partial x_i}\delta u_p\Big)\partial/\partial
x_p.
$$
It follows from this that relation \eqref{rita10e} can be rewritten
in the equivalent form
\begin{equation}\label{rita27}
    \delta\mathbf R_i=\Big(\frac{\partial}{\partial x_i}\delta u_k\Big)
    \frac{\partial}{\partial x_k}
\big( \mathbf W^\top\big)\mathbf J\mathbf W-\frac{\partial}{\partial
x_i} \big(\delta \mathbf W^\top\big)\mathbf J\mathbf W -
\frac{\partial}{\partial x_i} \big(\mathbf W^\top\big)\mathbf
J\delta \mathbf W   .
\end{equation}
Hence it suffices to prove that relations
\eqref{rita8}-\eqref{rita9} yields \eqref{rita27}. To this end
notice that relation \eqref{rita8a} and the identity
 $V_{ik}\partial/ \partial \xi_k=\partial/\partial x_i$  imply
\begin{equation*}
 \chi_k\frac{\partial}{\partial \xi_k}\mathbf R_i=\delta u_k
 \frac{\partial }{\partial x_k}\mathbf R_i,
\quad V_{ik}\frac{\partial}{\partial \xi_k}(\mathbf J\,
\boldsymbol\Gamma)= \frac{\partial}{\partial x_i}(\mathbf J\,
\boldsymbol\Gamma).
 \end{equation*}
Using these identities we can rewrite relation \eqref{rita9c} in the
form
\begin{equation}\label{rita28}
    \delta\mathbf R_i=\delta u_k\frac{\partial }{\partial x_k}\mathbf R_i+
    \frac{\partial}{\partial x_i}(\mathbf J\, \boldsymbol\Gamma)+
    \mathbf R_i\boldsymbol\Gamma+(\mathbf R_i\boldsymbol\Gamma)^\top.
\end{equation}
Let us calculate all terms in the right hand side. It follows from
\eqref{rita8d}
 and \eqref{rita8a}
that
$$
\delta \mathbf W=\chi_i\,\frac{\partial}{\partial\xi_i}\mathbf W+
\mathbf W\boldsymbol\Gamma =\delta u_k\frac{\partial}{\partial
x_k}\mathbf W+\mathbf W\boldsymbol\Gamma
$$
or equivalently
$$
\boldsymbol\Gamma=\mathbf W^{-1}\delta\mathbf W-\delta u_k\mathbf
W^{-1} \frac{\partial}{\partial x_k}\mathbf W
$$
Since $\mathbf J\mathbf W^{-1}=\mathbf W^\top\mathbf J$ we conclude
from this that
$$
\mathbf J\boldsymbol\Gamma=\mathbf W^\top\mathbf J
 \delta \mathbf W-\delta u_k\mathbf W^{\top}\frac{\partial}{\partial x_k}\mathbf W,
$$
which leads to the equality
\begin{equation}\label{rita29}\begin{split}
\frac{\partial}{\partial x_i}(\mathbf J\boldsymbol\Gamma)=
\frac{\partial \mathbf W^\top}{\partial x_i} \mathbf J\delta\mathbf
W+\mathbf W^\top\mathbf J \frac{\partial \delta\mathbf W}{\partial
x_i} -\\\delta u_k\Big(\frac{\partial \mathbf W^\top}{\partial x_i}
\mathbf J\frac{\partial \mathbf W}{\partial x_k}+\mathbf W^\top
\mathbf J\frac{\partial^2\mathbf W}{\partial x_i\partial x_k}\Big)
-\frac{\partial \delta u_k}{\partial x_i}\mathbf W^\top\mathbf J
\frac{\partial\mathbf W}{\partial x_k}.
\end{split}\end{equation}
Next, equality \eqref{rita21} implies
\begin{equation}\label{rita30}
    \delta u_k\frac{\partial}{\partial x_k}\mathbf R_i=-\delta u_k
    \Big(\frac{\partial^2\mathbf W^\top}{
    \partial x_i\partial x_k}\, \mathbf J \mathbf W+
    \frac{\partial\mathbf W^\top}{
    \partial x_i}\,\mathbf J\, \frac{\partial\mathbf W}{
\partial x_k}\Big)
\end{equation}
Notice that relations \eqref{rita21} and \eqref{rita23e} imply
\begin{equation}\label{rita31}\begin{split}
  \mathbf R_i\boldsymbol\Gamma+(\mathbf R_i\boldsymbol\Gamma)^\top=
  \delta \mathbf W^\top \mathbf J\frac{\partial \mathbf W}{\partial x_i} -
   \frac{\partial \mathbf W^\top}{\partial x_i}\,\mathbf J\delta \mathbf W^\top+\\
   \delta u_k\Big( \frac{\partial \mathbf W^\top}{\partial x_i}\mathbf J
   \frac{\partial \mathbf  W}{\partial x_k}-
   \frac{\partial \mathbf W^\top}{\partial x_k}
   \mathbf J \frac{\partial \mathbf W}{\partial x_i}\Big).
\end{split}\end{equation}
Substituting  \eqref{rita29}-\eqref{rita31} into \eqref{rita28} we
obtain
\begin{equation*}
    \delta\mathbf R_i=-\frac{\partial \delta u_k}{\partial x_i}
    \mathbf W^\top \mathbf J\frac{\partial \mathbf W}{\partial x_k}+
    \mathbf W^\top\mathbf J\frac{\partial \delta\mathbf W}{\partial x_i}
+\delta\mathbf W^\top\mathbf J \frac{\partial \mathbf W}{\partial
x_i}-\delta u_k\frac{\partial^2}{\partial x_i\partial
x_k}\big(\mathbf W^\top\mathbf J\mathbf W\big).
\end{equation*}
On the other hand, the identity $\mathbf W^\top\mathbf J\mathbf
W=\mathbf J$ implies
$$
\frac{\partial^2}{\partial x_i\partial x_k} \big(\mathbf
W^\top\mathbf J\mathbf W\big)=0, \quad \mathbf W^\top \mathbf J
\frac{\partial \mathbf W}{\partial x_k}+ \frac{\partial \mathbf
W^\top}{\partial x_k}\mathbf J\mathbf W=0.
$$
Thus we get
\begin{equation}\label{rita31}
 \delta\mathbf R_i=\frac{\partial \delta u_k}{\partial x_i}
 \frac{\partial \mathbf W^\top}{\partial x_k}\mathbf J\mathbf W
+\mathbf W^\top\mathbf J\frac{\partial \delta\mathbf W}{\partial
x_i} +\delta\mathbf W^\top\mathbf J\frac{\partial \mathbf
W}{\partial x_i}
\end{equation}
Relation \eqref{rita10b}, which was proved in Step 2, yields
$\delta\mathbf W^\top\mathbf J\mathbf W+\mathbf W^\top\mathbf
J\delta\mathbf W=0$ differentiation both sides of this equality
gives
$$
\mathbf W^\top\mathbf J\frac{\partial \delta\mathbf W}{\partial
x_i}+ \delta\mathbf W^\top\mathbf J\frac{\partial \mathbf
W}{\partial x_i}= -\frac{\partial}{\partial x_i} \big(\delta \mathbf
W^\top\big)\mathbf J\mathbf W - \frac{\partial}{\partial x_i}
\big(\mathbf W^\top\big)\mathbf J\delta \mathbf W .
$$
Substituting this relation into \eqref{rita31} we obtain desired
identity \eqref{rita27}, and the assertion follows.

\subsection{Proof of $(\mathbf i\mathbf i)$.} Choose an arbitrary
analytic
\begin{equation}\label{rita33}
    \boldsymbol\varphi=(\beta,\varphi_0, \mathbf u, \mathbf w,
    W_{11}, W_{12}, W_{21})
\end{equation}
and consider the vector field
\begin{equation}\label{rita34}
    \Theta(\boldsymbol\varphi)=\big(\mathbf u, \mathbf v, \mathbf w, \mathbf V,\boldsymbol \Lambda, \mathbf W,
    \mathbf R_i\big),
\end{equation}
defined by \eqref{anna17}. Let $\boldsymbol \theta$ be a
corresponding canonical mapping defined by \eqref{anna1}.  Next
choose an arbitrary vector field
\begin{equation}\label{rita35}
\delta\boldsymbol\varphi=(\delta\mathbf \beta, \delta\mathbf
\varphi_0, \delta\mathbf u, \delta\mathbf w, \delta W_{11},\delta
W_{12},\delta W_{21})\in X_{\sigma,d-1}.
\end{equation}
and the corresponding vector field
\begin{equation}\label{rita35x}
    \delta\boldsymbol\Theta=\big(\, \delta\mathbf u, \delta\mathbf v,
    \delta\mathbf w, \delta\mathbf V, \delta\mathbf W, \delta\boldsymbol\Lambda,
    \delta\mathbf R_i\,\big).
\end{equation}
Let us consider the associated mapping
\begin{equation}\label{rita36}
\delta\boldsymbol\theta=\left(
\begin{array}{c}
\delta \mathbf u(\mathbf \xi)\\ \delta \mathbf v(\boldsymbol\xi)+
\delta \mathbf V(\boldsymbol\xi)\boldsymbol\eta +\delta \boldsymbol
\Lambda(\boldsymbol\xi)
\boldsymbol\zeta+\boldsymbol\zeta^\top\delta \mathbf R(\boldsymbol \xi)\boldsymbol\zeta\\
 \delta \mathbf w(\boldsymbol \xi)+\delta \mathbf W(\boldsymbol\xi)\boldsymbol\zeta
\end{array}\right),
\end{equation}
with the coefficients given by \eqref{gala2} , i. e.,
\begin{gather}\nonumber
\delta\mathbf V=-\mathbf V\,\delta \nabla_\xi\delta\mathbf
u\,\mathbf V, \quad \delta W_{22}=\frac{1}{W_{11}}\big(W_{12}\delta
W_{21}+W_{21}\delta W_{12}-
W_{22}\delta W_{11}\big),\\
\label{rita100}\delta \mathbf W=\left(
\begin{array}{cc}
\delta W_{11}&\delta W_{12}\\
\delta W_{21}&\delta W_{22}
\end{array}
\right),\\\nonumber \delta\mathbf v=\delta\beta+\mathbf
V(\nabla\delta\varphi_0-w_2\nabla\delta w_1- \delta w_2 \nabla w_1)+
\delta \mathbf V\,(\nabla \varphi_0-w_2\nabla w_1),\\\nonumber
\delta\boldsymbol\Lambda=\delta \mathbf V \, \nabla w\, \mathbf
J\mathbf W-\mathbf V\,
 \nabla(\delta\mathbf w)\mathbf J\mathbf W-\mathbf W\, \nabla\mathbf w\, \mathbf J\,
 \delta\mathbf W,\\\nonumber
\delta\mathbf R_i=-\delta V_{ik}
\frac{\partial}{\partial\xi_k}(\mathbf W^{\top}) \mathbf J\mathbf W-
 V_{ik} \frac{\partial}{\partial\xi_k}(\delta\mathbf W^{\top})\mathbf J\mathbf W-
 V_{ik} \frac{\partial}{\partial\xi_k}(\mathbf W^{\top})\mathbf J\, \delta\mathbf W.
\end{gather}
Our task is to find to find a vector field
\begin{equation*}
    \boldsymbol\Upsilon=(\boldsymbol\nu,  \psi_0,  \boldsymbol\chi,\boldsymbol\lambda, \Gamma_{11},
    \Gamma_{12}, \Gamma_{21})\in X_{\sigma, d-1}(r).
\end{equation*}
such that the corresponding vector field
\begin{equation}\label{rita35xy}
    \mathfrak Z:=\delta\boldsymbol\Theta\equiv D_\varphi
    \boldsymbol\Theta(\boldsymbol\varphi_0)[\boldsymbol\Upsilon],
\end{equation}
\begin{equation}\label{rita35xyz}
\mathfrak Z\equiv \Big(\boldsymbol\chi,\, \boldsymbol\mu,\, \boldsymbol
\lambda,\,
 -\nabla\boldsymbol\chi,\, \boldsymbol\Gamma,\, \nabla (\mathbf J\boldsymbol\lambda),\, \partial_{\xi_i}
 (\mathbf J\, \boldsymbol\Gamma)\, \Big),\end{equation}
with
\begin{equation}\label{rita35xyza}
\boldsymbol\mu=\boldsymbol\nu+\nabla\psi_0, \quad \text{Tr~}\boldsymbol\Gamma=0.
\end{equation}
satisfies the relations \eqref{sima15}, i.e.,
\begin{subequations}\label{rita36x}
\begin{gather}\label{rita36xa}
\boldsymbol\chi\,= \,\mathbf V^\top \,\delta\mathbf
u\\\label{rita36xb} \boldsymbol\lambda\,=\, \mathbf
W^{-1}\delta\mathbf w-\chi_i\,\mathbf W^{-1}
\frac{\partial}{\partial \xi_i} \mathbf w\\\label{rita36xc}
\boldsymbol\Gamma\,=\, \mathbf W^{-1}\delta\mathbf W-\chi_i\,\mathbf
W^{-1} \frac{\partial}{\partial \xi_i} \mathbf W,\\\label{rita36xd}
\boldsymbol\mu\,=\, \mathbf V^{-1}\Big(\delta\mathbf W+\chi_i\,
\frac{\partial}{\partial \xi_i} \mathbf v-
\boldsymbol\Lambda\boldsymbol\lambda\Big)\\\label{rita36xe}
\nabla\psi_0=\boldsymbol\mu-\boldsymbol\nu,
 \quad \boldsymbol\nu=\frac{1}{(2\pi)^{n-1}}\int_{\mathbb T^{n-1}}
  \boldsymbol\mu\, d\boldsymbol\xi=\delta\boldsymbol\beta.
\end{gather}
\end{subequations}
In order to prove \eqref{rita36x}, we consider  the associated canonical mapping
\begin{equation}\label{rita101}
\boldsymbol\pi=\left(
\begin{array}{c}
\boldsymbol\chi(\boldsymbol\xi)\\
\boldsymbol\mu-(\nabla_\xi \boldsymbol\chi)\boldsymbol\eta+\nabla_\xi(\mathbf J\boldsymbol\lambda)\boldsymbol\zeta+\frac{1}{2}\nabla_\xi(\mathbf J\boldsymbol\Gamma\boldsymbol\zeta\cdot \boldsymbol\zeta)\\
\boldsymbol\lambda+\boldsymbol\Gamma\boldsymbol\zeta\end{array}\right).
\end{equation}
Notice that relations \eqref{rita36x} can be written in the form
\begin{equation}\label{rita102}
    \delta \boldsymbol\theta\,=\, \boldsymbol\theta'\,\boldsymbol \pi.
\end{equation}
Here the Jacobi matrix $\boldsymbol\theta'$ is defined by
\eqref{lena2}-\eqref{lena3}. Since $\boldsymbol \theta$ and
$\delta\boldsymbol\theta$ are given, it suffices to prove that a
solution $\boldsymbol\pi$ to \eqref{rita102} admits  representation
\eqref{rita101}. Since, in view of \eqref{anna3}, the matrix
$\boldsymbol\theta'$ is symplectic,   a solution of \eqref{rita102}
is given by
\begin{equation}\label{rita103}
    \boldsymbol\pi=-\mathbf J_{2n}  {\boldsymbol\vartheta'}^\top \mathbf J_{2n}\delta \boldsymbol\vartheta. \end{equation}
Here the matrix $\mathbf J_{2n}$ is defined by \eqref{anna3}.
Substituting \eqref{rita36} into \eqref{rita103} we conclude that
the mapping $\boldsymbol\pi$ admits the representation
\begin{equation}\label{rita104}
\boldsymbol\pi=\left(
\begin{array}{c}
\boldsymbol\chi(\boldsymbol\xi)\\
\boldsymbol\mu(\boldsymbol\xi)+\mathbf
A(\boldsymbol\xi)\boldsymbol\eta+\mathbf
B(\boldsymbol\xi)\boldsymbol\zeta+\frac{1}{2}\boldsymbol
\zeta^\top \boldsymbol C(\boldsymbol\xi)\boldsymbol\zeta\\
\boldsymbol\lambda(\boldsymbol\xi)+\boldsymbol\Gamma(\boldsymbol\xi)\boldsymbol\zeta\end{array}\right).
\end{equation}
Here $\boldsymbol\chi$, $\boldsymbol\mu$, $\boldsymbol \lambda$ are
$2\pi$- periodic vector fields, $\mathbf A$, $\mathbf B$ and
$\boldsymbol\Gamma$ are $2\pi$ periodic matrices, and a
vector-valued function $\boldsymbol\zeta^\top \boldsymbol
C(\boldsymbol\xi)\boldsymbol\zeta$  is in the form
\begin{equation}\label{rita105}
  \boldsymbol\zeta^\top \boldsymbol C(\boldsymbol\xi)\boldsymbol\zeta=\big(
  \boldsymbol\zeta^\top \mathbf C_1(\boldsymbol\xi)\boldsymbol\zeta,\dots,\boldsymbol\zeta^\top \mathbf C_{n-1}(\boldsymbol\xi)\boldsymbol\zeta\big),
\end{equation}
where $\mathbf  C_i$ are $2\pi$-periodic matrices. It follows
directly from \eqref{rita104}, \eqref{rita36}, and expression
\eqref{anna3} for the Jacobi matrix $\boldsymbol\theta'$ that
\begin{equation}\label{rita106}\begin{split}
\boldsymbol \chi=(\mathbf I_{n-1}+\mathbf u')^{-1} \delta\mathbf u,\\
\boldsymbol\lambda=\mathbf W^{-1}\delta\mathbf w-\chi_i\mathbf
W^{-1}\frac{\partial}{\partial \xi_i}
\mathbf w,\\
\boldsymbol\mu=\mathbf V^{-1}\big(\, \delta\mathbf v-
\chi_i\frac{\partial}{\partial \xi_i}
\mathbf v-\boldsymbol\Lambda\boldsymbol\lambda\,\big),\\
\boldsymbol\Gamma=\mathbf W^{-1}\delta\mathbf w-\chi_i\mathbf
W^{-1}\frac{\partial}{\partial \xi_i} \mathbf W.
\end{split}\end{equation}
Hence it remains to prove that
\begin{equation}\label{rita107}
\mathbf A=-\nabla\boldsymbol\chi, \quad \mathbf B=\nabla(\mathbf
J\boldsymbol\lambda)
\end{equation}
\begin{equation}\label{rita108}
    \mathbf C_i=\frac{\partial}{\partial \xi_i}(\mathbf J\boldsymbol\Gamma), \quad
    \mathbf J\boldsymbol\Gamma+\boldsymbol\Gamma^\top\mathbf J=0,
\end{equation}
\begin{equation}\label{rita109}
   \boldsymbol\mu= \boldsymbol\nu+\nabla\psi_0.
\end{equation}
To this end notice that ${\boldsymbol\theta'}^\top\mathbf
J_{2n}\boldsymbol\theta'=\mathbf J_{2n}$, which yields
\begin{equation}\label{rita110}
    {(\delta\boldsymbol\theta)'}^\top\,\,\mathbf J_{2n}\,\,\boldsymbol \theta'+
    {\boldsymbol\theta'}^\top\,\,\mathbf J_{2n}\,\,(\delta\boldsymbol \theta)'=0.
\end{equation}
Now set $\boldsymbol\sigma=(\boldsymbol\xi, \boldsymbol\eta,
\boldsymbol\zeta)\in R^{2n}$ and notice that the vector-valued
function $\boldsymbol \pi$ takes it value in $\mathbb R^{2n}$.
Denote by $\pi_k$ the components of $\boldsymbol\pi$.
Differentiation of the equality $\delta\boldsymbol
\theta=\boldsymbol \theta'\, \boldsymbol\pi$ leads to the identity
$$
\delta\boldsymbol \theta'=\boldsymbol \theta'\boldsymbol\pi'+
\pi_k\frac{\partial}{\partial\sigma_k} \boldsymbol\theta'.
$$
Substituting this result in \eqref{rita11} we arrive at
$$
{\boldsymbol\pi'}^\top\, {\boldsymbol\theta'}^\top\,\mathbf J_{2n}
\boldsymbol\theta'+ {\boldsymbol\theta'}^\top\, \mathbf J_{2n}\,
\boldsymbol\theta'\,\boldsymbol\pi'+
\pi_k\Big\{{\boldsymbol\theta'}^\top\, \mathbf
J_{2n}\frac{\partial}{\partial\sigma_k} \boldsymbol\theta'
+\frac{\partial}{\partial\sigma_k}
{\boldsymbol\theta'}^\top\,\mathbf J_{2n}\,
\boldsymbol\theta'\Big\}. =0.$$ Notice that
$$
{\boldsymbol\theta'}^\top\, \mathbf
J_{2n}\frac{\partial}{\partial\sigma_k} \boldsymbol\theta'
+\frac{\partial}{\partial\sigma_k}
{\boldsymbol\theta'}^\top\,\mathbf J_{2n}\,
\boldsymbol\theta'=\frac{\partial}{\partial\sigma_k}
\big({\boldsymbol\theta'}^\top\, \mathbf J_{2n}\,\boldsymbol\theta'
\big)\equiv\frac{\partial}{\partial\sigma_k} \mathbf J_{2n}=0
$$
Recalling that ${\boldsymbol\theta'}^\top\, \mathbf
J_{2n}\,\boldsymbol\theta'=\mathbf J_{2n}$ we obtain
\begin{equation}\label{rita111}
  {\boldsymbol\pi'}^\top\,\mathbf J_{2n} +
\mathbf J_{2n}\,\boldsymbol\pi' =0.
\end{equation}
It follows from \eqref{rita104} that
\begin{equation}\label{rita112}
\boldsymbol\pi'=\left(
\begin{array}{ccc}
P_{11},&0,&0\\
P_{21},&P_{22},&P_{23}\\
P_{31},&0,& P_{33}\end{array}\right), \quad\mathbf{J}_{2n}=\left(
\begin{array}{ccc}
0&\mathbf{I}_{n-1}&0\\
-\mathbf{I}_{n-1}&0&0
\\0&0&\mathbf{J}
\end{array}
\right),
\end{equation}
where
$$
P_{11}=\boldsymbol\chi', \quad P_{21}\big(\,\boldsymbol\mu+\mathbf
A\boldsymbol\eta+\mathbf B\boldsymbol\zeta+\frac{1}{2}\boldsymbol
\zeta^\top \boldsymbol C\boldsymbol\zeta\,\big)'_\xi
$$
$$
P_{22}=\mathbf A\quad P_{23}=\mathbf B+\mathbf C\boldsymbol\zeta,
$$
$$
P_{31}=\boldsymbol\lambda'_\xi+(\boldsymbol\Gamma\boldsymbol\zeta)'_\xi,
\quad P_{33}=\boldsymbol\Gamma,
$$
Here $\mathbf C\boldsymbol \lambda$ is $(n-1)\times 2$ matrix with
the entries $(\mathbf C\boldsymbol\lambda)_{ij}=C_{i,
jp}\,\lambda_p$. Substituting \eqref{rita112} into \eqref{rita111}
we obtain four nontrivial matrix equations
\begin{equation*}\begin{split}
 P_{22}=-P_{11}^\top, \quad P_{23}=(\mathbf J P_{31}^\top,\\
 P_{33}^\top\mathbf J+\mathbf J P_{33} =0, \quad P_{21}=P_{21}^\top.
\end{split}\end{equation*}
The first equation gives $\mathbf A=-\nabla\boldsymbol\chi$. The
second gives two nontrivial relation
$$
\mathbf B={(\mathbf J\boldsymbol\lambda)'_\xi}^\top\equiv \nabla
(\mathbf J\boldsymbol\lambda), \quad \mathbf C\boldsymbol
\zeta=\nabla_\xi(\mathbf J\boldsymbol\Gamma\boldsymbol\zeta).
$$
The latter equality is equivalent to relations $\mathbf
C_i=\partial(\mathbf J\boldsymbol\Gamma)/\partial\xi_i$. Finally
notice that relation $P_{21}=P_{21}^\top$ gives
$\boldsymbol\mu'_\xi=(\boldsymbol\mu'_\xi)^\top$. In other words,
the Jacobi matrix of the vector-field $\boldsymbol\mu$ is symmetric.
It follows that the vector field $\boldsymbol\mu$ is potential.
Hence it has the representation
$\boldsymbol\mu=\boldsymbol\nu+\nabla \psi_0$, where
$\boldsymbol\nu$ is a constant vector and $\psi_0$ is
$2\pi$-periodic function with zero mean. Therefore the matrices
$\mathbf A$, $\mathbf B$, $\boldsymbol\Gamma$ and the vector field
$\boldsymbol\mu$ satisfy conditions \eqref{rita107}-\eqref{rita109}.
This completes the proof of Theorem \ref{sima12}.

\setcounter{equation}{0}
\renewcommand{\theequation}{\arabic{equation}C}

\section{Proof of Lemma \ref{luna11}}\label{proofluna11}
 Equation \eqref{luna8b} along with the equality $\mathbf V^{-1}=\nabla(\boldsymbol\xi
 +\mathbf u)$ imply
 $$
 \boldsymbol\omega^\top\nabla (\boldsymbol\xi+\mathbf u)=\nabla_y H_m(\boldsymbol\xi+\mathbf u, \mathbf v, \mathbf w),
 $$
 which yields the first relation in \eqref{luna12}. Next, recalling formula \eqref{luna9} for $\boldsymbol\Lambda$  , we obtain
 $$
 \Big(\frac{\partial H_m}{\partial \mathbf y}(\boldsymbol\xi+\mathbf u, \mathbf v, \mathbf w) \boldsymbol\Lambda\Big)_q=-\frac{\partial H_m}{\partial  y_i}(\boldsymbol\xi+\mathbf u, \mathbf v, \mathbf w)V_{ij}\frac{\partial w_k}{\partial\xi_j} J_{kp}W_{pq},
 $$
 We also have
 $$
  \Big(\frac{\partial H_m}{\partial \mathbf z}(\boldsymbol\xi+\mathbf u, \mathbf v, \mathbf w) W\Big)_q
 =\frac{\partial H_m}{\partial  z_p}(\boldsymbol\xi+\mathbf u, \mathbf v, \mathbf w)W_{pq}
 $$
 From this and \eqref{luna8c} we obtain
 $$
 \Big\{-\frac{\partial H_m}{\partial  y_i}(\boldsymbol\xi+\mathbf u, \mathbf v, \mathbf w)V_{ij}\frac{\partial w_k}{\partial\xi_j} J_{kp}+\frac{\partial H_m}{\partial  z_p}(\boldsymbol\xi+\mathbf u, \mathbf v, \mathbf w)\Big\}W_{pq}=0.
 $$
Notice that in view of \eqref{luna8b}, we have $ {\partial
H_m}/{\partial  y_i}V_{ij}=\omega_j$. Thus we get
$$
 \Big\{-\boldsymbol\partial \mathbf w^\top\mathbf J
 +\frac{\partial H_m}{\partial \mathbf z}\Big\} \mathbf W=0.
 $$
 which yields the third inequality in \eqref{luna12}.
 In order to derive the second equality in \eqref{luna12} we differentiate both sides
 of \eqref{luna8a} with respect to $x_i$ to obtain
\begin{equation}\label{luna13}
\frac{\partial H_m}{\partial x_i}(\xi+\mathbf u, \mathbf v,\mathbf
w)+\frac{\partial  H_m}{\partial y_n}(\xi+\mathbf u, \mathbf
v,\mathbf w)\frac{\partial v_n}{\partial x_i}+\frac{\partial
H_m}{\partial z_p}(\xi+\mathbf u,\mathbf v,\mathbf w)\frac{\partial
w_p}{\partial x_i}=0.
\end{equation}
Since $\mathbf V=(\mathbf I+\mathbf u')^{-\top}$, we have $\nabla_x=
\mathbf V\nabla_\xi$. It follows that
$$
\frac{\partial}{\partial y_n}H_m(\xi+\mathbf u, \mathbf v,\mathbf
w)\frac{\partial}{\partial x_n} =\frac{\partial}{\partial
y_n}H_m(\xi+\mathbf u, \mathbf v,\mathbf
w)V_{nj}\frac{\partial}{\partial \xi_j}.
$$
 From this and \eqref{luna8b} we obtain
 \begin{equation}\label{luna14}
 \frac{\partial}{\partial y_n}H_m(\xi+\mathbf u, \mathbf v,\mathbf w)\frac{\partial}{\partial x_n}=
 \omega_i\frac{\partial}{\partial \xi_i}=\boldsymbol\partial.
 \end{equation}
 On the other hand, relation \eqref{luna9} implies
 $\mathbf v=\boldsymbol\beta+\nabla_x\varphi_0-w_1\nabla_x w_2$. Thus we get
 $$
 \frac{\partial v_n}{\partial x_i}=\frac{\partial^2 \varphi_0}{\partial x_n\partial x_i}
 -w_1\frac{\partial^2 w_2}{\partial x_n\partial x_i}-\frac{\partial w_1}{\partial x_n}
 \frac{\partial w_2}{\partial x_i}.
 $$
 Combining this identity with \eqref{luna14}
 we arrive at
 \begin{equation*}\begin{split}
  \frac{\partial}{\partial y_n}H_m(\xi+\mathbf u, \mathbf v,\mathbf w)\frac{\partial v_n}{\partial x_i}
 = \boldsymbol\partial \frac{\partial \varphi_0}{\partial x_i} \\- w_2\boldsymbol\partial \frac{\partial w_1}{\partial x_i}-\boldsymbol\partial w_1 \frac{\partial w_2}{\partial x_i}=\\
  \boldsymbol\partial v_i-\boldsymbol\partial w_1\frac{\partial w_2}{\partial x_i}+
 \boldsymbol\partial w_2 \frac{\partial w_1}{\partial x_i}
 \end{split}\end{equation*}
 From this and the latter relation in \eqref{luna12} we obtain
 \begin{equation*}\begin{split}
 \frac{\partial H_m}{\partial y_n}(\xi+\mathbf u, \mathbf v,\mathbf w)\frac{\partial v_n}{\partial x_i}
 = \\
  \boldsymbol\partial v_i-\frac{\partial H_m}{\partial z_2} \frac{\partial w_2}{\partial x_i}-
 \frac{\partial H_m}{\partial z_1} \frac{\partial w_1}{\partial x_i}
 \end{split}\end{equation*}
 Substituting this equality into \eqref{luna13} we obtain the second equality in \eqref{luna12}.
 $\square$

\setcounter{equation}{0}
\renewcommand{\theequation}{\arabic{equation}D}

\section{Proof of Theorem \ref{structural3}}\label{astructural3}
Direct calculations lead to the following formulae
\begin{subequations}\label{luna21}
\begin{equation}\label{luna21a}\begin{split}
\delta\Phi_1=\frac{\partial H_m}{\partial\mathbf  x}
(\boldsymbol\xi+\mathbf u, \mathbf v, \mathbf w)\delta\mathbf u
+\\\frac{\partial H_m}{\partial\mathbf  y} (\boldsymbol\xi+\mathbf
u, \mathbf v, \mathbf w)\delta\mathbf v+ \frac{\partial
H_m}{\partial\mathbf  z} (\boldsymbol\xi+\mathbf u, \mathbf v,
\mathbf w)\delta\mathbf w,
\end{split}\end{equation}

\begin{equation}\label{luna21b}\begin{split}
\delta\Phi_2=\Big\{\,\frac{\partial H_m}{\partial\mathbf  y}
(\boldsymbol\xi+\mathbf u, \mathbf v, \mathbf
w)\delta\boldsymbol\Lambda+ \frac{\partial H_m}{\partial\mathbf  z}
(\boldsymbol\xi+\mathbf u, \mathbf v, \mathbf w)\delta\mathbf W\,\Big\}^\top+\\
\Big\{\, (\delta\mathbf u)^\top\big(\frac{\partial^2
H_m}{\partial\mathbf x\partial\mathbf  y} (\boldsymbol\xi+\mathbf u,
\mathbf v, \mathbf w)\boldsymbol\Lambda+ \frac{\partial^2
H_m}{\partial\mathbf x\partial\mathbf  z} (\boldsymbol\xi+\mathbf u,
\mathbf v, \mathbf w)\mathbf W
\big)\,\Big\}^\top+\\
\Big\{\,(\delta\mathbf v)^\top\big(\frac{\partial^2
H_m}{\partial\mathbf \mathbf  y^2} (\boldsymbol\xi+\mathbf u,
\mathbf v, \mathbf w)\boldsymbol\Lambda+ \frac{\partial^2
H_m}{\partial\mathbf y\partial\mathbf  z} (\boldsymbol\xi+\mathbf u,
\mathbf v, \mathbf w)\mathbf W \big)\,\Big\}^\top
+\\
\Big\{\,(\delta\mathbf w)^\top \big(\frac{\partial^2
H_m}{\partial\mathbf z\partial\mathbf  y} (\boldsymbol\xi+\mathbf u,
\mathbf v, \mathbf w)\boldsymbol\Lambda+ \frac{\partial^2
H_m}{\partial\mathbf  z^2} (\boldsymbol\xi+\mathbf u, \mathbf v,
\mathbf w)\mathbf W \big)\,\Big\}^\top,
\end{split}\end{equation}
\begin{equation}\label{luna21c}\begin{split}
\delta\Phi_3=\big\{\frac{\partial H_m}{\partial\mathbf  y}
(\boldsymbol\xi+\mathbf u, \mathbf v, \mathbf w)\delta\mathbf V
+(\delta\mathbf u)^\top\frac{\partial^2 H_m}{\partial\mathbf
x\partial\mathbf  y}\big
(\boldsymbol\xi+\mathbf u, \mathbf v, \mathbf w)\}^\top+\\
\big\{(\delta\mathbf v)^\top\frac{\partial^2 H_m}{\partial\mathbf
y^2} (\boldsymbol\xi+\mathbf u, \mathbf v, \mathbf w)\mathbf V+
(\delta\mathbf w)^\top\frac{\partial^2 H_m}{\partial\mathbf
z\partial\mathbf  y} (\boldsymbol\xi+\mathbf u, \mathbf v, \mathbf
w)\mathbf V\}^\top,
\end{split}\end{equation}
\begin{equation}\label{luna21d}\begin{split}
\delta\Phi_4=\frac{\partial H_m}{\partial y_i}
(\boldsymbol\xi+\mathbf u, \mathbf v, \mathbf w)\delta\mathbf
R_i+(\mathbf A+ \mathbf B+\mathbf C)+
(\mathbf A+\mathbf B+\mathbf C)^\top+\\
(\delta\mathbf u)^\top\mathbf P^{(x)}+ (\delta\mathbf v)^\top
\mathbf P^{(y)}+(\delta\mathbf w)^\top \mathbf P^{(z)},
\end{split}\end{equation}
\begin{equation}\label{luna21e}
    \delta\Phi_5=\frac{1}{(2\pi)^{n-1}}\int_{\mathbb T^{n-1}}\delta w_1\, d\boldsymbol\xi,
\end{equation}
where
\begin{equation}\label{luna21e}\begin{split}
\mathbf A=\delta\boldsymbol\Lambda^\top\frac{\partial^2
H_m}{\partial\mathbf  y^2}
(\boldsymbol\xi+\mathbf u, \mathbf v, \mathbf w)\boldsymbol\Lambda,\\
\mathbf B=\delta\mathbf W^\top\frac{\partial^2 H_m}{\partial\mathbf
z\partial\mathbf  y} (\boldsymbol\xi+\mathbf u, \mathbf v, \mathbf
w)\boldsymbol\Lambda +\mathbf W^\top\frac{\partial^2
H_m}{\partial\mathbf z\partial\mathbf  y}
(\boldsymbol\xi+\mathbf u, \mathbf v, \mathbf w)\delta\boldsymbol\Lambda,\\
\mathbf C=\delta\mathbf W^\top\frac{\partial^2 H_m}{\partial\mathbf
z^2} (\boldsymbol\xi+\mathbf u, \mathbf v, \mathbf w)\mathbf W,
\end{split}\end{equation}
\begin{equation}\label{luna21f}\begin{split}
\mathbf P^{(\tau)}=\boldsymbol\Lambda^\top \frac{\partial^3
H_m}{\partial \boldsymbol\tau\partial\mathbf  y^2}
(\boldsymbol\xi+\mathbf u, \mathbf v, \mathbf w)\boldsymbol\Lambda+
\boldsymbol\Lambda^\top \frac{\partial^3 H_m}{\partial
\boldsymbol\tau\partial\mathbf  y\partial\mathbf z}
(\boldsymbol\xi+\mathbf u, \mathbf v, \mathbf w)\mathbf W+
\\\mathbf W^\top\frac{\partial^3 H_m}{\partial \tau\partial\mathbf  z\partial\mathbf y}
(\boldsymbol\xi+\mathbf u, \mathbf v, \mathbf w)\boldsymbol\Lambda+
\mathbf W^\top\frac{\partial^3 H_m}{\partial
\boldsymbol\tau\partial\mathbf  z^2} (\boldsymbol\xi+\mathbf u,
\mathbf v, \mathbf w)\mathbf W,
\end{split}\end{equation}
\end{subequations}
where $\boldsymbol\tau=\mathbf x, \mathbf y, \mathbf z$. The
variations $\delta\mathbf V$, $\delta\mathbf W$,
$\delta\boldsymbol\Lambda$, and $\delta\mathbf v$ are defined in
terms of the components of $\delta\boldsymbol\varphi$ by relations
\eqref{gala2}.

In view of the Second Structure Theorem \eqref{sima12} the mapping
\eqref{luna27} has the inverse given by formulae \eqref{sima13}. In
particular, we have
\begin{gather}\nonumber
\delta\mathbf u\,=\, \chi_i\frac{\partial}{\partial \xi_i}
\,(\boldsymbol\xi +\mathbf u), \quad\delta\mathbf v\,=\, \mathbf V
\boldsymbol\mu+\boldsymbol\Lambda \boldsymbol\lambda+\chi_i\,
\frac{\partial}{\partial \xi_i} \mathbf v,\\\label{mina1}
\delta\mathbf w\,=\, \mathbf W \boldsymbol \lambda+\chi_i\,
\frac{\partial}{\partial \xi_i} \mathbf w.
\end{gather}
Substituting these equalities in representation \eqref{luna21a} for
$\delta\Phi_1$ we arrive at
\begin{equation}\begin{split}\label{mina2}
    \delta\Phi_1= \frac{\partial H_m}{\partial\mathbf y}\mathbf V\boldsymbol\mu+\Big(
\frac{\partial H_m}{\partial\mathbf x}\frac{\partial}{\partial
\xi_i}(\text{Id}+\mathbf u)+ \frac{\partial H_m}{\partial\mathbf
y}\frac{\partial\mathbf v}{\partial \xi_i}+ \frac{\partial
H_m}{\partial\mathbf z}\frac{\partial\mathbf w}{\partial
\xi_i}\Big)\chi_i +\\\Big(\frac{\partial H_m}{\partial\mathbf
y}\boldsymbol\Lambda+\frac{\partial H_m}{\partial\mathbf z}\mathbf
W\Big)\boldsymbol\lambda
\end{split}\end{equation}
Here the derivatives of $H_m$ are calculated at the point
$(\boldsymbol\xi+\mathbf u, \mathbf v, \mathbf w)$. It follows from
formula \eqref{luna8a} for $\Phi_1$ that
\begin{equation}\label{mina3}
    \frac{\partial H_m}{\partial\mathbf x}\frac{\partial}{\partial \xi_i}(\text{Id}+\mathbf u)+
\frac{\partial H_m}{\partial\mathbf y}\frac{\partial\mathbf
v}{\partial \xi_i}+ \frac{\partial H_m}{\partial\mathbf
z}\frac{\partial\mathbf w}{\partial
\xi_i}=\frac{\partial\Phi_1}{\partial\xi_i}.
\end{equation}
Next, formula \eqref{luna8c} for $\Phi_3$ and the equality
$\boldsymbol\mu=\nabla\psi_0+\delta\boldsymbol\beta$ imply
\begin{equation}\label{mina4}
 \frac{\partial H_m}{\partial\mathbf y}\mathbf V\boldsymbol\mu=\boldsymbol\omega^\top \boldsymbol\mu+
 \Big(\frac{\partial H_m}{\partial\mathbf y}\mathbf V-\boldsymbol\omega^\top \Big)\boldsymbol\mu=
\boldsymbol\partial\psi_0+\boldsymbol\omega^\top
\cdot\delta\boldsymbol\beta+ \Phi_3^\top\boldsymbol\mu.
\end{equation}
Formula \eqref{luna8b} for $\Phi_2$ yields
\begin{equation}\label{mina5}
    \Big(\frac{\partial H_m}{\partial\mathbf y}\boldsymbol\Lambda+\frac{\partial H_m}{\partial\mathbf z}\mathbf W\Big)\lambda=\Phi_2^\top\boldsymbol\lambda.
\end{equation}
Substituting \eqref{mina3}-\eqref{mina5} into \eqref{mina2} we
obtain the desired representation \eqref{luna26a} for
$\delta\Phi_1$. Let us turn to the representation for
$\delta\Phi_2$. Equalities \eqref{sima13d}, \eqref{sima13f} in the
second structural theorem imply
\begin{gather}\nonumber
\delta\boldsymbol\Lambda\,=\, \mathbf V\nabla_\xi(\mathbf
J\boldsymbol \lambda)+\chi_i\, \frac{\partial}{\partial \xi_i}
\boldsymbol \Lambda+\boldsymbol\lambda^\top\mathbf
R+\boldsymbol\Lambda\boldsymbol\Gamma,\\\label{mina6} \delta\mathbf
W\,=\, \mathbf W \boldsymbol \Gamma+\chi_i\,
\frac{\partial}{\partial \xi_i} \mathbf W.
\end{gather}
Substituting these relations along with equalities \eqref{mina1}
into \eqref{luna21b} we obtain
\begin{gather}\nonumber
\delta\Phi_2^\top=\frac{\partial H_m}{\partial\mathbf y}\mathbf
V\nabla_\xi(\mathbf J\boldsymbol \lambda) +\frac{\partial
H_m}{\partial \mathbf y}(\boldsymbol\lambda^\top\mathbf
R)+\\\nonumber \chi_i\frac{\partial H_m}{\partial\mathbf y}\,
\frac{\partial  \boldsymbol \Lambda}{\partial \xi_i}+ \frac{\partial
H_m}{\partial\mathbf y} \boldsymbol\Lambda\boldsymbol\Gamma
+\chi_i\frac{\partial H_m}{\partial\mathbf z}\frac{\partial\mathbf
W}{\partial \xi_i} +\frac{\partial H_m}{\partial\mathbf z}\mathbf
W\boldsymbol\Gamma+\\\nonumber
\chi_i\frac{\partial}{\partial\xi_i}(\boldsymbol\xi+\mathbf
u)^\top\Big(\frac{\partial^2H_m}{\partial\boldsymbol
x\partial\boldsymbol y}\boldsymbol\Lambda+
\frac{\partial^2H_m}{\partial\boldsymbol x\partial\boldsymbol
z}\mathbf W\Big)+\\\nonumber \chi_i\frac{\partial \mathbf
v^\top}{\partial\xi_i}\Big(\frac{\partial^2H_m}{ \partial\boldsymbol
y^2}\boldsymbol\Lambda+ \frac{\partial^2H_m}{\partial\boldsymbol
y\partial\boldsymbol z}\mathbf W\Big)+\chi_i\frac{\partial\mathbf
w^\top}{\partial\xi_i}\Big(\frac{\partial^2H_m}{\partial\boldsymbol
z\partial\boldsymbol y}\boldsymbol\Lambda+
\frac{\partial^2H_m}{\partial\boldsymbol z^2}\mathbf
W\Big)+\\\nonumber
\boldsymbol\lambda^\top\boldsymbol\Lambda^\top\Big(\frac{\partial^2H_m}{
\partial\boldsymbol y^2}\boldsymbol\Lambda+
\frac{\partial^2H_m}{\partial\boldsymbol y\partial\boldsymbol
z}\mathbf W\Big)+\boldsymbol\lambda^\top \mathbf
W^\top\Big(\frac{\partial^2H_m}{\partial\boldsymbol
z\partial\boldsymbol y}\boldsymbol\Lambda+
\frac{\partial^2H_m}{\partial\boldsymbol z^2}\mathbf
W\Big)+\\\label{mina7} \boldsymbol\mu^\top\mathbf
V^\top\Big(\frac{\partial^2H_m}{ \partial\boldsymbol
y^2}\boldsymbol\Lambda+ \frac{\partial^2H_m}{\partial\boldsymbol
y\partial\boldsymbol z}\mathbf W\Big).
\end{gather}
Since $\partial H_m/\partial \mathbf y$ is the row vector with the
components $\partial H_m/\partial y_i$, we have
$$
\frac{\partial H_m}{\partial \mathbf
y}(\boldsymbol\lambda^\top\mathbf R)=
\boldsymbol\lambda^\top(\frac{\partial H_m}{\partial  y_i}\mathbf
R_i).
$$
Recall that the matrices $\mathbf R_i$and  $\Phi_4$ are symmetric. From this and formula \eqref{luna8d}, we conclude that
\begin{gather}\nonumber
\frac{\partial H_m}{\partial \mathbf
y}(\boldsymbol\lambda^\top\mathbf R)+
\boldsymbol\lambda^\top\boldsymbol\Lambda^\top\Big(\frac{\partial^2H_m}{
\partial\boldsymbol y^2}\boldsymbol\Lambda+
\frac{\partial^2H_m}{\partial\boldsymbol y\partial\boldsymbol
z}\mathbf W\Big)
\\\label{mina8}+\boldsymbol\lambda^\top
\mathbf W^\top\Big(\frac{\partial^2H_m}{\partial\boldsymbol
z\partial\boldsymbol y}\boldsymbol\Lambda+
\frac{\partial^2H_m}{\partial\boldsymbol z^2}\mathbf
W\Big)=\boldsymbol\lambda^\top \Phi_4^\top
\end{gather}
Next, it follows from \eqref{luna8b} that
\begin{gather}\label{mina9}
\frac{\partial H_m}{\partial\mathbf y}\, \frac{\partial  \boldsymbol
\Lambda}{\partial \xi_i} +\frac{\partial H_m}{\partial\mathbf
z}\frac{\partial\mathbf W}{\partial \xi_i} +
\frac{\partial}{\partial\xi_i}(\boldsymbol\xi+\mathbf
u)^\top\Big(\frac{\partial^2H_m}{\partial\boldsymbol
x\partial\boldsymbol y}\boldsymbol\Lambda+
\frac{\partial^2H_m}{\partial\boldsymbol x\partial\boldsymbol
z}\mathbf W\Big)+\\\nonumber \frac{\partial \mathbf
v^\top}{\partial\xi_i}\Big(\frac{\partial^2H_m}{ \partial\boldsymbol
y^2}\boldsymbol\Lambda+ \frac{\partial^2H_m}{\partial\boldsymbol
y\partial\boldsymbol z}\mathbf W\Big)+\frac{\partial\mathbf
w^\top}{\partial\xi_i}\Big(\frac{\partial^2H_m}{\partial\boldsymbol
z\partial\boldsymbol y}\boldsymbol\Lambda+
\frac{\partial^2H_m}{\partial\boldsymbol z^2}\mathbf
W\Big)=\frac{\partial\Phi_2^\top}{\partial\xi_i}.
\end{gather}
We also have
\begin{equation}\label{mina10}
\frac{\partial H_m}{\partial\mathbf
y}\boldsymbol\Lambda\boldsymbol\Gamma+\frac{\partial
H_m}{\partial\mathbf z}\mathbf W\boldsymbol\Gamma
=\Phi_2^\top\boldsymbol\Gamma.
\end{equation}
Substituting \eqref{mina8}-\eqref{mina10} into \eqref{mina7} we
obtain
\begin{equation}\label{mina11}\begin{split}
    \delta\Phi_2^\top=\frac{\partial H_m}{\partial\mathbf y}\mathbf V\nabla_\xi(\mathbf J\boldsymbol \lambda)+\boldsymbol\mu^\top\mathbf V^\top\Big(\frac{\partial^2H_m}{ \partial\boldsymbol y^2}\boldsymbol\Lambda+
\frac{\partial^2H_m}{\partial\boldsymbol y\partial\boldsymbol z}\mathbf W\Big)+\\
     \chi_i\frac{\partial\Phi_2^\top}{\partial\xi_i}+\Phi_2^\top\boldsymbol\Gamma+
    \boldsymbol\lambda^\top \Phi_4^\top
\end{split}\end{equation}
Expression  \eqref{luna29} for the matrix $\mathbf T$ implies
\begin{equation}\label{mina12}
    \boldsymbol\mu^\top\mathbf V^\top\Big(\frac{\partial^2H_m}{ \partial\boldsymbol y^2}\boldsymbol\Lambda+
\frac{\partial^2H_m}{\partial\boldsymbol y\partial\boldsymbol
z}\mathbf W\Big)=\boldsymbol\mu^\top \boldsymbol T^\top.
\end{equation}
On the other hand, formula \eqref{luna8c} for the operator $\Phi_3$
yields
\begin{equation*}\begin{split}
   \frac{\partial H_m}{\partial\mathbf y}\mathbf V\nabla_\xi(\mathbf J\boldsymbol \lambda)=
   \boldsymbol\omega^\top\nabla_\xi(\mathbf J\boldsymbol \lambda) +\Phi_3^\top \nabla_\xi(\mathbf J\boldsymbol \lambda)=\\(\boldsymbol \partial (\mathbf J\boldsymbol\lambda))^\top+\Big(
  (\mathbf J\boldsymbol \lambda)'_\xi \Phi_3\Big)^\top =(\mathbf J\boldsymbol \partial \boldsymbol\lambda)^\top+\Big(
  (\mathbf J\boldsymbol \lambda)'_\xi \Phi_3\Big)^\top.
\end{split}\end{equation*}
Substituting this identity along with \eqref{mina12} into
\eqref{mina11} we arrive at desired representation \eqref{luna26b}
for $\delta\Phi_2$.

Our  next task is to prove identity
\eqref{luna26c} for $\delta \Phi_3$. It follows from identity
\eqref{sima13e} in the Second Structural Theorem that
\begin{equation*}
    \delta\mathbf V\,=\, -\mathbf V\nabla_\xi\boldsymbol \chi+\chi_i\,
\frac{\partial \mathbf V}{\partial \xi_i}
\end{equation*}
Substituting this identity along with \eqref{mina1} into formula
\eqref{luna21c} for $\delta\Phi_3$ we obtain
\begin{gather}\nonumber
\delta\Phi_3^\top =-\frac{\partial H_m}{\partial \mathbf y}\mathbf
V\nabla_\xi\boldsymbol\chi
+\chi_i\frac{\partial H_m}{\partial \mathbf y}\frac{\partial \mathbf V}{\partial \xi_i} +\\
\nonumber \chi_i\frac{\partial}{\partial \xi_i}(\boldsymbol
\xi+\mathbf u)^\top\frac{\partial^2 H_m}{\partial\mathbf x
\partial\mathbf y}+\chi_i\frac{\partial \mathbf v^\top}{\partial
\xi_i}\frac{\partial^2 H_m}{\partial\mathbf
y^2}+\chi_i\frac{\partial \mathbf w^\top}{\partial
\xi_i}\frac{\partial^2 H_m}{\partial\mathbf z\partial\mathbf
y}+\\\label{mina13} \boldsymbol\mu^\top \mathbf
V^\top\frac{\partial^2 H_m}{\partial\mathbf y^2}\mathbf V+
\boldsymbol\lambda^\top \boldsymbol\Lambda^\top\frac{\partial^2
H_m}{\partial\mathbf y^2} \mathbf V+\boldsymbol\lambda^\top\mathbf
W^\top\frac{\partial^2 H_m}{\partial\mathbf z\partial\mathbf
y}\mathbf V.
\end{gather}
It follows from formula \eqref{luna8c} for $\Phi_3$ that
\begin{equation}\label{mina14}\begin{split}
-\frac{\partial H_m}{\partial \mathbf y}\mathbf
V\nabla_\xi\boldsymbol\chi
=-\boldsymbol\omega^\top\nabla_\xi\boldsymbol\chi-\Phi_3^\top\nabla_\xi\boldsymbol\chi
=-(\boldsymbol \partial
\chi)^\top-(\boldsymbol\chi'_\xi\Phi_3)^\top.
\end{split}\end{equation}
Next, we have
\begin{gather}\nonumber
\chi_i\frac{\partial H_m}{\partial \mathbf y}\frac{\partial \mathbf
V}{\partial \xi_i} + \nonumber \chi_i\frac{\partial}{\partial
\xi_i}(\boldsymbol \xi+\mathbf u)^\top\frac{\partial^2
H_m}{\partial\mathbf x \partial\mathbf
y}+\\\label{mina15}\chi_i\frac{\partial \mathbf v^\top}{\partial
\xi_i}\frac{\partial^2 H_m}{\partial\mathbf
y^2}+\chi_i\frac{\partial \mathbf w^\top}{\partial
\xi_i}\frac{\partial^2 H_m}{\partial\mathbf z\partial\mathbf
y}=\chi_i\frac{\partial\Phi_3^\top}{\partial \xi_i}
\end{gather}
On the other hand, formulae \eqref{luna29} for $\mathbf S$ and
$\mathbf T$ yield
\begin{equation}\label{mina16}
\boldsymbol\mu^\top \mathbf V^\top\frac{\partial^2
H_m}{\partial\mathbf y^2}\mathbf V+ \boldsymbol\lambda^\top
\boldsymbol\Lambda^\top\frac{\partial^2 H_m}{\partial\mathbf y^2}
\mathbf V+\boldsymbol\lambda^\top\mathbf W^\top\frac{\partial^2
H_m}{\partial\mathbf z\partial\mathbf y}\mathbf V=(\mathbf S
\boldsymbol\mu)^\top +(\mathbf T^\top\boldsymbol\lambda)^\top.
\end{equation}
Substituting \eqref{mina14}-\eqref{mina16} into \eqref{mina13} we
arrive at the desired identity \eqref{luna21c}.

It remains to prove representation \eqref{luna21d} for
$\delta\Phi_4$. Recall that this relation
  includes the matrices $\mathbf A$, $\mathbf B$, $\mathbf C$, given by \eqref{luna21e}, and the matrices $\mathbf P^{(\tau)}$, $\boldsymbol\tau=\mathbf x, \mathbf y, \mathbf z$
given by \eqref{luna21f}. Substituting representations
\eqref{mina6} into \eqref{luna21e} we obtain
\begin{equation}\label{mina17}
    \mathbf A=\boldsymbol\Lambda^\top\frac{\partial^2 H_m}{\partial \mathbf y^2}\boldsymbol\Lambda+\chi_i\frac{\partial \boldsymbol\Lambda^\top}{\partial\xi_i}\frac{\partial^2 H_m}{\partial \mathbf y^2}\boldsymbol\Lambda+
    (\mathbf J\boldsymbol\lambda)'_\xi\mathbf V^\top\frac{\partial^2 H_m}{\partial \mathbf y^2}\boldsymbol\Lambda+(\boldsymbol\lambda^\top \mathbf R)^\top\frac{\partial^2 H_m}{\partial \mathbf y^2}\boldsymbol\Lambda,
\end{equation}
\begin{gather}\label{mina18}
 \mathbf B=\chi_i\frac{\partial \mathbf W^\top}{\partial\xi_i}\frac{\partial^2 H_m}{\partial \mathbf z\mathbf y}\boldsymbol\Lambda+\boldsymbol \Gamma^\top\frac{\partial^2 H_m}{\partial \mathbf z\mathbf y}\boldsymbol\Lambda+\\\nonumber
\mathbf W^\top \frac{\partial^2 H_m}{\partial \mathbf z\mathbf
y}\mathbf V\nabla_\xi(\mathbf J\boldsymbol\lambda) +\chi_i\mathbf
W^\top \frac{\partial^2 H_m}{\partial \mathbf z\mathbf
y}\frac{\partial\boldsymbol\Lambda}{\partial\xi_i} +\mathbf W^\top
\frac{\partial^2 H_m}{\partial \mathbf z\mathbf
y}(\boldsymbol\lambda^\top\mathbf R)+ \mathbf W^\top
\frac{\partial^2 H_m}{\partial \mathbf z\mathbf
y}\boldsymbol\Lambda\boldsymbol\Gamma,
\end{gather}
\begin{equation}\label{mina19}
    \mathbf C=\chi_i\frac{\partial \mathbf W^\top}{\partial\xi_i}\frac{\partial^2 H_m}{\partial \mathbf z\mathbf y}\mathbf W+\boldsymbol\Gamma^\top W^\top\frac{\partial^2 H_m}{\partial \mathbf z\mathbf y}\mathbf W
\end{equation}
Next, expression \eqref{luna21f} for $\mathbf P^{(\tau)}$ and
representations \eqref{mina1} implies
\begin{gather}\nonumber
\delta\mathbf u^\top \mathbf P^{(x)}+\delta\mathbf v^\top \mathbf
P^{(y)}+ \delta\mathbf w^\top \mathbf
P^{(z)}=\chi_i\Big(\frac{\partial\mathbf u^\top}{\partial\xi_i}
\mathbf P^{(x)}+\frac{\partial\mathbf v^\top}{\partial\xi_i} \mathbf
P^{(y)}+
\frac{\partial\mathbf w^\top}{\partial\xi_i} \mathbf P^{(z)}\Big)+\\
\label{mina20} \boldsymbol\mu^\top \mathbf
P^{(y)}+\boldsymbol\lambda^\top(\boldsymbol\Lambda^\top\mathbf
P^{(y)}+ \mathbf W^\top\mathbf P^{(z)})
\end{gather}
\end{appendix}
Substituting \eqref{mina17} \eqref{mina20} into \eqref{luna21d} and
recalling the representation
$$
\delta\mathbf R_i=\frac{\partial}{\partial \xi_i}(\mathbf
J\boldsymbol \Gamma)+\mathbf R_i \boldsymbol\Gamma+(\mathbf R_i
\boldsymbol\Gamma)^\top+\chi_i\, \frac{\partial}{\partial \xi_i}
\mathbf R_i
$$
in the second structural Theorem \ref{sima12} we arrive at the
identity
\begin{gather}\nonumber
\delta\Phi_4=\frac{\partial H_m}{\partial y_i}V_{ik}\frac{\partial}{\partial \xi_k}(\mathbf J\boldsymbol\Gamma)+\mathbf Q\boldsymbol\Gamma+\boldsymbol\Gamma^\top\mathbf Q^\top+\\
\label{mina21} \chi_k \mathbf N_k+\mathbf U+\mathbf U^\top +\mathbf
K+\mathbf K^\top +\boldsymbol\Sigma.
\end{gather}
Here
\begin{gather}\nonumber
\mathbf Q=\frac{\partial H_m}{\partial y_i}\mathbf R_i+
\boldsymbol\Lambda^\top \frac{\partial^2 H_m}{\partial \mathbf y^2}
\boldsymbol\Lambda+\\\label{mina22} \mathbf W^\top \frac{\partial^2
H_m}{\partial\mathbf z\partial \mathbf y} \boldsymbol\Lambda
+\Big(\mathbf W^\top \frac{\partial^2 H_m}{\partial\mathbf z\partial
\mathbf y} \boldsymbol\Lambda\Big)^\top+\mathbf W^\top
\frac{\partial^2 H_m}{\partial\mathbf z^2} \mathbf
W=\boldsymbol\Omega+\Phi_4,
\end{gather}
\begin{gather}\nonumber
\mathbf N_k=\frac{\partial H_m}{\partial y_i}\frac{\partial\mathbf
R_i}{\partial\xi_k} +\boldsymbol\Lambda^\top \frac{\partial^2
H_m}{\partial \mathbf y^2}
\frac{\partial\boldsymbol\Lambda}{\partial\xi_k}+\frac{\partial\boldsymbol\Lambda^\top}{\partial\xi_k}
\frac{\partial^2 H_m}{\partial \mathbf y^2}
\boldsymbol\Lambda+\\\nonumber \mathbf W^\top \frac{\partial^2
H_m}{\mathbf z\partial \mathbf y}
\frac{\partial\boldsymbol\Lambda}{\partial\xi_k}+\frac{\partial\mathbf
W^\top}{\partial\xi_k} \frac{\partial^2 H_m}{\partial \mathbf
z\partial \mathbf y}
\boldsymbol\Lambda+\\
\Big(\mathbf W^\top \frac{\partial^2 H_m}{\partial\mathbf
z\partial\mathbf y}
\frac{\partial\boldsymbol\Lambda}{\partial\xi_k}+\frac{\partial\mathbf
W^\top}{\partial\xi_k} \frac{\partial^2 H_m}{\partial \mathbf
z\mathbf y} \boldsymbol\Lambda\Big)^\top+\\\nonumber \mathbf W^\top
\frac{\partial^2 H_m}{\partial\mathbf z^2} \frac{\partial\mathbf
W}{\partial\xi_k}+\frac{\partial\mathbf W^\top}{\partial\xi_k}
\frac{\partial^2 H_m}{\partial \mathbf z^2} \mathbf
W+\\\label{mina23} +\frac{\partial\mathbf u^\top}{\partial\xi_i}
\mathbf P^{(x)}+\frac{\partial\mathbf v^\top}{\partial\xi_i} \mathbf
P^{(y)}+ \frac{\partial\mathbf w^\top}{\partial\xi_i} \mathbf
P^{(z)}=\frac{\partial\Phi_4}{\partial\xi_k}.
\end{gather}
Next notice that in view of \eqref{luna8c},
$$
\frac{\partial}{\partial \xi_i}(\mathbf J\boldsymbol
\Gamma)=\boldsymbol\partial (\mathbf
J\boldsymbol\Gamma)+\Phi_{3,k}\frac{\partial}{\partial\xi_k}(\mathbf
J\boldsymbol\Gamma).
$$
Substituting this relation along with \eqref{mina22}, \eqref{mina23}
into \eqref{mina21} we arrive at the identity
\begin{gather}\nonumber
\delta\Phi_4=\boldsymbol\partial (\mathbf J\boldsymbol\Gamma)
+\boldsymbol\Omega\boldsymbol\Gamma+\boldsymbol\Gamma^\top\boldsymbol\Omega^\top
+\\\nonumber
\Phi_4\boldsymbol\Gamma+\boldsymbol\Gamma^\top\Phi_4^\top+\chi_k\frac{\partial\Phi_4}{\partial\xi_k}+
\Phi_{3,k}\frac{\partial}{\partial\xi_k}(\mathbf J\boldsymbol\Gamma)+\\
\label{mina24} \mathbf U+\mathbf U^\top +\mathbf K+\mathbf K^\top
+\boldsymbol\Sigma.
\end{gather}
Notice that $\mathbf U$ is the matrix-valued linear form of
$\partial\lambda_i/ \partial\xi_j$. Hence it admits the
representation
\begin{equation}\label{mina25}
  \mathbf U+\mathbf U^\top=\frac{\partial\lambda_i}{\partial\xi_j}\mathbf U_{ij},
\end{equation}
where $\mathbf U_{ij}$ are symmetric matrix-valued functions. Next,
$\mathbf K$ and $\boldsymbol \Sigma$ are matrix-valued linear form
of  $\lambda_i$ and $\mu_j$. Hence there are  symmetric matrix-valued
functions $\mathbf E_i$ and $\mathbf K_i$ such that
\begin{equation}\label{mina26}
    \mathbf K+\mathbf K^\top +\boldsymbol\Sigma=\mu_i\boldsymbol E_i+\lambda_i\mathbf K_i.
\end{equation}
Substituting \eqref{mina25} and \eqref{mina26} into \eqref{mina24}
we obtain the desired representation \eqref{luna28d}.

\setcounter{equation}{0}
\renewcommand{\theequation}{\arabic{equation}E}

\section{Proof of Theorems \ref{lita210} and \ref{figa5}}
\subsection{Proof of Theorem \ref{lita210}}\label{prooflita205}

This section is devoted to the proof of solvability of the following
problem
\begin{subequations}\label{nisa4}
\begin{gather}\label{nisa4a}
 \boldsymbol\partial \psi_0+\delta q+ \delta p\cdot ( w_1-\alpha)+
    \frac{1}{2} \delta M (w_1-\alpha)^2=F_1,\\
    \label{lita4b}
\mathbf J\boldsymbol\partial\boldsymbol\lambda +\mathbf
T\boldsymbol\mu
    +\mathbf W^\top\delta\mathbf p+
    \delta M(w_1-\alpha)\mathbf W^\top\mathbf e_1 =F_2,\quad
    \boldsymbol\mu =\nabla \psi_0+\delta\boldsymbol\beta\\
\label{nisa4c}
 -\boldsymbol\partial \boldsymbol\chi+\mathbf S\boldsymbol\mu+
 \mathbf T^\top \boldsymbol\lambda =F_3,
\\\label{nisa4d}
\boldsymbol\partial(\mathbf J\boldsymbol\Gamma)
+\boldsymbol\Omega\boldsymbol\Gamma+
(\boldsymbol\Omega\boldsymbol\Gamma)^\top +\mathbf U_{ij}
\frac{\partial \lambda_i}{\partial\xi_j}+ \mu_i\mathbf
E_i+\lambda_i\mathbf K_i+   \\\nonumber +\mathbf W^\top\delta\mathbf
M\mathbf W=F_4,\quad \Gamma_{11}=-\Gamma_{22},
\end{gather}
\begin{gather}\label{nisa4g}
 \frac{1}{(2\pi)^{n-1}}\int_{\mathbb T^{n-1}}\big[ \psi_0-\delta \boldsymbol
\beta\cdot\mathbf u+
w_2\big(\mathbf W\boldsymbol\lambda\big)_1\big]\, d \xi=f_1\\
\label{nisa4h} \frac{1}{(2\pi)^{n-1}} \int_{\mathbb T^{n-1}}\mathbf V^{-\top}\boldsymbol\chi\,
d\xi=f_3,\quad  \frac{1}{(2\pi)^{n-1}}\int_{\mathbb T^{n-1}}(\mathbf
W\boldsymbol\Gamma)_{12}\, d\xi=f_4\\
\label{nisa4f}
 \frac{1}{(2\pi)^{n-1}}\int_{\mathbb T^{n-1}}\Big((\mathbf W\boldsymbol\lambda)_1+
    \boldsymbol\chi\cdot \nabla w_1\Big) \, d
    \boldsymbol\xi=f_5.
\end{gather}
\end{subequations}
The following theorem constitutes the existence and uniqueness of
solutions to problem \eqref{nisa4}
\begin{theorem}\label{nisa210} Let a fixed $\sigma\in [1/2,1]$, $d\geq 2$,
and the matrix $\mathbf K_0=\mathbf S_0-\mathbf t_0\otimes \mathbf
t_0$ given by \eqref{sandt}, satisfies the condition $\text{~det~}
\mathbf K_0\neq 0$.  Furthermore, assume that
\begin{equation}\label{nisa6xx}
\|w_1-\alpha\|_{\sigma,0}+ \|\mathbf W-\mathbf I\|_{\sigma_0}+ \|\mathbf u'-\mathbf I\|_{\sigma,0}\leq c r
\end{equation}
and
\begin{equation}\label{nisa6x}
    \|\mathbf S-\mathbf S_0\|_{\sigma,0}+ \|\mathbf T-\mathbf T_0\|_{\sigma,0}\leq c(r+|\varepsilon|),
\end{equation}
\begin{equation}\label{nisa7x}
    \|\mathbf U_{ij}\|_{\sigma,0}+
\|\mathbf
E_i\|_{\sigma,0}+\|\mathbf K_i\|_{\sigma,0} \leq c.
\end{equation}
Then there are $\varepsilon_0>0$ and  $r_0>0$ with
the following properties. For every
$$
r\leq r_0 \quad
|\varepsilon|\leq \varepsilon_0, \quad ,
$$
$$
0\leq \sigma_0<\sigma_1<\sigma,\quad 1/4\geq\sigma_1, \sigma\in [1/4,1),
$$
and for all
$$
\mathbf F=(F_1, F_2, F_3, F_4,f_1,f_3,f_4,
f_5)\in \mathcal A_{\sigma_1,0}\times \mathcal A_{\sigma_1,0}^2
\times \mathcal A_{\sigma_1,0}^{n-1}\times \mathcal A_{\sigma_1,0}^4\times\mathbb C\times \mathbb
C^{n-1}\mathbb C\times \mathbb C,
$$
with $F_4=F_4^\top$,
 problem \eqref{nisa4} has a unique solution $$
(\psi_0, \boldsymbol\lambda, \boldsymbol\chi, \boldsymbol\Gamma,
\delta\boldsymbol \beta, q,  p, \delta M)\in
\mathcal A_{\sigma_0, 0}\times \mathcal A_{\sigma_0,0}^{2}\times
\mathcal A_{\sigma,0}^{n-1}\times \mathcal A_{\sigma_0,0}^4\times
\mathbb C^{n-1}\times \mathbb C^{3}.
$$
This solution admits the estimate
\begin{equation}\label{lita213}
\|(\psi_0, \boldsymbol\lambda, \boldsymbol\chi,
\boldsymbol\Gamma)\|_{\sigma_0,0}+ |(\delta\boldsymbol \beta,
q,  p, \delta M)| \leq
{c}{(\sigma_1-\sigma_0)^{-8n-12}} \|\mathbf F\|_{\sigma_1,0}
\end{equation}
where  the constant $c$ is independent of  $\varepsilon_0$, $r_0$,
and $\sigma_i$.
\end{theorem}
Notice that identity $\mathbf V=(\mathbf I+\mathbf u')^{-\top}$ implies  that
$$
\|\mathbf V^{-\top}-\mathbf I\|_{\sigma_0}\leq cr.
$$
The rest of the section is devoted to the proof of this theorem. Our strategy
 is the following. First we prove the existence and
uniqueness of solutions to a triangle truncated problem.  Then we
prove the solvability of equations \eqref{lita4} using the
contraction mapping principle.
Our considerations are based on the following existence and uniqueness results for
model differential equations with constant coefficients
\paragraph{Equations with constant coefficients}
Recall the denotations
$$
\overline{\boldsymbol\mu}= \frac{1}{(2\pi)^{n-1}}\int_{\mathbb T^{n-1}}\boldsymbol\mu(\boldsymbol\xi)\, d\boldsymbol\xi,\quad
{\boldsymbol\mu}^*={\boldsymbol\mu}-\overline{\boldsymbol\mu}
$$
\begin{lemma}\label{nisa6} For every $g\in \mathcal A_{\sigma_1,0}$ with $\overline{g}=0$, the
equations
\begin{gather}\label{nisa7}
 \boldsymbol\partial \psi_0=g, \quad \overline{\psi}_0=0
\end{gather} have a unique solution
$\psi_0\in \mathcal A_{\sigma_0,0}$  such that
\begin{equation}\label{nisa8}
    \|\psi_0\|_{\sigma_0,0}\leq {c}{(\sigma_1-\sigma_0)^{-\tau}}
\| g\|_{\sigma_1,0}, \quad |\psi_0|_{s}\leq {c}
| g|_{s+\tau}.
\end{equation}
where $\tau=n+1$ and $0\leq \sigma_0 <\sigma_1$ are  arbitrary numbers.
\end{lemma}
\begin{proof}
Substituting the decomposition
$$
\psi_0(\xi)= \sum\limits_{\mathbf s\in
\mathbb Z^{n-1}\setminus\{0\}} \widehat{\psi_0}(\mathbf s) e^{i
\mathbf s\boldsymbol\xi},\quad \widehat{\psi_0}(\mathbf
s)=(2\pi)^{\frac{1-n}{2}} \int_{\mathbb T^{n-1}}\psi_0(\boldsymbol
\xi)e^{-i \mathbf s\boldsymbol\xi}\, d \boldsymbol\xi,
$$
into equation \eqref{lita6a} we can rewrite this equation into the
equivalent form
\begin{equation}\label{lita11}\begin{split}
\widehat{\psi_0}(\mathbf s)= (i\,\boldsymbol\omega\cdot \mathbf
s)^{-1} \,\,\widehat{g}(\mathbf s)\text{~~for~~} \mathbf s\in
\mathbb Z^{n-1}\setminus \{0\}, \quad \widehat{\psi_0}(0)=0
\end{split}\end{equation}
The diophantine condition implies
\begin{equation}\label{nisa10}
|(\boldsymbol\omega\cdot \mathbf s)^{-1}|\leq c |\mathbf s|^{-n},
\end{equation}
which gives
\begin{equation}\label{nisa9}
    |\widehat{\psi_0}(\mathbf s)|\leq (1+|\mathbf s|)^n |\widehat{g}(\mathbf s)|
\end{equation}
It is well-known that for every $\sigma\in [0,1)$ and $\sigma'\in
(0, \sigma)$,
 and for every measurable function $g:\mathbb T^{n-1}\to \mathbb C$, we have
$$
\sup_{\mathbf s\in \mathbb Z^{n-1}}e^{\sigma|\mathbf s|}|\widehat
g(\mathbf s)|\leq \|g\|_{\sigma,0}, \quad \|g\|_{\sigma',0}\leq
\frac{c}{\sigma-\sigma'} \sup_{\mathbf s\in \mathbb
Z^{n-1}}e^{\sigma|\mathbf s|}|\widehat g(\mathbf s)|,
$$
where $c$ is independent of $g$, $\sigma$, and $\sigma'$.
>From this and \eqref{nisa9} that
\begin{equation}\begin{split}
e^{(\sigma_0+\sigma_1)|\mathbf s|/2}|\widehat{\psi_0}(\mathbf s)| \leq
c(1+|\mathbf s|)^{n}
e^{(\sigma_0+\sigma_1)|\mathbf s|/2}||\widehat{g}(\mathbf s)|\leq\\
c(\sigma_1-\sigma_0)^{-n}e^{\sigma_1|\mathbf s|}||\widehat{g}(\mathbf
s)| \leq c(\sigma_1-\sigma_0)^{-\tau} \|g\|_{\sigma_1,0}.
\end{split}\end{equation}
It follows that
$$
\|\psi_0\|_{\sigma_0,0}\leq \frac{c}{\sigma_1-\sigma_0}\sup_{\mathbf
s\in \mathbb Z^{n-1}}e^{ (\sigma_0+\sigma_1)|\mathbf
s|/2}|\widehat{\psi_0}(\mathbf s)|\leq c(\sigma_1-\sigma_0)^{-\tau} \|g\|_{\sigma_1,0}.
$$
Thus we get the first estimate in \eqref{nisa8}. The second obviously follows from \eqref{nisa9} and the definition of the Sobolev space $H_r$.
\end{proof}

\begin{lemma}\label{lita21lemma} Let $0\leq \sigma_0<\sigma_1$. Let
 $\mathbf H\in \mathcal A_{\sigma_1,0}$ satisfies the condition $\overline{\mathbf H}=0$.
Then the equation
\begin{equation}\label{lita21}
\mathbf J \, \boldsymbol\partial{\boldsymbol\lambda}+\boldsymbol\Omega
{\boldsymbol\lambda}=\mathbf H
\end{equation}
has a unique  solution $\boldsymbol \lambda\in \mathcal
A_{\sigma_0,0}$,
 satisfying the condition $\overline{\boldsymbol\lambda}=0$. This solution admits the estimates
\begin{equation}\label{lita22}
    \|\boldsymbol\lambda\|_{\sigma_0,0}\leq {c}{(\sigma_1-\sigma_0)^{-2n-3}}
    \,\,\|\mathbf H\|_{\sigma_1,0}.
\end{equation}
\begin{equation}\label{lita22xx}
    |\boldsymbol\lambda|_{s}\leq {c}
    \,\,|\mathbf H|_{s+2n+3}.
\end{equation}

\end{lemma}
\begin{proof} Equation \eqref{lita21} is equivalent to the following system of the linear
algebraic equations for the Fourier coefficients
$\widehat{\boldsymbol\lambda}(\mathbf s)$, $\mathbf s\in \mathbb
Z^{n-1}$, of  $\boldsymbol\lambda$,
\begin{equation*}
    i(\boldsymbol\omega\cdot \mathbf s)\,\mathbf J \,\,\widehat{\boldsymbol\lambda}(\mathbf s)
         +\boldsymbol\Omega \,\,\widehat{\boldsymbol\lambda}(\mathbf s)=
         \widehat{\mathbf H}(\mathbf s),
         \quad \mathbf s\in \mathbb Z^{n-1}\setminus\{0\},\quad
         \widehat{\boldsymbol\lambda}(0)=0.
\end{equation*}
This equality along with   expressions \eqref{ee1.02} and \eqref{zabyli} for the
 matrices $\mathbf J$ and $\boldsymbol\Omega$  implies
\begin{equation*}\begin{split}
 \widehat{\lambda_1}(\mathbf s)=-\big(k+(\boldsymbol\omega\cdot \mathbf s)^2\big)^{-1}
 \,\,\big(\widehat{H_1}(\mathbf s)-i(\boldsymbol\omega\cdot \mathbf s)
  \widehat{H_2}(\mathbf s)\big),\,
 \widehat{\lambda_1}(\mathbf s)= \widehat{H_2}(\mathbf s)+
 i(\boldsymbol\omega\cdot \mathbf s)
 \widehat{\lambda_1}(\mathbf s).
 \end{split}\end{equation*}
Recall that $\text{Re~} k\geq 0$. From this and the diophantine
estimate \eqref{nisa10}
we conclude that
\begin{equation*}
 | \widehat{\boldsymbol\lambda}(\mathbf s)|\leq c |\mathbf s|^{2n+2}\,\,
 |\widehat{\mathbf H}(\mathbf s)|,\quad
\mathbf s\in\mathbb Z^{n-1}.
\end{equation*}
Thus we get
\begin{gather*}
\|\boldsymbol\lambda\|_{\sigma_0,0}\leq c(\sigma_1-\sigma_0)^{-1}
\sup_{\mathbf s \in\mathbb Z^{n-1}}
\,\,e^{(\sigma_0+\sigma_1)|\mathbf s|/2}\,\,\,|\widehat{\boldsymbol \lambda}(\mathbf s)|\\
c(\sigma_1-\sigma_0)^{-2n-3}\sup_{\mathbf s\in\mathbb Z^{n-1}}\,\,
e^{(\sigma_1|\mathbf s|}\,\,\,|\widehat{\mathbf H}(\mathbf s)|\leq
c(\sigma_1-\sigma_0)^{-2n-3}\,\,\|\mathbf H\|_{\sigma_1,0},
\end{gather*}
and the lemma follows.
\end{proof}
Now we consider the matrix differential equation
\begin{gather}\label{nisa12}
    \boldsymbol\partial(\mathbf J\boldsymbol\Gamma) +
    \boldsymbol\Omega\boldsymbol\Gamma+
(\boldsymbol\Omega\boldsymbol\Gamma)^\top  + \delta\mathbf
M=\mathbf G,\\\label{nisa14}\Gamma_{11}=-\Gamma_{22},\quad \int_{\mathbb
T^{n-1}}\Gamma_{12}\, d\xi=f,
\end{gather}
 Her $\delta\mathbf M= \delta M\text{diag}(1,0)$ is the unknown matrix.
\begin{lemma}\label{nisa15}
For all $(\mathbf G, f)\in A_{\sigma_1,0}\times \mathbb
C$ and for all $0\leq \sigma_0<\sigma_1$,
 equations \eqref{nisa12}-\eqref{nisa14} have a unique solution
$(\boldsymbol\Gamma, \delta M)\in \mathcal A_{\sigma_0,0}^4\times
\mathbb C$. This solution admits the estimate
\begin{equation}\label{nisa15}
    \|\boldsymbol\Gamma\|_{\sigma_0,0}+|\delta \mathbf M|\leq c(\sigma_1-
    \sigma_0)^{-3n-2}
   (\|\mathbf G\|_{\sigma_1,0}+| f|).
\end{equation}
\end{lemma}
\begin{proof}
Rewrite equations \eqref{nisa12}-\eqref{nisa14} in the form
\begin{equation}\label{nisa16}\begin{split}
    \boldsymbol\Omega \overline{\boldsymbol \Gamma}+\big(\boldsymbol\Omega
     \overline{\boldsymbol \Gamma}\big)^\top+\delta\mathbf M=\overline{\mathbf G},\\
   \overline{\Gamma}_{11}=- \overline{\Gamma}_{22}, \quad \overline{\Gamma}_{12}=f_4, \quad
  \delta\mathbf M=\text{diag~} (\delta M,0).
\end{split}\end{equation}
\begin{equation}\label{nisa17}\begin{split}
\boldsymbol\partial(\mathbf J\boldsymbol\Gamma^*)
+\boldsymbol\Omega\boldsymbol\Gamma^*+
(\boldsymbol\Omega\boldsymbol\Gamma^*)^\top =\mathbf G^*, \quad
\Gamma_{11}^*=-\Gamma_{22}^*.
\end{split}\end{equation}
The unique solution to equations \eqref{nisa16} is given by
\begin{equation*}
    \overline{\boldsymbol\Gamma}=\left(
\begin{array}{cc}
-\overline{G}_{22}/2, &f_4\\
k f_4+\overline{G}_{21}&\overline{G}_{22}/2
\end{array}
\right), \quad \delta \mathbf M =\left(
\begin{array}{cc}
\overline{G}_{11}-k\overline{G}_{22}, &0\\
0& 0
\end{array}\right).
\end{equation*}
It follows that
\begin{equation}\label{nisa18}
    |\overline{\boldsymbol\Gamma}| +|\delta M|\leq c |\mathbf G|.
\end{equation}
Next,  using the Fourier transform we can write equation
\eqref{nisa17} in the form of the matrix equation
$$
i(\boldsymbol \omega\cdot \boldsymbol s)\mathbf J
\widehat{{\boldsymbol \Gamma}^*}(\mathbf s) +\boldsymbol\Omega
\widehat{{\boldsymbol \Gamma}^*}(\mathbf s)+\big( \boldsymbol\Omega
\widehat{{\boldsymbol \Gamma}^*}(\mathbf s)\big)^\top=
\widehat{{\boldsymbol G}^*}(\mathbf s), \quad \widehat{
\Gamma_{11}^*}(\mathbf s)=-\widehat{ \Gamma_{22}^*}(\mathbf
s)\text{~for} \,\,\mathbf s\in \mathbb Z^2\setminus\{0\},
$$
Recall that
\begin{equation}\label{nisa19}
 \widehat{{\boldsymbol \Gamma}^*}(0)=\widehat{{\boldsymbol G}^*}(0)=0.
 \end{equation}
In it turn, this  equation is equivalent to the system
of linear algebraic equations
\begin{gather*}
i(\boldsymbol\omega\cdot \mathbf s)\,\widehat{
\Gamma_{21}^*}(\mathbf s)
+2k\,\widehat{ \Gamma_{22}^*}(\mathbf s)\,=\,\widehat{ G_{11}^*}(\mathbf s),\\
i(\boldsymbol\omega\cdot \mathbf s)\,\widehat{
\Gamma_{22}^*}(\mathbf s)\, +\,\widehat{ \Gamma_{21}^*}\,-\,k\,
\widehat{ \Gamma_{12}^*}(\mathbf s)\,=
\,\widehat{ G_{12}^*}(\mathbf s),\\
-i(\boldsymbol\omega\cdot \mathbf s)\,\widehat{
\Gamma_{12}^*}(\mathbf s)
\,+\,2\, \widehat{ \Gamma_{22}^*}(\mathbf s)\,=\,\widehat{ G_{22}^*}(\mathbf s),\\
\widehat{ \Gamma_{11}^*}(\mathbf s)=-\widehat{
\Gamma_{22}^*}(\mathbf s)\text{~for} \,\,\mathbf s\in \mathbb
Z^2\setminus\{0\}.
\end{gather*}
Straightforward calculations give
$$
 \widehat{ \Gamma_{12}^*}(\mathbf s)\,=\,\frac{2}{(\boldsymbol\omega\cdot \mathbf s)^2+4k}
 \Big(\frac{1}{i(\boldsymbol\omega\cdot \mathbf s)}\big(\widehat{ G_{11}^*}(\mathbf s)
 -k\widehat{ G_{22}^*}(\mathbf s)\big)+
i(\boldsymbol\omega\cdot \mathbf s) \widehat{ G_{22}^*}(\mathbf s)-
\widehat{ G_{12}^*}(\mathbf s)\Big)
$$
and
$$
\widehat{ \Gamma_{22}^*}(\mathbf s)=\frac{1}{2}\widehat{
G_{22}^*}(\mathbf s)+\frac{1}{2} i(\boldsymbol\omega\cdot \mathbf s)
\widehat{ \Gamma_{12}^*}(\mathbf s),
$$
$$
\widehat{ \Gamma_{21}^*}(\mathbf s)=\widehat{ G_{12}^*}(\mathbf s)+
k\widehat{ \Gamma_{12}^*}(\mathbf s) -i(\boldsymbol\omega\cdot
\mathbf s) \widehat{ \Gamma_{22}^*}(\mathbf s)
$$
$$
\widehat{ \Gamma_{11}^*}(\mathbf s)=-\widehat{
\Gamma_{22}^*}(\mathbf s).
$$
These identities   and the diophantine estimate \eqref{nisa10}
imply the estimate
$$
| \widehat{{\boldsymbol \Gamma}^*}(\mathbf s)|\leq c|\mathbf
s|^{3n+1}|
 \widehat{{\mathbf G}^*}(\mathbf s)|\text{~for}
\,\,\mathbf s\in \mathbb Z^2.
$$
Set $\varsigma=(\sigma_1-\sigma_0)/3$. Arguing as in the proof of
Lemmas  \ref{lita21lemma} and \ref{nisa15} we obtain
\begin{equation*}\begin{split}
   \| {\boldsymbol \Gamma}^*\|_{\sigma_0,0}\leq c \varsigma^{-1}\sup_{\mathbf s}
   e^{(\sigma_0+\varsigma)|\mathbf s|}|\widehat{{\boldsymbol \Gamma}^*}(\mathbf s)|\leq\\
 c\varsigma^{-1}\sup_{\mathbf s}|\mathbf s|^{3n+1}
   e^{(\sigma_0+\varsigma)|\mathbf s|}|\widehat{{G}^*}(\mathbf s)| \leq
c\varsigma^{-3n-2}\sup_{\mathbf s}
   e^{(\sigma_0+\varsigma)|\mathbf s|}|\widehat{{G}^*}(\mathbf s)|\\\leq
   c\varsigma^{-3n-2}\| {\mathbf G}^*\|_{\sigma_0+\varsigma,0}.
\end{split}\end{equation*}
Combining this estimate with \eqref{nisa18} we arrive at
\begin{equation}\label{nisa21}
    \| {\boldsymbol \Gamma}\|_{\sigma_0,0}+|\delta\mathbf M|\leq
    c\varsigma^{-3n-2}\| {\boldsymbol G}\|_{\sigma_0+\varsigma,0}.
\end{equation}
This completes the proof of the lemma.

\end{proof}

\paragraph{Truncated equations.}
Our next task is to prove the existence and uniqueness of analytic solutions to linear equations with variable coefficients which can be regarded as simplification of the main problem \eqref{nisa4}. First we consider the shortened version of this problem which can be formulated as follows.
It is necessary to find periodic functions $\psi_0$, $\boldsymbol \lambda$, $\boldsymbol\chi$ and parameters $p$, $q$ and $\beta$ satisfying the following equations.
\begin{subequations}\label{nisa23}
\begin{gather}\label{nisa23a}
 \boldsymbol\partial \psi_0+ q=F_1,\\
    \label{nisa23b}
\mathbf J\boldsymbol\partial\boldsymbol\lambda +\boldsymbol\Omega \boldsymbol\lambda+\mathbf
T\boldsymbol\mu
    + p \,\,\mathbf e_1=F_2,\quad \boldsymbol\mu=\delta\boldsymbol\beta+\nabla\psi_0\\
\label{nisa23c}
 -\boldsymbol\partial \boldsymbol\chi+
 \mathbf S\boldsymbol\mu+\mathbf T^\top \boldsymbol\lambda =F_3,
\end{gather}
\begin{gather}\label{nisa23f}
\overline{ \psi_0}=f_1, \quad \overline{\boldsymbol\chi}=f_3, \quad \overline{\lambda_1}=f_5.
\end{gather}
\end{subequations}
\begin{proposition}\label{nisa24}
Under the assumptions of Theorem \ref{nisa210}, there exist
$r_0>0$, $\varepsilon_0>0$, and $c>0$ such that for all
$
r\leq  r_0$ ,  $ |\varepsilon|\leq
\varepsilon_0,$
and for all $(f_5, f_3, f_1)\in \mathbb C^{n+1}$ ,  $F_i\in
\mathcal A_{\sigma_1,0}$,  and for all $0\leq \sigma_0<\sigma_1$, equations \eqref{nisa23} have a unique
solution analytic solution.
This solution admits the
estimate
\begin{multline}\label{nisa25}
\|\psi_0\|_{\sigma_0,0}+\|\boldsymbol\lambda\|_{\sigma_0,0}  +
 \|\boldsymbol\chi\|_{\sigma_0,0}+\\| p|+| q|+
 |\delta\boldsymbol\beta|\leq c(\sigma_1-\sigma_0)^{-6n-8}
(\|\mathbf F\|_{\sigma_1,0}+|\mathbf f|).
\end{multline}
\end{proposition}
\begin{proof} We have
\begin{equation}\label{nisa26}
 \boldsymbol\lambda=   \overline{\boldsymbol\lambda}+\boldsymbol\lambda^*, \quad
  \boldsymbol\chi=f_3+ \boldsymbol\chi^*
\end{equation}
Substituting this decomposition into \eqref{nisa23} we obtain two system of equations
\begin{gather}\label{nisa27a}
 \overline{ \psi_0}=f_1,\quad\overline{\boldsymbol\chi}=f_3, \quad \overline{\lambda_1}=f_5\\
    \label{nisa27b}
\boldsymbol\Omega \overline{\boldsymbol\lambda} +\overline{\mathbf
T}\boldsymbol\beta
    + p \,\mathbf e_1=\overline{F_2}-\overline{\mathbf T \nabla\psi_0},\\
\label{nisa27c}
 \overline{\mathbf S} \boldsymbol\beta
 +\overline{\mathbf T}^\top \overline{\boldsymbol\lambda} =\overline{F_3}
 -\overline{\mathbf S\nabla\psi_0}-\overline{\mathbf T^\top \boldsymbol\lambda^*},
\end{gather}
and
\begin{gather}\label{nisa28a}
 \boldsymbol\partial \psi_0^*=F_1^*,\\
    \label{nisa28b}
\mathbf J\boldsymbol\partial\boldsymbol\lambda^* +\boldsymbol\Omega \boldsymbol\lambda^*+
(\mathbf
T\nabla\psi_0^*)^*
 =F_2^*-{\mathbf
T}^*\boldsymbol\beta,\\
\label{nisa28c}
 -\boldsymbol\partial \boldsymbol\chi^*+
 (\mathbf S\nabla\psi_0^*)^*+(\mathbf T^\top \boldsymbol\lambda^*) =F_3^*-
 \mathbf S^*\boldsymbol\beta-(\mathbf T^\top)^* \overline{ \boldsymbol\lambda},
\end{gather}
The rest of the proof is based on the following auxiliary lemma.

\begin{lemma}\label{lita15lemma}. Under the assumptions of Theorem  \ref{nisa4}
there exist $r_0>0$ and $\varepsilon_0>0$, $c>0$ such that for all
$
r\leq r_0$
and
$|\varepsilon|\leq \varepsilon_0,
$
and for given $\nabla \psi_0, \lambda^*\in H_0$, the system   \eqref{nisa27a}-\eqref{nisa27c}
has a unique solution, which admits the representation
\begin{subequations} \label{nisa29um}
\begin{gather}\label{nisa29i}
{\boldsymbol\beta}=\mathcal M(\nabla\psi_0^*,\boldsymbol\lambda^*)+\mathfrak F_\beta,\\
\label{nisa29j}
\overline{\boldsymbol\lambda}=\mathcal L(\nabla\psi_0^*,\boldsymbol\lambda^*)+\mathfrak F_\lambda,\\
\label{nisa29k}
p=\mathcal P(\nabla\psi_0^*,\boldsymbol\lambda^*)+\mathfrak F_p.
\end{gather}
\end{subequations}
Here the constant vectors $\mathfrak F_\beta$, $\mathfrak F_\lambda$, $ \mathfrak F_p$
and the linear functionals $\mathcal M$, $ \mathcal L$, $\mathcal P$
admit the estimates
\begin{equation}\label{nisa36overline}
    |\mathfrak F_\beta|\leq c (|\overline{|F}_2|+|\overline{F}_3|+|f_5|)
\end{equation}

\begin{equation}\label{nisa36}
    |\mathfrak F_\beta|+ |\mathfrak F_\lambda|+ |\mathfrak F_p|\leq c (|F_2|_{-s}+|F_3|_{-s}+|f_5|)
\end{equation}
\begin{equation}\label{nisa37}
    |\mathcal M|+ |\mathcal L|+ |\mathcal P|\leq c(\varepsilon_0+r_0) (|\psi^*|_{-s}+|\lambda^*|_{-s})
\end{equation}
where $s>0$ is an arbitrary  number and $c$ is independent on $\varepsilon $ and $r$.
\end{lemma}

\begin{proof}
Let us consider  the system of linear algebraic
equations
\begin{equation}\label{nisa30}\begin{split}
 \boldsymbol\Omega  \overline{\boldsymbol\lambda}+
 \overline{\mathbf T}\,\,\delta\boldsymbol\beta+p\,\,\mathbf e_1=
 \mathbf a,\\
\overline{\mathbf T}^\top\overline{\boldsymbol\lambda}+
\overline{\mathbf S}\,\,\delta\boldsymbol\beta
=\mathbf b, \\
 \overline{\lambda_1}=f_5.
\end{split}\end{equation}
Rewrite the first two equations  in the form
\begin{equation*}\begin{split}
   \boldsymbol\Omega  \overline{\boldsymbol\lambda}+ \overline{\mathbf T_0}\,\,
   \delta
   \boldsymbol\beta+ p\,\,\mathbf e_1=
 \mathbf a+(\mathbf T_0-\overline{\mathbf T})\delta\boldsymbol\beta,\\
{\mathbf T_0}^\top\,\,\overline{\boldsymbol\lambda}+{\mathbf S_0}\,\,
\delta\boldsymbol\beta =\mathbf b+(\mathbf T_0^\top-\overline{\mathbf
T}^\top)\,\overline{\boldsymbol\lambda}+
(\mathbf S_0-\overline{\mathbf S})\,\,\delta\boldsymbol\beta.
\end{split}\end{equation*}
Notice that
$$
\overline{\boldsymbol\lambda}=\overline{\lambda}_2\mathbf e_2+ f_5
\mathbf e_1,
$$
where $\mathbf e_i$ are the basis vectors with the components
$e_{ij}=\delta_{ij}$. Thus we  get the linear system of the
equations for $\overline{\lambda_2}$ and $\delta\boldsymbol\beta$
\begin{equation}\label{nisa30x}\begin{split}
\overline{\lambda_2}+\mathbf t_0\cdot
\delta\boldsymbol\beta=a_2+\{(\mathbf T_0-
\overline{\mathbf T})\delta\boldsymbol\beta\}_2,\\
\mathbf S_0\delta\boldsymbol \beta+{\lambda_2}\, \mathbf t_0=\mathbf
b+ f_5(\mathbf T_0^\top-\overline{\mathbf T}^\top)\mathbf e_1
+\\\overline{\lambda}_2(\mathbf T_0^\top-\overline{\mathbf T}^\top)
\mathbf e_2+(\mathbf S_0-\overline{\mathbf
S})\delta\boldsymbol\beta.
\end{split}\end{equation}
  Express
$\overline{\lambda_2}$ in terms of $\delta\boldsymbol\beta$ using the
first equation in \eqref{nisa30x}. Substituting the result into the second equation in
\eqref{nisa30x} we obtain the following equation for
$\delta\boldsymbol\beta$
\begin{equation*}\begin{split}
    \mathbf K_0\delta\boldsymbol\beta=(\mathbf S_0-\overline{\mathbf S})
    \delta\boldsymbol\beta+
    \{(\mathbf T_0-\overline{\mathbf T})\delta\boldsymbol\beta\}_2
    \big((\mathbf T_0^\top-\overline{\mathbf T}^\top){\boldsymbol e}_2- \mathbf t_0\big)-\delta\boldsymbol\beta\cdot \mathbf t_0
   \big((\mathbf T_0^\top-\overline{\mathbf T}^\top){\boldsymbol e}_2
    \\+\Big(
    \mathbf b-a_2 \mathbf t_0+a_2(\mathbf T_0^\top-
    \overline{\mathbf T}^\top){\boldsymbol e}_2
     +f_5(\mathbf T_0^\top-\overline{\mathbf T}^\top){\boldsymbol e}_1\Big),
\end{split}\end{equation*}
where the matrix $\mathbf K_0= \mathbf S_0-\mathbf t_0\otimes
\mathbf t_0$ has a
 bounded inverse.
Thus we get the following equation for $\delta\boldsymbol\beta$
\begin{equation}\label{nisa31}
    \delta\boldsymbol\beta-\mathbf A\delta\boldsymbol\beta=\mathbf K_0^{-1}(\mathbf b-
    a_2 \mathbf t_0 +a_2(\mathbf T_0^\top-
    \overline{\mathbf T}^\top){\boldsymbol e}_2
    +f_5(\mathbf T_0^\top-\overline{\mathbf T}^\top){\boldsymbol e}_1),
\end{equation}
where the linear mapping $\mathbf A: \mathbb C^2\to \mathbb C^2$ is
given by
$$
\mathbf A:\delta\boldsymbol\beta\mapsto \mathbf K_0^{-1}
\Big\{(\mathbf S_0-\overline{\mathbf S})\delta\boldsymbol\beta+
    \{(\mathbf T_0-\overline{\mathbf T})\delta\boldsymbol\beta\}_2
    \big((\mathbf T_0^\top-\overline{\mathbf T}^\top){\boldsymbol e}_2-
     \mathbf t_0\big)-\delta\boldsymbol\beta\cdot \mathbf t_0
   \big((\mathbf T_0^\top-\overline{\mathbf T}^\top){\boldsymbol e}_2 \Big\}.
$$
In view  of inequalities \eqref{nisa6x}, the
mapping $\mathbf A$
 admits the estimate
$|\mathbf A\boldsymbol\delta\beta|\leq
c(\varepsilon_0+r_0)|\delta\boldsymbol\beta|$.
 Choosing $\varepsilon_0$ and $r_0$ so small that $c(\varepsilon_0+r_0)\leq 1/2$, we obtain that equation \eqref{nisa31} has the only solution which admits the representation
 \begin{equation}\label{nisa31}
    \delta\boldsymbol\beta=(1-\mathbf A)^{-1}\mathbf K_0^{-1}\big(\mathbf b-
    a_2 \mathbf t_0+a_2(\mathbf T_0^\top-
    \overline{\mathbf T}^\top){\boldsymbol e}_2
    +f_5(\mathbf T_0^\top-\overline{\mathbf T}^\top){\boldsymbol e}_1\big)
\end{equation}
 When system \eqref{nisa31} is solved the vector $\overline{\boldsymbol\lambda}$ and the scalar $p$ are restored
 by the relations
\begin{equation}\label{nisa32}
\overline{\boldsymbol\lambda}= f_5\mathbf e_1+\{a_2+((\mathbf T_0-\overline{\mathbf
T})\boldsymbol\beta)\cdot \mathbf e_2\}\mathbf e_2 -\mathbf t_0\cdot \delta\boldsymbol\beta\mathbf e_2
\end{equation}
\begin{equation}\label{nisa33}
p=a_1+ kf_5+\{(\mathbf T_0-\overline{\mathbf T})\boldsymbol\beta\}_1
\end{equation}
In view of \eqref{nisa27b}-\eqref{nisa27c} the vectors  $\boldsymbol\beta$, $\boldsymbol\lambda$ and the constant $p$ ~ satisfy equations \eqref{nisa30} with the right hand sides
\begin{equation*}
    \mathbf a=\overline{F_2}-\overline{\mathbf T \nabla\psi_0},\quad
 \mathbf b=\overline{F_3}
 -\overline{\mathbf S\nabla\psi_0}-\overline{\mathbf T^\top \boldsymbol\lambda^*}
\end{equation*}
Substituting these expressions into \eqref{nisa31}-\eqref{nisa33} we obtain desired representation
\eqref{nisa29i}-\eqref{nisa29k} in which
\begin{gather}\label{nisa34a}
    \mathfrak F_{\beta}=(1-\mathbf A)^{-1}\mathbf K_0^{-1}\Big(\overline{F_3}-(\overline{F}_2 \cdot \mathbf e_2) \overline{\mathbf t}_0+
    (\overline{F}_2\cdot \mathbf e_2)(\mathbf T_0^\top-
    \overline{\mathbf T}^\top){\boldsymbol e}_2\\\nonumber +f_5(\mathbf T_0^\top-\overline{\mathbf T}^\top)\mathbf e_1\Big),\\\label{nisa34b}
     \mathfrak F_{\lambda}=f_5\mathbf e_1 +\overline{F}_2 \mathbf e_2+\Big(
     (\mathbf T_0 -\overline{\mathbf T}) \mathfrak F_\mu\cdot \mathbf e_2\Big)\mathbf e_2-(\mathbf t_0\cdot \mathfrak F_\mu)\mathbf e_2,\\\label{nisa34c}
 \mathfrak F_{p}=k f_5+\overline{F}_2+(\mathbf T_0 -\overline{\mathbf T}) \mathfrak F_\mu\cdot \mathbf e_1\\,
\end{gather}
and
\begin{gather}\label{nisa35a}
\mathcal M=(1-\mathbf A)^{-1}\mathbf K_0^{-1}\Big\{\big(\overline{\mathbf T\nabla \psi_0^*}\cdot \mathbf e_2\big)\mathbf t_0-\big(\overline{\mathbf T\nabla \psi_0^*}\cdot \mathbf e_2\big)
(\mathbf T_0^\top-
    \overline{\mathbf T}^\top){\boldsymbol e}_2\\\nonumber
-\overline{\mathbf S\nabla\psi_0^*}-\overline{\mathbf T\boldsymbol\lambda^*}\Big\},\\
\label{nisa35b}
\mathcal L= -(\overline{\mathbf T\nabla \psi_0^*}\cdot \mathbf e_2)\mathbf e_2-
((\mathbf T_0 -\overline{\mathbf T}) \mathcal M\cdot \mathbf e_2)\mathbf e_2
-(\mathbf t_0\cdot \mathcal M)\mathbf e_2,\\
\label{nisa35c}
\mathcal P= -(\overline{\mathbf T\nabla\psi_0^*})\cdot \mathbf e_2+
(\mathbf T_0 -\overline{\mathbf T}) \mathcal M\cdot \mathbf e_1.
\end{gather}
It remains to prove that these vectors and functionals satisfy inequalities \eqref{nisa36}-\eqref{nisa37}.
We begin with the observation that $\mathfrak F_\beta$, $\mathfrak F_{\lambda}$,
and $\mathfrak F_p$ are linear integral functionals with analytic kernels of
the linear space of functions $F_2$, $F_3$ and constants $f_5$. We have
\begin{equation*}
    |\mathfrak F_\beta|\leq c (|\overline F_2|+|\overline F_3|+ |f_5|),
\end{equation*}
which immediately gives \eqref{nisa36overline}. In order to prove \eqref{nisa36}
we note that the Cauchy inequality
$
|\overline{uv}|\leq |u|_s|v|_{-
s}
$ implies
$$
|\overline{ F}_i|\,=\,|\overline {1 \cdot F_i}|\,\leq \, |1|_{s}\,|F_i|_{-s}
\,\leq c \, |F_i|_{-s},
$$
gives \eqref{nisa36}. Next notice that the $\mathcal M$, $\mathcal L$, and $\mathcal P$
are linear integral functionals
 with analytical kernels. In particular,  they are continuous in every Banach space $H_s$.
  Hence estimate \eqref{nisa37}
 is almost trivial. We prove it for the functional $\mathcal M$.
 The same proof works for $\mathcal L$ and $\mathcal P$.
 In order to estimate $\mathcal M$, notice that
\begin{multline*}
|\,\overline{\mathbf T\nabla \psi_0^*}\,|= |\,\overline{(\mathbf T-\mathbf T_0)\nabla\psi^*}\,|\leq
  |\mathbf T-\mathbf T_0|_{s+1}|\nabla\psi^*|_{-s-1}\\\leq c  \|\mathbf T-\mathbf T_0\|_{\sigma,0}
  |\psi^*|_{-s}\leq cr_0 |\psi^*|_{-s}.
 \end{multline*}
Repeating these arguments we obtain
\begin{equation*}
   |\overline{\mathbf S\nabla \psi_0^*}|+|\overline{\mathbf T\boldsymbol\lambda^*}|\, \leq\,
   cr_0\,(|\psi^*|_{-s}+|\boldsymbol\lambda^*|_{-s}).
\end{equation*}
Combining  this result withy  the obvious inequality
$$
|\mathcal M| \leq c(|\,\overline{\mathbf T\nabla \psi_0^*}\,|+ |\overline{\mathbf S\nabla \psi_0^*}|+|\overline{\mathbf T\boldsymbol\lambda^*}|)
$$
we obtain estimate \eqref{nisa37} for $\mathcal M$.
\end{proof}

Let us turn to the proof of Proposition \ref{nisa24}. Our strategy is the following. First, we
use representations \eqref{nisa29i}-\eqref{nisa29k} in order to reduce the system \eqref{nisa27a} \eqref{nisa28c} to the closed system of equations for the deviations $\psi_0^*$, $\boldsymbol\lambda^*$,
and $\boldsymbol\chi^*$. Then we reduce the obtained system to an abstract operator equation.
We solve this equation in the space of bounded functions by using contraction mapping principle. Finally we prove that the obtained solution is analytic. Let us consider the basic system of equations \eqref{nisa28a}-\eqref{nisa28c}.
Assume that $\varepsilon_0$ and $r_0$ meet all requirements  of Lemma \ref{lita15lemma}.
It follows from representations \eqref{nisa29i}-\eqref{nisa29k} in this lemma that we can rewrite equations \eqref{nisa28a}-\eqref{nisa28c} in the equivalent form
\begin{gather}\nonumber
 \boldsymbol\partial \psi_0^*=G_1^*,\\
    \label{nisa43b}
\mathbf J\boldsymbol\partial\boldsymbol\lambda^* +\boldsymbol\Omega \boldsymbol\lambda^*+
(\mathbf
T\nabla\psi_0^*)^*
 =G_2^*-{\mathbf
T}^*\mathcal M(\psi_0^*, \boldsymbol\lambda^*),\\
\nonumber
 -\boldsymbol\partial \boldsymbol\chi^*+
 (\mathbf S\nabla\psi_0^*)^*+(\mathbf T^\top \boldsymbol\lambda^*) =G_3^*-
 \mathbf S^*\mathcal M(\psi_0^*, \boldsymbol\lambda^*)-(\mathbf T^\top)^* \mathcal L(\psi_0^*, \boldsymbol\lambda^*),
\end{gather}
where
$$
G_1^*=F_1^*, \quad G_2^*=F_2^*-{\mathbf
T}^*\mathfrak F_\beta,\quad G_3^*=F_3^*-\mathbf S^*\mathfrak F_\beta-(\mathbf T^\top)^* \mathfrak F_\lambda
$$
It is worth noting that equations \eqref{nisa43b} and relations \eqref{nisa29i}-\eqref{nisa29k}
forms the system of equations which is equivalent to  equations \eqref{nisa23a}-\eqref{nisa23f}
It follows from estimate \eqref{nisa36}
that
\begin{equation}\label{nisa44}
    \|G_1^*\|_{\sigma_1,0}+\|G_2^*\|_{\sigma_1,0}+\|G_3^*\|_{\sigma_1,0}\leq
  c( \|F_1^*\|_{\sigma_1,0}+\|F_2^*\|_{\sigma_1,0}+\|F_3^*\|_{\sigma_1,0}  +|f_5|).
\end{equation}
Now introduce the linear operator
$$
\Xi_1: (F_1^*,F_2^*, F_3^*)\, \to\, (\psi_0^*, \boldsymbol\lambda^*, \boldsymbol\chi^*)
$$
which assigns to every vector $(F_1^*,F_2^*, F_3^*)$ the solution of the following system of equations
\begin{gather}\label{nisa39a}
 \boldsymbol\partial \psi_0^*=F_1^*,\\
    \label{nisa39b}
\mathbf J\boldsymbol\partial\boldsymbol\lambda^* +\boldsymbol\Omega \boldsymbol\lambda^*+
(\mathbf
T\nabla\psi_0^*)^*
 =F_2^*,\\
\label{nisa39c}
 -\boldsymbol\partial \boldsymbol\chi^*+
 (\mathbf S\nabla\psi_0^*)^*+(\mathbf T^\top \boldsymbol\lambda^*) =F_3^*.
\end{gather}
Denote by $\mathcal Z_\sigma$ the subspace of the Banach space $\mathcal A_{\sigma_0}\times \mathcal
A_{\sigma_0}^2\times \mathcal A_{\sigma,0}^{n-1}$ which consists of all functions with with zero mean value.
Let us show that for every $0\leq \sigma_0<\sigma_1\leq \sigma$, the operator
$\Xi_1: \mathcal Z_{\sigma_1}\to \mathcal Z_{\sigma_0}$ is continuous and
\begin{equation}\label{nisa40}
    \|\Xi_1(F_1^*,F_2^*, F_3^*)\|_{Z_{\sigma_0}}\leq c(\sigma_1-\sigma_0)^{-4n-6} \|\Xi_1(F_1^*,F_2^*, F_3^*)\|_{Z_{\sigma_1}}
\end{equation}
 Choose an arbitrary $(F_1^*,F_2^*, F_3^*)\in \mathcal Z_{\sigma_1}$. Since the system \eqref{nisa39a}-\eqref{nisa39c} is triangular, the existence of solution to this system  obviously follows from Lemmas \ref{nisa6} and \ref{lita21lemma}. Next, estimate \eqref{nisa8} in Lemma \ref{nisa6} implies
\begin{equation}\label{nisa41}
    \|\psi^*_0\|_{2/3\sigma_0+1/3\sigma_1, 0} \leq c(\sigma_1-\sigma_0)^{-n-2}\|F_1^*\|_{\sigma_1,0}.
\end{equation}
Using this inequality and applying Lemma \ref{lita21lemma} to equation \eqref{nisa39b}
we arrive at the estimate
\begin{multline}\label{nisa42}
 \|\boldsymbol \lambda^*\|_{2/3\sigma_0+1/3\sigma_1, 0} \leq c(\sigma_1-\sigma_0)^{-2n-3}
   \big( \|\psi^*_0\|_{1/3\sigma_0+2/3\sigma_1, 0} +
    \|F_2^*\|_{1/3\sigma_0+2/3\sigma_1,0}\big)\\\leq
    c(\sigma_1-\sigma_0)^{-3n-5}\big(\|F_1^*\|_{\sigma_1,0}+
   \|F_2^*\|_{\sigma_1,0}\big).
\end{multline}
Finally applying Lemma \ref{nisa6} to equation \eqref{nisa39c} we obtain
\begin{multline}\label{nisa43}
 \|\boldsymbol \chi^*\|_{\sigma_0, 0} \leq c(\sigma_1-\sigma_0)^{-n-2}
   \big( \|\psi^*_0\|_{2/3\sigma_0+1/3\sigma_1, 0} + \|\boldsymbol\lambda^*\|_{2/3\sigma_0+1/3\sigma_1, 0}
  +  \|F_3^*\|_{2/3\sigma_0+1/3\sigma_1,0}\big)\\\leq
    c(\sigma_1-\sigma_0)^{-4n-6}\big(\|F_1^*\|_{\sigma_1,0}+
   \|F_2^*\|_{\sigma_1,0}+\|F_3\|_{\sigma_1,0}\big).
\end{multline}
Combining estimates \eqref{nisa41}-\eqref{nisa43} we obtain \eqref{nisa40}.
Now introduce the second linear operator
$$
\Xi_2:(\psi_0^*, \boldsymbol\lambda^*, \boldsymbol\chi^*) \to \Big(0, {\mathbf
T}^*\mathcal M(\psi_0^*, \boldsymbol\lambda^*),
 \mathbf S^*\mathcal M(\psi_0^*, \boldsymbol\lambda^*)+(\mathbf T^\top)^* \mathcal L(\psi_0^*, \boldsymbol\lambda^*)\Big)^\top
$$
Since the embedding $\mathcal A_{0, 0} \to H_{-s}$ is bounded for $s\geq 0$,  estimate
\eqref{nisa37} yields the inequality
\begin{equation}\label{nisa45}
    \|\Xi_2(\psi_0^*, \boldsymbol\lambda^*, \boldsymbol\chi^*)\|_{\mathcal Z_\sigma}\leq
    c(\varepsilon_0+r_0)\|(\psi_0^*, \boldsymbol\lambda^*, \boldsymbol\chi^*)|_{\mathcal Z_0}.
\end{equation}
Now we can  rewrite system \eqref{nisa43b} in the form of the operator equation
\begin{equation}\label{nisa46}
  (\psi_0^*, \boldsymbol\lambda^*, \boldsymbol\chi^*)=\Xi_1\Xi_2(\psi_0^*, \boldsymbol\lambda^*, \boldsymbol\chi^*)+\Xi_1(G_1^*, G_2^*, G_3^*).
\end{equation}
Estimates \eqref{nisa40}, \eqref{nisa45}, and the inequality $1/4<\sigma$ imply
\begin{equation}\label{nisa47}
  \|\Xi_1\Xi_2(\psi_0^*, \boldsymbol\lambda^*, \boldsymbol\chi^*) \|_{\mathcal Z_{0}}\leq c(\varepsilon_0+r_0)\|(\psi_0^*, \boldsymbol\lambda^*, \boldsymbol\chi^*) \|_{\mathcal Z_{0}}
\end{equation}
It follows that the norm of the operator $\Xi_1\Xi_2:\mathcal Z_0\to \mathcal Z_0$ does not exceed $c(\varepsilon_0+r_0)$.  Choosing $\varepsilon_0$ and $r_0$  sufficiently small and applying the contraction mapping principle we conclude that operator equation \eqref{nisa46} has a unique solution in the space $\mathcal Z_0$. This solution satisfies the inequality
\begin{equation}\label{nisa48}
  \|(\psi_0^*, \boldsymbol\lambda^*, \boldsymbol\chi^*)\|_{\mathcal Z_0}\leq c\|\Xi_1(G_1^*, G_2^*, G_3^*)\|_{\mathcal Z_0}.
\end{equation}
Let us prove that the obtained solution is analytic.
Since $1/4\leq \sigma_1$, estimates \eqref{nisa40} with $\sigma_0=0$ and estimate  \eqref{nisa44} imply the inequality
\begin{multline}\label{nisa49}
    \|\Xi_1(G_1^*, G_2^*, G_3^*)\|_{\mathcal Z_0}\leq  \|(G_1^*, G_2^*, G_3^*)\|_{\mathcal Z_{\sigma_1}}
\\\leq  c( \|F_1^*\|_{\sigma_1,0}+\|F_2^*\|_{\sigma_1,0}+\|F_3^*\|_{\sigma_1,0}  +|f_5|).
\end{multline}
Hence the solution to equation \eqref{nisa46} admits the estimate
\begin{equation*}
  \|(\psi_0^*, \boldsymbol\lambda^*, \boldsymbol\chi^*)\|_{\mathcal Z_0}\leq c
   ( \|F_1^*\|_{\sigma_1,0}+\|F_2^*\|_{\sigma_1,0}+\|F_3^*\|_{\sigma_1,0}  +|f_5|),
\end{equation*}
which along with \eqref{nisa45} leads to the inequality
\begin{equation}\label{nisa50}
   \|\Xi_2(\psi_0^*, \boldsymbol\lambda^*, \boldsymbol\chi^*)\|_{\mathcal Z_\sigma}\leq  c
   ( \|F_1^*\|_{\sigma_1,0}+\|F_2^*\|_{\sigma_1,0}+\|F_3^*\|_{\sigma_1,0}  +|f_5|).
\end{equation}
On the other hand, inequality \eqref{nisa40} yields
\begin{multline*}
\|\Xi_1\Big(\Xi_2(\psi_0^*, \boldsymbol\lambda^*, \boldsymbol\chi^*)+(G_1^*, G_2^*, G_3^*)\Big)\|_{\mathcal Z_{\sigma_0}}\\\leq
c(\sigma_1-\sigma_0)^{-4n-6}\Big(\|\Xi_2(\psi_0^*, \boldsymbol\lambda^*, \boldsymbol\chi^*)\|_{\mathcal Z_{\sigma_1}}+\|(G_1^*, G_2^*, G_3^*)\|_{\mathcal Z_{\sigma_1}}\Big)
\end{multline*}
Combining this result with \eqref{nisa50} and \eqref{nisa44} we finally obtain
\begin{multline*}
\|\Xi_1\Big(\Xi_2(\psi_0^*, \boldsymbol\lambda^*, \boldsymbol\chi^*)+(G_1^*, G_2^*, G_3^*)\Big)\|_{\mathcal Z_{\sigma_0}}\\\leq
c(\sigma_1-\sigma_0)^{-4n-6}
( \|F_1^*\|_{\sigma_1,0}+\|F_2^*\|_{\sigma_1,0}+\|F_3^*\|_{\sigma_1,0}  +|f_5|)
\end{multline*}
Using this result  and equation \eqref{nisa46} we finally obtain that the solution to system  \eqref{nisa43b} admits the estimate
\begin{multline}\label{nisa51}
  \|\psi_0^*\|_{\sigma_0,0}+ \|\boldsymbol \lambda^*\|_{\sigma_0,0}+\|\boldsymbol \chi^*\|_{\sigma_0,0}\\\leq c(\sigma_1-\sigma_0)^{-4n-6} c(|f_5|+ \|F_1\|_{\sigma_1,0}+\|F_2\|_{\sigma_1,0}+\|F_3\|_{\sigma_1,0} )
\end{multline}
Hence we proof that for all sufficiently small $\varepsilon_0$ and $r_0$ system \eqref{nisa43b}
has an analytic solution satisfying inequality \eqref{nisa51}. Next notice that  the mean value $\overline{\boldsymbol\chi}=f_3$ and the constant $q=f_1$.  Recall that the vector $\delta\boldsymbol\beta$, the mean value $\boldsymbol\lambda^*$, and the constant $p$ are connected with the deviations $\nabla\psi^*$, $\boldsymbol
\lambda^*$ and $\boldsymbol
\chi^*$ by relations \eqref{nisa29i}-\eqref{nisa29k}. It follows from this relations and estimates
\eqref{nisa36}-\eqref{nisa37} that
\begin{multline*}
|\overline{\boldsymbol\chi}|+|\delta\boldsymbol\beta|+|\overline{\boldsymbol\lambda}|+|p|+|q|
\leq c(|f_1|+|f_3|+|f_5|+ \|F_1\|_{\sigma_1,0}+\|F_2\|_{\sigma_1,0}+\|F_3\|_{\sigma_1,0} )\\
+c(\|\psi_0^*\|_{\sigma_0,0}+ \|\boldsymbol \lambda^*\|_{\sigma_0,0}+\|\boldsymbol \chi^*\|_{\sigma_0,0})
 \end{multline*}
Combining this result with \eqref{nisa51}
we obtain
\begin{multline*}
|\overline{\boldsymbol\chi}|+|\delta\boldsymbol\beta|+|\overline{\boldsymbol\lambda}|+|p|+|q|
\\\leq c(\sigma_1-\sigma_0)^{-4n-6}(|f_1|+|f_3|+|f_5|+ \|F_1\|_{\sigma_1,0}+\|F_2\|_{\sigma_1,0}+\|F_3\|_{\sigma_1,0} ).
 \end{multline*}
This result, estimate \eqref{nisa51}, and decomposition \eqref{nisa26} imply  desired estimate
\eqref{nisa25}. This completes the proof of Proposition \ref{nisa24}.

\end{proof}

Now we consider the extended truncated system, which includes the extra equation for the matrix-valued
function $\boldsymbol\Gamma$. This extended system is formulated as follows
\begin{subequations}\label{lita6}
\begin{gather}\label{lita6a}
 \boldsymbol\partial \psi_0+ q=F_1,\\
    \label{lita6b}
\mathbf J\boldsymbol\partial\boldsymbol\lambda +\mathbf
T\boldsymbol\mu
    +p\mathbf e_1 =F_2,\quad \boldsymbol\mu=
    \delta\boldsymbol\beta+\nabla\psi_0\\
\label{lita6c}
 -\boldsymbol\partial \boldsymbol\chi+
 \mathbf S\boldsymbol\mu+\mathbf T^\top \boldsymbol\lambda =F_3,\\
\label{lita6f}
\overline{ \psi_0}=f_1, \quad
\overline{\boldsymbol\chi}=f_3\quad \overline{\lambda_1}\,
=f_5,
\end{gather}

\begin{gather}
\label{lita6d}
\boldsymbol\partial(\mathbf J\boldsymbol\Gamma)
+\boldsymbol\Omega\boldsymbol\Gamma+
(\boldsymbol\Omega\boldsymbol\Gamma)^\top +
\\\nonumber\mathbf U_{ij}\frac{\partial \lambda_i}{\partial\xi_j}+
\mu_i\mathbf E_i+\lambda_i\mathbf K_i+ \delta\mathbf M=F_4,\\
\label{lita6g}
\Gamma_{11}=-\Gamma_{22},\quad
\overline{\Gamma}_{12}=f_4, \quad\delta\mathbf M=\delta M \text{~diag~} (1,0).
\end{gather}
\end{subequations}
The following proposition guarantees the well-posedness of the
truncated problem \eqref{lita6}.

\begin{proposition}\label{lita7} Under the assumptions of Theorem
\ref{nisa210},  there are $\varepsilon_0>0$ and  $r_0$ with
the following properties. For every
$$
(\alpha, k)\in \Sigma_{\varrho}, \quad\|\boldsymbol\varphi-
\boldsymbol\varphi(\alpha)\|_{\sigma, d}\leq r_0, \quad
\boldsymbol\varphi(\alpha)= (0,0,0,\boldsymbol\alpha,1,0,0), \quad
|\varepsilon|\leq \varepsilon_0, \quad ,
$$
$$
0\leq \sigma_0<\sigma_1<\sigma,\quad \sigma_1> 1/4
$$
$$\mathbf f\in \mathbb C^{4},\quad \mathbf F\in \mathcal F_{\sigma_1,0}.
$$
Then problem \eqref{lita6} has a unique solution $$ (\psi_0,
\boldsymbol\lambda, \boldsymbol\chi, \boldsymbol\Gamma,
 \delta\boldsymbol \beta, q,p, \delta\mathbf M)\in \mathcal A_{\sigma_0, 0}
 \times \mathcal A_{\sigma_0,0}^{2}\times \mathcal A_{\sigma,d}^{n-1}\times
 \mathcal A_{\sigma_0,0}^4\times \mathbb C^{n-1}\times \mathbb C^{3}.
$$
This solution admits the estimate
\begin{equation}\label{lita8}
\|(\psi_0, \boldsymbol\lambda, \boldsymbol\chi,
\boldsymbol\Gamma)\|_{\sigma_0,0}+ |(\delta\boldsymbol \beta, \delta
q, \delta p, \delta\mathbf M)|\leq
{c}{(\sigma_1-\sigma_0)^{-8n-12}} (\|\mathbf
F\|_{\sigma_1,0}+|\mathbf f|)
\end{equation}
where  the constant $c$ is independent of  $\varepsilon_0$, $r_0$, and
$\sigma$.
\end{proposition}
\begin{proof}
We begin with the observation that equations \eqref{lita6a}-\eqref{lita6f}
are independent of $\Gamma$. Moreover, system \eqref{lita6a}-\eqref{lita6f}
coincides with system \eqref{nisa23a}-\eqref{nisa23f}.
Applying Proposition \ref{nisa24} we conclude that this system has a unique analytic solution.
This solution admits the estimate
\begin{multline}\label{nisa55}
\|\psi_0\|_{\sigma_0,0}+\|\boldsymbol\lambda\|_{\sigma_0,0}  +
 \|\boldsymbol\chi\|_{\sigma_0,0}+| p|+| q|\\+
 |\delta\boldsymbol\beta|\leq c(\sigma_1-\sigma_0)^{-6n-8}
(\|\mathbf F\|_{\sigma_1,0}+|\mathbf f|).
\end{multline}
Hence it suffices to analyze equation \eqref{lita6d}.
Rewrite it in the form
\begin{equation}\label{nisa56}
\boldsymbol\partial(\mathbf J\boldsymbol\Gamma)
+\boldsymbol\Omega\boldsymbol\Gamma+
(\boldsymbol\Omega\boldsymbol\Gamma)^\top +
\delta\mathbf M= F_4-\mathbf U_{ij}\frac{\partial \lambda_i}{\partial\xi_j}-
\mu_i\mathbf E_i-\lambda_i\mathbf K_i
\end{equation}
Applying inequality \eqref{nisa55} we obtain
\begin{gather}\nonumber
    \|\mathbf U_{ij}\frac{\partial \lambda_i}{\partial\xi_j}-
\mu_i\mathbf E_i-\lambda_i\mathbf K_i\|_{2/3\sigma_0+1/3\sigma_1,0}\leq
c(\sigma_1-\sigma_0)^{-1}(\|\boldsymbol\lambda\|_{1/2\sigma_0+1/2\sigma_1, 0}\\\label{nisa57}+
c\|\psi_0\|_{1/2\sigma_0+1/2\sigma_1, 0} +|\delta\boldsymbol\beta|)\leq
c(\sigma_1-\sigma_0)^{-6n-9}
(\|\mathbf F\|_{\sigma_1,0}+|\mathbf f|).
\end{gather}
Applying Lemma \ref{nisa15} we conclude that equation \eqref{nisa56} has a unique analytic  solution
which satisfies condition \eqref{lita6g} and admits the estimate
\begin{gather*}
\|\boldsymbol\Gamma\|_{\sigma_0,0}+ | \delta\mathbf M|\leq
{c}{(\sigma_1-\sigma_0)^{-2n-3}}\,\, (\,\|\mathbf
F_4-\mathbf U_{ij}\frac{\partial \lambda_i}{\partial\xi_j}-
\mu_i\mathbf E_i-\lambda_i\mathbf K_i\|_{2/3\sigma_0+1/3\sigma_1,0}+|f_4|\,\big)
\end{gather*}
It follows  from this and \eqref{nisa57} that
\begin{gather*}
\|\boldsymbol\Gamma\|_{\sigma_0,0}+ | \delta\mathbf M|\leq
{c}{(\sigma_1-\sigma_0)^{-8n-12}} (\|\mathbf F\|_{\sigma_1,0}+|\mathbf f|).
\end{gather*}
Combining this result with \eqref{nisa55} we obtain \eqref{lita8} and the proposition follows.

\end{proof}

\paragraph{Proof of Theorem \ref{nisa210}} We are now in a position to complete the proof of Theorem
\ref{nisa210}. The proof is based on the Proposition \ref{lita7} and the contraction mapping principle. By technical reasons it is convenient to introduce the following denotations.
Introduce the vectors
$$
 \mathfrak w=(\psi_0,
\boldsymbol\lambda, \boldsymbol\chi, \boldsymbol\Gamma,
 \delta\boldsymbol \beta, q,p, \delta M), \quad
\mathfrak f=(F_1, F_2, F_3, F_4, f_1, f_3, f_4, f_5).
$$
Denote by $\mathcal X_\sigma$ the closed subspace of the Banach space
$$\mathcal A_{\sigma_0, 0}
 \times \mathcal A_{\sigma_0,0}^{2}\times \mathcal A_{\sigma,d}^{n-1}\times
 \mathcal A_{\sigma_0,0}^4\times \mathbb C^{n-1}\times \mathbb C^{3},$$
 which consists of all vector functions $\mathfrak w$ such that $\text{~tr~}\boldsymbol\Gamma=0$.
Denote by $\mathcal Y_\sigma$ the closed subspace of the Banach space
$$\mathcal A_{\sigma_0, 0}
 \times \mathcal A_{\sigma_0,0}^{2}\times \mathcal A_{\sigma,d}^{n-1}\times
 \mathcal A_{\sigma_0,0}^4\times \mathbb C\times \mathbb C^{n-1} \times\mathbb C^{n-1}\times \mathbb C$$
which consists of all vectors $\mathfrak f$ such that the matrix $F_4$ is symmetric.
Rewrite main system \eqref{nisa4} in the form.
\begin{subequations}\label{nisa58}
\begin{gather}\label{nisa58a}
 \boldsymbol\partial \psi_0+ q=F_1+\Xi_{3,1}(p,\delta M),\\
    \label{nias58b}
\mathbf J\boldsymbol\partial\boldsymbol\lambda +\boldsymbol\Omega \boldsymbol\lambda+\mathbf
T\boldsymbol\mu +p\mathbf e_1=
F_2+\Xi_{3,2}(p,\delta M),\quad
    \boldsymbol\mu =\nabla \psi_0+\delta\boldsymbol\beta\\
\label{58c}
 -\boldsymbol\partial \boldsymbol\chi+\mathbf S\boldsymbol\mu+
 \mathbf T^\top \boldsymbol\lambda =F_3,
\\\label{nisa58d}
\boldsymbol\partial(\mathbf J\boldsymbol\Gamma)
+\boldsymbol\Omega\boldsymbol\Gamma+
(\boldsymbol\Omega\boldsymbol\Gamma)^\top +\mathbf U_{ij}
\frac{\partial \lambda_i}{\partial\xi_j}+ \mu_i\mathbf
E_i+\lambda_i\mathbf K_i+\delta\mathbf M =  \\\nonumber F_4+\Xi_{3,4}(\delta M),\quad \Gamma_{11}=-\Gamma_{22},
\end{gather}
\begin{gather}
\label{nisa58e}
 \overline{\psi_0}=f_1+\Xi_{3,5}(\boldsymbol\beta, \boldsymbol\lambda)
\\
\label{nisa58f} \overline{\boldsymbol\chi}
=f_3+\Xi_{3,6}(\boldsymbol\chi),
\\\label{nisa58g}
\overline{\Gamma_{12}}=f_4+\Xi_{3,7}(\boldsymbol\Gamma)
\\\label{nisa58h}
 \overline{\lambda_1}=f_5+\Xi_{3,8}(\boldsymbol\lambda, \boldsymbol\chi),
\end{gather}
\end{subequations}
where
\begin{gather}\nonumber
\Xi_{3,1}(p,\delta M)=-  p ( w_1-\alpha)-
    \frac{1}{2} \delta M (w_1-\alpha)^2,\\
    \nonumber
\Xi_{3,2}(p,\delta M)=-p(\mathbf W^\top-\mathbf I)\mathbf e_1-
    \delta M(w_1-\alpha)\mathbf W^\top\mathbf e_1,
    \\
\nonumber
\Xi_{3,3}=0, \quad \Xi_{3,4}(\delta M)=\delta\mathbf M-\mathbf W^\top \delta\mathbf  M\mathbf W,
\\ \nonumber
\Xi_{3,5}(\boldsymbol\beta, \boldsymbol\lambda)=\boldsymbol\delta\beta\cdot \overline{\mathbf u}
-\overline{w_2\mathbf W\boldsymbol\lambda\cdot \mathbf e_1}
\\\nonumber
\Xi_{3,6}(\boldsymbol\chi)=\overline{(\mathbf I-\mathbf V^{-\top})\boldsymbol\chi},
\quad
\Xi_{3,7}(\boldsymbol\Gamma)=\overline{\Big((\mathbf I-\mathbf W)\boldsymbol\Gamma\Big)_{12}}
\\\nonumber
\Xi_{3,8}(\boldsymbol\lambda, \boldsymbol\chi)=
\overline{(\mathbf I-\mathbf W)\boldsymbol\lambda\cdot \mathbf e_1}-\overline{\boldsymbol\chi\cdot
\nabla w_1}.
\end{gather}
It is necessary to prove that, under the assumptions of Theorem \ref{nisa210}, the system of equations
\eqref{nisa58} for  all $\mathfrak f\in \mathcal Y_{\sigma_1}$ has a unique solution $\mathfrak w\in \mathcal X_{\sigma_0}$ satisfying the inequality
\begin{equation}\label{nisa59}
    \|\mathfrak w\|_{\mathcal X_{\sigma_0}}\leq c(\sigma_1-\sigma_0)^{-8n-12}\,\,\|\mathfrak f\|_{\mathcal Y_{\sigma_1}}
\end{equation}
Recall that
\begin{equation}\label{nisa60x}
    0\leq \sigma_0<\sigma_1\leq \sigma, \quad 1/4\leq \sigma_1.
\end{equation}
The proof is based on the following auxiliary lemma.
\begin{lemma}\label{nisa60} Under the assumptions of Theorem \ref{nisa210} the operator
$\Xi_3=\big(\,\Xi_{3, i}\,\big)_{1\leq i\leq 8} $ admits  the estimate
\begin{equation}\label{nisa61}
    \|\Xi_3(\mathfrak w)\|_{\mathcal Y_{\sigma}}\leq cr_0 \|\mathfrak w\|_{\mathcal X_0}.
\end{equation}
\end{lemma}
\begin{proof}
Obviously we have
\begin{multline*}
    \|\Xi_{3, 1}(p, \delta M)\|_{\mathcal Y_\sigma}+ \|\Xi_{3, 2}(p, \delta M)\|_{\mathcal Y_\sigma} +\|\Xi_{3, 4}(\delta M)\|_{\mathcal Y_\sigma}\leq\\
    c(|p|+|\delta M|)(\|w_1-\alpha\|_{\sigma,0}+\|\mathbf W-\mathbf I\|_{\sigma,0})
\end{multline*}
It follows from the estimate
$$
(\alpha, k)\in \Sigma_{\varrho}, \quad\|\boldsymbol\varphi-
\boldsymbol\varphi(\alpha)\|_{\sigma, d}\leq r, \quad
\boldsymbol\varphi(\alpha)= (0,0,0,\boldsymbol\alpha,1,0,0),
$$
in condition of Theorem \ref{nisa210} and the relations $\boldsymbol\varphi=
(\boldsymbol\beta, \varphi_0, \mathbf u, \mathbf w, W_{11}, W_{12}, W_{21})$,
 $\text{~det~}\mathbf W=1$ that
 \begin{equation}\label{nisa63}
\|\mathbf u\|_{\sigma,0}+\|w_2\|_{\sigma,0}+ \|w_1-\alpha\|_{\sigma,0}+\|\mathbf W-\mathbf I\|_{\sigma,0}\leq c r\leq cr_0.
 \end{equation}
 Thus we get
 \begin{equation}\label{nisa64}
    \|\Xi_{3, 1}(p, \delta M)\|_{\mathcal Y_\sigma}+ \|\Xi_{3, 2}(p, \delta M)\|_{\mathcal Y_\sigma} +
    \|\Xi_{3, 4}(\delta M)\|_{\mathcal Y_\sigma}\leq  cr_0(|p|+|\delta M|)
 \end{equation}
 Next, the Cauchy inequality  implies
 \begin{multline*}
    |\Xi_{3, 5}|+ |\Xi_{3, 6}| +|\Xi_{3, 7}| +
    |\Xi_{3, 8}| \leq
    c(|\delta\boldsymbol\beta|+|\boldsymbol \lambda|_0+|\boldsymbol \chi|_0+|\boldsymbol \Gamma|_0
    )\times\\(|\mathbf u|_{0}+|\mathbf W-\mathbf I|_0 +|w_2\mathbf W|_0+|\nabla w_1|_{0}+
    |\mathbf I -\mathbf V^{-1}|_0)
\end{multline*}
Recall that $I-\mathbf V^{-\top}\equiv -{\mathbf u'}$ and the embedding
$$
\mathcal A_{\sigma, 0}\hookrightarrow H_1\hookrightarrow H_0, \quad \mathcal A_{0,0}\hookrightarrow
H_0
$$
is bounded. From this and \eqref{nisa63} we obtain
\begin{multline*}
  |\mathbf u'|_{0}+|\mathbf W-\mathbf I|_0 +|w_2\mathbf W|_0+|\nabla w_1|_{0}+
    |\mathbf I -\mathbf V^{-\top}|_0\\\leq
  |\mathbf u|_{1}+|\mathbf W-\mathbf I|_0 +|w_2\mathbf W|_0+| w_1^*|_{1}
 \\ \leq  \|\mathbf u\|_{\sigma,0}+\|\mathbf W-\mathbf I\|_{\sigma,0} +\|w_2\mathbf W\|_{\sigma,0}
  +\|w_1^* \|_{\sigma,0}\leq c r_0
\end{multline*}
and
\begin{equation*}
|\boldsymbol \lambda|_0+|\boldsymbol \chi|_0+|\boldsymbol \Gamma|_0
\leq c(\|\boldsymbol \lambda\|_{0,0}+\|\boldsymbol \chi\|_{0,0}+\|\boldsymbol \Gamma\|_{0,0}).
\end{equation*}
Combining the obtained results we arrive at the estimate
\begin{equation}\label{nisa65}
    |\Xi_{3, 5}|+ |\Xi_{3, 6}| +|\Xi_{3, 7}| +
    |\Xi_{3, 8}| \leq cr_0(|\boldsymbol\beta|+\|\boldsymbol \lambda\|_{0,0}+
    \|\boldsymbol \chi\|_{0,0}+\|\boldsymbol \Gamma\|_{0,0}).
\end{equation}
It remains to note that the desired estimate \eqref{nisa61} obviously follows from
\eqref{nisa64} and \eqref{nisa65}.
\end{proof}
Let us turn to the proof of the Theorem \ref{nisa210}. Denote by $\Xi_4$
the linear operator which assigns to every $\mathfrak f\in Y_{\sigma_1}$ the solution
$\mathfrak w$ of problem \eqref{lita6}. It follows from the Proposition \eqref{lita7}
that for small $\varepsilon_0$ and $r_0$, the operator $\Xi_4: \mathcal Y_{\sigma_1}
\to \mathcal X_{\sigma_0}$ is bounded and
\begin{equation}\label{nisa66}
    \|\Xi_4 \mathfrak f\|_{\sigma_0}\leq c (\sigma_1-\sigma_0)^{-6n-8}
    \|\mathfrak f\|_{\mathcal Y_{\sigma_1}}.
\end{equation}
In particular, we have
\begin{equation}\label{nisa67}
    \|\Xi_4 \mathfrak f\|_{\sigma_0}\leq c \|\mathfrak f\|_{\mathcal Y_{\sigma_1}}.
\end{equation}
Now we can rewrite system \eqref{nisa58} in the form of the operator equation
\begin{equation}\label{nisa68}
    \mathfrak w\,=\,\Xi_4\,( \,\Xi_3\, \mathfrak w+ \mathfrak f\,).
\end{equation}
By virtue of \eqref{nisa61} and \eqref{nisa67} the operator $\Xi_4\Xi_3: \mathcal X_0\to \mathcal X_0$
is bounded and its norm does not exceed $cr_0$. Choosing $r_0$ sufficiently small
and applying the contraction mapping principle we conclude that operator equation \eqref{nisa68}
has a unique solution in the space $\mathcal X_0$.  In view of \eqref{nisa67},
this solution admits the estimate
\begin{equation*}
    \|\mathfrak w\|_{\mathcal X_0}\,\leq \, c\, \|\mathfrak f\|_{Y_{\sigma_1}}.
\end{equation*}
Let us prove that this solution is analytic.
Notice that estimate  \eqref{nisa61} implies
\begin{equation*}
 \|\Xi_3\mathfrak w\|_{\mathcal Y_{\sigma_1}}\leq c
     \|\mathfrak w\|_{\mathcal X_0}\leq c \|\mathfrak f\|_{Y_{\sigma_1}}
\end{equation*}
Hence
$$
\|\Xi_3\mathfrak w+\mathfrak f\|_{\mathcal Y_{\sigma_1}}\leq c
     \|\mathfrak w\|_{\mathcal X_0}+\|\mathfrak f\|_{\mathcal Y_{\sigma_1}}
     \leq c \|\mathfrak f\|_{Y_{\sigma_1}}.
$$
Combining this result with \eqref{nisa66} we obtain
desired estimate \eqref{nisa59}. This completes the proof of Theorem \ref{nisa210}.

\subsection{Proof of Theorem \ref{figa5}}\label{proothfiga5}
We split the proof into two parts. First we prove the existence and uniqueness of solutions to problem
\eqref{figa4},
\begin{subequations}\label{krit1}
\begin{gather}\label{krit1a}\boldsymbol\partial\, \boldsymbol\mu\,=\,
-p\nabla w_1 -  \nabla g_1,\\
\label{krit1b}
\mathbf J\boldsymbol\partial\, \boldsymbol\lambda\,+\, \boldsymbol\Omega\, \boldsymbol\lambda\,+\,
\mathbf T\, \boldsymbol\mu+p\mathbf W^\top \mathbf e_1=\mathbf g_2\\
\label{krit1c}
-\boldsymbol\partial\, \boldsymbol\lambda\,+\,\mathbf S\boldsymbol\mu\, +\, \mathbf T^\top
\boldsymbol\lambda=0
\end{gather}
\begin{gather}\label{krit1d}
\overline{\{W\boldsymbol\lambda\}\cdot \mathbf e_1}\,\,+\,\,\overline{\{\boldsymbol\chi\cdot \nabla w_1\}}
=\gamma, \\\label{krit1e}
\overline{\{\mathbf V^{-\top}\boldsymbol\chi\}}\,=\, 0.
\end{gather}
\end{subequations}
in the space of analytic function and establish estimate \eqref{figa6}. Next, we estimate the obtained
solution in the Sobolev space $H_s$. We begin with the observation that problem \eqref{krit1}
is the very particular case of general problem \eqref{nisa4}. Notice  that, in view of conditions
\eqref{kira5extra} and \eqref{kira16extra} of Theorem \ref{figa5},
the matrices $\mathbf W$, $\mathbf V$, $\mathbf S$, $\mathbf T$ and the function $w_1$
 meet all requirements of Theorem \ref{nisa210} with $\sigma$ replaced by $\sigma/2$ and $r_0=c\varepsilon_0$.
Next set

\begin{equation}\label{krit3}
\mathbf U_{ij}=0, \quad
\mathbf E_i=0, \quad \mathbf K_i=0,
\end{equation}
$$
\sigma_0=\sigma/4, \quad \sigma_1=\sigma/2\geq 1/4
$$
Obviously these quantities  satisfy all conditions of
Theorem \ref{nisa210}. Now
introduce the temporary notation
\begin{equation}\label{krit5}\begin{split}
    F_1=-g_1^*, \quad F_2=\mathbf g_2, ,\quad  F3=0, \quad F_4=0,\\
\quad f_1=0, \quad f_3=0, \quad f_4=0,\quad  f_5=\gamma,
\end{split}\end{equation}
Since $r_0\leq c\varepsilon_0$, it follows from Theorem \ref{nisa210}
that for a suitable choice of $\varepsilon_0$ and for  all $|\varepsilon|\leq \varepsilon_0$,
the equations \eqref{nisa4} with the righthand sides \eqref{krit5} have a unique analytic solution. This solution
satisfies inequalities \eqref{lita213}. In particular, we have
\begin{equation}\label{krit4}
|(\psi_0, \boldsymbol\lambda, \boldsymbol\chi)\|_{3/8\sigma,0}+
 |(\delta\boldsymbol \beta,  p)| \leq
{c}(
\| g_1^*\|_{\sigma/2,0}   + \|\mathbf g_2^*\|_{\sigma/2,0} +|\gamma|).
\end{equation}
By virtue of condition \eqref{krit3} and the equalities $f_4=0$, $F_4=0$, the matrices $\boldsymbol\Gamma$
and $\delta \mathbf M$
satisfy the homogeneous equation
$$
\boldsymbol\partial(\mathbf J\boldsymbol\Gamma)
+\boldsymbol\Omega\boldsymbol\Gamma+
(\boldsymbol\Omega\boldsymbol\Gamma)^\top+\delta\mathbf M = \delta\mathbf M-
\mathbf W^\top \delta\mathbf  M\mathbf W,
$$
$$
\overline{\Gamma_{12}}=\overline{\Big((\mathbf I-\mathbf W)\boldsymbol\Gamma\Big)_{12}}
$$
Applying  Lemma \ref{lita21lemma} and the contraction mapping principle we conclude that
$\boldsymbol\Gamma=\delta\mathbf M=0$. Since $\overline{F}_1=0$, we also have $q=0$.
Hence the functions $\psi_0$, $\boldsymbol \lambda$ and $\boldsymbol \chi$ satisfy the equations
\begin{subequations}\label{krit6}
\begin{gather}\label{krit6a}\boldsymbol\partial\, \psi_0\,=\,
-p w_1^* -   g_1^*,\\
\label{krit6b}
\mathbf J\boldsymbol\partial\, \boldsymbol\lambda\,+\, \boldsymbol\Omega\, \boldsymbol\lambda\,+\,
\mathbf T\, \boldsymbol\mu+p\mathbf W^\top \mathbf e_1=\mathbf g_2\\\nonumber
\boldsymbol\mu=\delta\boldsymbol\beta+\nabla\psi_0,
\\
\label{krit6c}
-\boldsymbol\partial\, \boldsymbol\lambda\,+\,\mathbf S\boldsymbol\mu\, +\, \mathbf T^\top
\boldsymbol\lambda=0
\end{gather}
\begin{gather}\label{krit6d}
\overline{\{W\boldsymbol\lambda\}\cdot \mathbf e_1}\,\,+\,\,\overline{\{\boldsymbol\chi\cdot \nabla w_1\}}
=\gamma, \\\label{krit6e}
\overline{\{\mathbf V^{-\top}\boldsymbol\chi\}}\,=\, 0.
\end{gather}
\end{subequations}
It follows that $\boldsymbol\mu, \boldsymbol\lambda$, $\boldsymbol\chi$ and $p$ satisfy equations
\eqref{krit1}. Inequalities \eqref{krit4} and  the obvious inequality
$$
\|\boldsymbol\mu\|_{\sigma/4,0}\leq c(|\boldsymbol\beta|+\|\psi_0\|_{3/8\sigma,0}
$$
implies the estimate
\begin{equation}\label{krit7}
    |(\boldsymbol\mu, \boldsymbol\lambda, \boldsymbol\chi)\|_{\sigma/4,0}+
 | p| \leq
{c}(
\| g_1^*\|_{\sigma/2,0}   + \|\mathbf g_2^*\|_{\sigma/2,0} +|\gamma|).
\end{equation}
which  yields \eqref{figa6}. It remains to estimate the solution to problem
\eqref{krit1} in the Sobolev spaces. Now we change the denotations and set
\begin{equation}\label{krit8}\begin{split}
    F_1=-pw_1^*-g_1*, \quad F_2=p(\mathbf I-\mathbf W^\top) \mathbf e_1+\boldsymbol g_2, \quad F_3=0,\\
    f_1=0, \quad f_3=\overline{(\mathbf I-\mathbf V^{-\top}) \boldsymbol \chi}\quad
    f_5=\overline{(\mathbf I-\mathbf W)\boldsymbol\lambda\mathbf e_1}-
    \overline{\boldsymbol\chi\cdot \nabla w_1}
\end{split}\end{equation}
Now  we estimate
these quantities in the Sobolev space  $H_s$.
Obviously we have
\begin{equation}\label{krit9}\begin{split}
|F_1|_{s}\leq |p| |w_1^*|_{s}+|g_1^*|_{s}\leq c \varepsilon_0 |p|+|g_1^*|_{s},\\
|F_2|_{s}\leq |p| |\mathbf I-\mathbf W|_{s}+|\mathbf g_2|_{s}\leq c \varepsilon_0 |p|+|\mathbf g_2|_{s}.
\end{split}\end{equation}
We also have
\begin{equation}\label{krit9overline}
    |\overline{F}_2|\leq c \varepsilon_0|p|+|\overline{\mathbf g_2}|
\end{equation}
The Cauchy inequality implies
\begin{equation*}\begin{split}
|f_3|\leq |\mathbf I-\mathbf V^{-\top}|_{-s}\, |\boldsymbol\chi|_{s}\leq \|\mathbf I-\mathbf V^{-\top}\|_{\sigma/2,0}
\, |\boldsymbol\chi|_{s}\leq c\varepsilon_0|\boldsymbol\chi|_{s}.
\end{split}\end{equation*}
\begin{equation*}\begin{split}
|f_5|\leq |\mathbf I-\mathbf W|_{-s}\, |\boldsymbol\lambda|_{s}+|\nabla w_1|_{-s}|\boldsymbol\chi|_{s}\leq\\
c\|\mathbf I-\mathbf W\|_{\sigma/2,0} |\boldsymbol\lambda|_{s}+\|w_1^*\|_{\sigma/2,0}
|\boldsymbol\chi|_{s}\leq c\varepsilon_0(|\boldsymbol\lambda|_{s}+|\boldsymbol\chi|_{s})
\end{split}\end{equation*}
Combining the obtained results we arrive at the inequality
\begin{equation}\label{krit10}\begin{split}
|f_3|+|f_5|\leq c\varepsilon_0(|\overline{\boldsymbol\lambda}|+|\overline{\boldsymbol\chi}|)
+c\varepsilon_0(|\boldsymbol\lambda^*|_{s}+|\boldsymbol\chi^*|_{s})
\end{split}\end{equation}
With this notation equations \eqref{krit6} can be written in the form
\begin{subequations}\label{krit11}
\begin{gather}\label{krit11a}
 \boldsymbol\partial \psi_0+ q=F_1,\\
    \label{krit11b}
\mathbf J\boldsymbol\partial\boldsymbol\lambda +\boldsymbol\Omega \boldsymbol\lambda+\mathbf
T\boldsymbol\mu
    + p \,\,\mathbf e_1=F_2,\quad \boldsymbol\mu=\delta\boldsymbol\beta+\nabla\psi_0\\
\label{krit11c}
 -\boldsymbol\partial \boldsymbol\chi+
 \mathbf S\boldsymbol\mu+\mathbf T^\top \boldsymbol\lambda =F_3,
\end{gather}
\begin{gather}\label{krit11f}
\overline{ \psi_0}=f_1, \quad \overline{\boldsymbol\chi}=f_3, \quad \overline{\lambda_1}=f_5.
\end{gather}
\end{subequations}
Hence the  these equations coincide with equations \eqref{nisa23} at least formally. It follows from Lemma
\ref{lita15lemma} that a solution to system \eqref{krit1} admits the representation \eqref{nisa29um}.
recalling identities  $\boldsymbol\beta=\overline{\boldsymbol\mu}$ and
$\nabla \psi_0=\boldsymbol\mu^*$ we can rewrite this representation in the form
\begin{subequations} \label{krit29um}
\begin{gather}\label{krit29i}
\overline{\boldsymbol\mu}=\mathcal M(\boldsymbol\mu^*,\boldsymbol\lambda^*)+\mathfrak F_\beta,\\
\label{krit29j}
\overline{\boldsymbol\lambda}=\mathcal L(\boldsymbol\mu^*,\boldsymbol\lambda^*)+\mathfrak F_\lambda,\\
\label{krit29k}
p=\mathcal P(\boldsymbol\mu^*,\boldsymbol\lambda^*)+\mathfrak F_p.
\end{gather}
\end{subequations}
It follows from estimates \eqref{nisa36} and \eqref{nisa37} in Lemma \ref{lita15lemma}
that the constant vectors $\mathfrak F_\beta$, $\mathfrak F_\lambda$, $ \mathfrak F_p$
and the linear functionals $\mathcal M$, $ \mathcal L$, $\mathcal P$
admit the estimates
\begin{equation}\label{krit12}
    |\mathfrak F_\beta|+ |\mathfrak F_\lambda|+ |\mathfrak F_p|\leq c (|F_2|_{s}+|F_3|_{s}+|f_5|),
\end{equation}
\begin{equation}\label{krit13}
    |\mathcal M|+ |\mathcal L|+ |\mathcal P|\leq c\varepsilon_0 (|\boldsymbol\mu^*|_{s}+
    |\boldsymbol\lambda^*|_{s}).
\end{equation}
Moreover, estimate \eqref{nisa36overline} yields the inequality
\begin{equation}\label{krit12overline}
    |\overline{\mathfrak F_\beta}|\leq c|\overline{F}_2|.
\end{equation}
Combining estimates \eqref{krit9}, \eqref{krit10} and \eqref{krit12}, \eqref{krit13}
we  arrive at the inequality
\begin{equation*}\begin{split}
|\mathfrak F_\beta| +|\mathfrak F_\lambda|+|\mathfrak F_p|+|\mathcal M|+|\mathcal L| +|\mathcal P|
\leq
    c\varepsilon_0(|\overline{\boldsymbol\lambda}|+|\overline{\boldsymbol\chi}|+|p|)+\\
+c\varepsilon_0(|\boldsymbol\lambda^*|_{s}+|\boldsymbol\chi^*|_{s})+
c(|g_1^*|_{s}+|\mathbf g_2|_{s}).
\end{split}\end{equation*}
>From this and \eqref{krit29um} we obtain
\begin{equation*}\begin{split}
|\overline{\boldsymbol\lambda}|+|\overline{\boldsymbol\chi}|+|\overline{\boldsymbol\mu}|+|p|\leq
    c\varepsilon_0(|\overline{\boldsymbol\lambda}|+|\overline{\boldsymbol\chi}|+|p|)+\\
+c\varepsilon_0(|\boldsymbol\lambda^*|_{s}+|\boldsymbol\chi^*|_{s})+
c(|g_1^*|_{s}+|\mathbf g_2|_{s}).
\end{split}\end{equation*}

Choosing $\varepsilon_0$ sufficiently small we finally obtain
\begin{equation}\label{krit14}\begin{split}
|\overline{\boldsymbol\lambda}|+|\overline{\boldsymbol\chi}|+|\overline{\boldsymbol\mu}|+|p|\leq
c\varepsilon_0(|\boldsymbol\lambda^*|_{s}+|\boldsymbol\chi^*|_{s})+
c(|g_1^*|_{s}+|\mathbf g_2|_{s}).
\end{split}\end{equation}
Next, equations \eqref{krit1a}-\eqref{krit1c} imply the equalities
\begin{gather*}
\boldsymbol\partial\, \boldsymbol\mu^*\,=\,
-p\nabla w_1 -  \nabla g_1,\\
\boldsymbol\partial\, \boldsymbol\lambda^*\,+\, \boldsymbol\Omega\, \boldsymbol\lambda^*=\mathbf g_2-
\mathbf T\, \boldsymbol\mu^*-(\mathbf T\, \overline{\boldsymbol\mu}+p\mathbf W^\top \mathbf e_1)
\\
\boldsymbol\partial\, \boldsymbol\chi^*\,=(\mathbf S\boldsymbol\mu^*\, +\, \mathbf T^\top
\boldsymbol\lambda^*)+(\mathbf S\overline{\boldsymbol\mu}\, +\, \mathbf T^\top
\overline{\boldsymbol\lambda })
\end{gather*}
Applying Lemma \ref{nisa6}  to the first equation and noting that the embedding of every Sobolev space into $\mathcal A_{\sigma/2,0}$
is bounded we obtain
\begin{equation*}
    |\boldsymbol\mu^*|_{s+3n+4} \leq c(|p| +|g_1^*|_{s+4n+6})
\end{equation*}
Next, applying Lemma \ref{lita21lemma} to the second equation we obtain
\begin{equation*}
  |\boldsymbol\lambda^*|_{s+n+1}  \leq c |\mathbf g_2|_{s+3n+4}
  + c|\boldsymbol\mu^*|_{s+3n+4}+ c(|p|+|\overline{\boldsymbol\mu}).
\end{equation*}
Applying Lemma \ref{nisa6}  to the third equation we obtain
\begin{equation*}
 |\boldsymbol\chi^*|_{s}  \leq
   c(|\boldsymbol\mu^*|_{s+n+1}+|\boldsymbol\lambda^*|_{s+n+1})+ c(|p|+|\overline{\boldsymbol\mu}
  +|\overline{\boldsymbol\lambda}|)
\end{equation*}
Combining the obtained results we arrive at the inequality
\begin{equation}\label{krit15}
   |\boldsymbol\mu^*|_{s} + |\boldsymbol\lambda^*|_{s} + |\boldsymbol\chi^*|_{s}\leq c
   (|g_1^*|_{s+4n+6}+ |\mathbf g_2|_{s+4n+6})+c(|p|+|\overline{\boldsymbol\mu}
  +|\overline{\boldsymbol\lambda}|).
\end{equation}
Combining \eqref{krit14} and \eqref{krit15}, and choosing $\varepsilon_0$ sufficiently small  we obtain
the desired estimate
\begin{equation}\label{krit17}
    |\boldsymbol\mu|_{s} + |\boldsymbol\lambda|_{s} + |\boldsymbol\chi|_{s}+|p|\leq c
   (|g_1^*|_{s+4n+6}+ |\mathbf g_2|_{s+4n+6}),
\end{equation}
which gives estimate \eqref{figa7} in Theorem \ref{figa5}.
Next, estimates \eqref{krit9overline} and \eqref{krit12overline} yields
\begin{equation}\label{krit14overline}
   |\overline{\mathfrak F_\beta}|\leq   c \varepsilon_0|p|+|\overline{\mathbf g_2}|
\end{equation}
which along with \eqref{krit29i} and \eqref{krit13} implies
\begin{equation*}
    |\overline{\boldsymbol\mu}|\leq |\overline{\mathbf g_2}| +c\varepsilon_0(
    |p|+|\boldsymbol\mu^*|_{s}+
    |\boldsymbol\lambda^*|_{s})
\end{equation*}
Combining this result withy \eqref{krit17}, we finally obtain
\begin{equation*}
   |\overline{\boldsymbol\mu}|\leq |\overline{\mathbf g_2}|+c\varepsilon_0
   (|g_1^*|_{s+4n+6}+ |\mathbf g_2|_{s+4n+6}),
\end{equation*}
which gives estimate \eqref{figa7overline} in Theorem \ref{figa5}.
This completes the proof of Theorem \ref{figa5}.

\end{document}